\documentclass[a4paper, english]{amsart}
\usepackage[T1]{fontenc}
\usepackage[latin9]{inputenc}
\usepackage{color}
\usepackage[english]{babel}
\usepackage{amstext}
\usepackage{amsthm}
\usepackage{amsmath}
\usepackage{amssymb}
\usepackage{yfonts}

\usepackage{tikz}
\usetikzlibrary{
  decorations.pathmorphing,%
  decorations.markings,%
  decorations.pathreplacing,%
  decorations.text,%
  calc,%
  patterns,%
  shapes.geometric,%
  arrows,
  decorations.shapes,
  positioning,
  plotmarks,
  shadings,
  math,
  intersections
}
\tikzset{>=latex} 
\tikzset{font=\small}
\tikzset{mark size=1.5pt, mark options=thin}
\tikzset{pin distance=4pt,
  every pin edge/.style={<-, thin, shorten <= -2pt}}

\usepackage{graphicx}
\usepackage{subfigure}
\usepackage[unicode=true,pdfusetitle,
 bookmarks=true,bookmarksnumbered=false,bookmarksopen=false,
 breaklinks=false,pdfborder={0 0 0},pdfborderstyle={},backref=false,colorlinks=true]
 {hyperref}
\definecolor{myblue}{rgb}{0,0,0.6}
\hypersetup{colorlinks=true, linkcolor=myblue,citecolor=myblue,filecolor=myblue,urlcolor=myblue}

\usepackage[a4paper]{geometry}
\newcommand{\e}{\epsilon}
\newcommand{\Gsc}{{\Gamma_{D}}}
\newcommand{\N}[1]{\left\|#1\right\|}

\newcommand{\dist}{\operatorname{dist}}
\newcommand{\muin}{\mu^{\rm{in}}}
\newcommand{\muout}{\mu^{\rm{out}}}

\usepackage[show]{ed}
\newcommand{\Diff}{\operatorname{Diff}}

\makeatletter

\newcommand{\supp}{\operatorname{supp}}

\newcommand{\Ell}{\operatorname{Ell}}
\newcommand{\WF}{\operatorname{\WFh }}
\newcommand{\Id}{\operatorname{Id}}
\newcommand{\comp}{\operatorname{comp}}

\newtheorem{theorem}{Theorem}[section]
\newtheorem{cor}[theorem]{Corollary}
\newtheorem{definition}[theorem]{Definition}
\newtheorem{lem}[theorem]{Lemma}
\newtheorem{prop}[theorem]{Proposition}

\newtheorem{condition}[theorem]{Condition}
\newtheorem{assumption}[theorem]{Assumption}

\newtheorem*{claim*}{Claim}

\numberwithin{equation}{section}
\numberwithin{table}{section}
\numberwithin{figure}{section}

\newtheorem{rem}[theorem]{Remark}
\newtheorem*{rem*}{Remark}

\newtheorem{experiment}[theorem]{Experiment}

\newcommand{\tfa}{\text{ for all }}
\newcommand{\tfor}{\text{ for }}

\newcommand{\tif}{\text{ if }}
\newcommand{\ton}{\text{ on }}
\newcommand{\tas}{\text{ as }}
\newcommand{\tand}{\text{ and }}
\newcommand{\tst}{\text{ such that }}

\newcommand{\bre}{\begin{rem}}
\newcommand{\ere}{\end{rem}}
\newcommand{\bit}{\begin{itemize}}
\newcommand{\eit}{\end{itemize}}
\newcommand{\ben}{\begin{enumerate}}
\newcommand{\een}{\end{enumerate}}
\newcommand{\beq}{\begin{equation}}
\newcommand{\eeq}{\end{equation}}
\newcommand{\beqs}{\begin{equation*}}
\newcommand{\eeqs}{\end{equation*}}
\newcommand{\bpf}{\begin{proof}}
\newcommand{\epf}{\end{proof}}
\newcommand{\ble}{\begin{lem}}
\newcommand{\ele}{\end{lem}}

\definecolor{dalcol}{rgb}{0,0.3,0}
\newcommand{\dal}[1]{{\color{dalcol}{#1}}}

\definecolor{jecol}{rgb}{0.4,0,0}
\definecolor{escol}{rgb}{0,0,0.8}
\definecolor{estcol}{rgb}{0,0.5,0}
\definecolor{esnewcol}{rgb}{0,0.5,0}

\newcommand{\tendi}{\rightarrow \infty}
\newcommand{\tendo}{\rightarrow 0}

\usepackage{soul}

\newcommand{\PML}{{\rm pml}}

\newcommand{\Rimp}{R}
\newcommand{\hFEM}{{h_{\rm FEM}}}
\newcommand{\eps}{\varepsilon}
\newcommand{\half}{\frac{1}{2}}

\newcommand{\bcN}{\mathcal{N}}
\newcommand{\bcD}{\mathcal{D}}
\newcommand{\PadeP}{\mathcal{P}_{\MPade,\NPade}}
\newcommand{\PadeQ}{\mathcal{Q}_{\MPade,\NPade}}
\newcommand{\GammaI}{{\GammaIR}}
\newcommand{\GammaIinf}{\Gamma_{\rm{tr}}^\infty}
\newcommand{\GammaIR}{\Gamma_{{\rm tr}, R}}
\newcommand{\nudtr}{\nu_d^{\rm tr}}
\newcommand{\nujtr}{\nu_j^{\rm tr}}
\newcommand{\nuntr}{\nu_n^{\rm tr}}
\newcommand{\tr}{{\rm tr}}
\newcommand{\Rea}{\mathbb{R}}
\newcommand{\noi}{\noindent}
\newcommand{\cI}{\mathcal{I}}
\newcommand{\cH}{\mathcal{H}}
\newcommand{\cV}{\mathcal{V}}

\newcommand{\xiin}{\xi_1^{\rm in}}
\newcommand{\xiout}{\xi_1^{\rm out}}
\newcommand{\pin}{p^{\rm in}}
\newcommand{\pout}{p^{\rm out}}
\newcommand{\Bout}{\mathcal B^{\rm out}}
\newcommand{\Bin}{\mathcal B^{\rm in}}
\newcommand{\tin}{t^{\rm in}}
\newcommand{\tout}{t^{\rm out}}

\newcommand{\projx}{\pi_{{\mathbb{R}^d}}}
\newcommand{\projxM}{\pi_{M}}

\newcommand{\WFh}{\operatorname{WF}_{h}}

\newcommand{\vol}{\operatorname{vol}}
\newcommand{\dvol}{d\!\vol}
\newcommand{\boundary}{{\partial M}}

\newcommand{\mvanish}{{m_{\rm vanish}}}
\newcommand{\mmult}{{m_{\rm mult}}}
\newcommand{\morderzero}{{m_{\rm ord}}}

\newcommand{\alpharef}{\alpha^{\rm ref}}

\newcommand{\crayone}{{c_{\rm ray,1}}}
\newcommand{\craytwo}{{c_{\rm ray,2}}}
\newcommand{\craythree}{{c_{\rm ray,3}}}
\newcommand{\crayfour}{{c_{\rm ray,4}}}
\newcommand{\crayfive}{{c_{\rm ray,5}}}

\newcommand{\Cfrak}{C_{\mathfrak R}}
\newcommand{\Crefone}{C_{\rm ref}}

\newcommand{\domain}{\Omega_R}
\newcommand{\tildedomain}{{\widetilde{\Omega}_R}}
\newcommand{\domaintwo}{\Omega}
\newcommand{\Omegaminus}{\Omega_-}
\newcommand{\Omegaplus}{\Omega_+}

\newcommand{\MPade}{\mathsf{M}}%
\newcommand{\NPade}{\mathsf{N}}

\usepackage{scalerel,stackengine}
\stackMath
\newcommand\reallywidehat[1]{%
\savestack{\tmpbox}{\stretchto{%
  \scaleto{%
    \scalerel*[\widthof{\ensuremath{#1}}]{\kern-.6pt\bigwedge\kern-.6pt}%
    {\rule[-\textheight/2]{1ex}{\textheight}}
  }{\textheight}%
}{0.5ex}}%
\stackon[1pt]{#1}{\tmpbox}%
}
\newcommand{\mythmname}[1]{\emph{(#1.)}}

\begin{document}
\title[Local absorbing boundary conditions for high-frequency waves]{Local absorbing boundary conditions on fixed domains give order-one errors for high-frequency
waves}
\author{Jeffrey Galkowski} 
\address{Department of Mathematics, University College London, 25 Gordon Street, London, WC1H 0AY, UK}
\email{J.Galkowski@ucl.ac.uk}

\author{David Lafontaine} 
\address{Department of Mathematical Sciences, University of Bath, Bath, BA2 7AY, UK}
\email{D.Lafontaine@bath.ac.uk}

\author{Euan A.~Spence}
\address{Department of Mathematical Sciences, University of Bath, Bath, BA2 7AY, UK}
\email{E.A.Spence@bath.ac.uk}

\date{\today}
\maketitle

\begin{abstract}
We consider approximating the solution of the Helmholtz exterior Dirichlet problem for a nontrapping obstacle, with boundary data coming from plane-wave incidence, by the solution of the corresponding boundary value problem where the exterior domain is truncated and a local absorbing boundary condition coming from a Pad\'e approximation (of arbitrary order) of the Dirichlet-to-Neumann map is imposed on the artificial boundary (recall that the simplest such boundary condition is the impedance boundary condition).  
We prove upper- and lower-bounds on the relative error incurred by this approximation, both in the whole domain and in a fixed neighbourhood of the obstacle (i.e.~away from the artificial boundary).
Our bounds are valid for arbitrarily-high frequency, with the artificial boundary fixed, and show that the relative error is bounded away from zero, independent of the frequency, and regardless of the geometry of the artificial boundary.
\end{abstract}


\tableofcontents

\vspace{-1cm}

\section{Introduction and statement of the main results}\label{sec:intro}

\subsection{Informal discussion of the main results, their context, and their novelty}

\subsubsection*{Background on absorbing boundary conditions.}

Wave-scattering problems are usually posed in unbounded domains. However, when computing approximations to the solutions of such problems via discretisation methods in the domain, such as finite-element methods (as opposed to discretisation methods on the boundary such as boundary-element methods), an artificial boundary is introduced so that the computational domain is finite. The question then arises of what boundary condition to impose on this artificial boundary. 
If the exact Dirichlet-to-Neumann map for the domain exterior to the artificial boundary is used as the boundary condition, then the solution of the truncated problem is exactly the restriction to the truncated domain of the solution of the scattering problem. However, the Dirichlet-to-Neumann map is a nonlocal operator and is expensive to compute. 

Since the late 1970s, starting with the papers  \cite{Li:75,EnMa:77,EnMa:77a,EnMa:79,BaTu:80,BaGuTu:82}, 
there has been much research on designing \emph{local} boundary conditions to impose on the artificial boundary,
with these boundary conditions approximating the (nonlocal) Dirichlet-to-Neumann map.
Since the goal is for these boundary conditions to ``absorb'' waves hitting this boundary, and not reflect them back into the computational domain, they are often called ``absorbing'' or ``non-reflecting'' boundary conditions.
These boundary conditions are now standard tools in the numerical simulation of waves propagating in unbounded domains; see, e.g., the reviews \cite{Gi:91,Ha:97,Ts:98,Ha:99,Gi:04}, \cite[\S3.3]{Ih:98}.

\subsubsection*{The error incurred by absorbing boundary conditions}

The following natural and important question then arises:~what is the error between the solution of the truncated problem and the solution of the true scattering problem, and how does this error depend on the following factors?

\vspace{-1ex}

\bit
\item[(i)] The shape of the artificial boundary.
\item[(ii)] The distance of the artificial boundary from the scatterer.
\item[(iii)] The position in the computational domain where the error is measured (e.g., is the error smaller away from the artificial boundary than near it?).
\item[(iv)] \emph{Either} the time (for problems posed in the time domain) \emph{or} the frequency of the waves (for problems posed in the frequency domain).
\item[(v)] The order of the artificial boundary condition.
\eit

\vspace{-1ex}

Perhaps surprisingly, despite the decades-long interest in absorbing boundary conditions, there do not yet exist rigorous answers to many of these questions.

A summary of the existing answers to these questions is as follows: In the time domain, there exist error estimates describing how the error depends on the distance of the artificial boundary from the scatterer 
 \cite[Theorem 3.2]{BaTu:80}, \cite[Theorem 2.4]{DiJo:05}, on the order of the boundary conditions \cite[\S2.3]{Ha:97} (for fixed boundary), and on the average frequency present in the solution \cite[\S5]{HaRa:87}.
In the frequency domain for fixed frequency,
there exist error estimates describing how the error 
depends on the distance of the artificial boundary from the scatterer \cite[Theorems 4.1 and 4.2]{BaGuTu:82}, \cite[Theorem 3.1]{Go:82}.

\subsubsection*{The Helmholtz problem most studied by the numerical-analysis community: artificial boundary fixed and frequency arbitrarily high.}

One situation where, to our knowledge, there do not yet exist any estimates
on the error incurred by absorbing boundary conditions is in the frequency domain when the artificial boundary is fixed and the frequency is arbitrarily high. This situation is a ubiquitous model problem for numerical methods applied to the Helmholtz equation.

Indeed, the following is a non-exhaustive list of papers analysing numerical methods applied to this set up, with the analyses valid in the high-frequency limit with the domain fixed. We highlight that this list includes some of the most influential work in the numerical analysis of the Helmholtz equation from the last $\sim$15 years.\footnote{More specifically, all of the following papers consider \emph{either} the Helmholtz boundary-value problem \eqref{eq:BVPimp_new} below with the impedance boundary condition \eqref{eq:BVPimpbc_new} on the truncation boundary, \emph{or} the analogous boundary-value problem with variable coefficients in the PDE. 
}
\bit
\item 
Conforming FEMs (including continuous interior-penalty methods) 
\cite{ShWa:05, HaHu:08,MeSa:11,EsMe:12,ZhWu:13,Wu:13,EsMe:14,DuWu:15,ZhDu:15,DuZh:16,ChNi:18,BuNeOk:18,DiMoSp:19,Ch:19,GrSa:18,MeSaTo:20,ChNi:20,DuWuZh:20}.
\item Least-squares methods \cite{DeGoMuZi:12,ChQi:17,BeMe:18,HuSo:20,SoLe:20}.
\item DG methods based on piece-wise polynomials 
\cite{FeWu:09,FeWu:11,DeGoMuZi:12,FeXi:13,HoSh:13,MePaSa:13,CuZh:13,ChLuXu:13,MuWaYe:14,ZhDu:15,SaZe:15,WaWaZhZh:18,CaWu:20,ZhPaCh:20, ZhWu:21}.
\item Plane-wave/Trefftz-DG methods 
\cite{AmDjFa:09,HiMoPe:11,HiMoPe:14,AmChDiDjFi:14,HiMoPe:16,HuYu:18,MaPePi:19,YuHu:20,HuWa:20}.
\item Multiscale finite-element methods 
\cite{GaPe:15,BrGaPe:15,Pe:17,BaChGo:17,OhVe:18,ChVa:20,PeVe:20, MaAlSc:21, ChHoWa:21, HaPe:21, FrHaPe:21}.
\item Preconditioning methods \cite{GaGrSp:15,GrSpVa:17,GrSpZo:20,GoGrSp:20,LiXiSade:20,RaNa:20, GoGaGrLaSp:21, GoGrSp:21}.
\item Uncertainty-quantification methods \cite{FeLiLo:15,LiWaZh:18,GaKuSl:20}.
\eit

\subsubsection*{Informal summary of the results of this paper.}

The present paper proves error bounds on the accuracy of absorbing boundary conditions for 
the ubiquitous model problem discussed above. These bounds show how the error in this set up depends on each of the factors (i)-(v) described above, and all but one of our bounds are provably sharp.

More specifically, we consider the Helmholtz exterior Dirichlet problem 
with boundary data coming from plane-wave incidence when the artificial boundary is fixed and the frequency is arbitrarily high. 
We consider absorbing boundary conditions coming from a Pad\'e approximation (of arbitrary order) of the Dirichlet-to-Neumann map; recall that this popular class of boundary conditions was introduced in \cite{EnMa:77,EnMa:77a,EnMa:79} in the time-dependent setting.

These results are presented in \S\ref{sec:impedance} in the simplest-possible case of an impedance boundary condition, with these results illustrated in numerical experiments in \S\ref{sec:num}.
The results for the general Pad\'e case are presented  in \S\ref{subsec:error1} and \S\ref{subsec:error2}.
Our results about well-posedness of the truncated problem in \S\ref{sec:wellposed} are also new and of independent interest.
Of the results present in the existing literature, the results in this paper are closed to those of \cite{HaRa:87}, and we compare and contrast these two sets of results in \S\ref{sec:HaRa}.

\subsubsection*{How the results are obtained, and their novelty from the point of view of analysis.}

The main results are obtained using techniques from semiclassical analysis; i.e., rigorous analysis of PDEs with a large/small parameter, with the analysis explicit in that parameter. In this case the parameter is the large frequency of the Helmholtz equation.

More specifically we use \emph{semiclassical defect measures} \cite[Chapter 5]{Zworski_semi}, \cite[\S E.3]{DyZw:19}. These measures describe where the mass of Helmholtz solutions in phase space (i.e. the set of positions $x$ and momenta $\xi$) is concentrated in the high-frequency limit; for an informal discussion of Helmholtz defect measures, see \cite[\S9.1]{LaSpWu:19a}.  

The main novelty of this paper is in applying these semiclassical techniques to this long-standing numerical-analysis question of the accuracy of absorbing boundary conditions.
A large part of the analysis are delicate arguments  (in \S\ref{sec:mainproofs}) involving 
constructing geometric-optic rays and controlling their properties with respect to the distance of the artificial boundary from the scatterer, and the geometry of both the artificial boundary and the scatterer. Indeed, controlling the properties of these rays is what allows us to determine how the error depends on the factors (i)-(iii) above.
We highlight that the ideas behind the ray constructions are outlined in \S\ref{sec:idea_rays}, and their use in the defect-measure arguments is described informally in \S\ref{sec:outlinerays}.

In addition, the following two aspects of our paper are of independent interest in (non-numerical) analysis.
\bit
\item The arguments in \S\ref{sec:resolventestimates} that use defect measures to prove bounds on the solution operator 
over \emph{families} of domains (as opposed to a single one), with the bounds explicit in both frequency and the characteristic length scale of the domains.
\item The extension in \S\ref{sec:Miller} of the results in~\cite{Miller} about defect measures on the boundary to the case when the right-hand side of the Helmholtz equation is non-zero.
\eit

\subsubsection*{The wider context of absorbing boundary conditions in the numerical analysis of the Helmholtz equation}

Another important use of local absorbing boundary conditions in the numerical analysis of the Helmholtz equation is in domain-decomposition (DD) methods. 
This large interest began with the use of impedance boundary conditions for non-overlapping DD methods in \cite{De:91,BeDe:97} and the connection between absorbing boundary conditions and the optimal subdomain boundary conditions (involving appropriate Dirchlet-to-Neumann maps) 
was highlighted in 
\cite{NaRoSt:94,EnZh:98}.
Despite the large current interest in Helmholtz DD methods (see, e.g., the reviews in  \cite{GaZh:19}, \cite{GrSpZo:20}), there are no rigorous frequency-explicit convergence proofs for any practical DD method for the high-frequency Helmholtz equation, partly due to a lack of frequency-explicit bounds on the error when absorbing boundary conditions are used to approximate the appropriate Dirichlet-to-Neumann maps. 
We therefore expect the results and techniques in the present paper to be relevant for the frequency-explicit analysis of DD methods for the Helmholtz equation; preliminary results on this are given in \cite{LS1}.

\subsection{Overview of the main results in the simplest-possible setting}\label{sec:impedance}

In this section, we present a selection of our bounds on the error in their simplest-possible setting when an impedance boundary condition is imposed on the truncation boundary.
Our upper and lower bounds on the error when the absorbing boundary condition comes  from a general 
Pad\'e approximation of the Dirichlet-to-Neumann map are given in \S\ref{subsec:error1} and \ref{subsec:error2}, with results on the wellposedness of this problem in \S\ref{sec:wellposed}.

Let $\Omegaminus\subset\mathbb{R}^{d}$, $d\geq 2$, be a bounded open set such
that the open complement $\Omegaplus:=\mathbb{R}^{d} \setminus \overline{\Omegaminus}$
is connected, and let $\Gsc:=\partial\Omegaminus$ be $C^\infty$.
Given $k>0$ and $a\in\mathbb R ^d$ with $|a|=1$, let $u\in H_{\text{loc}}^{1}(\Omegaplus)$ be the solution to the
Helmholtz equation 
in $\Omegaplus$
\begin{subequations}\label{eq:BVP}
\begin{equation}
(\Delta +k^2) u=0\hspace{1em}\text{in }\Omegaplus,
\end{equation}
 with the Dirichlet boundary
condition 
\begin{equation}
u=\exp(ikx\cdot a ) \quad\text{ on } \Gsc
 \end{equation}
and satisfying the Sommerfeld radiation condition
\begin{equation}\label{eq:src}
\frac{\partial u}{\partial r} - i k u  = o \left( \frac{1}{r^{(d-1)/2}}\right)
\end{equation}
\end{subequations}
as $r:= |x|\rightarrow \infty$, uniformly in $\widehat{x}:= x/r$. (The technical reason we only consider Dirichlet boundary conditions on $\Gamma_D$, and not also Neumann boundary conditions, is discussed in Remark \ref{rem:Neumann} below.)

The physical interpretation of \eqref{eq:BVP} is that $u$ is minus the scattered wave for the plane-wave scattering problem with sound-soft boundary conditions; i.e., $\exp(i k x\cdot a) - u$ is the total field for the sound-soft scattering problem.

We assume throughout that the obstacle $\Omegaminus$ is \emph{nontrapping}, i.e.~all billiard trajectories 
(in the sense of \cite[\S24.3]{Ho:85}) starting in a neighbourhood of the convex hull of $\Omegaminus$ escape that neighbourhood after some uniform time. Without loss of generality, we assume that $\Omegaminus$ has characteristic length scale one (results explicit in the size of $\Omegaminus$ can then be obtained by a scaling argument). 
In principle, our arguments could also cover the case when the Helmholtz equation \eqref{eq:BVP} has variable coefficients, but the ray arguments would be more complicated, since the rays are no longer straight lines (at least in a neighbourhood of the scatterer).

Let $v$ be the solution of the analogous exterior Dirichlet problem, but with the exterior domain $\Omegaplus$ truncated, and an impedance boundary condition prescribed on the truncation boundary. 
More precisely, let $\tildedomain$ be such that $\tildedomain\subset B(0,MR)$ for some $M>0$, $\GammaIR:=\partial \tildedomain$
is $C^\infty$ and $\Omegaminus\Subset \tildedomain$, where $\Subset$ denotes compact containment. The subscripts $R$ 
on $\tildedomain$ and $\GammaIR$
emphasise that both have characteristic length scale $R$, and the subscript $\tr$ on $\GammaIR$ emphasises that this is the truncation boundary. We assume that the family $\{\GammaIR\}_{R\in[1,\infty)}$ is continuous in $R$ and is such that the limit $\GammaIinf:=\lim_{R\tendi}(\GammaIR/R)$ exists.
Let $\domain:= \tildedomain\setminus \overline{\Omegaminus}$, and let $v\in H^{1}(\domain)$ be the solution of
\begin{subequations}\label{eq:BVPimp_new}
\begin{align}
(\Delta+k^2)v=0\quad&\text{in }\domain,\label{eq:BVPimppde_new}\\
v=\exp(ikx\cdot a )
\quad&\text{on }\Gsc, \quad\text{ and}\\
\partial_n v- ik v =0
\quad&\text{on }\GammaIR.\label{eq:BVPimpbc_new}
\end{align}
\end{subequations}

\begin{theorem}[Lower and upper bounds when $\GammaI = \partial B (0,R)$]
\label{th:quant_new}
Suppose that  $\Omegaminus$ is nontrapping, $\Omegaminus \subset B(0,1)$, and $\GammaIR = \partial B (0,R)$ with $R \geq 1$. Then there exists $C_j=C_j(\Omegaminus)>0, j=1,2,$ such that for any $R\geq1 $, there exists $k_0(R, \Omegaminus)>0$ such that, for any direction $a$, the solutions to~\eqref{eq:BVP} and~\eqref{eq:BVPimp_new}, $u$ and $v$ respectively, satisfy
\beq\label{eq:mainbound_new}
\frac{C_1}{R^{2}}\leq 
\frac{\Vert u-v\Vert_{L^{2}(\domain)}}{\Vert u\Vert_{L^{2}(\domain)}}\leq \frac{C_2}{R^{2}},\qquad\tfa k \geq k_0.
\eeq
Furthermore, there exists $C_3=C_3(\Omegaminus)>0$ 
such that for any $R\geq 2$, there exists $k_1=k_1(R, \Omegaminus)>0$ such that, for any direction $a$, 
\beq\label{eq:subset_new}
\frac{\Vert u-v\Vert_{L^{2}(B(0,2)\setminus \Omega_-)}}{\Vert u\Vert_{L^{2}(B(0,2 )\setminus\Omega_-)}}\geq \frac{C_3}{R^{2}}\qquad\tfa k \geq k_1.
\eeq
\end{theorem}

\noindent Theorem~\ref{th:quant_new} shows that, for sufficiently high frequency, the error is proportional to $R^{-2}$ in \emph{both} the whole domain $\domain$ \eqref{eq:mainbound_new} \emph{and} a neighbourhood of the obstacle \eqref{eq:subset_new}.
 
We make two comments:
(i) The reason that $k_0$ and $k_1$ depends on $R$ is discussed below Theorem \ref{th:quant} (the more-general version of Theorem \ref{th:quant_new}).
(ii) When the impedance boundary condition is replaced by the more-general boundary condition corresponding to 
Pad\'e approximation, the only changes in \eqref{eq:mainbound_new} and \eqref{eq:subset_new} are in the powers of $R$ (see 
\eqref{eq:mainbound} and \eqref{eq:subset} below).

The following theorem shows that when $\GammaIinf$ 
is not a sphere centred at the origin, the relative error between $u$ and $v$ does not decrease with $R$. 

\begin{theorem}[Lower bound for generic $\GammaIR$]
\label{th:quant2_new}
Suppose that  $\Omegaminus$ is nontrapping, $\Omegaminus \subset B(0,1)$, and there exists $M>1$ such that 
\beqs
B(0,M^{-1}R) \subset \tildedomain \subset B(0,MR).
\eeqs
Assume that $\GammaIR$ is smooth and 
strictly convex and 
(i) $\GammaIinf$ is \emph{not} a sphere centred at the origin, and (ii) the convergence
$\GammaIR/R\rightarrow \GammaIinf$ is in $C^{0,1}$ globally and in $C^{1,\varepsilon}$ (for some $\varepsilon>0$) away from any corners of $\GammaIinf$.

Then there exists $C=C(\Omegaminus, \{\GammaIR\}_{R\in[1,\infty)})>0$ such that for all $R\geq 1$, there exists $k_0=k_0(R, \Omegaminus, \{\GammaIR\}_{R\in[1,\infty)})>0$ such that, for any direction $a$, the solutions
 to~\eqref{eq:BVP} and~\eqref{eq:BVPimp_new}, $u$ and $v$ respectively, satisfy
\beq\label{eq:general_lower_new}
\frac{\Vert u-v\Vert_{L^{2}(\domain)}}{\Vert u\Vert_{L^{2}(\domain)}} \geq C
\qquad\tfa k \geq k_0.
\eeq
\end{theorem}

\bre
We highlight that the constant $C$ in Theorem \ref{th:quant2_new} depends on the family $\{\GammaIR\}_{R\in[1,\infty)}$ (indexed by $R$), but is independent of the variable $R$ itself. This also applies in Theorems \ref{th:uniformestimates}, \ref{th:quant2}, and \ref{th:quant3} below.
\ere

We make four comments:
(i) Even under the more-general boundary condition corresponding to Pad\'e approximation, the lower bound analogous to \eqref{eq:general_lower_new} is still independent of $R$; see Theorem \ref{th:quant2} below.
(ii) The numerical experiments in \S\ref{sec:num} indicate that $k_0$ in Theorem \ref{th:quant2_new} is independent of $R$, and in fact a lower bound holds uniformly in $k$ and $R$; see Tables \ref{tab:square1} and \ref{tab:square2}.
(iii) Under further smoothness assumption on $\GammaIinf$, Theorem \ref{th:quant3} proves an upper bound on the relative error. 
(iv) The reason why the error decreases with $R$ when $\GammaIR= \partial B(0,R)$, but is independent of $R$ for generic $\GammaIR$ is explained 
in the text immediately after the statement of Theorem \ref{th:quant3}.

\subsection{Definitions of the boundary conditions corresponding to Pad\'e approximation of the Dirichlet-to-Neumann map}
\label{sec:setup}

We now consider a more-general truncated problem than \eqref{eq:BVPimp_new}. With $\Omegaminus, \tildedomain,$ and $\domain$ as in \S\ref{sec:impedance}, let $v\in H^{1}(\domain)$ be the solution of
\begin{subequations}\label{eq:BVPimp}
\begin{align}
(\Delta+k^2)v=0\quad&\text{in }\domain,\label{eq:BVPimppde}\\
v=\exp(ikx\cdot a )
\quad&\text{on }\Gsc, \quad\text{ and}\\
\bcN(k^{-1}\partial_n v)- i \bcD(v) =0
\quad&\text{on }\GammaIR.\label{eq:BVPimpbc}
\end{align}
\end{subequations}
where 
$\bcN\in \Psi^{2\NPade}(\GammaI)$, $\bcD\in \Psi^{2\MPade}(\GammaI)$ (i.e.~$\bcN$ and $\bcD$ are semiclassical pseudodifferential operators on $\GammaI$ of order $2\NPade $ and $2\MPade $ respectively) and both have real-valued principal symbols (see \S\ref{sec:appendix} for background material on semiclassical pseudodifferential operators).

While most of our analysis applies to much more general choices of $\bcN$ and $\bcD$, we focus on $\bcN$ and $\bcD$ corresponding to a Pad\'e approximation (up to terms that are lower order both in $k^{-1}$ and differentiation order) of the principal symbol of the Dirichlet-to-Neumann map; this class of $\bcN$ and $\bcD$ 
was introduced in \cite{EnMa:77,EnMa:77a,EnMa:79} in the time-dependent setting.
In the following assumption, $\Diff^{m}$ denotes the set of operators of the form
\beqs
A(x,k^{-1}D)=\sum_{j=0}^{m}a_j(x)\big(k^{-1}D\big)^j,
\eeqs
with $a_j \in C^\infty$, $D=-i\partial$,
Furthermore, we use Fermi normal coordinates $x=(x_1,x'), \xi=(\xi_1,\xi')$, with $\GammaI = \{x_1=0\}$, $x_1$ the signed distance to $\GammaIR$, $\partial_{x'}$, and $\partial_{x_1}$ orthogonal. We also let $r(x',\xi')$ denote the symbol of one plus the tangential Laplacian on $\GammaI$, i.e.
\beq\label{eq:r}
r(x',\xi'):=1 - |\xi'|_g^2
\eeq
where $|\cdot|_g$ is the norm induced on the co-tangent space (i.e.~the space of the Fourier variables $\xi'$ corresponding to the tangential variables $x'$) of $\GammaI$ from $\mathbb{R}^d$; see \S\ref{subsec:geo} for more details.

Let the coefficients  $(p_{\MPade,\NPade}^j)_{j=0}^{\MPade}$ and
$(q_{\MPade,\NPade}^j)_{j=1}^{\NPade}$ be defined so that
 $p(t)/q(t)$ is the Pad\'e approximant of of type $[\MPade,\NPade]$ at $t=0$ to $\sqrt{1-t}$, where
\beq\label{eq:pq}
p(t)=\sum_{j=0}^{\MPade} p_{\MPade,\NPade}^j t^j\quad \tand \quad q(t)= \sum_{j=0}^{\NPade} q_{\MPade,\NPade}^j t^j
\eeq
with $q_{\MPade,\NPade}^0=1$ and $p_{\MPade,\NPade}^{\MPade}, q_{\MPade,\NPade}^{\NPade}\neq 0$. This definition implies that
\beq\label{eq:Padeorder}
\sqrt{1-t} - \left(\sum_{j=0}^{\MPade} p_{\MPade,\NPade}^j t^j\right)\left(1+ \sum_{j=1}^{\NPade} q_{\MPade,\NPade}^j t^j\right)^{-1}= O\big( t^{\morderzero}\big)\quad\tas t\tendo
\eeq
where 
\beqs
\morderzero \geq \MPade+\NPade+1.
\eeqs

\begin{assumption}[Boundary condition corresponding to Pad\'e approximation]\label{ass:Pade}
We assume that
\beqs
\bcD
- \PadeP\big(x',k^{-1}D_{x'}\big) \in k^{-1}\Diff^{2\MPade -1},\qquad
\bcN
 - \PadeQ\big(x',k^{-1}D_{x'}\big)  \in k^{-1}\Diff^{2\NPade -1},
\eeqs
where 
\beqs
\PadeP(x',\xi') := \sum_{j=0}^{\MPade} p_{\MPade,\NPade}^j\big(1-r(x',\xi')\big)^j
\,\,\tand\,\,
 \PadeQ(x',\xi') := 1+\sum_{j=1}^{\NPade} q_{\MPade,\NPade}^j \big(1-r(x',\xi')\big)^j.
\eeqs
\end{assumption}

By \eqref{eq:r}, $\PadeP$ and $\PadeQ$ involve powers of $|\xi'|^2_g$. Since $|\xi'|_g^2$ is a quadratic form in the variables $\xi'$, the boundary condition \eqref{eq:BVPimpbc} involves differential operators, and is thus local.

Recall that the rationale behind these particular $\bcD$ and $\bcN$ consists of the following three points.

(i) The ideal condition to impose on $\GammaI$ is that the Neumann trace, $\partial_n v$, equals the Dirichlet-to-Neumann map for the exterior of $\tildedomain$ under the Sommerfeld radiation condition~\eqref{eq:src} applied to the Dirichlet trace, $v$ (see \S\ref{sec:applied2} and the references therein).

(ii) When $\tildedomain$ is strictly convex,
 the principal symbol of this Dirichlet-to-Neumann map (as a semiclassical pseudodifferential operator), away from glancing rays, i.e.~rays that are tangent to the boundary, equals $\sqrt{r(x',\xi')}=\sqrt{1 - |\xi'|_g^2}$; see Remark \ref{rem:applied1} for more details.

(iii) The definitions of $p(t)$ and $q(t)$ \eqref{eq:pq} imply that if $\bcD$ and $\bcN$ satisfy Assumption \ref{ass:Pade}, then the boundary condition \eqref{eq:BVPimpbc} corresponds to approximating 
$\sqrt{1 - |\xi'|_g^2}$ 
by the Pad\'e approximant of type $[\MPade,\NPade]$ at $|\xi'|^2_g=0$, i.e. at rays that are normal to the boundary.

The polynomials $p(t)$ and $q(t)$ are constructed based on their properties at $t=0$. However, the quantity $q(t) \sqrt{1-t}- p(t)$ can have other zeros in $t\in (0,1]$, which corresponds to the boundary condition \eqref{eq:BVPimpbc} not reflecting certain non-normal rays. 
We record here notation used later in the paper for these other zeros. Given $\MPade,\NPade$, let $\{t_j\}_{j=1}^{\mvanish}$ be the zeros of $q(t) \sqrt{1-t}- p(t)$ in $t\in (0,1]$ where $p(t)$ and $q(t)$ are defined by \eqref{eq:pq}. 
Then  $\mvanish<\infty$ since $q(t) \sqrt{1-t}- p(t)$ is analytic on $(-1,1)$, continuous at $1$, and $p(1)\neq 0$ (see Lemma \ref{lem:TH} below). Let $\mmult$ be the highest multiplicity of the zeros $\{t_j\}_{j=1}^{\mvanish}$. 

When $\bcN=\bcD=I$, \eqref{eq:BVPimpbc} is the impedance boundary condition 
\beq\label{eq:impedance1}
\partial_n v - i k v=0,
\eeq
and is covered by Assumption \ref{ass:Pade} with $\MPade=\NPade=0$, i.e. $p(t)=q(t)=1$. In this case, $\mvanish=0$, since $\sqrt{1-t}-1$ has no zeros for $t\in (0,1]$.

\subsection{Wellposedness of the truncated problem and $k$-explicit bound on its solution}\label{sec:wellposed}

\begin{theorem}\label{th:uniformestimates}
Let $\Omegaminus\Subset B(0,1)$ be a non-trapping obstacle, $M>0$, $\tildedomain\subset B(0,MR)$ be convex 
with smooth boundary $\GammaIR$ that is nowhere flat to infinite order 
and such that 
$\GammaIR/R\to \GammaIinf$ in $C^\infty$. 
Let $\bcN$ and $\bcD$ satisfy Assumption \ref{ass:Pade} with \emph{either} $\MPade=\NPade$ \emph{or} $\MPade=\NPade+1$.

There exists $C>0$ such that given $R\geq 1$, there exists $k_0=k_0(R)>0$ such that, 
given $f \in L^2(\domain)$, $g_D \in H^1(\Gamma_D)$, and $g_I\in L^2(\GammaI)$, 
if $k\geq k_0$, then 
the solution $v\in H^1(\domain)$ of
\begin{subequations}\label{eq:BVPimpdata}
\begin{align}
(\Delta+k^2)v=f\quad&\text{in }\domain,\label{eq:BVPimpdata2}\\
v=g_D\quad&\text{on }\Gsc, \quad\text{ and}\\
\bcN(k^{-1}\partial_n v) -i \bcD
(v) =g_I
\quad&\text{on }\GammaI
\end{align}
\end{subequations}
 exists, is unique, and satisfies 
\begin{align}\nonumber
&\N{\nabla v}_{L^2(\domain)} + k \N{v}_{L^2(\domain)} \\
&\qquad\leq C \Big(
R\N{f}_{L^2(\domain)} + R^{1/2}\big( \N{\nabla_{\Gamma_D}g_D}_{L^2(\Gamma_D)} + k \N{g_D}_{L^2(\Gamma_D)}\big) + R^{1/2} k\N{g_I}_{L^2(\GammaIR)}
\Big).
\label{eq:GenImpResolvent}
\end{align}
\end{theorem}

\noindent Note that results analogous to the wellposedness statement in Theorem~\ref{th:uniformestimates}  
in the time domain are given in \cite[Theorem 4]{TrHa:86}, \cite[Theorem 1]{EnMa:79} for problems where the spatial domain is a half-plane.

Because of the importance of the truncated problem in numerical analysis, proving bounds analogous to \eqref{eq:GenImpResolvent} when $v$ satisfies the impedance boundary condition 
\beq\label{eq:impedance}
\partial_n v- i kv=g_I \quad\ton \GammaIR
\eeq
has been the subject of many investigations in the literature, including 
\cite[\S8.1]{Me:95},  \cite{CuFe:06,He:07,BaYuZh:12,LiMaSu:13}, \cite[Remark 4.7]{MoSp:14}, \cite[\S2.1]{Ch:15},  \cite{BaYu:16,BaSpWu:16}, \cite[Appendix B]{ChNi:18}, \cite{SaTo:18}, \cite[Appendix A]{GrPeSp:18}, \cite{GrSa:18}. Indeed,
the bound \eqref{eq:GenImpResolvent} under the boundary condition \eqref{eq:impedance} and various assumptions on $\Omegaminus$ and $\tildedomain$
(often for star-shaped $\Omegaminus$ and $\tildedomain$ and sometimes with $\Omegaminus=\emptyset$) 
in \cite[Proposition 8.1.4]{Me:95},  \cite[Theorem 1]{CuFe:06}, \cite[Proposition 3.3]{He:07}, 
\cite[Theorem 1.8]{BaSpWu:16}, \cite[\S2.1.5]{Ch:15}, \cite[Theorem 22]{SaTo:18}, \cite[\S A.2]{GrPeSp:18}, \cite[Theorems 3.2 and 5.10]{GrSa:18} (with the last four references considering the variable-coefficient Helmholtz equation). 

To our knowledge, the  bound \eqref{eq:GenImpResolvent}, however, is the first $k$-explicit bound for a truncated Helmholtz problem where a local absorbing boundary condition is posed other than the impedance boundary condition \eqref{eq:impedance}.

\subsection{Bounds on the relative error in $\domain$}\label{subsec:error1}

All the results in this section proved under the assumption that $\bcN$ and $\bcD$ satisfy Assumption \ref{ass:1} with \emph{either} $\MPade=\NPade$ \emph{or} $\MPade=\NPade+1$, so the the truncated problem is wellposed by Theorem \ref{th:uniformestimates}.

\begin{theorem}[Lower bound for general strictly-convex $\GammaI$]
\label{th:lower}
If  $\Omegaminus$ is nontrapping and $\GammaI $ is strictly convex, then there exists $C=C(\domain, \MPade,\NPade)>0$ that depends continuously on $R$ and $k_0 = k_0(R, \domain, \MPade,\NPade)>0$, such that, for any direction $a$,
\beqs
\frac{\Vert u-v\Vert_{L^{2}(\domain)}}{\Vert u\Vert_{L^{2}(\domain)}}\geq C\qquad\tfa k\geq k_0.
\eeqs
\end{theorem}

The following three results prove bounds on the relative error that are explicit in $R$. 
Theorem \ref{th:quant} considers the case $\GammaIR = \partial B (0,R)$, and Theorems \ref{th:quant2} and \ref{th:quant3} consider the case when $\GammaIR/R$ tends to a limiting object that is not a sphere.

\begin{theorem}[Quantitative lower and upper bounds when $\GammaI = \partial B (0,R)$]
\label{th:quant}
Suppose that  $\Omegaminus$ is nontrapping, $\Omegaminus \subset B(0,1)$, and $\GammaIR = \partial B (0,R)$ with $R \geq 1$. Then, there exists $C_j=C_j(\Omegaminus, \MPade,\NPade )>0, j=1,2,$ such that for any $R\geq1 $, there exists $k_0(R, \Omegaminus, \MPade,\NPade)>0$ such that, for any direction $a$, 
\beq\label{eq:mainbound}
\frac{C_1}{R^{2\morderzero}}\leq 
\frac{\Vert u-v\Vert_{L^{2}(\domain)}}{\Vert u\Vert_{L^{2}(\domain)}}\leq \frac{C_2}{R^{2\morderzero}},\qquad\tfa k \geq k_0.
\eeq
\end{theorem}

The reason that $k_0$ 
in Theorem \ref{th:quant} depends on $R$ is because of the difference between the limits $k\tendi$ with $R$ fixed and $R\tendi$ with $k$ fixed. To illustrate this difference, consider the boundary conditions
\beq\label{eq:twobcs}
(\partial_n - i k) v= 0 \quad\tand \quad \left(\partial_n - i k+ \frac{d-1}{2r}\right) v= 0. 
\eeq
Both satisfy Assumption \ref{ass:Pade} with $\MPade=\NPade=0$, with, respectively $\bcN=1$, $\bcD=1$ and $\bcN=1, \bcD=1-k^{-1}i(d-1)(2r)^{-1}$.
Therefore, in both cases the error $\|u-v\|_{L^2(\domain)}/\|u\|_{L^2(\domain)} \sim R^{-2}$ for fixed $R$ as $k\tendi$ by Theorem \ref{th:quant}. However, for fixed $k$ as $r:=|x|\tendi$,
\beq\label{eq:error1}
(\partial_n - i k)(u-v)(x)= (\partial_n - i k)u(x) = O(r^{-(d+1)/2})_{L^\infty}=O(r^{-1})_{L^2(\partial B(0,r))}
\eeq
and
\begin{align}
\left(\partial_n - i k+ \frac{d-1}{2r}\right)(u-v)(x)= \left(\partial_n - i k+ \frac{d-1}{2r}\right)u(x) &= O(r^{-(d+3)/2})_{L^\infty}
=O(r^{-2})_{L^2(\partial B(0,r))}.
\label{eq:error2}
\end{align}
The fact that the right-hand sides of \eqref{eq:error1} and \eqref{eq:error2} are different shows that, while
the behaviour of $u-v$ for the two boundary conditions in \eqref{eq:twobcs} is the same as $k\tendi$ with $R$ fixed by Theorem \ref{th:quant}, the behaviour as $R\tendi$ with $k$ fixed is different. 
We expect that the bounds in this paper -- for the limit $k\tendi$ with $R$ fixed -- in fact hold uniformly when $R\ll k^\gamma$ for some $\gamma<1$ (i.e., when the large parameter $R$ is smaller than the large parameter $k$).

When the limiting object  $\GammaIinf$ is not a sphere, 
the lower and upper bounds are given separately in Theorems \ref{th:quant2} and \ref{th:quant3}, respectively. This is because 
the lower bound allows the limiting object to, e.g., have corners, whereas the upper bound requires the limiting object to be smooth.

\begin{theorem}[Quantitative lower bound for generic $\GammaIR$]
\label{th:quant2}
Suppose that  $\Omegaminus$ is nontrapping, $\Omegaminus \subset B(0,1)$, and there exists $M>0$ such that 
\beq\label{eq:OmegaR}
B(0,M^{-1}R) \subset \tildedomain \subset B(0,MR).
\eeq
Assume that $\GammaIR$ is smooth and strictly convex and such that (i) $\GammaIinf$ is \emph{not} a sphere centred at the origin, and (ii) the convergence
$\GammaIR/R\rightarrow \GammaIinf$ is in $C^{0,1}$ globally and in $C^{1,\varepsilon}$ (for some $\varepsilon>0$) away from any corners of $\GammaIinf$.

Then there exists $C=C(\Omegaminus, \MPade,\NPade ,\{\GammaIR\}_{R\in[1,\infty)})>0$ 
such that for all $R\geq 1$, there exists $k_0=k_0(R, \Omegaminus, \MPade,\NPade ,\{\GammaIR\}_{R\in[1,\infty)})>0$ such that, for any direction $a$,
\beqs
\frac{\Vert u-v\Vert_{L^{2}(\domain)}}{\Vert u\Vert_{L^{2}(\domain)}} \geq C
\qquad\tfa k \geq k_0.
\eeqs
\end{theorem}

\begin{theorem}[Quantitative upper bound for generic $\GammaIR$]
\label{th:quant3}
Suppose that  $\Omegaminus$ is nontrapping with $\Omegaminus \subset B(0,1)$. Suppose that, for every $R\geq 1$,  $ \tildedomain \subset B(0,MR)$,
$\GammaIR$ is smooth, convex, and nowhere flat to infinite order, and $(\GammaIR/R)\to \GammaIinf$  in $C^\infty$ as $R\tendi$. 
Then there exists $C=C(\Omegaminus, 
\MPade,\NPade , \{\GammaIR\}_{R\in[1,\infty)})>0$ such that for any $R\geq1 $, there exists $k_0 =
k_0(R, \Omegaminus, $ $\MPade, \NPade,  \{\GammaIR\}_{R\in[1,\infty)})>0$ such that for any $a\in\mathbb R^d$,
\beqs
\frac{\Vert u-v\Vert_{L^{2}(\domain)}}{\Vert u\Vert_{L^{2}(\domain)}}\leq C
\qquad\tfa k \geq k_0.
\eeqs
\end{theorem}

We now explain why the constants in the upper and lower bounds in Theorems \ref{th:lower}-\ref{th:quant3} decrease with $R$ when $\GammaIR=\partial B(0,R)$, but are independent of $R$ for generic $\GammaIR$.
Recall from \S\ref{sec:setup} that the boundary condition \eqref{eq:BVPimpbc} corresponds to approximating $\sqrt{r(x',\xi')}$ by a
Pad\'e approximant in $|\xi'|^2_g$,  with this approximation valid to order $\morderzero$ in $|\xi'|^2_g$ at $\xi'=0$ (i.e., rays hitting $\GammaI$ in the normal direction) by \eqref{eq:Padeorder};
recall also that there exists finitely-many other values of $|\xi'|^2_g$ such that $\PadeQ(x',\xi')\sqrt{r(x',\xi')} - \PadeP(x',\xi') =0$, which corresponds to there being finitely-many non-normal angles such that rays hitting $\GammaIR$ at these angles are not reflected by $\GammaIR$.
When $\GammaIR=\partial B(0,R)$ and $R$ is large, the rays originating from $\Omegaminus$ hit $\GammaIR$ in a direction whose angle with the normal decreases  with increasing $R$ (in fact the angle $< R^{-1}$; see Lemma \ref{lem:ballangle} below). Thus, if $R$ is sufficiently large, the finitely-many special non-normal angles are avoided, and the error for large $k$ decreases with $R$, with the accuracy depending on $\morderzero$; see Theorem \ref{th:quant}.
When $\GammaIinf$ is not a sphere centred at the origin, 
for every incident direction there exist
rays hitting $\GammaIinf$ at a fixed, non-normal angle that is also not one of the finitely-many special non-normal angles 
(see Lemma \ref{lem:ray3} below).
Since the Dirichlet-to-Neumann map is not approximated by the boundary condition \eqref{eq:BVPimpbc} at such an angle, the error is therefore independent of $R$ and $\morderzero$; see Theorems \ref{th:quant2} and \ref{th:quant3}.

\subsection{Bounds on the relative error in subsets of $\domain$}\label{subsec:error2}

Given the upper and lower bounds on the error in Theorems \ref{th:lower}-\ref{th:quant3}, a natural question is: is the error any smaller in a neighbourhood of the obstacle (i.e.~away from the artificial boundary)?

We focus on the case when \emph{either} $\GammaIR =\partial B(0,R)$ \emph{or} $\GammaIR$ is the boundary of a hypercube with smoothed corners.
We do this because the artificial boundaries most commonly used in applications are $\GammaIR =\partial B(0,R)$ and $\GammaIR$ is a hypercube, but in the latter case we need to smooth the corners for technical reasons.

\begin{theorem}[Quantitative lower bound on subset of $\domain$ when $\GammaIR= \partial B(0,R)$]
\label{th:local_ball}
Suppose that  $\Omegaminus$ is nontrapping, $\Omegaminus \subset B(0,1)$, and $\GammaIR = \partial B (0,R)$ with $R \geq 1$. Then, there exists $C=C(\Omegaminus, \MPade,\NPade )>0$ 
and $R_0= R_0(\MPade,\NPade )\geq 2$ 
such that for any $R\geq R_0$, there exists $k_0=k_0(R, \Omegaminus, \MPade,\NPade)>0$ such that, for any direction $a$, 
\beq\label{eq:subset}
\frac{\Vert u-v\Vert_{L^{2}(B(0,2)\setminus\Omegaminus)}}{\Vert u\Vert_{L^{2}(B(0,2 )\setminus\Omegaminus)}}\geq \frac{C}{R^{2\morderzero}}\qquad\tfa k \geq k_0.
\eeq
Furthermore, if $\MPade=\NPade=0$, then $R_0=2$.
\end{theorem}

That is, when $\GammaIR= \partial B(0,R)$, the error in $B(0,2)$ is bounded below, independently of $k$, and the lower bound has the same dependence on $R$ as for the error in $\domain$ (see Theorem \ref{th:quant}). The fact that we have explicit information about $R_0$ when $\MPade=\NPade=0$ is because in this case $\mvanish=0$, i.e.~there are no non-normal angles for which the reflection coefficient vanishes, and the proof is simpler.

\begin{theorem}\mythmname{Quantitative lower bound on subset of $\domain$ when $\GammaIR$ is the boundary of a smoothed hypercube}
\label{th:local_square} 
Suppose that  $\Omegaminus$ is nontrapping and $\Omegaminus \subset B(0,1)$. 
Let $\mathfrak C$ be the set of corners of $ [- R/2, R/2]^d$ and, given $\epsilon>0$, let 
\beqs
\mathfrak C_\epsilon := \bigcup_{x \in \mathfrak C} B(x, \epsilon);
\eeqs
i.e.~$\mathfrak C_\epsilon$ is a neighbourhood of the corners.
Then, there exists $C=C(\Omegaminus, \MPade,\NPade )>0$, and $\epsilon_0 = \epsilon_0(\Omegaminus)$ such that, for any $R\geq 4$, if $\GammaIR$ is smooth and
\beqs
\GammaIR \setminus \mathfrak C_\epsilon = \left[- \frac R2, \frac R2\right]^d \setminus \mathfrak C_\epsilon \quad\tfor 0<\epsilon\leq \epsilon_0,
\eeqs
then there exists $k_0=k_0(\GammaIR,\Omegaminus, \MPade,\NPade)>0$ such that, for any direction $a$, 
\beqs
\frac{\Vert u-v\Vert_{L^{2}(B(0,2)\setminus\Omegaminus)}}{\Vert u\Vert_{L^{2}(B(0,2)\setminus\Omegaminus)}}\geq \frac{C}{R^{(d-1)/2}},\hspace{1em}\quad\tfa k \geq k_0.
\eeqs
\end{theorem}

That is, when $\GammaIR$ is a smoothed hypercube, the error in $B(0,2)$ is bounded below independently of $k$, in a similar way to the error in $\domain$ (see Theorem \ref{th:quant2}). However, whereas the lower bound in Theorem \ref{th:quant2} is independent of $R$, Theorem \ref{th:local_square} allows for the possibility that the large-$k$-limit of the error in $B(0,2)\setminus\Omegaminus$ decreases with $R$.

\bre[Smoothness of boundaries]
Theorems \ref{th:lower}, \ref{th:quant}, \ref{th:quant2}, \ref{th:quant3}, and \ref{th:uniformestimates} are proved under the assumptions that $\Gamma_D$ and $\GammaIR$ are $C^\infty$, with Theorem \ref{th:uniformestimates} also assuming that $\GammaIinf$ is $C^\infty$.
In all these proofs one actually requires that these boundaries are $C^m$ for some unspecified $m$. One could in principle go through the arguments in the present paper, and those in \cite{Miller} about defect measures on the boundary (which we adapt in \S\ref{sec:defect}), to determine the smallest $m$ such that the results hold, but we have not done this. 
\ere

\subsection{Numerical experiments in 2-d illustrating some of the main results}\label{sec:num}

These numerical experiments all consider the simplest boundary condition satisfying Assumption \ref{ass:Pade}, i.e.~the impedance boundary condition $\partial_n v -i k v=0$, which is covered by Assumption \ref{ass:Pade} with $\bcN=\bcD=1$.

We first describe the set up common to Experiments \ref{exp:ball}, \ref{exp:butterfly}, and \ref{exp:trapping}. The set up for Experiments \ref{exp:square1} and \ref{exp:square2} is slightly different, and is described at the beginning of Experiment \ref{exp:square1}.

\noi \emph{The absorbing boundary condition.}
We let $\GammaI = \partial B(0,  \Rimp )$, for some specified $ \Rimp >0$, $d=2$, $\bcN=\bcD=1$ in \eqref{eq:BVPimpbc}; therefore $\MPade=\NPade=0$, $\morderzero=1$, and $\mvanish=0$.

\noi \emph{The PML solution used as a proxy for the exact solution.}
As a proxy for the solution $u$ to \eqref{eq:BVP}, we use $u_{\PML}$ defined to be the solution of the boundary value problem analogous to \eqref{eq:BVP} but truncated with a radial PML in an annular region $B(0, R_{\PML})\setminus B(0, \Rimp )$, with $R_{\PML}> \Rimp$, using the particular PML described in  \cite[\S3]{CoMo:98a}.
The one change from \cite[\S3]{CoMo:98a} is that we take the scaled variable to be independent of $k$, i.e.,
\beqs
\widetilde{\rho} = \rho + i\int_{\Rimp}^\rho  \sigma(s) \,d s, \quad\rho > \Rimp,
\eeqs
(compare to \cite[Second displayed equation on Page 2067]{CoMo:98a} where $\Rimp$ in our notation is $a$ in their notation). We choose 
\beq\label{eq:sigma}
\sigma(s) = (r- \Rimp)^2 /(R_{\PML}- \Rimp)^2.
\eeq
With this set up, the error between the PML solution and the solution to \eqref{eq:BVP} decreases exponentially with $k$ by \cite[Theorem 1.2]{GLS2}. (Note, in particular, that $\sigma = f_\theta'$ in the notation of \cite{GLS2} and thus the choice of $\sigma$ \eqref{eq:sigma} satisfies the regularity assumptions in \cite{GLS2} -- indeed, this particular $\sigma$ is given as an example in \cite[Equation 1.7]{GLS2}.)
The width of the PML, $R_{\PML}-\Rimp$ is chosen as a constant independent of $k$ (specified in each experiment) which is always larger than the largest wavelength considered.

\noi \emph{The FEM approximation space.}
The boundary value problems for $u_{\PML}$ and $v$ are discretised using the finite element method with P4 elements (i.e.~conforming piecewise polynomials of degree $4$) and implemented in FreeFEM++ \cite{He:12}. 
The finite-element approximations to $u_\PML$ and $v$ are denoted by $u_{\PML,\hFEM}$ and $v_{\hFEM}$ respectively, and the same mesh is used inside $\Omega_R$ when computing both.
We then compute the \emph{relative error}
\beq\label{eq:rel_error}
\frac{
\N{u_{\PML,\hFEM}- v_{\hFEM}}_{L^2(\domain)}
}{
\N{u_{\PML,\hFEM}}_{L^2(\domain)}
}.
\eeq
using an element-wise quadrature rule. In the figures we plot the total fields corresponding to $u_{\PML,\hFEM}$ and $v_{\hFEM}$, i.e.~$\exp(i k x\cdot a) - u_{\PML,\hFEM}$ and $\exp(i k x\cdot a) - v_{\hFEM}$ respectively; this is because the total field is easier to interpret than the scattered fields.

\noi \emph{Ensuring accuracy of the FEM solutions.
The relative $H^1$ errors in the FEM approximations of $u_{\PML}$ and $v$ are both controllably small, uniformly in $k$, if 
\beq\label{eq:meshcondition}
k\Rimp (\hFEM k)^{2p}= C
\eeq
for some $C>0$, independent of all parameters, where $p$ is the polynomial degree and $\hFEM$ is the meshwidth. This is proved in \cite[Theorems 4.9 and 5.3]{GS3}, following earlier results in  \cite[Theorem 5.1 and Corollary 5.2]{DuWu:15} for the impedance problem with no scatterer and  \cite[Theorem 4.4]{LiWu:19} for the PML problem with no scatterer and $p=1$ (see also \cite{ChGaNiTo:22} for related results).
Although $p=4$, we choose $\hFEM$ to satisfy \eqref{eq:meshcondition} with $p=3$. This choice ensures that the FEM error \emph{decreases} as $k\to \infty$, and thus 
the difference between $u_{\PML,\hFEM}-v_{\hFEM}$ and $u_{\PML}-v$ decreases as $k\to \infty$.
We choose $C>0$ (depending on $R$ and $p$) such that when $k=20$, $hk=1$ (i.e., there are $2\pi$ points per wavelength at $k=20$).
We use triangular elements, and thus there is a variational crime caused by approximating the curved boundaries $\GammaI$ and $\partial \Omega_-$; empirically this error is controlled if $hk$ is sufficiently small, and thus decreases as $k\to\infty$ under the meshthreshold \eqref{eq:meshcondition}.
The linear systems are solved using preconditioned GMRES, using the package ``ffddm'' with tolerance $10^{-6}$ and the preconditioner ORAS (Optimized Restricted Additive Schwarz), as described in \cite{ffddm:20}. 
}

\begin{experiment}[Scattering by ball, verifying Theorems \ref{th:quant_new}/\ref{th:quant}]\label{exp:ball}
We choose $\Gamma_D= \partial B(0,1)$, $ \Rimp =2$, $R_\PML = 2+0.5$, and $a= (1,0)$ (i.e.~the plane wave is incident from the left).
Figure \ref{fig:ball} shows the real parts of the total fields
\beq\label{eq:total_fields}
\Re\big(\exp(i k x\cdot a) - u_{\PML,\hFEM}),
\quad
\Re\big(\exp(i k x\cdot a) - v_{\PML,\hFEM}),
\quad\tand\quad
\Re\big(u_{\PML,\hFEM} - v_{\PML,\hFEM}).
\eeq
at $k=40$. 
We see the error is largest in the shadow of the scatterer near $\Gamma_D$.

Table \ref{tab:ball} then shows the relative error (defined by \eqref{eq:rel_error}) for increasing $k$ for $ \Rimp =2, 4, 8$. 
The errors in Table \ref{tab:ball} are constant for $ \Rimp $ fixed as $k$ increases, in agreement with Theorems \ref{th:quant_new}/\ref{th:quant}. 
The errors for $\Rimp=4$ are roughly 4 to 4.5 times smaller than the errors for $\Rimp=2$, and the errors for $\Rimp=8$ are roughly 4 times smaller than the errors for $\Rimp=4$.
Since $\morderzero=1$, the factor $R^{-2\morderzero}= R^{-2}$ in the  bound \eqref{eq:mainbound} means that we expect 
the error for $R=4$ to be 4 times smaller than that for $R=2$, and the error for $R=8$ to be 4 times smaller than that for $R=4$, 
at least when $k\geq k_0(R)$ (with $k_0(R)$ the unspecified constant in Theorems \ref{th:quant_new}/\ref{th:quant}).
\end{experiment}

\begin{figure}
\subfigure[Real part of total field corresponding to $u_{\PML,\hFEM}$]
  {
    \includegraphics[width=0.47\textwidth]{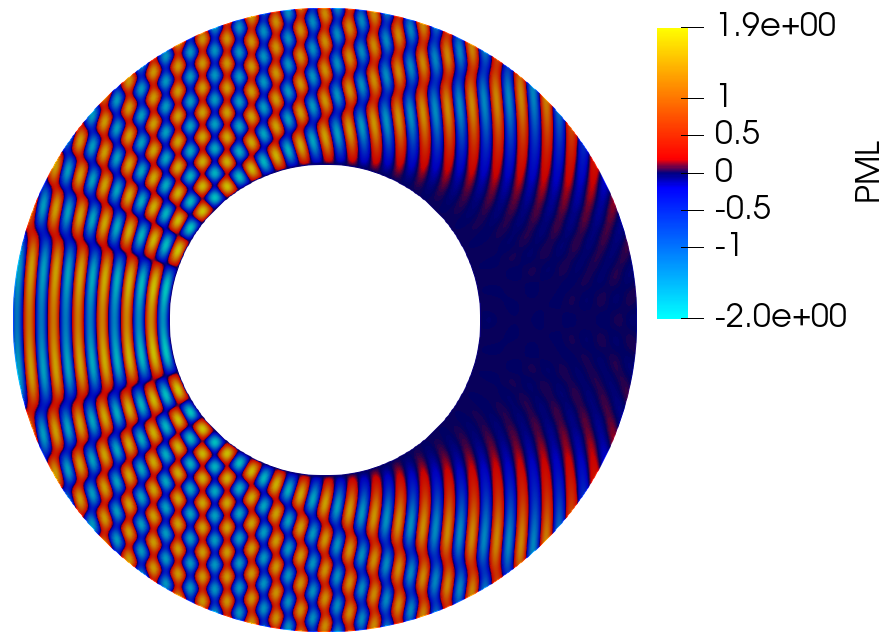}
  }
\subfigure[Real part of total field corresponding to $v_{\hFEM}$]
  {
    \includegraphics[width=0.47\textwidth]{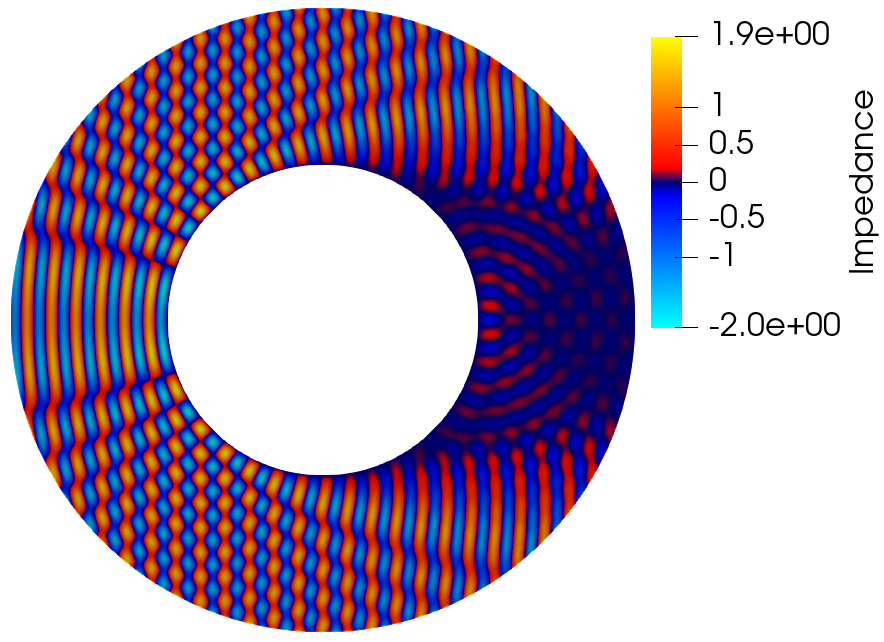}
  }
\subfigure[Real part of error $u_{\PML,\hFEM}- v_{\hFEM}$]
  {
    \includegraphics[width=0.47\textwidth]{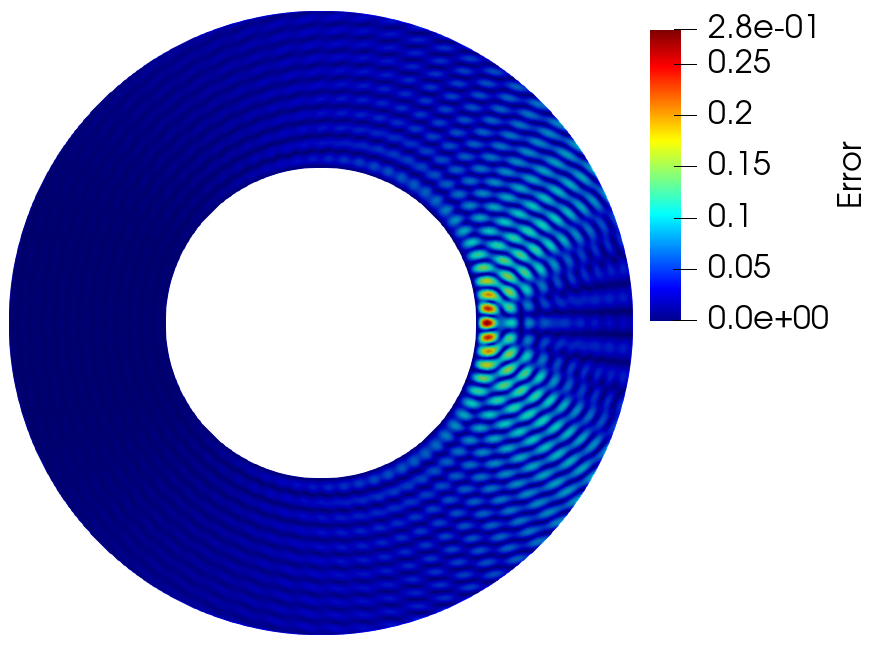}
  }
\caption{Scattering by a unit ball for $k=40$ (as described in Experiment \ref{exp:ball})}\label{fig:ball}
\end{figure}


\begin{table}[h!]
\begin{tabular}{c|c|c|c}
$k$ & Relative error for ball $\Rimp =2$ & Relative error for ball $\Rimp =4$ & Relative error for ball $\Rimp =8$ \\
\hline
20 &0.052557755 &0.012321440&0.0035458354\\
40 &0.050360302& 0.011438903&0.0029200006\\
80 &0.050034175&0.011050890&\\
160 &0.049001901  &
\end{tabular}
\caption{The relative error \eqref{eq:rel_error} against $k$ for scattering by a ball (described in Experiment \ref{exp:ball}) for two different values of $\Rimp$.}\label{tab:ball}
\end{table}

\begin{figure}
\subfigure[Real part of total field corresponding to $u_{\PML,\hFEM}$]
  {
    \includegraphics[width=0.47\textwidth]{
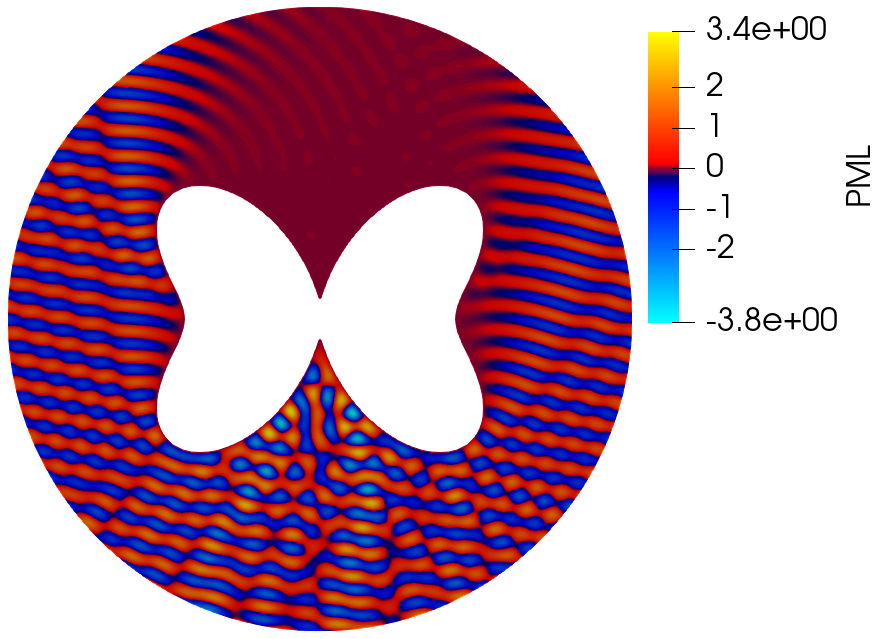}
  }
\subfigure[Real part of total field corresponding to $v_{\hFEM}$]
  {
    \includegraphics[width=0.47\textwidth]{
    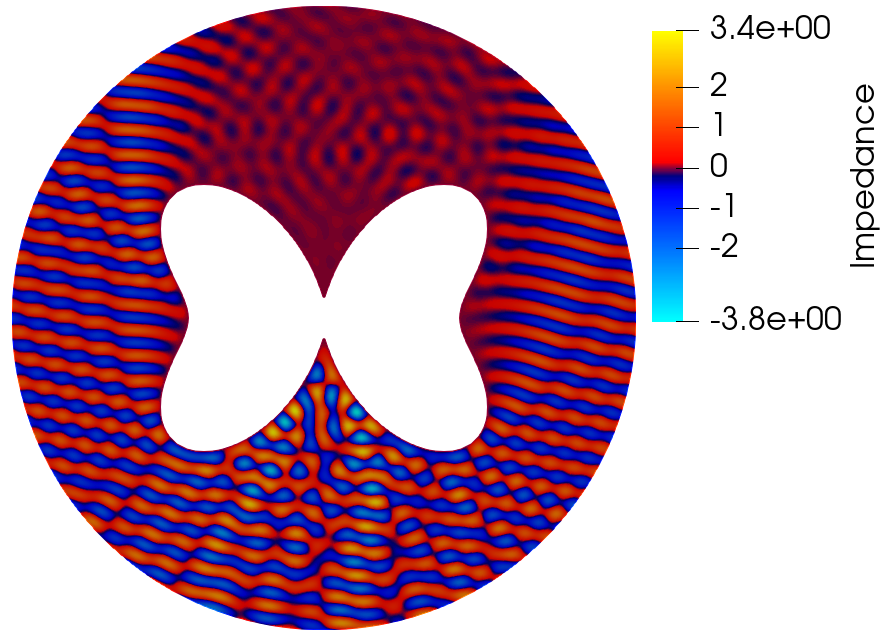}
  }
\subfigure[Real part of error $u_{\PML,\hFEM}- v_{\hFEM}$]
  {
    \includegraphics[width=0.47\textwidth]{
    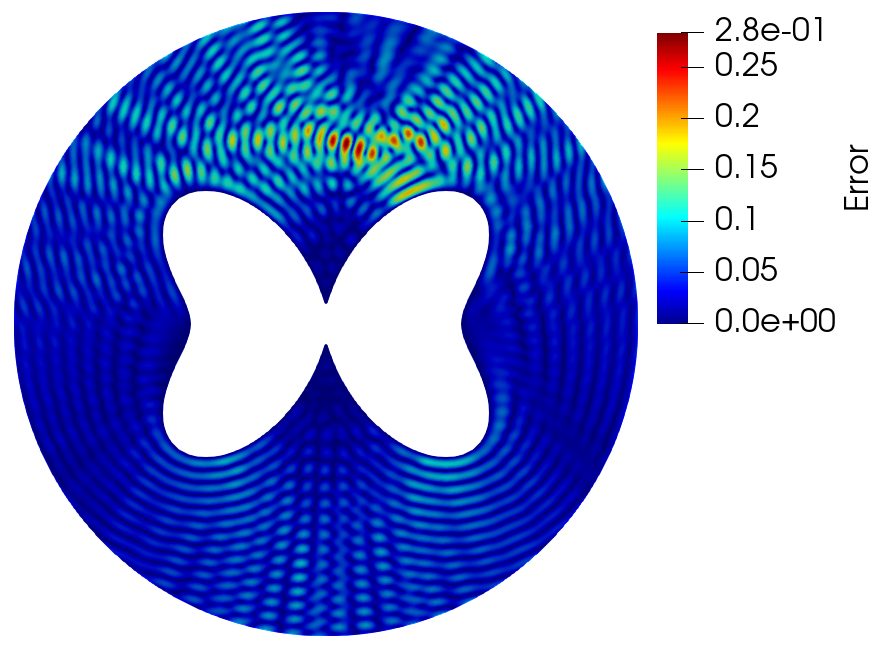}
  }
\caption{Scattering by a butterfly-shaped obstacle for $k=40$ (as described in Experiment \ref{exp:butterfly})}\label{fig:bean}
\end{figure}

\begin{experiment}[Scattering by a butterfly-shaped obstacle, verifying Theorems \ref{th:quant_new}/\ref{th:quant}]\label{exp:butterfly}
We choose
$\Gamma_D$ to be the curve defined in polar coordinates by 
\beqs
\Gamma_D :=\Big\{ (r,\theta)\, :\, r=(0.3+\sin^2(\theta))(1.4 \cos(2\theta)+1.5), \theta\in [0,2\pi)\, \Big\}
\eeqs
$ \Rimp =2$, and $R_\PML = 2+0.5$. We consider the two different incident plane waves corresponding to 
$a = (\cos(7\pi/16), \sin(7\pi/16))$ and $a = (\cos(\pi/16), \sin(\pi/16))$.

Figure \ref{fig:bean} shows the real parts of the total fields \eqref{eq:total_fields} at $k=40$ with $a = (\cos(7\pi/16), \sin(7\pi/16))$, computed with $p=2$ and $\hFEM= (2\pi/5) k^{-1-1/4}$.
In this case, the error is large in the shadow of the scatterer not only near $\Gamma_D$, but also away from the obstacle. The choice 
$a = (\cos(\pi/16), \sin(\pi/16))$ gives a qualitatively similar picture.

Table \ref{tab:butterfly} shows the relative error (defined by \eqref{eq:rel_error}) for this set up for increasing $k$ and the two different incident plane waves. 
For each $a$, the error is constant as $k$ increases, again in agreement with Theorems \ref{th:quant_new}/\ref{th:quant}.
While the errors depend on $a$, the results are consistent with the statement in Theorems \ref{th:quant_new}/\ref{th:quant} that the error can be bounded, from above and below, uniformly in $a$.
\end{experiment}

\begin{table}[h]
\begin{tabular}{c|c|c}
$k$ &  Relative error, incident angle $7\pi/16$ & Relative error, incident angle $\pi/16$ \\
\hline
20 &  0.066501411 & 0.060746128\\
40 & 0.063926342 &  0.061104428\\
80 & 0.063212656 & 0.058719452
\end{tabular}
\caption{The relative error \eqref{eq:rel_error} against $k$ for scattering by a butterfly-shaped obstacle (described in Experiment \ref{exp:butterfly}) and two different incident plane waves.}\label{tab:butterfly}
\end{table}

\begin{experiment}[Trapping created by the impedance boundary]\label{exp:trapping}
We choose $ \Rimp =2$, $R_\PML = 2+0.5$, $k=50$, $a = (10/\sqrt{104},2/\sqrt{104})$, and $\Omegaminus$ the polygon connecting the points
$(0.5,0.125)$, $(0.5,0.5)$, $(-0.5, 0.5)$, $(-0.5,-0.5)$, $(0.8,-0.5)$, $(0.8, -0.125)$, $(0.55, -0.125)$, $(0.55, -0.375)$,
$(-0.375, -0.375)$, $(-0.375, 0.375)$, $(0.25, 0.375)$, $(0.25, 0.125).$
The total fields are plotted in Figure \ref{fig:trapping}.

This set up is not included in Theorems \ref{th:quant_new}/\ref{th:quant}, since $\Omegaminus$ is trapping. However we include this experiment to show that artificial reflections from the impedance boundary $\Gamma_I$ can excite trapped waves not present in the PML solution
 (as long as the incident angle is chosen in a careful way depending on $\Omegaminus$, $k$, and the position of $\Gamma_I$).
\end{experiment}

\begin{figure}
\subfigure[Real part of total field corresponding to $u_{\PML,\hFEM}$]
  {
    \includegraphics[width=0.48\textwidth]{
    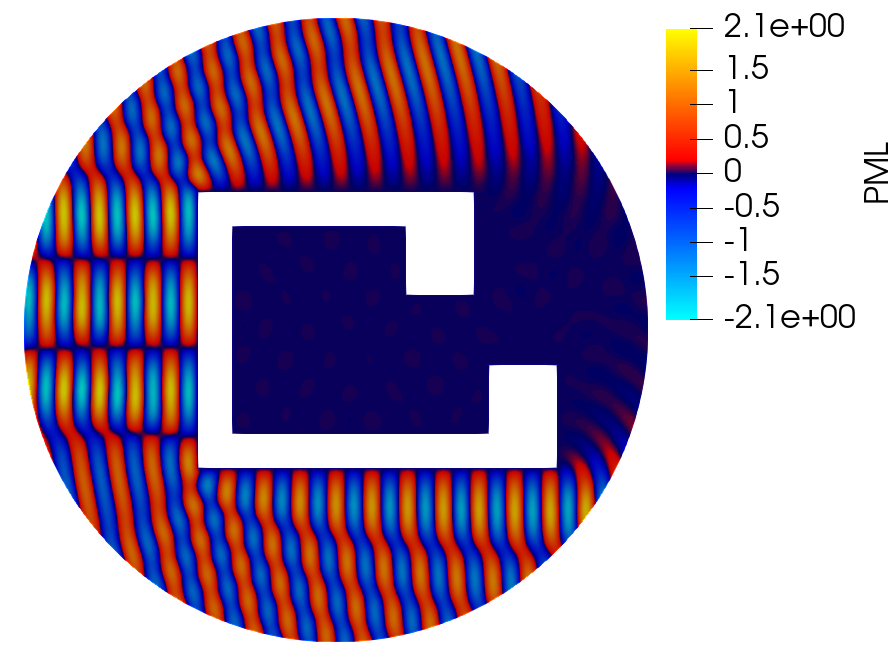}
  }
\subfigure[Real part of total field corresponding to $v_{\hFEM}$]
  {
    \includegraphics[width=0.48\textwidth]{
    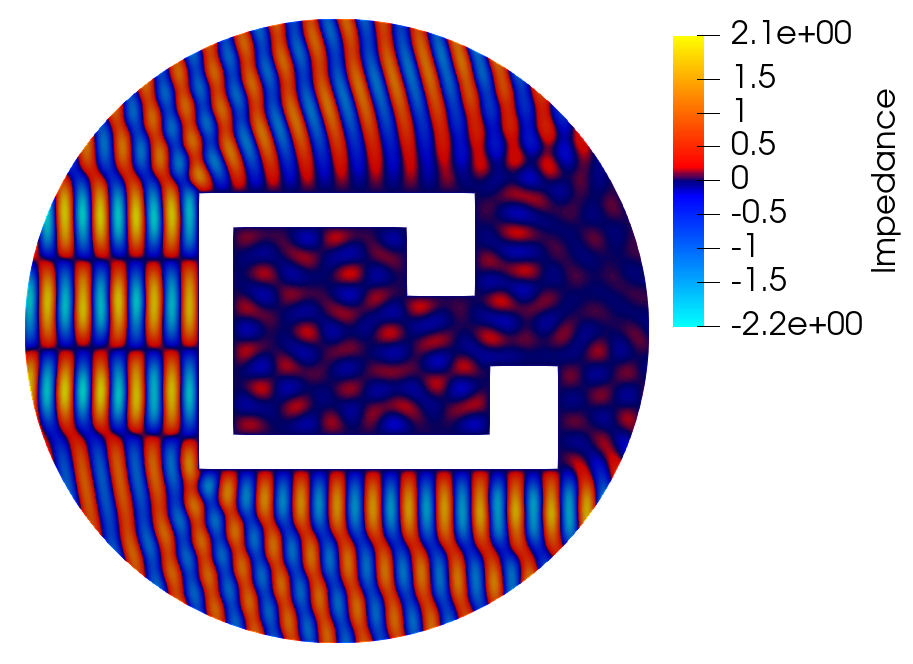}
  }
\subfigure[Real part of error $u_{\PML,\hFEM}- v_{\hFEM}$]
  {
    \includegraphics[width=0.48\textwidth]{
    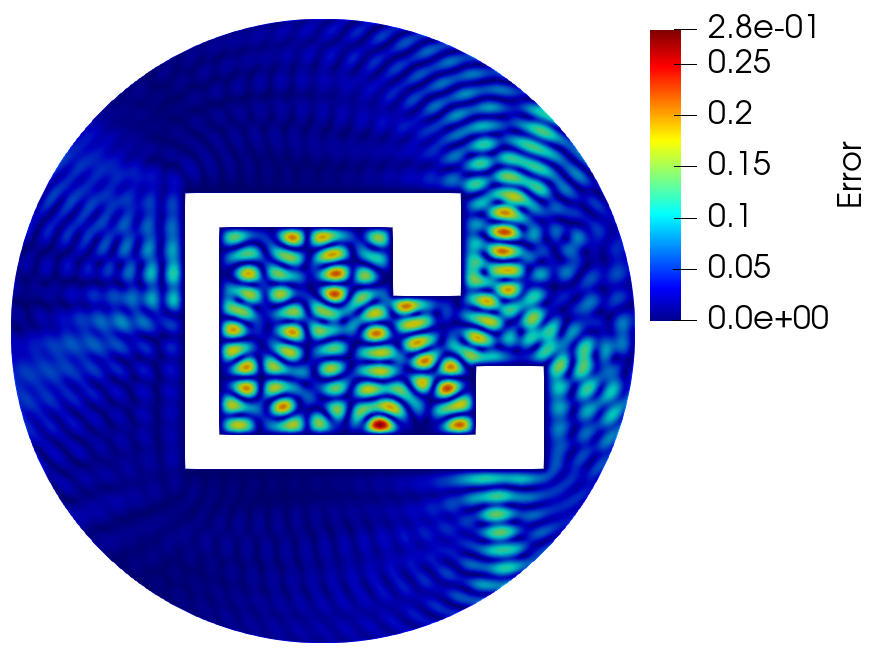}
  }
\caption{Scattering by a trapping obstacle for $k=50$ (as described in Experiment \ref{exp:trapping})}\label{fig:trapping}
\end{figure}

\begin{experiment}[Square $\Gamma_I$, investigating accuracy for increasing $k$ with $\Gamma_I$ fixed]\label{exp:square1}
Both this experiment and Experiment \ref{exp:square2} investigate the effect of a non-circular impedance boundary. $\Gamma_I$ is the square of side length $2R_{\rm square}$ centred at the origin. We still compute our proxy for $u$ using a radial PML, posing the boundary-value problem for $u_{\PML}$ on $B(0,3R_{\rm square}/2)$, with the PML region being $B(0,3R_{\rm square}/2+ 1/2)\setminus B(0,3R_{\rm square}/2)$. Observe that $\Gamma_I \subset B(0,3R_{\rm square}/2)$, and so $\Gamma_I$ is a fixed distance away from the PML region.
We choose $\Omegaminus= B(0,1)$, 
$R_{\rm square}=2, 4, 8 $ (observe that $\Gamma_D$ is then inside $\Gamma_I$ -- as required), and incident direction $a= (\cos \pi/8, \sin \pi/8)$. 
Table \ref{tab:square1} then shows the relative error for increasing $k$.

When $\GammaI= \partial B(0,R)$, Table \ref{tab:ball} showed the error decreasing by roughly a factor of $4$ as $R$ doubled. In Table \ref{tab:square1} we see very different behaviour: going from $R_{\rm square}=2$ to $R_{\rm square} =4$ the error decreases by less than a factor of $2$, and going from $R_{\rm square}=4$ to $R_{\rm square} =8$ the error does not decrease. Although this experiment is not covered by Theorems \ref{th:quant2_new}/\ref{th:quant2}, since the theorem requires $\GammaI$ to be smooth, the behaviour of the error is consistent with the main result of Theorems \ref{th:quant2_new}/\ref{th:quant2}, namely that when $\GammaIinf:= \lim_{R\tendi}(\GammaI /R)$ is not a ball centred at the origin, the relative error is bounded above and below, \emph{independent of $R$}, as $k$ increases.
\end{experiment}

\begin{table}[h]
\begin{tabular}{c|c|c|c}
$k$ &  Relative error for $R_{\rm square}=2$ & Relative error for $R_{\rm square}=4$ & Relative error for $R_{\rm square}=8$ \\
\hline
20 & 0.0832432
 & 
0.0582767
 & 0.0529081\\ 
40 & 0.0802578
&  0.0578435
&  0.0528049\\ 
80 & 0.0772090
& &
\end{tabular}
\caption{The relative error \eqref{eq:rel_error} against $k$ for scattering by the ball of radius $1$ with $\Gamma_I$ a square of side length $2R_{\rm square}$ centred at the origin and incident angle $\pi/8$ (described in Experiment \ref{exp:square1}).}\label{tab:square1}
\end{table}

\begin{experiment}[Square $\GammaI$,  investigating accuracy for increasing ${\rm dist}(\Gamma_I,0)$ with $k$ fixed]\label{exp:square2}
We now investigate the error when $\GammaI$ is a square as $R_{\rm square}$ increases with $k$ fixed. This situation is not covered by any of Theorems \ref{th:lower}-\ref{th:quant3}. However, we include this experiment since its results, along with those in Experiment \ref{exp:square1}, indicate that the lower bound in Theorems \ref{th:quant2_new}/\ref{th:quant2}
holds uniformly in $R$ and $k$.


To investigate the case when $R_{\rm square}$ increases with $k$ fixed, we consider an equivalent problem when $R_{\rm square}$ is fixed, $k$ increases, and the obstacle diameter decreases like $1/k$. 
The set up is as in Experiment \ref{exp:square1} with $R_{\rm square}=2$ (so the PML region is $B(0,3.5)\setminus B(0,3)$), 
$\Omegaminus= B(0, 10/k)$ (so that we need $k>5$ for $\Gamma_D$ to be inside $\Gamma_I$), and the incident direction $a=\pi/8$. 
Figure \ref{fig:square2} plots the relative error for this set up with $k=30$, and
Table \ref{tab:square2} displays the relative error \eqref{eq:rel_error} for $k=
20,40,80$. 
This set up is equivalent to $\Omegaminus= B(0,1)$, $k=10$, and $R_{\rm square}=
4,8,16$, and Table \ref{tab:square2} is labelled with these parameters. 

The fact that the last three entries of Table \ref{tab:square2} and the last entries in the second and third columns of Table \ref{tab:square1} are all around $0.05$ suggests that some value near $0.05$ is a lower bound on the relative error in \emph{both} the limit $k\tendi$ with $R_{\rm square}$ fixed \emph{and}  the limit $R_{\rm square}\tendi$ with $k$ fixed.
\end{experiment}

\begin{table}[h]
\begin{tabular}{c|c}
$R_{\rm square}$ &  Relative error \\
\hline
4 & 0.0593483\\
8 & 0.0532721\\
16 & 0.0515247\\
\end{tabular}
\caption{The relative error \eqref{eq:rel_error} against $R_{\rm square}$ for scattering by the ball of radius $1$ with $\Gamma_I$ a square of side length $2R_{\rm square}$ centred at the origin, $k=10$, and incident 
direction $a=(\cos(\pi/8), \sin(\pi/8))$ (described in Experiment \ref{exp:square2}).}
\label{tab:square2}
\end{table}

\begin{figure}
  {
    \includegraphics[width=0.48\textwidth]{
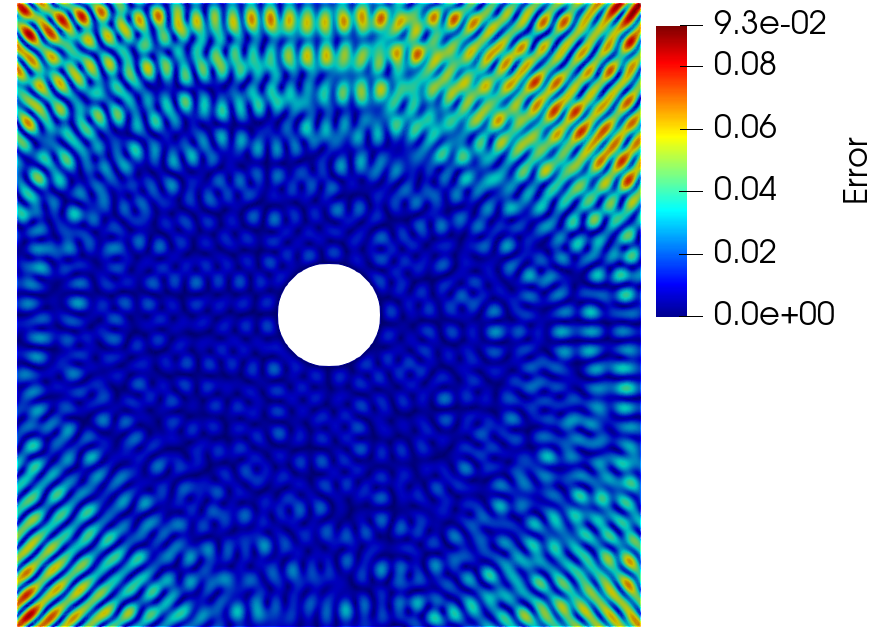}
  }

\caption{Real part of error $u_{\PML,\hFEM}- v_{\hFEM}$ for scattering by a circle radius one with $k=10$, 
$\Gamma_I$ the square of side length $12$ centred at the origin, and incident direction $a=(\cos(\pi/8), \sin(\pi/8))$ (as described in Experiment \ref{exp:square2})}\label{fig:square2}
\end{figure}

\subsection{Comparison to the results of \cite{HaRa:87}}\label{sec:HaRa}
Out of the existing results on absorbing boundary conditions in the literature, the closest to those in the present paper are in \cite{HaRa:87}. 
Indeed, \cite{HaRa:87} used microlocal methods to study the time-domain analogue of the problems \eqref{eq:BVP}/\eqref{eq:BVPimp} when $\Omega_-=\emptyset$ (i.e.,~no obstacle), and proved a bound on the error between the solutions of the analogues of \eqref{eq:BVP}/\eqref{eq:BVPimp} at an arbitrary time.

While the results of the present paper also use microlocal methods (using defect measures instead of propagation of singularities used in \cite{HaRa:87}), differences between the results of the present paper and the results of \cite{HaRa:87} are the following.
\bit
\item The constants in the main error bound in  \cite{HaRa:87} (\cite[Equation 5.1]{HaRa:87}) depend in an unspecified way on time. The results of the present paper hold uniformly for high-frequency in the frequency domain, which is analogous to proving results for arbitrarily-long times in the time domain.
\item The constants in the main error bound in \cite{HaRa:87} are not explicit in the distance of the artificial boundary from the origin. In contrast, the error bounds in Theorems \ref{th:quant2}-\ref{th:local_square} are explicit in $R$.
\item \cite{HaRa:87} does not have to deal with glancings rays because it assumes that (i) $\Omega_-=\emptyset$ and (ii) the data is supported away from the artificial boundary. In contrast, (i) we allow the obstacle $\Omega_-$ to be non-empty and have tangent points, and so have to deal with glancing here, and (ii) we allow $f$ in \eqref{eq:BVPimpdata2} to have support up to the boundary $\GammaIR$ (as is needed to use the bound \eqref{eq:GenImpResolvent} in, e.g., the analysis of finite-element methods); therefore a large part of the analysis in \S\ref{sec:resolventestimates} takes place at glancing.
\eit

\subsection{Outline of paper}

\S\ref{sec:defect} contains results about semiclassical defect measures of Helmholtz solutions, with these results used in proofs of both the upper and lower bounds in Theorems \ref{th:lower}-\ref{th:local_square}.

\S\ref{sec:outgoing} proves three results about outgoing solutions of the Helmholtz equation (i.e., solutions satisfying the Sommerfeld radiation condition \eqref{eq:src}), Lemmas \ref{lem:mass_u}, \ref{lem:impedancetrace}, and \ref{lem:impedancetrace2}, with the first used in the proof of the lower bounds, and the last two used in the proof of the upper bounds.

\S\ref{sec:resolventestimates} proves Theorem \ref{th:uniformestimates} (the wellposedness result). Important ingredients for this proof are the trace bounds of Theorem \ref{th:higherOrderBound}; since the proofs of these are long and technical, they are postponed to \S\ref{sec:traceHO}.

\S\ref{sec:mainproofs} proves Theorems \ref{th:lower}-\ref{th:local_square}. The upper bounds follow immediately from Theorem \ref{th:uniformestimates} and Lemma \ref{lem:impedancetrace}. However, the lower bounds require showing that there exist rays, created by the incident plane wave, that reflect off $\Gamma_D$ and hit $\GammaIR$ at an angle for which the reflection coefficient is not zero. Furthermore, to prove the qualitative bounds Theorems \ref{th:quant}-\ref{th:local_square} we need to control various properties of these rays explicitly in $R$. \S\ref{sec:outlinerays} outlines the ideas used to construct these rays. 

\noi\textbf{Acknowledgements.} The authors thank Martin Averseng (University of Bath), Shihua Gong (University of Bath), and Pierre-Henri Tournier (Sorbonne Universit\'e, CNRS) for their help in performing the numerical experiments in \S\ref{sec:num}. We thank Th\'eophile Chaumont-Frelet (INRIA, Nice) and Ivan Graham (University of Bath) for commenting on an early draft of the introduction; in particular we thank TCF for asking us about the error in subsets of $\domain$, prompting us to prove the results in \S\ref{subsec:error2}.
We also thank the anonymous referee for their constructive comments.  EAS and DL were supported by EPSRC grant EP/R005591/1.
This research made use of the Balena and Nimpus High Performance Computing (HPC) Services at the University of Bath.


\section{Results about defect measures of solutions of the Helmholtz equation}\label{sec:defect}

\subsection{Restatement of the boundary-value problems in semiclassical notation}\label{sec:conventions}

While we anticipate the vast majority of ``end users'' of Theorems \ref{th:lower}, \ref{th:quant}, \ref{th:quant2}, and \ref{th:quant3} will use the Helmholtz equation in the form \eqref{eq:BVP} with frequency $k$ (and be interested in the limit $k\tendi$), the tools and existing results from semiclassical-analysis that we use to prove these results are more convenient to write using the semiclassical parameter $h=k^{-1}$ (and the corresponding limit $h\tendo$).

The boundary-value problem \eqref{eq:BVP} therefore becomes, 
\begin{subequations}\label{eq:BVP2}
\begin{align}
(-h^{2}\Delta-1)u=0\qquad&\text{in }\Omegaplus,\label{eq:PDE2}\\
u=\exp(ix\cdot a /h)
\qquad&\text{on }\Gsc, \quad\text{ and}\label{eq:BC2}\\
h\frac{\partial u}{\partial r} - i  u  = o \left( \frac{1}{r^{(d-1)/2}}\right)\qquad&\text{as } r\tendi,
\end{align}
\end{subequations}
and the boundary-value problem \eqref{eq:BVPimp} becomes, 
\begin{subequations}\label{eq:BVPimp2}
\begin{align}
(-h^{2}\Delta-1)v=0\qquad&\text{in }\domain,\\
v=\exp(ix\cdot a /h)
\qquad&\text{on } \Gsc,\label{eq:BCimp2}\\
\bcN h\partial_n v - i\bcD v =0
\quad&\text{on }\GammaI.
\label{eq:BVPimp2imp}
\end{align}
\end{subequations}
In the rest of the paper, we use the ``$h$-notation'' instead of the ``$k$-notation''.

Appendix \ref{sec:appendix} recaps semiclassical pseudodifferential operators and associated notation.

\subsection{The Helmholtz equation posed a Riemannian manifold $M$}

While the main results of this paper concern the Helmholtz equation posed in $\domain \subset \Rea^d$, in the rest of this section 
(\S\ref{sec:defect}), in \S\ref{sec:resolventestimates}, and in \S\ref{sec:traceHO}, 
unless specifically indicated otherwise, we consider the Helmholtz equation posed on a Riemannian manifold $M$ with smooth boundary $\boundary$ and such that there exists a smooth extension $\widetilde{M}$ of $M$.
The reason we do this is that 
we expect the intermediary results of Theorems \ref{t:propagate} and  \ref{th:higherOrderBound} to be of interest in this manifold setting, independent of their application in proving the main results (Theorems \ref{th:lower}-\ref{th:local_square}).
This manifold setting involves the operator $P:= -h^2 \Delta_g-1$, where $\Delta_g$ is the metric Laplacian.
Nevertheless, for the reader unfamiliar with this set up, we highlight that $M$ can be replaced by $\domain$, $\widetilde{M}$ replaced by $\Rea^d$, and $\Delta$ replaced by $\Delta_g$, and all the statements and proofs remain unchanged.

\subsection{The local geometry and the flow} \label{subsec:geo}

Near the boundary $\boundary$,
we use Riemannian/Fermi normal coordinates $(x_{1},x')$, in which $\Gamma$
is given by $\{x_{1}=0\}$ and $\domain$ is $\{x_{1}>0\}$.  
The conormal and cotangent variables are given by $(\xi_{1},\xi')$. In these coordinates,
\beq\label{eq:P}
P:=-h^2\Delta_g -1= (hD_{x_1})^2 - R(x_{1},x',hD_{x'}) +h\big(a_1(x)hD_{x_1}+a_0(x,hD_{x'})\big).
\eeq
where $a_1\in C^\infty$, $a_0$ and $R$ are tangential pseudodifferential operators (in sense of \S\ref{sec:tangential}), with $a_0$ of order $1$, and $R$ of order $2$ with $h$-symbol $r(x_{1},x',\xi')$, with $r(0,x',\xi')=1-|\xi'|_{g_{\Gamma}}^2$ (where the metric $g_\Gamma$ in the norm is that induced by the boundary). That is, $r(0,x',\xi')$ is the symbol of one plus the tangential Laplacian; 
in what follows, we often abbreviate $r(0,x',\xi')$ to $r(x',\xi')$.

The fact that $P$ is self adjoint implies that $R$ is self adjoint, $a_1=\overline{a_1}$, and $[hD_{x_1},a_1]=a_0 -(a_0)^*$ (with the latter two conditions obtained by integration by parts in the $x_1$ variable near $\Gamma$).
Let $p$ denote the semiclassical principal symbol of $P:= -h^2 \Delta_g - 1$, i.e.~$p= |\xi|^2_g-1$. In a classical way (see, e.g., \cite[\S24.2 Page 423]{Ho:85}), the cotangent bundle to the boundary $T^{*}\boundary $
is divided in three regions, corresponding to the number of solutions
of the second order polynomial equation $p(\xi_{1})=0$:
\begin{itemize}
\item the \emph{elliptic region} $\mathcal{E}:=\big\{(x',\xi')\in T^{*}\boundary ,\ r(x',\xi')<0\big\}$,
where this equation has no solution, 
\item the \emph{hyperbolic region} $\mathcal{H}:=\big\{(x',\xi')\in T^{*}\boundary ,\ r(x',\xi')>0\big\}$,
where it has two distinct solutions 
\beq\label{eq:xiinout}
\xi_{1}^{\text{in}}=-\sqrt{r(x',\xi')}
\quad\tand\quad
\xi_{1}^{\text{out}}=\sqrt{r(x',\xi')},
\eeq
\item the \emph{glancing region} $\mathcal{G}:=\big\{(x',\xi')\in T^{*}\boundary ,\ r(x',\xi')=0\big\}$,
where it has exactly one solution, $\xi_{1}=0$.
\end{itemize}
The hyperbolic region plays a crucial role in obtaining the lower bounds in the main results, while we perform analysis near glancing to obtain the upper bounds.

With $p= |\xi|^2_g-1$ (i.e., the semiclassical principal symbol of $P:= -h^2 \Delta_g - 1$), the Hamiltonian vector field of $p$ is defined for compactly supported $a$ by 
\[
H_{p}a:=\left\{ p,a\right\},
\]
where $\{\cdot,\cdot\}$ denotes the Poisson bracket. 
Let $H_{p}^{*}$ denote the formal adjoint of $H_p a$, and let $\varphi_t(\rho)$ denote the \emph{generalised bicharacteristic flow} in $M$ (see \cite[\S24.3]{Ho:85}), defined such that
\beq\label{eq:flow}
(t,\rho)\in\mathbb{R}\times S^{*}_{\overline{M}}\widetilde{M}
\rightarrow\varphi_{t}(\rho)\in S^{*}_{\overline{M}}\widetilde{M}.
\eeq
In particular, when $M=\domain$ and $\widetilde{M}=\Rea^d$, 
$\varphi_{t}(\rho)\in S^*_{\overline{\domain}}\mathbb{R}^d:= \{(x,\xi) \in S^*\mathbb{R}^{d}, x\in \overline{\domain}\}=
 \{ x\in \overline{\domain}, \xi \in \Rea^d \text{ with } |\xi|=1\}$. By Hamilton's equations, away from the boundary of $M$, this flow satisfies $\dot{x}_i= 2 \xi_i$ and $\dot{\xi}_i=0$, so that it has speed $2$ (since $|\xi|=1$). Recall that the projection of the flow in the spatial variables are the \emph{rays}.

We now defined some projection maps.
Let $\projxM: T^*\widetilde{M} \to \widetilde{M}
$ be defined by $\projx(x,\xi)=x$.
Let $\pi_{\boundary}:  T_\boundary^{*}\widetilde{M}\cap \{p=0\} \rightarrow T^*\boundary$ be defined by 
\beq\label{eq:pigamma}
\pi_{\boundary}\big(0,x',\xi_{1},\xi'\big)=(x',\xi').
\eeq
Let $\pi_{\boundary,\rm in}:= \pi_{\boundary} |_{\xi_1<0}$ 
and let $\pi_{\boundary, {\rm out}}:= \pi_{\boundary} |_{\xi_1>0}$.

\bre[The Dirichlet-to-Neumann map away from glancing in local coordinates]\label{rem:applied1}
In the notation above, locally on $\GammaI$, the map $u \mapsto h D_{x_1}u = -h \partial_n u/i$ has semiclassical principal symbol $-\sqrt{r(x',\xi')}$. The minus sign in front of the square root is chosen since, when $\xi'=0$ (i.e.~$u$ corresponds to a normally-incident wave), the outgoing condition is that $h D_{x_1}u= -u$ (i.e.~$\partial_n u= i k u$), as opposed to $hD_{x_1}u= u$ (i.e.~$\partial_n u = -i k u$).
\ere

\subsection{Existence and basic properties of defect measures}

We first assume that $u\in L_{\operatorname{loc}}^2(\mathbb{R}^d)$ is a solution to 
\begin{equation}
\label{eq:helmholtzNoBoundary}
Pu:= (-h^2\Delta_g-1)u=hf \quad\text{ on } U, 
\qquad u|_{\mathbb{R}^d\setminus \overline{U}}=0,
\end{equation}
where $U\subset \mathbb{R}^d$ is open with smooth boundary $\Gamma$ and $f\in L^2_{\rm comp}(\Rea^d)$. When taking traces of $u$, we always do so from $U$ rather than from $\mathbb{R}^d\setminus \overline{U}$.
To define the defect measures associated with $u$ 
we need the following boundedness assumption.

\begin{assumption}\label{ass:1}
Given $\chi\in C_{c}^{\infty}(\mathbb{R}^{d})$, there exists $C>0$,
and $h_{0}>0$ such that for any $0<h\leq h_{0}$ 
\beqs
\Vert\chi u\Vert_{L^{2}(U)}+\Vert u\Vert_{L^2(\Gamma)}+\Vert h\partial_{n}u\Vert_{L^{2}(\Gamma)}\leq C.
\eeqs
\end{assumption}

We highlight that Assumption \ref{ass:1} is satisfied when the problem is nontrapping; see Lemma \ref{l:defectExist} below.

\begin{theorem}[Existence of defect measures] \label{th:ex_defect}
Suppose that $u_{h_k}$ solves~\eqref{eq:helmholtzNoBoundary} and satisfies Assumption~\ref{ass:1}. Then there exists a subsequence $h_{k_\ell}\to 0$ and non-negative Radon measures $\mu$ and $\mu^j$on
$T^*\widetilde{M}$, $\nu_{d}$, $\nu_{n}$, $\nu_{j}$ on $T^{*}\boundary $ such that for any symbol $b\in C_c^\infty(T^*\widetilde{M})$ and tangential symbol $a\in C_c^\infty(T^{*}\boundary )$,
as $\ell\tendi$
\begin{align}\label{eq:jointmeasure}
&\big\langle b(x,h_{k_\ell}D_{x})u,u\big\rangle\rightarrow\int b(x,\xi)\ d\mu, \qquad\qquad
\big\langle b(x,h_{k_\ell}D_{x})u,f\big\rangle\rightarrow\int b(x,\xi)\ d\mu^j, \\ \nonumber
&\big\langle a(x',h_{k_\ell}D_{x'})u,u\big\rangle_{\Gamma}\rightarrow\int a(x',\xi')\ d\nu_{d},\qquad
\big\langle a(x',h_{k_\ell}D_{x'})h_{k_\ell}D_{x_1}u,u\big\rangle_{\Gamma}\rightarrow\int a(x',\xi')\ d\nu_{j},\\ \nonumber
&\big\langle a(x',h_{k_\ell}D_{x'})h_{k_\ell}D_{x_1}u,h_{k_\ell}D_{x_1}u\big\rangle_{\Gamma}\rightarrow\int a(x',\xi')\ d\nu_{n}.
\end{align}
\end{theorem}
\begin{proof}[Reference for the proof]
See \cite[Theorem 5.2]{Zworski_semi}.
\end{proof}

\bre[The measure $\nu_j$]
The joint measure $\nu_j$ also describes pairings with the Neumann and Dirichlet traces swapped, 
since, by \eqref{eq:symbol}, 
\beqs
\big\langle a(x',h_{k_\ell}D_{x'})u,h_{k_\ell}D_{x_1}u\big\rangle_{\Gamma}
=\overline{\big\langle a(x',h_{k_\ell}D_{x'})^*h_{k_\ell}D_{x_1}u,u\big\rangle_{\Gamma}}
\rightarrow\overline{\int \overline{a}\ d\nu_{j}} =\int a\ d\nu_{j}.
\eeqs
\ere

We use the notation that $\mu(a) := \int a\, d \mu$ for the pairing of a function and a measure. 
We also use the notation that $b\mu (f) := \int f\, b\, d\mu$, where $b \in L^\infty(d \mu)$ and $f\in L^1(d\mu)$.

We now recall the following two fundamental results.
\begin{lem}[Invariance and support of defect measures]\label{lem:inv}
Let $u$ satisfy \eqref{eq:helmholtzNoBoundary} and let $\mu$ be a defect measure of $u$.

(i) In the interior of $U$, 
\begin{equation} \label{eq:inv_int}
\mu(H_pa)=-2\Im\mu^{j}(a)
\end{equation}
for all $a\in C_c^\infty(T^*U)$;
in particular, if $f=o(1)$ as $h\tendo$, then $\mu$ is invariant under the flow.

(ii) $\mu$ is supported in the characteristic set:
\begin{equation}
\supp \mu\cap T^*U\subset\Sigma_{p}:=\{p=0\}.\label{eq:carac}
\end{equation}
\end{lem}
\begin{proof}[References for the proof]
\eqref{eq:inv_int} was originally proved in \cite{Gerard}; see also \cite[Theorem 5.4]{Zworski_semi}, \cite[Theorem E.44]{DyZw:19}. \eqref{eq:carac} was proved in the framework with boundary by \cite[Lemma 1.3]{Miller}; see also \cite[Lemma 4.2]{GaSpWu:20}.
\end{proof}

Part (ii) of Lemma \ref{lem:inv} implies that $\mu$ is only supported on $|\xi|=1$; this is the reason why we only consider the flow \eqref{eq:flow} defined on $S^{*}_{\overline{M}}\widetilde{M}$.

\subsection{Evolution of defect measures under the flow}

\begin{lem}[Integration by parts]
\label{lem:intByParts}
Let $B_i\in C_c^\infty((-2\delta,2\delta)_{x_1};\Psi^{\ell_i}(\Rea^{d-1}))$,
 $i=1,2$, and let 
$B=B_0+B_1hD_{x_1}$. 
If
\begin{equation} \label{eq:B0B1}
B_1^*=B_1,\qquad B_0^*+[hD_{x_1},B_1]=B_0,
\end{equation}
then, for all $u \in C^\infty(\overline{M})$,
\begin{align}\nonumber
&\frac{i}{h}\big\langle  [P,B]u,u\big\rangle _{L^2(M)}=-\frac{2}{h}\Im\big\langle  Bu,Pu\big\rangle _{L^2(M)}  \\ \nonumber
&\hspace{0.5cm}
-\big\langle  B_1 hD_{x_1}u,hD_{x_1}u\big\rangle _{L^2(\boundary)} 
-\big\langle  \big(B_0+h(D_{x_1}B_1)-h(B_1 a_1 -\overline{a_1}B_1)\big)hD_{x_1}u,u\big\rangle _{L^2(\boundary)}\\
&\hspace{0.5cm} 
-\big\langle  B_0u,hD_{x_1}u\big\rangle _{L^2(\boundary)}
-\big\langle \big( h(D_{x_1}B_0)+B_1(R-ha_0)+h\overline{a_1}B_0\big)u,u\big\rangle _{L^2(\boundary)},
\label{eq:ibps}
\end{align}
\end{lem}

\begin{cor}\label{cor:ibps}
Let $u$ satisfy Assumption \ref{ass:1} and thus have defect measures as in Theorem~\ref{th:ex_defect}. 
Given $a\in C_c^\infty (T^*\widetilde{M})$, 
let 
\[
a_{\rm even}(x,\xi_{1},\xi'):=\frac{a(x,\xi_{1},\xi')+a(x,-\xi_{1},\xi')}{2},\quad a_{\rm odd}(x,\xi_{1},\xi'):=\frac{a(x,\xi_{1},\xi')-a(x,-\xi_{1},\xi')}{2\xi_{1}},
\]
so that $a(x,\xi_{1},\xi')= a_{\rm even}(x,\xi_{1},\xi') + \xi_1 a_{\rm odd}(x,\xi_{1},\xi')$.
Then
\begin{equation}
\mu(H_{p}a)=-2\Im \mu^j(a) -\nu_{n}(a_{\rm odd})-2\Re \nu_j(a_{\rm even})-\nu_{d}({r(x',\xi')}a_{\rm odd}).\label{eq:H_p_star}
\end{equation}
\end{cor} 

\begin{proof}[Proof of Lemma \ref{lem:intByParts}]
First recall that $R$ is self adjoint, $a_1=\overline{a_1}$, and $[hD_{x_1},a_1]=a_0 -(a_0)^*$; see \S\ref{subsec:geo}.
By integration by parts,
\beqs
\big\langle  (hD_{x_1})^2Bu,u\big\rangle _{L^2(M)}=\big\langle  Bu,(hD_{x_1})^2u\big\rangle _{L^2(M)}-\frac{h}{i}\Big[\big\langle  hD_{x_1}Bu,u\big\rangle _{L^2(\boundary)}+\big\langle  B u, hD_{x_1}u\big\rangle _{L^2(\boundary)}\Big],
\eeqs
and
\beqs
\big\langle  a_1 hD_{x_1}Bu,u\big\rangle _{L^2(M)}=\big\langle Bu, \big(a_1hD_{x_1}+[hD_{x_1},a_1]\big)u\big\rangle _{L^2(M)}-\frac{h}{i}\big\langle  Bu,a_1u\big\rangle _{L^2(\boundary)}
\eeqs
Using theses two identities, the expression for $P$ \eqref{eq:P}, the self-adjointness of $R$, 
and the fact that $[hD_{x_1}, a_1] = a_0-(a_0)^*$,
we obtain that
\begin{align}
\big\langle  PBu,u\big\rangle _{L^2(M)} =& \big\langle  Bu, Pu\big\rangle _{L^2(M)}- \frac{h}{i}\Big[\big\langle  hD_{x_1}Bu,u\big\rangle _{L^2(\boundary)}+\big\langle  B u, hD_{x_1}u\big\rangle _{L^2(\boundary)} + h \big\langle  Bu,a_1u\big\rangle _{L^2(\boundary)}\Big].
\label{eq:Euan1}
\end{align}
The definition of $B$ and the form of $P$ in \eqref{eq:P} imply that 
\begin{align}\nonumber
h D_{x_1} Bu &= B_1 (hD_{x_1})^2 u + \big( hD_{x_1}B_1+ B_0\big)(hD_{x_1}u) + (hD_{x_1}B_0)u,\\
&= B_1 \big( R-ha_0 -ha_1 hD_{x_1}\big) u + B_1 P u + \big( hD_{x_1}B_1 + B_0\big) (hD_{x_1}u) + (hD_{x_1}B_0)u.
\label{eq:Euan2}
\end{align}
Therefore, using \eqref{eq:Euan1} and \eqref{eq:Euan2}, we have 
\begin{align}\nonumber
&\frac{i}{h}\big\langle  [P,B]u,u\big\rangle _{L^2(M)}=\frac{i}{h}\big\langle  PBu,u\big\rangle _{L^2(M)} - \frac{i}{h}\big\langle  B(Pu),u\big\rangle _{L^2(M)} 
 \\ \nonumber
&=\frac{i}{h}\big\langle  Bu,Pu\big\rangle _{L^2(M)} - \frac{i}{h}\big\langle  B(Pu),u\big\rangle _{L^2(M)} 
\\ \nonumber
&\hspace{0.5cm}
-\big\langle  B_1 hD_{x_1}u,hD_{x_1}u\big\rangle _{L^2(\boundary)} 
-\big\langle  \big(B_0+h(D_{x_1}B_1)-h(B_1 a_1 -\overline{a_1}B_1)\big)hD_{x_1}u,u\big\rangle _{L^2(\boundary)}\\
&\hspace{0.5cm} 
-\big\langle  B_0u,hD_{x_1}u\big\rangle _{L^2(\boundary)}
-\big\langle [ h(D_{x_1}B_0)+B_1(R-ha_0)+h\overline{a_1}B_0]u,u\big\rangle _{L^2(\boundary)}
-\big\langle  B_1(Pu),u\big\rangle _{L^2(\boundary)}\label{eq:Euan3}
\end{align}
Next, using the definition of $B$, integration by parts, and \eqref{eq:B0B1}, we find that, for any $v,u$,
\begin{align}\nonumber
\big\langle  Bv,u\big\rangle _{L^2(M)} &=-\frac{h}{i}\big\langle  v,B_1^*u\big\rangle _{L^2(\boundary)}+ \big\langle v, B_0^* u + hD_{x_1}(B_1^* u)\big\rangle_{L^2(M)}\\
&=-\frac{h}{i}\big\langle  v,B_1u\big\rangle _{L^2(\boundary)}+ \big\langle  v,Bu\big\rangle _{L^2(M)}\label{eq:Euan4}
\end{align}
Letting $v=Pu$, combining \eqref{eq:Euan3} and \eqref{eq:Euan4}, and using the fact that $B_1=B_1^*$, we obtain
\begin{align*}\nonumber
&\frac{i}{h}\big\langle  [P,B]u,u\big\rangle _{L^2(M)}=\frac{i}{h}\big\langle  Bu,Pu\big\rangle _{L^2(M)} - \frac{i}{h}\big\langle  Pu,Bu\big\rangle _{L^2(M)} 
\\ \nonumber
&\hspace{0.5cm}
-\big\langle  B_1 hD_{x_1}u,hD_{x_1}u\big\rangle _{L^2(\boundary)} 
-\big\langle  \big(B_0+h(D_{x_1}B_1)-h(B_1 a_1 -\overline{a_1}B_1)\big)hD_{x_1}u,u\big\rangle _{L^2(\boundary)}\\
&\hspace{0.5cm} 
-\big\langle  B_0u,hD_{x_1}u\big\rangle _{L^2(\boundary)}
-\big\langle [ h(D_{x_1}B_0)+B_1(R-ha_0)+h\overline{a_1}B_0]u,u\big\rangle _{L^2(\boundary)},
\end{align*}
which is \eqref{eq:ibps}.
\end{proof}

\bpf[Proof of Corollary \ref{cor:ibps}]
Letting $h\tendo$ in \eqref{eq:ibps}, using the third equation in \eqref{eq:symbol} and the definitions of the measures in Theorem \ref{th:ex_defect}, we have 
\begin{equation}\label{eq:H_p_star2}
\mu(H_{p}b)=-2\Im \mu^j(b) -\nu_{n}(b_1)-2\Re \nu_j(b_0)-\nu_{d}(r \,b_1),
\end{equation}
where $b=\sigma(B)$, $b_i=\sigma(B_i)$.
The idea of the proof is to construct a $B$ satisfying the assumptions of Lemma \ref{lem:intByParts} with $\sigma(B_0)= a_{\rm odd}$ and $\sigma(B_1)= a_{\rm even}$ (and thus $\sigma(B)= a$).
Since \eqref{eq:H_p_star} is linear in $a$, without loss of generality, we assume that $a$ is real.
Since $a_{\rm even}$ and $a_{\rm odd}$ are both smooth, even functions of $\xi_1$, abusing notation slightly, we can write 
\beq\label{eq:the_end1}
a_{\rm even/odd}(x,\xi_1,\xi') = 
a_{\rm even/odd}(x, \xi_1^2,\xi').
\eeq
Let 
\beq\label{eq:the_end2}
\widetilde{a}_{\rm even}(x,\xi')= a_{\rm even}\big(x,r(x_1,x',\xi'),\xi'\big), \qquad 
\widetilde{a}_{\rm odd}(x,\xi')= a_{\rm odd}\big(x,r(x_1,x',\xi'),\xi'\big),
\eeq
and 
\beqs
\widetilde{a}(x,\xi')= 
\widetilde{a}_{\rm even}(x,\xi')+ \xi_1 \widetilde{a}_{\rm odd}(x,\xi').
\eeqs
Since $S^*\widetilde{M}= \{\xi_1^2 - r(x_1,x',\xi')=0\}$ and 
$H_p\big(\xi_1^2 - r(x_1,x',\xi')\big)=0$ (by \eqref{eq:P}),
\beqs
\widetilde{a}|_{S^*\widetilde{M}} = a|_{S^*\widetilde{M} }
\quad\tand
\quad
H_p a|_{S^*\widetilde{M}} = H_p\big(a|_{S^*\widetilde{M}}\big);
\eeqs
therefore 
\beqs
H_p a|_{S^*\widetilde{M}} = H_p\big(\widetilde{a}|_{S^*\widetilde{M}}\big).
\eeqs
Since $\mu$ is supported on $\{p=0\}$ by \eqref{eq:carac}, 
\beq\label{eq:the_end3}
\mu (H_p a)= \mu(H_p \widetilde{a}).
\eeq

Let 
\beqs
B_0(x,hD_{x'}):= \frac{\widetilde{a}_{\rm even}(x,hD_{x'})+ (\widetilde{a}_{\rm even}(x,hD_{x'}))^*}{2} + \frac{1}{2} \left[ hD_{x_1} ,  \frac{\widetilde{a}_{\rm odd}(x,hD_{x'})+ (\widetilde{a}_{\rm odd}(x,hD_{x'}))^*}{2}\right]
\eeqs
and 
\beqs
B_1(x,hD_{x'}):= \frac{\widetilde{a}_{\rm odd}(x,hD_{x'})+ (\widetilde{a}_{\rm odd}(x,hD_{x'}))^*}{2}. 
\eeqs
Then \eqref{eq:B0B1} is satisfied and, by \eqref{eq:symbol}, \eqref{eq:the_end2}, and \eqref{eq:the_end1},
\begin{align*}
\sigma(B_0)(x,\xi')&= \widetilde{a}_{\rm even}(x,\xi')=a_{\rm even}(x,\xi_1^2, \xi') \quad\ton S^*\widetilde{M}.
\end{align*}
Similarly, $\sigma(B_0)(x,\xi')= a_{\rm odd}(x,\xi_1^2, \xi')$, and thus $\sigma(B)= a(x,\xi_1,\xi')$ on $S^*\widetilde{M}$.
The result \eqref{eq:H_p_star} then follows from \eqref{eq:H_p_star2} and \eqref{eq:the_end3}.
\epf

\subsection{Properties of defect measures on the boundary}\label{sec:Miller}

In this subsection we review the calculations from~\cite{Miller}, adapting them to the case when the right-hand side of the PDE is non-zero.

\bre[Notation in \cite{Miller}]\label{rem:not_Miller}
Since our results rely heavily on the results of \cite{Miller}, we record here the correspondence between the notation in \cite{Miller} (on the left) and our notation (on the right):
\begin{equation*}
\Delta_{p}=4r,\quad k^{{\rm in/out}}=\xi_{1}^{{\rm in/out}},\quad  \sigma=\xi_{1},\quad s=x_{1},
\quad\dot \nu^N = 4 \nu_n, \quad \nu^{jN} = 2\nu_j.
\end{equation*}
\ere

Recall that $u$ has defect measure $\mu$, trace measures $\nu_d$, $\nu_n$, and $\nu_j$, and $f$ and $u$ have joint defect measure $\mu^j$. By \cite[Lemma 3.3]{GaSpWu:20}, $\mu^j(a)$ is absolutely continuous with respect to $\mu$, and $\mu^j=\beta d\mu$ for some $\beta\in L^1(d\mu)$; hence \eqref{eq:H_p_star} becomes
\begin{equation}
\label{e:flowMe}
\mu(H_pa+2\Im \beta a)=-\nu_n(a_{\rm odd})-2\Re \nu_j(a_{\rm even})-\nu_d(ra_{\rm odd}).
\end{equation}
For convenience, we define the differential operator 
$$
\mathcal{L}:=H_p+2\Im \beta.
$$

\begin{lem}\label{lem:charge}
There is a distribution $\mu^0$ on $T^*_{\boundary }\widetilde{M}$ supported in $\overline{B^*\boundary }$ such that 
\beq\label{eq:charge1}
\mathcal{L}^*(\mu 1_{x_1>0})=\delta(x_1)\otimes \mu^0,
\eeq
where $\otimes$ denotes tensor product of distributions.
Furthermore, on $\pi_{\boundary}^{-1}(\mathcal{H})$, 
\beq\label{eq:mu0delta}
\mu^0:=\delta\big(\xi_1- \xiin
\big)\otimes \muin- \delta\big(\xi_1-
\xiout
\big)\otimes \muout
\eeq
where $\mu^{\rm in/out}$ are positive measures on $T^*\boundary $ supported in $\mathcal{H}$, and $\xi^{\rm in/out}$ are defined by \eqref{eq:xiinout}.
\end{lem}

\begin{proof}
The proof follows~\cite[Proposition 1.7]{Miller}, replacing $H_p$ at every step by $\mathcal{L}$. In particular, by \eqref{e:flowMe}, $\mathcal{L}^*(\mu 1_{x_1>0}) $ is supported in $\{x_1=0\}$ and hence is of the form $\sum_{k=0}^{\ell}\delta^{(k)}(x_1)\otimes \mu_k$ where each $\mu_k$ is a distribution on $T^*_{\boundary }\widetilde{M}$. But, letting $\chi\in C_c^\infty(\mathbb{R})$ with $\chi^{(k)}(0)=1$, for $k\leq \ell$ and applying~\eqref{e:flowMe} to $a_\e=\e^\ell \chi(\e^{-1}x_1)b(x',\xi)$, we have for $\ell\geq 1$,
$$
\sum_{k=0}^\ell \e^{\ell-k}\mu_k(b)= \mu(1_{x_1>0}\mathcal{L}a_\e)= \mu\Big(1_{x_1>0}(\e^{\ell-1}H_p\chi +\e^\ell\chi  H_pb-2\Im \beta a)\Big)
\rightarrow 0 \tas \e\to 0.
$$
In particular, $\mu_k=0$ for $k\geq 1$, and \eqref{eq:charge1} follows.

The result \eqref{eq:mu0delta} about the structure of $\mu^0$ in the hyperbolic set follows by considering a small neighbourhood $\mathcal{V}$ in $T^*\boundary $ of a point $\rho\in \mathcal{H}$ and $\delta>0$ such that each geodesic trajectory of length $2\delta$ centered in $\pi_{\boundary}^{-1}(\mathcal{V})$ intersects the boundary exactly once. We may then use 
$$
(-\delta,\delta)\times \pi_{\boundary}^{-1}(\mathcal{V})\ni (t,\rho)\to \varphi_t(\rho)\in \mathcal{V}_\delta\subset T^*\widetilde{M}
$$
as coordinates on an open neighbourhood, $\mathcal{V}_\delta$ of $\pi_{\boundary}^{-1}(\mathcal{V})$. In these coordinates, writing $\widetilde{\mu}$ for the pull-back of $1_{x_1>0}\mu$ under $\varphi_t$, we obtain
$$
(\partial_t+2\Im \beta)\widetilde{\mu}=\delta(t)\otimes \mu^0.
$$

In particular, $\widetilde{\mu}$ is null $\mathcal{V}_{t_0}$ for any $t_0\in (-\delta,\delta)$, and testing by $\e\chi(t\e^{-1}) b$ with $0\leq b\in C_c^\infty(\pi_{\boundary}^{-1}(\mathcal{V}))$, and $\chi\in C_c^\infty(-\delta,\delta)$ with $t\chi'(t)<0$ on $|t|>0$, $\chi(0)=1$, we have  
$$
\widetilde{\mu}(\chi'(\e^{-1}t)b-2\e\Im \beta \chi(\e^{-1}t) b)=\mu^0(b).
$$
Now $\widetilde{\mu}$ is identically zero on $\pi_{\rm in}^{-1}(\mathcal{V})\times [0,\infty)$ and on $\pi_{\rm out}^{-1}(\mathcal{V})\times (-\infty,0])$. Therefore, for $b$ supported in $\pi_{\boundary  {\rm out}}^{-1}(\mathcal{V})$ 
$$
\mu^0(b)\leq \liminf_{\e \to 0}\widetilde{\mu}\Big(\big[\chi'(\e^{-1}t)b-2\e\Im \beta \chi(\e^{-1}t) b\big]1_{t>0}\Big)\leq 0.
$$
Similarly, for $b$ supported in $\pi_{\rm in}^{-1}(\mathcal{V})$, $\mu^0(b)\geq 0$. In particular, $\mu^0$ is a positive distribution on $\pi_{\rm in}^{-1}(\mathcal{H})$ and a negative distribution of $\pi_{\rm out}^{-1}(\mathcal{H})$, and the result follows. 
\end{proof}

Next, we decompose $\mu$ into its interior and boundary components, with the following lemma the analogue of \cite[Proposition 1.8]{Miller}.
\begin{lem}\label{lem:interior_boundary}
There is a positive measure $\mu^{\partial}$ on $\mathcal{G}\subset T^*_{\boundary }\widetilde{M}$ such that 
$$
\mu= 1_{x_1>0}\mu +\delta(x_1)\otimes \delta (H_px_1)\otimes \mu^{\partial}.
$$
\end{lem}
\begin{proof}
Let $\chi \in C_c^\infty(\mathbb{R})$ with $\chi(0)=\chi'(0)=1$  and $b\in C_c^\infty(\mathbb{R}\times T^*\mathbb{R}^{n-1})$. Then, with $a_\e=\e \chi(x_1\e^{-1})b(x,\xi')$, \eqref{e:flowMe} implies that
$$
\mu(\mathcal{L}a_\e)= -2\e\Re \nu_j(b)
$$
Now, 
$$
\mathcal{L}a_\e= 2\chi'(x_1\e^{-1})H_px_1b+O(\e).
$$
Therefore, by the dominated convergence theorem, 
$$
\mu(\mathcal{L}a_\e)\to \mu(1_{x_1=0}bH_px_1)
$$
and, since $|\nu_j(b)|<\infty$, 
$$
\mu(1_{x_1=0}bH_px_1)=0.
$$
Since $b$ was arbitrary, $\mu$ decomposes as claimed.
\end{proof}

The following lemma is the analogue of~\cite[Lemma 1.9]{Miller}.

\begin{lem}
\label{l:elliptic}
On $\mathcal{E}$ (i.e.~$r <0$), $\Re \nu_j=0$ and $\nu_n=-r\nu_d$. 
\end{lem}
\begin{proof}
Let $\chi \in C^\infty(\mathbb{R})$ with $\chi \equiv 1$ on $(-\infty, -1]$ and $\supp \chi \subset (-\infty,0)$. Let $b=b(x,\xi')\in C_c^\infty$ and define $b_\e=\chi(\e^{-1}r)b$. Then, by~\eqref{e:flowMe} together with the fact that $\supp \mu\subset S^*M$, 
$$
0=\mu(H_pb_\e+2\Im \beta b_\e)=-2\Re \nu_j(b_\e)
$$
Sending $\e\to 0^+$, we obtain 
$$
0=2\Re \nu_j(b1_{r<0}). 
$$
Since $b$ was arbitrary, $\nu_j1_{r<0}=0$. Replacing $b$ by $b(x,\xi')\xi_1$ and applying the same argument, we obtain 
$$
\nu_n1_{r<0}=-r\nu_d1_{r<0}.
$$
\end{proof}

Next, we prove the analogue of~\cite[Proposition 1.10]{Miller}
\begin{lem}
\label{lem:key_Miller}On the hyperbolic set $\mathcal{H}$,

\noindent (i)
\beq\label{eq:key_Miller0}
2\muout=  \sqrt{r(x',\xi')}\nu_{d} + 2\Re \nu_{j}+\frac{1}{\sqrt{r(x',\xi')}}\nu_{n},\qquad
2\muin=\sqrt{r(x',\xi')}\nu_{d}-2\Re \nu_{j}+\frac{1}{\sqrt{r(x',\xi')}}\nu_{n}.
\eeq
(ii) If $\muin=0$ on some Borel set $\mathcal{B}\subset \mathcal{H}$, then 
\beq\label{eq:muout}
\muout=2\Re \nu_{j}=2\sqrt{r(x',\xi')}\nu_{d}=\frac{2 }{\sqrt{r(x',\xi')}}\nu_{n}.
\eeq
(iii) If 
\beq\label{eq:alpha_key}
-2 \Re  \nu_j=(\Re \alpha)\nu_{d}=4(\Re \alpha)|\alpha|^{-2}\nu_{n}
\eeq
on some Borel set $\mathcal{B}\subset \mathcal{H}$ for $\alpha$ a complex valued function
such that $\alpha+2\sqrt{r(x',\xi')}$ is never zero
on $\mathcal{B}$ then
\beq\label{eq:key_Miller1}
\muout=\alpharef \muin,
\eeq
where
\beq\label{eq:key_Miller2}
\alpharef :=\left|\frac{2\sqrt{r(x',\xi')}-\alpha}{2\sqrt{r(x',\xi')}+\alpha}\right|^{2}\text{ on }\mathcal{B},
\eeq
where the superscript ``{\rm ref}'' stands for ``reflected''. If instead, $\alpha-2\sqrt{r}$ is never zero, then 
$$
(\alpharef )^{-1}\muout=\muin.
$$
\end{lem}
\begin{proof}
(i) By combining Lemmas \ref{lem:charge} and \ref{lem:interior_boundary},
\beq\label{eq:mutemp1}
\mathcal{L}^*\mu=\delta(x_1)\otimes \mu^0+\mathcal{L}^*(\delta(x_1)\otimes \delta(H_px_1)\otimes \mu^{\partial}).
\eeq
Let $\chi \in C^\infty(\mathbb{R})$ with $\chi \equiv 0$ on $(-\infty,1]$ and $\chi \equiv 1$ on $[2,\infty)$. For $a\in C_c^\infty(\mathbb{R}\times T^*\boundary )$ (so $a=a(x_1,x',\xi')$), let $a_\e=\chi(\e^{-1}|H_px_1|)a$. Since $H_p x_1= 2\xi_1$, $a_\e= a$ for $|\xi_1|\geq \e$ and $a_\e=0$ for $|\xi_1|\leq \e/2$.
Combining \eqref{eq:mutemp1} and \eqref{e:flowMe}, and
using the facts that $a_\e$ is even in $\xi_1$ and $a_\eps=0$ for $|H_p x_1|\leq \eps/2$, we find that
$$
\mu^0(a_\e|_{x_1=0})=\mu(\mathcal{L}a_\e)=-2\Re \nu_j(a_\e|_{x_1=0}).
$$
By \eqref{eq:mu0delta},
\beqs
\chi\big(2|\xiin|/\e\big)\muin\big(a|_{x_1=0}\big)- \chi\big(2|\xiout|/\e\big)\muout\big(a|_{x_1=0}\big)=2\Re \nu_j\big(a_\e|_{x_1=0}\big).
\eeqs
Therefore, by the dominated convergence theorem,
\beq\label{eq:system1}
\muin-\muout=-2\Re \nu_j \quad\ton \mathcal{H}.
\eeq
Similarly, since $a_\e \xi_1$ is an odd function of $\xi_1$,  \eqref{eq:mutemp1} and \eqref{e:flowMe} imply that
$$
\mu^0(a_\e\xi_1|_{x_1=0})=\mu(\mathcal{L}a_\e\xi_1)= -\nu_d(ra_\e|_{x_1=0})-\nu_n(a_\e|_{x_1=0}).
$$
By \eqref{eq:mu0delta},
\beqs
\xiin\chi\big(2|\xiin|/\e\big)\muin\big(a|_{x_1=0}\big)-\xiout \chi\big(2|\xiout|/\e\big)\muout\big(a|_{x_1=0}\big)=-\nu_d(ra_\e|_{x_1=0})-\nu_n(a_\e|_{x_1=0}).
\eeqs
Therefore, by the dominated convergence theorem,
\beq\label{eq:system2}
-\sqrt{r}(\muin+\muout)=-r\nu_d-\nu_n \quad\ton \mathcal{H}.
\eeq
The  result \eqref{eq:key_Miller0} now follows from solving \eqref{eq:system1} and \eqref{eq:system2} for $\muin$ and $\muout$. 

(ii) By the Cauchy--Schwarz inequality and similar reasoning used in the proof of \cite[Lemma 3.3]{GaSpWu:20},
\beq\label{eq:CS}
|\nu_j|\leq \sqrt{\sqrt{r} \nu_d} \sqrt{\nu_n/\sqrt{r}}. 
\eeq
By \eqref{eq:key_Miller0}, when $\muin=0$, 
\beq\label{eq:consist1}
2\Re \nu_j = \sqrt{r} \nu_d + \nu_n/\sqrt{r},
\eeq
However, for both \eqref{eq:CS} and \eqref{eq:consist1} to hold, we must have $\sqrt{r}\nu_d = \nu_n/\sqrt{r}$, and \eqref{eq:key_Miller1} follows.

(iii) The equation \eqref{eq:key_Miller1} follows from using \eqref{eq:alpha_key} in \eqref{eq:key_Miller0}.
\end{proof}

\begin{lem}
\label{l:glancing}
$$
-H_p^2x_1 \mu^{\partial}=4\nu_n1_{\mathcal{G}}.
$$
In particular, $\mu^{\partial}$ is supported in $H_p^2x_1\leq 0$ and $\nu_n1_{\mathcal{G}}$ does not charge $H_p^2x_1\geq 0$.
\end{lem}
\begin{proof}
We follow~\cite[Lemma 4.7]{GaSpWu:20}. Since $H_px_1=2\xi_1$, 
$$
H_p(2a(x,\xi)\xi_1)=aH_p^2x_1+2\xi_1H_pa.
$$
Now, put $a_\e=\chi(\e^{-1}x_1)\chi(\e^{-1}r(x,\xi'))2 a\xi_1$ where $\chi \in C_c^\infty(\mathbb{R})$ has $\chi \equiv 1$ near $0$. 
Then,
$$
H_pa_\e=a\chi(\e^{-1}x_1)\chi(\e^{-1}r)H_p^2x_1+O(1)\Big(|\chi'(\e^{-1}x_1)|+|\chi'(\e^{-1}r)|+\e^{1/2}\Big),
$$
where we have used that on $S^*M$, $H_pr=-H_p\xi_1^2=O(\xi_1)$. Then, by the dominated convergence theorem,
$$
\mu(H_pa_\e)\to \frac{1}{2}\mu^{\partial}\big([H_p^2x_1] a\big).
$$
Using~\eqref{e:flowMe}, we have
$$
\mu(H_pa_\e)=-2\mu(\Im \beta a_\e)-\nu_d(2\chi(\e^{-1}r)ra)-\nu_n(2\chi(\e^{-1}r)a).
$$
Using the dominated convergence theorem again, using that $\xi_1=O(\sqrt{r})$ on $S^*M$, we have
$$
\mu(2\Im \beta a_\e)\to 0,
$$
and hence
$$
\tfrac{1}{2}\mu^{\partial}\big([H_p^2x_1]a\big)=-\nu_n(2a1_{\mathcal{G}}),
$$
as claimed.
\end{proof}

\begin{lem}
\label{l:finalEq}
Let $q=q(x_1,x_1\xi_1,x',\xi)\in C_c^\infty(T^*\widetilde{M})$. Then, 
$$
\mu(H_pq)=-2\Im \beta \mu(q)+(\muin-\muout)(q|_{x_1=0})+\tfrac{1}{2} \mu^{\partial}(\Re(\dot{n}^j) H_p^2x_1q|_{x_1=0}).
$$
where $\dot{n}^j \nu_n=\nu_j$.
\end{lem}
\begin{proof}
By Lemma~\ref{lem:key_Miller}, 
$$
\mu(\mathcal{L}q)=-2\Re \nu_j(q|_{x_1=0})=(\muin-\muout)(q|_{x_1=0})-2\Re \nu_j(1_{\mathcal{G}}q|_{x_1=0}).
$$
Now, since $\nu_j\ll \nu_n$ we may write $\nu_j=\dot n^j\nu_n$ and use Lemma~\ref{l:glancing} to obtain
$$
-2\Re \nu_j(1_{\mathcal{G}}q|_{x_1=0})=-2\Re \nu_n(\dot n^j1_{\mathcal{G}}q|_{x_1=0})=\frac{1}{2}\mu^{\partial}((\Re\dot n^j) H_p^2x_1 q|_{x_1=0}),
$$
and the claim follows.
\end{proof}

\begin{theorem}
\label{t:propagate}
Suppose that $\boundary$ is nowhere tangent to $H_p$ to infinite order. Then, for $q\in C_c^\infty(^bT^*M)$
\beq\label{eq:propagate}
\pi_*\mu(q\circ\varphi_t)-\pi_*\mu(q)=\int_0^t\Big(-2\Im \pi_*\mu^j+\delta(x_1)\otimes(\muin-\muout)+\frac{1}{2}(\Re \dot{n}^j)H_p^2x_1\mu1_{\mathcal{G}}1_{x_1=0}\Big)(q\circ\varphi^s)ds,
\eeq
where $^bT^*M$ denotes the $b$-cotangent bundle to $M$ and $\pi: T^*M\rightarrow \,^bT^*M$ is defined by  $\pi(x_1,x',\xi_1,\xi'):= (x_1,x',x_1\xi_1,\xi')$ (see~\cite[Section 4.2]{GaSpWu:20}).
\end{theorem}

\bpf
This result is analogous to \cite[Lemma 4.8]{GaSpWu:20},  except that \cite[Lemma 4.8]{GaSpWu:20} only considers zero Dirichlet boundary conditions, and thus only $-2\Im \pi_*\mu^j$ appears on the right-hand side of \cite[Equation 4.3]{GaSpWu:20} 
compared to \eqref{eq:propagate} (note that \cite{GaSpWu:20} defines the joint measure $\mu^j$ differently to \eqref{eq:jointmeasure}, with the result that the signs of $\mu^j$ are changed here compared to in \cite{GaSpWu:20} -- compare the definitions \cite[Equation 3.1]{GaSpWu:20} and \eqref{eq:jointmeasure}, and then the sign change in the propagation statements \cite[Lemma 4.4]{GaSpWu:20} and \eqref{eq:H_p_star}).

Examination of the proof of~\cite[Lemma 4.8]{GaSpWu:20} shows that the only time absolute continuity of the measure $\mu_1$ in that proof is used is in the higher-order glancing set. Therefore, since Lemma~\ref{l:finalEq} shows that $\mu(H_pq)=\mu_1(q)$ for some measure that is absolutely continuous with respect to $\mu$ on the glancing set, the result \eqref{eq:propagate} follows in exactly the same way as in~\cite[Equation 4.3 and Lemma 4.8]{GaSpWu:20}.
\epf

\subsection{Linking Lemma \ref{lem:key_Miller} to concepts in the applied literature}\label{sec:applied2}
The summary is that $\alpharef $ in \eqref{eq:key_Miller2} is the square of the reflection coefficient describing how plane waves interact with the boundary condition 
\beq\label{eq:appliedbc}
hD_{x_1} v(0,x')= -\frac{\alpha(x', hD_{x'})}{2}v(0,x'),
\eeq
where $\alpha$ is a semiclassical pseudodifferential operator. Indeed, when $\alpha=2$, the boundary condition \eqref{eq:appliedbc} corresponds to the first-order impedance boundary condition $(hD_{x_1}+1)v=0$ at $x_1=0$, i.e.~$(-\partial_{x_1}- i k )v=0$ (since $h=k^{-1}$). The Helmholtz solution 
\beqs
v(x) = \exp\big( ik \big( \xi' \cdot x' - \sqrt{1-|\xi'|^2}\, x_1 \big) \big) + R\exp\big( ik \big( \xi'\cdot x' + \sqrt{1-|\xi'|^2}\, x_1 \big),
\eeqs
in the half-plane $x_1>0$, 
corresponds to an incoming plane wave with unit amplitude, and an outgoing plane wave with amplitude $R$. Imposing the boundary condition 
$(\partial_{x_1}- i k )v=0$ at $x_1=0$, we obtain that 
\beqs
R= \frac{ \sqrt{1-|\xi'|^2}-1
}{
 \sqrt{1-|\xi'|^2}+1
}
\eeqs
which equals $\sqrt{\alpharef }$ when $\alpha=2$ (since $r(x',\xi')= \sqrt{1- |\xi'|^2}$ when $\Gamma$ is flat).

The interpretation of $\sqrt{\alpharef }$ as the reflection coefficient is consistent with the relation $\muout = \alpharef  \muin$ in \eqref{eq:key_Miller1}. Indeed, the defect measure of the solution $v$ of \eqref{eq:BVPimp} records where the mass of the solution is concentrated in phase space $(x,\xi)$ in the high-frequency limit $h\tendo$ (see, e.g., the discussion and references in \cite[\S9.1]{LaSpWu:19a}). 
The relation $\muout = \alpharef  \muin$ therefore describes how much mass of $|v|^2$ (since the defect measure is quadratic in $v$) is reflected from $\GammaI$.

The expression for $\alpharef $ in \eqref{eq:key_Miller2} shows that, to minimise reflection from $\GammaI$ (i.e.~to make $\alpharef $ small), $\alpha/2$ must approximate the symbol of the Dirichlet-to-Neumann map $\sqrt{r(x',\xi')}$; recall the discussion in \S\ref{sec:setup} and see, e.g.~\cite[\S3.3.2]{Ih:98} for similar discussion in this frequency-domain setting, and, e.g., \cite[Pages 631-632]{EnMa:77a}, \cite[Equation 1.12]{EnMa:79}, 
\cite[\S2.2]{Ts:98}, and \cite[\S3]{Gi:04} for analogous discussion in the time domain.

\subsection{Relationship between boundary measures and the measure in the interior}

The goal of this subsection is to prove Lemma \ref{lem:interpr} relating the measures $\muin$ and $\muout$ to the measure $\mu|_{T^*U}$. We first introduce some notation.

Recall that $\pi_{\boundary}$ is defined by \eqref{eq:pigamma}; let 
\beqs
p^{\rm out/in} : \mathcal{H} \rightarrow  \mathcal \pi_{\boundary} ^ {-1} \mathcal H \cap \big\{ \xi_1 = \xi^{\rm out/in} \big\} \subset T^*_\boundary \widetilde{M}
\eeqs
be defined by
\beq\label{eq:pinout}
p^{\rm out/in} (x',\xi'):=\big(0,x', \xi^{\rm out/in} (x',\xi'),\xi'\big)
\eeq
(i.e., $p^{\rm out/in}$ takes a point in $\cH$ and gives it outgoing/incoming normal momentum).

For $q \in \mathcal H$, let
\beq\label{eq:tout}
\tout (q) = \sup \Big\{ t>0 \,:\, \projxM  \varphi_t\big(\pout(q))\cap \big( \Gamma  \setminus  \{\projxM(q)\} \big)= \emptyset\Big\};
\eeq
i.e.~$\tout (q)$ is the positive time at which the flow starting at $t=0$ from $\pout(q)$ hits $\Gamma$ again.
Similarly, let
\beqs
\tin  (q) = \inf \Big\{ t<0 \,:\, \projxM  \varphi_t\big(\pin(q) \big)\cap \big( \Gamma  \setminus  \{\projxM(q)\} \big)= \emptyset\Big\};
\eeqs
i.e.~$\tin  (q)$ is the negative time at which the flow starting at $t=0$ from $\pin(q)$ hits $\Gamma$ again.

Given $\mathcal V\subset \mathcal H$, let $\Bout (\mathcal V), \Bin(\mathcal V)  \subset T^*U$ be defined by
\begin{align*}
\Bout (\mathcal V)  &:= \bigcup_{q\in\mathcal V} \Big\{ \varphi_{t}\big(\pout(q)\big), \,\,\, 0<t<\tout(q) \Big\}, \quad\text{ and } \\ 
\Bin(\mathcal V)  &:= \bigcup_{q\in\mathcal V} \Big\{ \varphi_{t}\big(\pout(q)\big), \,\,\, \tin (q) <t<0 \Big\}.
\end{align*}
i.e.~$\Bout(\mathcal V)$ is the union of the outgoing flows from points in $\mathcal V$ up to their times $\tout $ and 
i.e.~$\Bin(\mathcal V)$ is the union of the incoming flows from points in $\mathcal V$ up to their (negative) times $t_{\rm in}$.

The whole point of these definitions is that in $\Bout(\mathcal V)$ we can work in geodesic coordinates 
\beqs
(\rho,t) \in \big( \pi_{\boundary} ^ {-1} \mathcal V \cap \big\{ \xi_1 = \xiout \big\} \big) \times \mathbb R_+ = \pout(\mathcal V) \times \mathbb R_+,
\eeqs
defined for $(x,\xi)\in \mathcal{B}$ 
by $(x,\xi) = \varphi_t (\rho)$ (in a similar way to in the proof of Lemma \ref{lem:charge}).
Similarly, in $\Bin(\mathcal V)$ we work in geodesic coordinates 
\beqs
(\rho,t) \in \big( \pi_{\boundary} ^ {-1} \mathcal V \cap \big\{ \xi_1 = \xiin\big\} \big) \times \mathbb R_-
= \pin(\mathcal V) \times \mathbb R_-.
\eeqs

In the following
lemma, recall that the pushforward measure $f_* \mu$ is defined by $(f_* \mu)(\mathcal B) = \mu(f^{-1}(\mathcal B))$.

\begin{lem}[Relationship between boundary measures and the measure in the interior]
\label{lem:interpr}
Let $u$ satisfy \eqref{eq:helmholtzNoBoundary} with $f=o(1)$ as $h\tendo$, and let $\mu$ be a defect measure of $u$.
Let $\muout, \muin$ be defined by Lemma \ref{lem:charge}.
Then, in the geodesic coordinates described above,
$$
\mu = \big(\pout _*(2\sqrt{r}\muout )\big) \otimes dt \; \text{ on } \Bout(\mathcal V) \quad\tand\quad
\mu  = \big(\pin _*(2\sqrt{r}\muin )\big) \otimes  dt \;  \text{ on } \Bin (\mathcal V),
$$
where $dt$ denotes Lebesgue measure in $t$ and $\otimes$ denotes product measure.
\end{lem}

\begin{proof}
We prove the result for $\Bout(\mathcal V)$; the proof for $\Bin (\mathcal V)$ is similar. By 
Part (i) of Lemma \ref{lem:inv}, $\mu$ is invariant away from the boundary, therefore $\mu$ is invariant on $\pm t >0$ (away from $\Gamma$). Since
the flow is generated by $\partial_t$ in geodesic coordinates, and, in these coordinates, $\Bout \subset\{t>0\}$,
$$
\mu=\mu(\rho,t) = \mu_1 (\rho) \otimes \mathbf 1 _{t>0}dt,
$$
for some $\mu_1$.
Since $\mu|_{ x_1 < 0} = 0$,
$$
\mu_1 = \mu_1\bold 1_{\xi_1>0}, 
$$
and thus, on $\Bout$
\begin{equation}\label{eq:mu1}
\mu = \mu_1 (\rho) \bold 1_{\xi_1>0}\otimes \mathbf 1 _{t>0}dt, 
\end{equation}
from which 
\beq\label{eq:compare11}
\partial_t \mu =  \mu_1 (\rho)  \bold 1_{\xi_1>0}\otimes \delta(t). 
\eeq
On the other hand, 
since $x_1=0$ is $t=0$ in geodesic coordinates,
Lemma \ref{lem:charge} implies that
\beq\label{eq:compare12}
H_p^*\mu = \mathcal{L}^* \mu=(2\sqrt{r})\delta(t)\otimes\delta\big(\xi_{1}-\xiin\big)\otimes \muin-(2\sqrt{r})\delta(t)\otimes\delta\big(\xi_{1}-\xiout\big)\otimes\muout,
\eeq
where the factors of $2\sqrt{r}$ arise because $|\partial x_1/\partial t|= 2 |\xi_1| =2\sqrt{r}$. 

Therefore, since $\Bout(\mathcal V) \subset  \mathcal \pi_{\boundary} ^ {-1} \mathcal V \cap \{ \xi_1 = \xiout \big\}$ and $\partial_t \mu = -H_p^* \mu$, comparing \eqref{eq:compare11} and \eqref{eq:compare12}, we find that $\mu_1 = \pout_*(2\sqrt{r}\muout) $ in $\Bout$
(note that $\pout_*$ appears because $\rho = \pout(q)$ for $q\in \mathcal V$ and $\muout$ acts on $\mathcal V$). The result then follows from \eqref{eq:mu1}.
\end{proof}

The following corollary of Lemma \ref{lem:interpr} 
is an essential ingredient of our proofs of the lower bounds in Theorems \ref{th:lower}, \ref{th:quant},  \ref{th:quant2}, \ref{th:local_ball}, and \ref{th:local_square}.

\begin{cor}
\mythmname{Relationships between incoming boundary measures, outgoing boundary measures, and measures in the interior}\label{cor:explain}
Let $u$ be a solution of (\ref{eq:helmholtzNoBoundary}), and let $\mu$ be any defect measure of $u$. 

(i)  (Between two pieces of the boundary.) Let $\mathcal{V}_1 \subset \mathcal{H}$. 
Assume that $\sup_{q \in \mathcal{V}_1} \tout (q)<\infty$, 
and that $\pi_{\boundary} (\varphi_{\tout (q)}(\pout(q))) \in \mathcal H$ for all $q \in \mathcal V_1$. Let
\beqs
\mathcal{V}_2 := 
 \bigcup_{q \in \mathcal{V}_1}\pi_{\boundary}\Big( \varphi_{\tout (q)}\big(\pout(q)\big)\Big)
 \subset \mathcal H
\eeqs
(i.e.~$\mathcal{V}_2$ is the union of the outgoing flows from points in $\mathcal{V}_1$, projected into $T^{*}\boundary $).
Then 
\beq\label{eq:key1}
(2\sqrt{r}\muin)(\mathcal V_2) = (2\sqrt{r}\muout)(\mathcal V_1).
\eeq

(ii) (Between the boundary and the interior.)
Let $\mathcal{V} \subset \mathcal{H}$ and $A \subset T^* U$. 
Then
\beq \label{eq:key2a}
\mu\big(A) \geq  \bigg(\inf_{q\in\mathcal V} \int^{\tout(q)}_0   \mathbf 1_A\big(\varphi_t(\pout(q)) \big) \; dt\bigg) 
(2\sqrt{r}\muout)(\mathcal V)
\eeq
and
\beq \label{eq:key2b}
\mu\big(A) \geq  \bigg(\inf_{q\in\mathcal V} \int_{\tin (q)}^0   \mathbf 1_A\big(\varphi_t(\pin(q)) \big) \; dt\bigg) 
(2\sqrt{r}\muin)(\mathcal V).
\eeq
\end{cor}
The integrals on the right-hand sides of \eqref{eq:key2a} and \eqref{eq:key2b} are the shortest times that elements of $\mathcal{V}$ spend in $A$ under, respectively, the outgoing forward flow and the incoming backward flow, with the flows considered until they hit $\Gamma$ again.

\begin{proof}[Proof of Corollary \ref{cor:explain}]
(i) The definition of $\mathcal V_2$ implies that
$$
\mathcal B_{\rm{out}}(\mathcal V_1) = \mathcal B_{\rm{in}}(\mathcal V_2);
$$
let $\mathcal B$ denote this set.
In $\mathcal B$, we work in both sets of geodesic coordinates: 
\beqs
(\rho_1, t_1) \in \pout(\mathcal V_1)
 \times \mathbb R_+
\quad\text{ and } \quad
 (\rho_2, t_2) \in \pin(\mathcal V_2) 
 \times \mathbb R_-
 \eeqs
as defined above.
The coordinates $(\rho_{j}(q), t_{j}(q))$, $j=1,2,$ of $q\in \mathcal B$ satisfy
\beq\label{eq:storm1}
t_1  = t_+(\pout(\rho_1))+t_2
\quad\text{ and }\quad
\rho_2 = 
\varphi_{\tau(\rho_1)}(\rho_1)  =: \Phi^{1\rightarrow 2}(\rho_1).
\eeq
The first equation in \eqref{eq:storm1} implies that $dt_1 = dt_2$.
By  Lemma \ref{lem:interpr}, in $\mathcal B$,
$$
\mu = \big(\pout_*(2\sqrt{r}\muout) \big)_{\mathcal V_1} (\rho_1) \otimes dt_1 = \big(\pin_*(2\sqrt{r}\muin) \big)_{\mathcal V_2}(\rho_2) \otimes  dt_2,
$$
where the subscripts $\mathcal V_1$ and $\mathcal V_2$ show on which neighbourhood of $\mathcal H$ $\pout$, $\pin$, $\muout$, and $\muin$ are considered. This last equality and the second equation in \eqref{eq:storm1} imply that 
$$
\big(\pin_*(2\sqrt{r}\muin) \big)_{\mathcal V_2} = \Phi^{1\rightarrow 2}_* \big(\pout_*(2\sqrt{r}\muout) \big)_{\mathcal V_1}.
$$
Then
\begin{align*}
(2\sqrt{r}\muin)(\mathcal V_2) &=  \big(\pin_*(2\sqrt{r}\muin) \big)_{\mathcal V_2}\big(\pin(\mathcal V_2)),\\
&=  \big(\pin_*(2\sqrt{r}\muin) \big)_{\mathcal V_2}\big(\pi_{\boundary}^{-1} \mathcal V_2 \cap \{ \xi_1 = \xiin \}\big),\\
& =  \Phi^{1\rightarrow 2}_* \big(\pout_*(2\sqrt{r}\muout) \big)_{\mathcal V_1}\big(\pi_{\boundary}^{-1} \mathcal V_2 \cap \{ \xi_1 = \xiin \}\big),  \\
&=   \big(\pout_*(2\sqrt{r}\muout) \big)_{\mathcal V_1}\big( ( \Phi^{1\rightarrow 2})^{-1} (\pi_{\boundary}^{-1} \mathcal V_2 \cap \{ \xi_1 = \xiin \})\big),\\
& = \big(\pout_*(2\sqrt{r}\muout) \big)_{\mathcal V_1}\big( \pi_{\boundary}^{-1} \mathcal V_1 \cap \{ \xi_1 = \xiout\}\big), \\
&= \big(\pout_*(2\sqrt{r}\muout) \big)_{\mathcal V_1}\big( \pout(\mathcal V_1) ),\\
& = (2\sqrt{r}\muout)(\mathcal V_1). 
\end{align*}
(ii) We prove \eqref{eq:key2a}; the proof of \eqref{eq:key2b} is similar.
Using Lemma \ref{lem:interpr} along with the definitions of $\mathcal B_{\rm out}$, $\tout $, and the geodesic coordinates, we have
 \begin{align*}
\mu\big(\mathcal B_{\rm{out}}(\mathcal V) \cap A) &= \Big(\big(\pout_*(2\sqrt{r}\muout)\big) \otimes dt \Big) \big(\mathcal B_{\rm{out}}(\mathcal V) \cap A\big), \\
&= 
\int_{\pout(\mathcal V)}
\int^{\tout(\pi_{\boundary}(\rho))}_0  \mathbf 1_A (\rho,t) \; dt \,d\big(\pout_*(2\sqrt{r}\muout)\big)(\rho),\\
&=  
\int_{\pout(\mathcal V)}
 \int^{\tout(\pi_{\boundary}(\rho))}_0  \mathbf 
 1_{A}\big(\varphi_t(\rho)\big)
  \; dt \,d\big(\pout_*(2\sqrt{r} \muout)\big)(\rho),
\end{align*}
where we have used the fact that the point represented in geodesic coordinates by $(\rho,t)$ is in $A$ iff $\varphi_t(\rho)\in A$.
Using the change of variables $\rho= \pout(q)$, for $q \in \mathcal V$, and then Fubini's theorem,
we then have that
\begin{align*}
\mu\big(A) 
&\geq \int_{ \mathcal V} \int^{\tout(q)}_0  \mathbf 
1_A\big(\varphi_t(\pout(q)) \big)\; dt \,d(2\sqrt{r}\muout)(q),\\
&\geq  \bigg(\inf_{q\in\mathcal V} \int^{\tout(q)}_0 
\mathbf 1_A\big(\varphi_t(\pout(q)) \big)
 \; dt\bigg) (2\sqrt{r}\muout)(\mathcal V),
\end{align*}
as required.
\end{proof}

\subsection{The reflection coefficient on $\GammaI$}\label{sec:reflection_coefficient}

To understand how the defect measures of the solution $v$ of the truncated problem \eqref{eq:BVPimp} are affected by the artificial boundary $\GammaI$, we now show that the hypotheses of Part (iii) of Lemma \ref{lem:key_Miller} are satisfied, and get expressions for the numerator and denominator in the reflection coefficient $\alpharef $ in \eqref{eq:key_Miller2}.

\begin{lem} \label{lem:hierarchy_coef} 
If $v$ is the solution to \eqref{eq:BVPimp} and
\beq\label{eq:alpha}
\alpha(x',\xi') = 2\frac{\sigma(\bcD)(x',\xi')}{\sigma(\bcN)(x',\xi')},
\eeq
then, in the hyperbolic set $\mathcal H$  of $\GammaI$ 
\beq\label{eq:reflection1}
-2 \Re  \nu_j=(\Re \alpha)\nu_{d}=4(\Re \alpha)|\alpha|^{-2}\nu_{n}.
\eeq
\end{lem}

Combining \eqref{eq:alpha}, \eqref{eq:reflection1},  \eqref{eq:key_Miller1}, and \eqref{eq:key_Miller2}, we obtain the following corollary.

\begin{cor}\label{cor:reflection}
Let $v$ be the solution of \eqref{eq:BVPimp}, and let $\mu$ be a defect measure of $v$. Then, in the hyperbolic set $\mathcal{H}$ on $\GammaI$,
\eqref{eq:key_Miller1} holds with 
\beq\label{eq:reflectioncoefficient}
\alpharef=
\left|
\frac{\sqrt{r} -\sigma(\bcD)/\sigma(\bcN)}
{\sqrt{r} +\sigma(\bcD)/\sigma(\bcN)}
\right|^{2}
\eeq
\end{cor}

\begin{proof}[Proof of Lemma \ref{lem:hierarchy_coef}]
We prove that 
\beq\label{eq:measure1}
\sigma(\bcD)(x',\xi')d\nudtr=-
\sigma(\bcN)(x',\xi')d\nujtr.
\eeq
and
\beq\label{eq:measure2}
\big(\sigma(\bcD)(x',\xi')\big)^2d\nudtr
=\big(\sigma(\bcN)(x',\xi')\big)^2
d\nuntr.
\eeq
The result then follows from Part (iii) of Lemma \ref{lem:key_Miller}, since \eqref{eq:measure1} and \eqref{eq:measure2} imply that \eqref{eq:alpha_key} is satisfied.

For $a\in C_c^\infty(T^{*}\GammaIR)$, if the traces of $v$ have associated defect measures, then, as $h\tendo$,
\beq\label{eq:compare1}
\Big\langle a(x',hD_{x'})\bcN (hD_{x_1}v),v\Big\rangle\to \int a(x',\xi')\sigma(\bcN)(x',\xi')\,d\nujtr.
\eeq
On the other hand, in local coordinates, the boundary condition \eqref{eq:BVPimp2imp} is 
\beq\label{eq:localimp}
\bcN hD_{x_1}v + \bcD v=0,
\eeq
so that
\begin{align}\nonumber
\Big\langle a(x',hD_{x'})
\bcN hD_{x_1}v,v\Big\rangle&=-\Big\langle a(x',hD_{x'})
\bcD 
v,v\Big\rangle\\
&\to -\int a(x',\xi')\sigma(\bcD)(x',\xi')d\nudtr.\label{eq:compare2}
\end{align}
Comparing \eqref{eq:compare1} and \eqref{eq:compare2}, we obtain \eqref{eq:measure1}. 

We now use a similar, but slightly more involved, argument to obtain \eqref{eq:measure2}.
First observe that if $\sigma(B)$ is real and the trace of $w$ has an associated defect measure $d \mu$, then 
\begin{align}\nonumber
\big\langle  a(x', hD_{x'}) B w, B w \big\rangle &=\big\langle  B^* a(x', hD_{x'}) B w,  w \Big\rangle \\ \nonumber
&=\Big\langle a(x', hD_{x'}) B^2 + a(x', hD_{x'})(B^*-B)B + [B,a(x', hD_{x'})]B w, w\Big\rangle\\
&\to \int a(x',\xi')\big(\sigma(B)(x',\xi')\big)^2 d \mu \label{eq:abDw}
\end{align}
as $h\to 0$, since both $B^*-B$ and $[B,a(x', hD_{x'})]$ are $O(h)$(see 
\eqref{eq:symbol} and \cite[Proposition E.17]{DyZw:19}).
Therefore, \eqref{eq:abDw} with $B= \bcN$ and $w=hD_{x_1}v$ implies that
\begin{align}
\Big\langle a(x',hD_{x'})\bcN hD_{x_1}v,\bcN hD_{x_1}v\Big\rangle\to \int a(x',\xi')\big(\sigma(\bcN)(x',\xi')\big)^2 d\nuntr.\label{eq:compare3}
\end{align}
On the other hand by \eqref{eq:localimp} and  \eqref{eq:abDw} (with  $B= \bcD$ and $w=v$),
\begin{align}\nonumber
\Big\langle a(x',hD_{x'})\bcN hD_{x_1}v,\bcN hD_{x_1}v\Big\rangle&=\Big\langle a(x',hD_{x'})\bcD v,
\bcD v\Big\rangle\\
&\rightarrow \int a(x',\xi')\big(\sigma(\bcD)(x',\xi')\big)^2 d\nudtr.\label{eq:compare4}
\end{align}
Comparing \eqref{eq:compare3} and \eqref{eq:compare4}, we find \eqref{eq:measure2}.
\end{proof}

\subsection{The mass produced by the Dirichlet boundary data on $\Gamma_D$}

\begin{lem} \label{lem:nuD}
Suppose that $h_\ell\to 0$ and $a_\ell\to a$, then the defect measure of 
$$
e^{i x\cdot a_\ell/h_{\ell}}|_{\Gamma_D}
$$
is given by
$$
\dvol_{\Gamma_D} \otimes \delta_{\xi' = (a_{T(x')})^{\flat}},
$$
where $\dvol_{\Gamma_D}$ denotes Lebesgue measure on $\Gamma_D$, $a_{T(x')} := a - (a\cdot n (x')) n(x')$ is the tangential component of the direction $a$ at the point $x'$, $(\cdot)^{\flat}$ denotes the lowering map $T\Gamma_D\to T^{*}\Gamma_D$ given by the metric,
and $\delta$ denotes Dirac measure.
\end{lem}

\begin{proof} 
By using a partition of unity argument, it is sufficient to work locally in a neighbourhood of a point $x_0\in \Gamma_D$.
We work in Euclidean coordinates $\tt{x}$ such that in a neighbourhood of $x_0$, 
$$
\Gamma_{D}=\{(\gamma(\tt{x}'),\tt{x'})\}.
$$
If $a_\ell=(\tt{a_1,a'})$,  then, since $n(\tt{x'})= (1, -\nabla \gamma(\tt{x'}))/\sqrt{1+ |\nabla \gamma(\tt{x'})|^2}$, 
$$
a-(a \cdot n({\tt{x}'}))n({\tt{x}'})=\Big(\frac{{\tt{a_1}}|\nabla \gamma({\tt{x}}')|^2+\langle{\tt{ a}}', \nabla \gamma ({\tt{x}}')\rangle}{1+|\nabla \gamma({\tt{x}}')|^2},{\tt{a'}}-\frac{\langle {\tt{a'}},\nabla \gamma({\tt{x}}')\rangle -{\tt{a_1}}}{1+|\nabla \gamma({\tt{x}}')|^2}\nabla \gamma({\tt{x}}')\Big),
$$
and the metric on $\Gamma_D$ in the ${\tt{x}'}$ coordinates is 
$$
g_{ij}({\tt{x}}')=\delta_{ij}+\partial_{x_i}\gamma({\tt {x}}')\partial_{x_j}\gamma({\tt {x}}'), \qquad i,j=2,\dots,n.
$$
Therefore, since we identify the tangent space of $\Gamma_D$ with $\partial_{\tt{x}_i}$ $i=2,\dots n$
\begin{align*}
(a^T)^\flat&
={\tt{a}}'-\frac{\langle {\tt{a'}}, \nabla \gamma({\tt{x}}')\rangle-{\tt{a_1}}}{1+|\nabla \gamma({\tt{x}'})|^2}\nabla \gamma({\tt{x}'})+\Big(
\frac{{\tt{a_1}}|\nabla \gamma({\tt{x}}')|^2+\langle{\tt{ a}}', \nabla \gamma ({\tt{x}}')\rangle}{1+|\nabla \gamma({\tt{x}}')|^2}
\Big)\nabla \gamma({\tt{x'}})
\\
&={\tt{a}}'+\frac{{\tt{a_1}}-\langle {\tt{a'}}, \nabla \gamma({\tt{x}}')\rangle}{1+|\nabla \gamma({\tt{x}'})|^2}\nabla \gamma({\tt{x}'})+\Big(\langle \tt{a}',\nabla \gamma({\tt{x}'})\rangle +\frac{|\nabla \gamma({\tt{x}'})|^2}{1+|\nabla\gamma({\tt{x}'})|^2}({\tt{a_1}}-\langle{\tt{a'}},\nabla \gamma({\tt{x'}})\rangle)\Big)\nabla \gamma({\tt{x'}})\\
&={\tt{a'}}+{\tt{a_1}}\nabla \gamma({\tt{x}'}).
\end{align*}
Let $u_\ell =e^{i x\cdot a_\ell/h_{\ell}}|_{\Gamma_D}$; the previous calculation implies that $u_\ell({\tt x'}) = \exp((i/h)( {\tt a}_\ell' \cdot{\tt x}'+ {\tt a}_{\ell,1}\gamma({\tt x}'))$. 
By change of variable for the semiclassical quantisation (see, e.g., \cite[Theorem 9.3, p. 203]{Zworski_semi}, 
\begin{align*}
\big\langle b({\tt{x}'}, h_\ell D_{{\tt{x}}'}) u_\ell, u_\ell \big\rangle_{\Gsc}&= \int_{\Gsc} \big(b({\tt{x}}', h_\ell D_{x'})u_\ell \big) ({\tt{x}'}) \, \overline{ u_\ell ({\tt{x}'})} \; d{\tt{x}}'\\  &\hspace{-2.5cm}=  \int_{\Gsc} \big( b({\tt{x}}',  h_\ell D_{{\tt{x}'}}) u_\ell \big) ({\tt x}') \, \overline{ u_\ell ({\tt x}')}  \sqrt{1+|\nabla \gamma({\tt x}')|^2} \; d{\tt x}' + O(h_\ell) \\
  &\hspace{-2.5cm}= (2\pi h_\ell)^{-n+1}\int_{\Gsc} \int_{\Gsc} \int _{\mathbb R ^{n-1}} e^{\frac ih ({\tt x}'-{\tt y}')\cdot \xi'} b({\tt x}',\xi') \\ 
  &\hspace{0.75cm} \times e^{\frac ih ({\tt a}_\ell'\cdot {\tt y}' + {\tt a}_{\ell,1}\gamma({\tt y}'))}e^{-\frac ih ({\tt a}'_{\ell}\cdot {\tt x}' + {\tt a}_{\ell,1} \gamma({\tt x}'))}\sqrt{1+|\nabla \gamma({\tt x}')|^2}\;d{\xi}'d{\tt y}'d{\tt x}' + O(h_\ell).
\end{align*}
Observe that for ${\tt x}'$ fixed, the phase 
\begin{align*}
\Phi({\tt y}',\xi') 
&=  ({\tt x}'-{\tt y}')\cdot \xi' + {\tt a}_\ell'\cdot {\tt y}' + {\tt a}_{\ell,1} \gamma({\tt y}') - {\tt a}_\ell'\cdot {\tt x}' - {\tt a}_{\ell,1} \gamma({\tt x}'),\\
&=  ({\tt x}'-{\tt y}')\cdot (\xi'-  {\tt a}_\ell') + {\tt a}_{\ell,1}( \gamma({\tt y}') - \gamma({\tt x}'))
\end{align*}
is stationary 
(i.e.~$\partial_{{\tt y}'} \Phi = \partial_{{\tt \xi}'}\Phi=0$)
if and only if
$$
({\tt y}',\xi') = ({\tt x}', {\tt a}_\ell' +  \nabla \gamma({\tt x}') {\tt a}_{\ell,1}),
$$
where it is additionally non-degenerate. Consequently, by stationary phase (see, e.g., \cite[\S3.5]{Zworski_semi})
\begin{align*}
\big\langle b({\tt{x}}', h_\ell D_{{\tt{x}'}}) u_\ell, u_\ell \big\rangle_{\Gsc}&= \int_{\Gsc} b\bigg({\tt x}',{\tt a}_\ell' +  \nabla \gamma({\tt x}') {\tt a}_{\ell,1}\bigg)\sqrt{1+|\nabla \gamma({\tt x}')|^2} d{\tt x}' + O(h_\ell)\\
&=\int_{\Gsc} b\bigg({\tt x}',(a^T)^{\flat}({\tt{x}'})\bigg)\sqrt{1+|\nabla \gamma({\tt x}')|^2} d{\tt x}' + O(h_\ell).
\end{align*}
The result follows by letting $\ell\to \infty$.
\end{proof}


\section{Properties of outgoing solutions of the Helmholtz equation}\label{sec:outgoing}

The goal of this section is to prove three lemmas (Lemmas \ref{lem:mass_u}, \ref{lem:impedancetrace}, and \ref{lem:impedancetrace2}), the first two of which concern the solution to the exterior Dirichlet problem:
\begin{equation}
\label{eq:BVP3}
\begin{cases}
(-h^2\Delta-1)u=0\quad&\text{ in }\Omegaplus,\\
u=g\quad&\text{ on }\Gamma_D,\\
h\partial_r u-iu=o(r^{(1-d)/2})&\text{ as } r\tendi;
\end{cases}
\end{equation}
observe that the problem~\eqref{eq:BVP2} is a special case of~\eqref{eq:BVP3} with $g=e^{ia\cdot x/h}$.

\begin{lem}
\label{lem:mass_u}  
Suppose that $\Omegaminus\Subset B(0,1)$ is non-trapping. Then there is $C_0>0$ such that for all $R\geq 1$ there is $h_0>0$ such that for $u_h$ solving~\eqref{eq:BVP3}
$$
\Vert u_h\Vert_{H_h^1(B(0,R)\setminus\overline{\Omegaminus})}\leq C_0R^{1/2}\|g\|_{H_h^{1}(\Gamma_D)},\qquad 0<h<h_0.
$$
\end{lem}

\begin{lem}\label{lem:impedancetrace}
Let $\bcN, \bcD$ be as in \S\ref{sec:setup} (i.e., $\bcN\in \Psi^{2\NPade}(\GammaI)$, $\bcD\in \Psi^{2\MPade}(\GammaI)$ and both have real-valued principal symbols). There exists $C>0$  such that for any $R>1$ 
there exists $h_0(R)>0$ such that for $0<h\leq h_0(R)$ the solution $u$ of \eqref{eq:BVP3} satisfies
\beqs
\N{(\bcN hD_{n}-\bcD) u}_{L^2(\GammaIR)}
\leq  C \frac{\Upsilon (R) }{R^{1/2}}\N{u}_{L^2(\domain)},
\eeqs
where $n(x)$ is the normal vector field to $\GammaIR$, and
\begin{align}\nonumber
\Upsilon (R) := &\sup \bigg\{ \big|\sigma(\bcN)(x',\xi')n(x)\cdot\xi-\sigma(\bcD)(x',\xi')\big| 
+ \big|H_p\big(\sigma(\bcN)(x',\xi')n(x)\cdot\xi-\sigma(\bcD)(x',\xi')\big)\big|\,\\
&\qquad\qquad : x\in \GammaIR,\,\,\bigg|\xi \cdot \frac {x}{|x|} - 1\bigg| \leq \frac{C}{|x|^2},\,\, |\xi|=1\bigg\}.\label{eq:Upsilon}
\end{align}
\end{lem}

The quantity $\Upsilon(R)$ controls, on all rays that are approximately radial, the reflection coefficient as well as the change of the reflection coefficient under the Hamiltonian flow.

\begin{lem}[Bounds on $\Upsilon(R)$]
\label{lem:impedancetrace2} 
If $\bcN$ and $\bcD$ satisfy Assumption \ref{ass:Pade}, then the following hold.

\noi (i) There exists $C_1>0$, independent of $R$, such that if $\GammaIR = \partial B(0,R)$, then $\Upsilon(R) \leq C_1 R^{-2\morderzero}$.

\noi (ii) There exists $C_2>0$, independent of $R$, such that 
if $\GammaIR$ is $C^2$ uniformly in $R$, then $\Upsilon(R) \leq C_2$.
\end{lem}

Regarding Lemma \ref{lem:mass_u}: this result gives us a lower bound on $ 1/\|u\|_{L^2(\domain)}$, and we use this in proving the $R$-explicit lower bounds on the relative error in Theorems \ref{th:quant}, \ref{th:quant2}, \ref{th:quant3}. The analogue of this result without the explicit dependence of the constant on $R$ was proved in \cite[Theorem 3.5]{BaSpWu:16}. 

Regarding Lemmas \ref{lem:impedancetrace} and \ref{lem:impedancetrace2}: the upper bounds in Theorem \ref{th:quant} and in Theorem \ref{th:quant3} follow from 
applying Theorem \ref{th:uniformestimates} to $u-v$ and then using these two lemmas. 

\subsection{Proof of Lemma \ref{lem:mass_u}}

We define the directly-incoming set $\mathcal{I}$ by
\begin{equation}
\label{eq:incoming}
\mathcal{I}:=\bigg\{\rho\in S^{*}\domain,\text{ s.t. }\projx \bigg(\bigcup_{t\geq0}\varphi_{-t}(\rho)\bigg)\cap\Omegaminus=\emptyset\bigg\},
\end{equation}
where we recall that $\projx $ denotes projection in the $x$ variable.
The following lemma reflects the fact that
$u$ is an outgoing solution.

\begin{lem} \label{lem:WFu}
If $u$ solves~\eqref{eq:BVP3} with $\|g\|_{H_h^1}\leq C$, then
$$
\WFh (u)\cap \mathcal{I}=\emptyset.
$$
In particular, there exists $C>0$, sufficiently large, such that
$$
\WFh (u)\cap \{|x|>C\} \subset \left\{ \ \left| \xi - \frac{x}{|x|} \right| < \frac{C}{|x|},\quad \left|\xi\cdot \frac{x}{|x|}-1\right|\leq \frac{C}{|x|^2} \ \right\}.
$$

\end{lem} 
\begin{proof}
Let $R_D$ be the outgoing resolvent for 
$$
(-h^2\Delta-1)w=f,\qquad w|_{\Gsc}=0,
$$ 
i.e., $w = R_D f$.
Fix $0<R_1<R_2$ such that $\Omegaminus\subset B(0,R_1)$, and let $\chi_i\in C_c^\infty(B(0,R_{2}))$, $i=0,1,2$, with $\chi_i\equiv 1$ on $B(0,R_1)$, $\supp \chi_i\subset \{\chi_{i+1}\equiv 1\}$.
We now extend the Dirichlet boundary data off $\Gamma_D$ by letting $\widetilde{g}$ be the solution of 
\begin{align*}
(-h^2\Delta -1)\widetilde{g}=0 &\quad\text{ in } \Omegaplus\cap B(0,R_1),\\
\widetilde{g} = g &\quad\ton \Gamma_D,\\
(hD_n -1)\widetilde{g} =0 &\quad\ton \partial B(0,R_1).
\end{align*}
We now show that $u$ can be expressed as an outgoing resolvent plus a function with compact support. To this end, let
$$
v:=u-\chi_0 \widetilde{g}-R_D\big([-h^2\Delta,\chi_0]\widetilde{g})\big),
$$
and observe that $(-h^2 \Delta -1)v=0$.
Since the Dirichlet Laplacian is a black box Hamiltonian in the sense of~\cite[Chapter 4]{DyZw:19}, by~\cite[Theorem 4.17]{DyZw:19}, the radiation condition for $u$ implies that $v=0$ and hence $u=\chi_0\widetilde{g}+R_D([-h^2\Delta,\chi_0]\widetilde{g}).$ Now, by, e.g.,~\cite[Theorem 4.4]{DyZw:19}, the range of $(1-\chi_2)R_D$ lies in the range of $R_0\chi_1 $ where $R_0$ denotes the free resolvent. In particular, by the outgoing property of $R_0$ (see e.g.~\cite[Theorem 3.37]{DyZw:19}) 
\begin{equation}
\label{eq:tempOutgoing}
\WFh (u)\cap \{|x|>R_2+1\}\subset \bigcup_{t\geq 0}\varphi_t
( S_{B(0,R_2)}^*\mathbb{R}^d).
\end{equation}

Now, suppose that $A\subset \mathcal{I}$, where $\mathcal{I}$ is as in~(\ref{eq:incoming}). Then, for $t_0\geq 0$ large enough, 
$$\varphi_{-t_0}(A) \subset \{|x|>R_2+1\}$$
and, moreover, 
$$
\bigcup_{t\leq -t_0}\varphi_{t}(A)\cap S_{B(0,R_1)}^*\mathbb{R}^d=\emptyset.
$$
Therefore, by~(\ref{eq:tempOutgoing}), $\varphi_{-t_0}(A)\cap \WFh (u)=\emptyset.$ Now, since $(h^2\Delta+1)u=0$, and 
$$
\bigcup_{-t_0\leq t\leq 0} \varphi_t(A)\cap S^*_{\Gsc}\mathbb{R}^d=\emptyset,
$$
by propagation of singularities (see e.g.~\cite[Appendix E.4]{DyZw:19}), $A\cap \WFh (u)=\emptyset$.

Now, suppose $(x,\xi)\in \WFh (u)\cap \{|x|\geq R\}$. Then, $(x,\xi)\notin \mathcal{I}$ and, in particular, there is $t\geq 0$ such that $\varphi_{-t}(x,\xi)\in S^*_{\Omegaminus}\mathbb{R}^d.$ 
Let
$$
t_0=\inf\{ t\geq 0\,:\, \varphi_{-t}(x,\xi)\in S^*_{\Omegaminus}\mathbb{R}^d\}
$$
and $(x_0,\xi_0)=\varphi_{-t_0}(x,\xi)$. Then, $|x_0|\leq R_1$, $t_0\geq \frac{R-R_0}{2}$, $\xi=\xi_0$, and 
$$
x=x_0+2t_0\xi_0.
$$
Observe that 
$$
|x_0+2t\xi_0|=\sqrt{|x_0|^2+4t\langle x_0,\xi_0\rangle +4t^2}=2t\sqrt{1+|x_0|^2t^{-2}+2t^{-1}\langle x_0,\xi_0\rangle}=2t+O(t^{-1}|x_0|^2).
$$
Then consider 
$$
\left|\frac{x}{|x|}-\xi\right|=\left|\frac{x_0+2t\xi_0}{|x_0+2t\xi|_0}-\xi_0\right|=\left|\frac{x_0+\xi_0 O(t^{-1}|x_0|^2)}{|x_0+2t\xi_0|}\right|=O(t^{-1}|x_0|)=O\left(\frac{R_1}{|x|-R_1}\right).
$$
In particular, if $R\geq 2R_1$, $|x|-R_1\geq \frac{1}{2}|x|$.

Next, observe that
$$
\xi\cdot \frac{x}{|x|}=\frac{x_0\cdot \xi_0+2t}{|x_0+2t\xi_0|},\qquad 
|x_0+2t\xi_0|^2=|x_0|^2+4t^2+4tx_0\cdot \xi_0
$$
so that 
$$
\frac{1}{|x_0+2t\xi_0|}=\frac{1}{2t}\Big(1-\frac{x_0\cdot \xi_0}{2t}+O(R_1^2t^{-2})\Big).
$$
In particular, 
$$
\frac{x_0\cdot \xi_0+2t}{|x_0+2t\xi_0|}=1+\frac{x_0\cdot \xi_0}{2t}-\frac{x_0\cdot \xi_0}{2t}+O(R_1^2t^{-2})=1+O(R_1^2t^{-2})=1+O\left(\frac{R_1^2}{(|x|-R_1)^2}\right).
$$
Taking $|x|\geq 2R_1$ completes the proof.
\end{proof}

\begin{cor} \label{cor:outgo}
There exists $t_0>0, r_0>0$ such that, if $u$ solves~\eqref{eq:BVP3} and has defect measure $\mu$, then for any $r\geq r_0$, if $(x,\xi) \in \supp \mu$ with $|x| = r$, then,  for $0\leq t \leq r-t_0$,
$$ |x(\varphi_{-t}(x,\xi))|^2 = |x -2 t\xi|^2 = (r-2t)^2+O(tr^{-1}).$$
\end{cor}

\begin{proof}
This follows from Lemma \ref{lem:WFu} by observing that, by the definition of defect measures, $\supp  \mu \subset \WFh  (u)$; then, if $|x| = r$ and $ |\xi| = 1$ with $| \xi\cdot \tfrac{x}{|x|} -1 | < \frac{C}{r^2}$, then $x\cdot\xi \geq r-\frac{1}{r}$.
\end{proof}

By the definitions of $\WFh (u)$ and $\mathcal{I}$, another corollary of Lemma \ref{lem:WFu} is the following lemma, originally proved in \cite[Proposition 3.5]{Bu:02} (see also \cite[Lemma 3.4]{GaSpWu:20}).

\begin{lem}
\label{lem:null_incom}Suppose that $u$ solves~\eqref{eq:BVP3} and has defect measure $\mu$. Then $\mu(\mathcal{I})=0$.
\end{lem}

We now prove Lemma \ref{lem:mass_u}.

\begin{proof}[Proof of Lemma \ref{lem:mass_u}]
Suppose that the lemma fails. Then there exist  $R\geq 1$, $\e>0$, $(h_{\ell}, g_\ell)$ such that $h_{\ell}\rightarrow 0$ as $\ell \to \infty$ and such that 
\begin{equation}
\label{e:contra1}
\|u_{h_{\ell}}\|_{H_{h_\ell}^1(B(0,R)\setminus\overline{\Omegaminus})}=  1\quad\tand\quad \|g_{\ell}\|_{H_{h_{\ell}}^1(\Gamma_D)}\leq \frac{1}{R^{1/2}(C_0+\e)}.
\end{equation}
Let $w_{\ell}$ solve
$$
(-h^2_{\ell}\Delta -1)w_{\ell}=0,\qquad w_{\ell}|_{\Gamma_D}=g_{\ell},\qquad (hD_n-1)w_{\ell}|_{\partial B(0,1)}=0.
$$
Since Lemma \ref{lem:mass_u} is not used in the proof of Theorem~\ref{th:uniformestimates}, the upper bound in this latter result implies that 
there exists a $C_1>0$ such that
$$
\|w_{\ell}\|_{H_{h_{\ell}}^1(B(0,1)\setminus \overline{\Omegaminus})}\leq C_1\|g_{\ell}\|_{H_{h_\ell}^1(\Gamma_D)}.
$$
Let $\chi \in C_c^\infty(B(0,1))$ with $\chi \equiv 1$ near $\Gamma_D$ and put $v_{\ell}=u_{\ell}- \chi w_{\ell}$ so that 
$$
\begin{cases}
(-h_{\ell}^2\Delta-1)v_{\ell}=-(-h_\ell^2\Delta_{\ell}-1)\chi w_{\ell}=: h_\ell f_{\ell}\\
v_{\ell}|_{\Gamma_D}=0\\
(h_{\ell}D_n-1)v_{\ell}=o(r^{(1-d)/2}),
\end{cases}
$$
and $\|f_{\ell}\|_{L^2}\leq C_2\|w_\ell\|_{H_h^1}\leq C_2C_1\|g_{\ell}\|_{H_h^1(\Gamma_D)}$, $\supp f_{\ell}\subset B(0,1)$.  In particular, by e.g.~\cite[Theorem 1]{GaSpWu:20} there is $C_3>0$ such that for any $\psi\in C_c^\infty$ with $\psi\equiv 1$ on $B(0,1)$ and $\supp \psi \subset B(0,R_0)$, and any $h_\ell$ small enough,
\begin{equation}
\label{e:spider}
\|\psi v_{\ell}\|_{H^1_{h_\ell}}\leq  C_3 R_0\|f_{\ell}\|_{L^2}\leq R_0 C_1C_2C_3\|g_{\ell}\|_{H_h^1(\Gamma_D)}.
\end{equation}
Now, taking $C_0\geq C_1(3C_2C_3+1)$ the proof is complete for $1\leq R\leq 2$. To see this, observe that using~\eqref{e:spider} with $R_0=3$ and $\psi \equiv 1$ on $B(0,2)$
$$
\| u_{h_\ell}\|_{H_{h_\ell}^1 (B(0,2)\setminus\overline{\Omegaminus})}\leq \|\psi(v_\ell +\chi w_\ell)\|_{H_{h_\ell}^1}\leq \|\psi v_\ell\|_{H_{h_\ell}^1} +\|\chi w_\ell\|_{H_{h_\ell}^1}\leq C_1(3C_2C_3+1)R^{1/2}\|g\|_{H_h^1}< 1
$$
which contradicts~\eqref{e:contra1}.

Now, for $R\geq 2$, we can pass to a subsequence in $\ell$, and assume that $v_{{\ell}}$ has defect measure $\mu$. 
By Lemma~\ref{lem:null_incom}, $\mu(\mathcal{I}\cap T^*\widetilde{M}\setminus \supp f)=0$ and 
$$
\mu(H_pa)=0,\qquad a\in C_c^\infty(T^*\widetilde{M}\setminus \supp f).
$$ 

Therefore, since $\supp f\subset B(0,1)$
$$
\supp \mu \cap T^*\widetilde{M}\setminus B(0,2)\subset \bigcup_{t\geq 0}\varphi_t\Big(\big\{ (x,\xi)\,:\, |x|=2, \,\exists s>0\text{ s.t. } \varphi_{-s}(x,\xi)\in T^*B(0,1)\big\}\Big).
$$
In particular, since $\mu$ is invariant under $\varphi_t$ on $T^*(\mathbb{R}^d\setminus B(0,1))$,
\begin{align*}
\mu (T^*B(0,R)\setminus B(0,2))&\leq \mu\Big(\bigcup_{0\leq t\leq \sqrt{R^2-4}}\varphi_t\Big(\big\{ (x,\xi)\,:\, |x|=2,\,\,\exists s>0\text{ s.t. } \varphi_{-s}(x,\xi)\in T^*B(0,1)\big\}\Big)\Big)\\
&=\sqrt{R^2-4}\,\mu\Big(\bigcup_{-1\leq t\leq 0}\varphi_t\Big(\big\{ (x,\xi)\,:\, |x|=2,\,\exists s>0\text{ s.t. } \varphi_{-s}(x,\xi)\in T^*B(0,1)\big\}\Big)\Big)\\
&\leq \sqrt{R^2-4}\,\lim_{\ell \to \infty}\|v_\ell\|_{L^2(B(0,2)-B(0,1))}^2\\
&\leq 9C_1^2C_2^2C_3^2\sqrt{R^2-4}\lim_{\ell \to \infty}\|g_{\ell}\|^2_{H_h^1(\Gamma_D)}\\
&\leq \frac{3C_1C_2C_3\sqrt{R^2-4}}{R(C_0+\e)^2}.
\end{align*}
By~\cite[Lemma 4.2]{GaSpWu:20}
$$
\mu\Big(|\xi|^2{1}_{T^*B(0,R)\setminus B(0,2)}\Big)\geq \limsup_{\ell \to \infty} \|v_\ell\|_{H_{h_\ell}^1(B(0,R)\setminus B(0,5/2))}^2.
$$
Therefore, using~\eqref{e:spider} with $R_0=3$, $\psi \equiv 1$ on $B(0,5/2)$,
$$
\limsup_{\ell \to \infty}\|v_{\ell}\|^2_{H_{h_\ell}^1(B(0,R))}\leq \frac{9C^2_3C^2_2C^2_1(1+\sqrt{R^2-4})}{R(C_0+\e)^2}.
$$
Hence, letting 
$$
C_0=C_1\max\left(3C_2C_3+1,\sup_{R\geq 2}\frac{3C_3C_2\sqrt{1+\sqrt{R^2-4}}+1}{R^{1/2}}\right),
$$
we have
$$
\limsup_{\ell \to \infty}\|u_{h_\ell}\|_{H_{h_\ell}^1(B(0,R))}\leq  \frac{3C_3C_2C_1\sqrt{1+\sqrt{R^2-4}}+C_1}{R^{1/2}(C_0+\e)}<1,
$$
which contradicts~\eqref{e:contra1}.
\end{proof}

\subsection{Proof of Lemmas \ref{lem:impedancetrace} and \ref{lem:impedancetrace2}}

In the next lemma, we identify $S^*\GammaI$ with a subset of $S^*\mathbb{R}^d$.
\begin{lem}
\label{l:traceEst}
Suppose that $A\in \Psi^m(\mathbb{R}^d)$ and $\WFh '(A)\cap S^*\GammaI=\emptyset.$ Then there is $C>0$ such that
$$
\|Au\|_{L^2(\GammaI)}\leq C\|Au\|_{L^2} +Ch^{-1}\|PAu\|_{L^2}+O(h^\infty)\|u\|_{L^2}.
$$
\end{lem}
\begin{proof}
First, note that for $B\in \Psi^0$ with $\WFh '(B)$ supported away from $S^*\mathbb{R}^d$, we can write using the elliptic parametrix construction (Lemma~\ref{l:ellip}) that there is $E\in \Psi^{-2}$ such that 
$$
BAu=EPAu+O(h^\infty)_{\Psi^{-\infty}}.
$$
In particular, by the Sobolev embedding as in~\cite[Lemma 5.1]{Ga:19} see also~\cite[Lemma 7.10]{Zworski_semi}, 
\begin{align*}
\|BAu\|_{L^2(\GammaI)}\leq Ch^{-1/2}\|BAu\|_{H^1_{h}}&\leq Ch^{-1/2}\|EPAu\|_{H_h^1}+O(h^\infty)\|u\|_{L^2}
\\&\leq Ch^{-1/2}\|PAu\|_{L^2}+O(h^\infty)\|u\|_{L^2}.
\end{align*}
Therefore, we can assume that 
\begin{equation*}
\WFh '(A)\subset \big\{1-\delta \leq |\xi|^2\leq 1+\delta\big\}
\end{equation*}
for any $\delta>0$. Next, if $\WFh '(A)\cap S^*_{\GammaI}\mathbb{R}^d=\emptyset$, then there is $\chi\in C_c^\infty(\mathbb{R}^d)$ with $\chi\equiv 1$ in a neighbourhood of $\GammaI$ such that 
$$
\chi A=O(h^\infty)_{\Psi^{-\infty}}.
$$
In particular, 
$$
\|\chi Au|_{\GammaI}\|_{L^2(\GammaI)}=O(h^\infty)\|u\|_{L^2}.
$$

By using a partition of unity, we can work locally, assuming that $\GammaI=\{x_1=0\}$ as in \S\ref{subsec:geo}. 
We can then assume that $\WFh '(A)\subset \{|x_1|<\delta\}$. Write $A=a(x,hD)$ where $d(\supp a, \{r(x,\xi)=0\})>\epsilon>0$ and $\supp a\subset \{|x_1|<\delta\}$ for some $\epsilon >0$. Then, choosing $\delta>0$ small enough, we have $|\xi_1|>0$ on $\supp a$ and hence there is $e\in C_c^\infty(T^*\mathbb{R}^d)$ with $|e|>0$ on $\supp a$ and such that 
$$
e(x,\xi)(\xi_1-b(x,\xi'))a(x,\xi')=(-\xi_1^2+r(x,\xi'))a(x,\xi).
$$
Therefore,
$$
\|(hD_{x_1}-b(x,hD_{x'}))Au\|_{L^2}\leq C\|PAu\|_{L^2}+O(h)\|Au\|_{L^2};
$$
the result then follows by applying~\cite[Lemma 7.11]{Zworski_semi}.
\end{proof}

\begin{lem} \label{lem:crown}
Let $u$ be the solution to \eqref{eq:BVP3}.
For any $\eta>0$, there exists $R_0>0$ such that, for $R\geq R_0$ and $h$ small enough (depending on $R$)
\beq\label{eq:crown}
\Vert u \Vert_{L^2(B(0,R+1)  \setminus B(0,R-1))} \leq (\sqrt{2}+\eta) {R^{-\frac{1}{2}}} \Vert u \Vert _ {L^2(B(0,R))}.
\eeq
\end{lem}

\begin{proof}
We define $A_{r_0, r_1} := \overline{B(0,r_0)}  \setminus B(0,r_1)$.
First, observe that it is sufficient to prove that there exists $R_1(\eta)>0$ such that, for any $R\geq R_1$ and any $u$ solving~\eqref{eq:BVP3} having defect measure $\mu$,
\begin{equation} \label{eq:crown_defect}
\mu (T^* A_{R+1, R-1}) < \frac{(\sqrt{2}+\eta)^2}{R}  \mu (T^* B(0,R)).
\end{equation}
Indeed, if \eqref{eq:crown} fails, then there exists $\eta>0$ and $h_n \rightarrow 0$ and $g_n\in H_h^1(\Gamma_D)$ such that, for $u(h_n)$ solving~\eqref{eq:BVP3} with $g=g_n$ and some $R\geq R_1 (\eta)$,
\begin{equation}
\label{e:spider2}
\Vert u(h_n) \Vert_{L^2(A_{R+1,R-1})} >  \frac{\sqrt{2}+\eta}{R^{1/2}} \Vert u(h_n) \Vert _ {L^2(B(0,R))}.\qquad \Vert u(h_n) \Vert _ {L^2(B(0,R))}=1.
\end{equation}
Then, passing to a subsequence, we can assume that $u(h_n)$ has defect measure $\mu$. 
Let $\epsilon>0$ be arbitrary. Take $\chi^\epsilon_0$ equal to one in $A_{R+1,R-1}$ and supported in $A_{R+1+\epsilon,R-1-\epsilon}$ and $\chi^\epsilon_1$ supported in $B(0,R)$ and equal to one in $B(0,R-\epsilon)$. The estimate~\eqref{e:spider2} implies
$$
\Vert \chi_0^\epsilon u(h_n) \Vert_{L^2} >  \frac{\sqrt{2}+\eta}{R^{1/2}} \Vert \chi_1^\epsilon u(h_n) \Vert _ {L^2},
$$
passing to the limit $h_n \rightarrow 0$ and using e.g.~\cite[Lemma 4.2]{GaSpWu:20} we obtain 
\beqs
\mu((\chi_0^\epsilon)^2)\geq \frac{(\sqrt{2}+\eta)^2}{R} \mu((\chi_1^\epsilon)^2),
\eeqs
which in turn implies, by the support
properties of $\chi_{0,1}$,
$$
\mu (T^* A_{R+1+\epsilon, R-1-\epsilon}) \geq \frac{(\sqrt{2}+\eta)^2}{R}\mu (T^* B_{R-\epsilon}).
$$
In particular, sending $\e\to 0^+$, and using monotonicity of measures
$$
 \mu (T^* A_{R+1, R-1}) \geq \frac{(\sqrt{2}+\eta)^2}{R} \mu (T^* B_{R}),
$$
which contradicts~\eqref{eq:crown_defect}.

We therefore only need to prove (\ref{eq:crown_defect}). The definition of defect measures implies $\supp \  \mu \subset \text{WF}_h(u)$, thus, by Lemma \ref{lem:WFu},
$$
\supp   \mu \,\cap \{|x|>C\}\subset \left\{ \left|  \xi\cdot\frac{x}{|x|} -1 \right| < \frac{C}{|x|^2} \right\}.
$$
Now, invariance of defect measures away from the obstacle combined with the above implies that, for $r_0>C+2$, so that $\Omegaminus \subset B(0,r_0 - 2)$, and $0\leq t\leq1$,
$$
 \mu (T^* A_{r_1, r_0}) =  \mu \left( \varphi_{-t} \left(T^* A_{r_1, r_0} \cap \left\{ |\xi| = 1, \ \left|  \xi\cdot\frac{x}{|x|} - 1\right| < \frac{C}{|x|^2}\right\} \right) \right).
$$
By Corollary \ref{cor:outgo}, there exist $C_0,C_1,C_2>0$ such that 
\begin{equation*}
\begin{gathered}
\varphi_{-\frac{1}{2}-C_0R^{-2}}\big(T^* A_{R+1,R-1}\cap \supp {\mu}\big)\cap T^*\big\{|x|\geq R\big\}=\emptyset,\\
\varphi_{-1-1C_0R^{-2}}\big(T^* A_{R+1,R-1}\cap \supp {\mu}\big)\subset T^*\big\{|x|< R-1\big\}.
\end{gathered}
\end{equation*}
Fix $r_0>0$ such that $\Omegaminus\Subset B(0,r_0)$. Then, for $0\leq 2t\leq R-1-r_0$, we have $\varphi_{-t}(S^* A_{R+1,R-1})\cap B(0,r_0)=\emptyset$. Therefore, using the fact that $\langle x,\xi\rangle>0$ on $\supp {\mu}\cap T^*A_{R+1,R-1}$, we have 
\begin{equation}
\label{eq:nonSelfLoop}
\varphi_{-t}(T^* A_{R+1,R-1}\cap \supp {\mu})\cap T^* A_{R+1,R-1}\cap \supp {\mu}=\emptyset\,\,\tfor\,\, t\in \left[1+C_0R^{-2},\frac{ R-1-r_0}{2}\right].
\end{equation}
Now, let $T_{1,R}:=(R-1-r_0)/2$ and $T_{0,R}:=1+C_0R^{-2}$ and consider
$$
f_{T,R}(x,\xi):=\int_{T_{0,R}}^{T_{1,R}} 1_{T^*A_{R+1,R-1}\cap \supp {\mu}}\circ\varphi_{t}(x,\xi)dt.
$$
We claim that $0\leq f_{T,R}\leq T_{0,R}$; to see this, suppose that $\varphi_t(x,\xi)\in T^*A_{R+1,R-1}\cap \supp {\mu}$ and $\varphi_{s}(x,\xi)\in T^*A_{R+1,R-1}\cap \supp {\mu}$ with $T_{0,R}\leq s\leq t-T_{0,R}$ and $t\leq T_{1,R}$. Then,
$$
\varphi_{-(t-s)}(x,\xi)\in T^*A_{R+1,R-1}\cap \supp {\mu},\qquad (x,\xi)\in T^*A_{R+1,R-1}\cap \supp {\mu}
$$
and $T_{0,R}\leq t-s\leq T_{1,R}$, contradicting~(\ref{eq:nonSelfLoop}).

Now, since ${\mu}$ is $\varphi_t$ invariant,
$$
(T_{1,R}-T_{0,R})\,{\mu}(1_{T^*A_{R+1,R-1}})={\mu}(f_{T,R}(x,\xi))\leq T_{0,R}\,{\mu}(B(0,R)).
$$
In particular, 
\begin{align*}
{\mu}(1_{T^*A_{R+1,R-1}})\leq \frac{T_0,R}{T_{1,R}-T_{0,R}}{\mu}(B(0,R))\leq \frac{2}{R}(1+O(R^{-1})){\mu}(B(0,R)).
\end{align*}
Choosing $R>0$ large enough yields \eqref{eq:crown_defect}, and the proof is complete.
\end{proof}

We now prove Lemmas \ref{lem:impedancetrace} and \ref{lem:impedancetrace2}.

\bpf[Proof of Lemma \ref{lem:impedancetrace}]
Let $\widetilde{n}$ be a smooth extension of the normal vector field to $\GammaIR$, $n_R(x)$ and $C_0>0$ so that the conclusions of~Lemma~\ref{lem:WFu} hold, and, $\widetilde{\bcN}$, $\widetilde{\bcD}$ smooth extensions of $\bcN$ and $\bcD$. Next, fix $\e>0$ such that
$$
\sup \Big\{ \big|\widetilde{\bcN}hD_{\widetilde{n}}-\widetilde{\bcD}\big|+\big|H_p(\widetilde{\bcN}hD_{\widetilde{n}}-\widetilde{\bcD})\big|\,:\, \dist(x,\GammaIR)<\e, \, \Big|\xi\cdot \frac{x}{|x|}-1\Big|\leq \frac{C_0}{|x|^2},\, \big||\xi|-1\big|<\e\Big\}\leq 2\Upsilon(R).
$$
 and let $\chi$ be smooth, supported in 
$$
\Gamma_\e:=\big\{x\,:\, \dist(x,\GammaIR)<\e\big\},
$$
and equal to one near $\GammaIR$. By Lemma~\ref{lem:WFu}, we can find $Z \in \Psi(\mathbb{R}^d)$ with $\WFh '(Z)\cap \mathcal{I}=\emptyset$ such that 
$$
\chi u=\chi Zu+O_{C^\infty}(h^\infty\|u\|_{L^2}).
$$
Now, since $\tildedomain$ is convex, and $\Omegaminus\Subset \tildedomain$, $S^*\GammaI\subset \mathcal{I}$. In particular, by Lemma~\ref{l:traceEst},
\begin{align*} 
\Vert (\bcN hD_n-\bcD) u \Vert_{L^2(\GammaI)} &=\Vert (\bcN hD_n-\bcD) \chi Zu \Vert_{L^2(\GammaI)} +O(h^\infty)\|u\|_{L^2}\\
&\hspace{-2cm}\leq C\Vert(\widetilde{\bcN} hD_{\widetilde{n}}-\widetilde{\bcD})\chi Zu\Vert_{L^2} +Ch^{-1}\Vert (-h^2\Delta-1)(\widetilde{\bcN} hD_{\widetilde{n}}-\widetilde{\bcD})\chi Zu\Vert_{L^2}+ O(h^\infty)\|u\|_{L^2}\\
&\hspace{-2cm}=C\Vert (\widetilde{\bcN} hD_{\widetilde{n}}-\widetilde{\bcD})\chi u\Vert_{L^2} +Ch^{-1}\Vert (-h^2\Delta-1)(\widetilde{\bcN} hD_{\widetilde{n}}-\widetilde{\bcD})\chi u\Vert_{L^2}+ O(h^\infty)\Vert u\Vert_{L^2}\\
&\hspace{-2cm} \leq C\Vert (\widetilde{\bcN} hD_{\widetilde{n}}-\widetilde{\bcD}) \chi u \Vert_{L^2} +  Ch^{-1}  \Vert (\widetilde{\bcN} hD_{\widetilde{n}}-\widetilde{\bcD})(-h^2\Delta-1) \chi u \Vert_{L^2} \\ 
&  \qquad +  Ch^{-1}  \Vert [-h^2\Delta-1,\widetilde{\bcN} hD_{\widetilde{n}}-\widetilde{\bcD}] \chi u \Vert_{L^2}+O(h^\infty) \Vert u\Vert_{L^2},
\end{align*} 
and, using the fact that $(-h^2 \Delta - 1)u = 0$,
\begin{align}\nonumber
\Vert (\widetilde{\bcN} hD_{\widetilde{n}}-\widetilde{\bcD}) u \Vert_{L^2(\GammaI)} 
&\leq \Vert (\widetilde{\bcN} hD_{\widetilde{n}}-\widetilde{\bcD})\chi u \Vert_{L^2}  +  h^{-1}  \Vert (\widetilde{\bcN} hD_{\widetilde{n}}-\widetilde{\bcD}) [h^2 \Delta + 1,\chi]  u \Vert_{L^2} \\
&\qquad+  h^{-1} \Vert [-h^2\Delta-1,\widetilde{\bcN} hD_{\widetilde{n}}-\widetilde{\bcD}] \chi u \Vert_{L^2}.
 \label{eq:upper_dec}
\end{align}
Let 
$$
R_1:=\sup\Big\{R\,:\, \GammaIR\cap B(0,C_0+1)\neq \emptyset\Big\}.
$$
Then, for $1\leq R\leq R_1$, the proof is completed, since $\|Bu\|_{H_h^1}+h^{-1}\|[B,(-h^2\Delta-1)]u\|_{L^2}\leq C_B\|u\|_{L^2}$ for any $B\in \Psi^\infty$.  We now consider the case $R\geq C_0$. 

Observe that, by Lemma \ref{lem:WFu},
\begin{equation} \label{eq:up_WFstat}
\WFh \big(  \chi u \big) \subset \supp\chi \cap \WFh ( u ) \subset \left\{ \Big| \xi\cdot\frac{x}{|x|} -1 \Big| < \frac{C}{|x|^2}, \ x\in\Gamma_\epsilon ,\,|\xi|=1\right\}.
\end{equation}
Now, let $\widetilde{\chi}\in C_c^\infty(\mathbb{R}^d)$ with $\widetilde{\chi}\equiv 1$ on $\supp\chi$ with $\supp \widetilde{\chi}\subset \Gamma_\e$,  and $\psi\in C_c^\infty(T^*\mathbb{R}^d)$ with
$$
\supp\psi\subset \left\{ \left| \xi\cdot\frac{x}{|x|}  -1 \right| \leq \frac{2C}{|x|^2},\quad \big||\xi|-1\big|<\e\right\},$$
with $\psi\equiv 1$ on 
$$\left\{ \left|\xi\cdot\frac{x}{|x|}  - 1\right| < \frac{C}{|x|^2},\,|\xi|=1 \right\}.$$ and $\Psi := \text{Op}_h (\psi \widetilde \chi)$. By (\ref{eq:up_WFstat})
$$
 \Vert (\widetilde{\bcN} hD_{\widetilde{n}}-\widetilde{\bcD})\chi u \Vert_{L^2}  =   \Vert \Psi (\widetilde{\bcN} hD_{\widetilde{n}}-\widetilde{\bcD})\chi u \Vert_{L^2} + O(h^\infty) \Vert \chi u \Vert_{L^2},
 $$
where $\Psi  (\widetilde{\bcN} hD_{\widetilde{n}}-\widetilde{\bcD})$ has principal $h$-symbol 
\beq\label{eq:Lambda}
\Lambda(x,\xi) := \psi \widetilde \chi (\widetilde{\bcN}(x,\xi)\xi \cdot \widetilde{n}(x) - \widetilde{\bcD}(x,\xi)),
\eeq
and thus $\Psi  (\widetilde{\bcN}hD_{\widetilde{n}}-\widetilde{\bcD}) = \text{Op}_h(\Lambda) + O(h)_{L^2\rightarrow L^2}$, and then, by\cite[Theorem 5.1]{Zworski_semi},
$$
\Vert \Psi \chi u \Vert_{L^2} \leq \left(\sup \big|\Lambda(x,\xi)\big|+O\big(h^{1/2}\big)\right)\Vert \chi u\Vert_{L^2}.
$$
However, by the support properties of $\widetilde \chi$ and $\psi$ and the definition \eqref{eq:Lambda} of $\Lambda$,
$$
\sup \big|\Lambda(x,\xi)\big| \leq \Upsilon(R),
$$
and it follows that, given $R>0$, there exists $h_0(R)>0$ such that, for $0 < h \leq h_0$,
\begin{equation} \label{eq:upper_e1}
 \Vert  (\widetilde{\bcN}hD_{\widetilde{n}}-\widetilde{\bcD})\chi u \Vert_{L^2}  \lesssim \Upsilon(R) \Vert {\chi} u \Vert_{L^2}.
\end{equation}
 
 On the other hand, by Lemma \ref{lem:WFu}, 
 $$
 \WFh ( [-h^2 \Delta - 1,\chi]u) \subset \Big\{ \big| \xi\cdot \frac{x}{|x|} - 1\big| < \frac{C}{|x|^2}, \ x\in\Gamma_{\e},\,|\xi|=1 \Big\};
 $$
we  obtain in the same way as before, reducing $h_0$ if necessary, that for $0 < h \leq h_0$ 
 \begin{equation} \label{eq:upper_commu}
\Vert  (\widetilde{\bcN}hD_{\widetilde{n}}-\widetilde{\bcD}) [-h^2 \Delta - 1,\chi]  u  \Vert_{L^2}  \lesssim  \Upsilon(R) \Vert  [-h^2 \Delta - 1,\chi] u \Vert_{L^2} \lesssim   \Upsilon(R) h \Vert \chi_0 u \Vert_{H^1_h},
 \end{equation}
where $\chi_0$ is supported in the support of $\widetilde \chi$ and equal to one on the support of $\chi$. But, since  $(-h^2 \Delta - 1)u = 0$, $u$ has $h$-wavefront set in $\{ |\xi|^2 = 1 \}$, thus so does $\widetilde \chi u$, and it follows that, taking $\eta$ compactly supported near one
\begin{align}
\Vert  \chi_0 u \Vert_{H_h^1} &=  \Vert \text{Op}_h(\eta(|\xi|^2)) \chi_0 \widetilde{\chi} u \Vert_{H_h^1} + O(h^\infty) \Vert \widetilde{\chi} u \Vert_{L^2} \label{eq:eqnorms}\\
& =  \Vert \text{Op}_h(\eta(|\xi|^2) \xi \chi_0) \widetilde{\chi} u \Vert_{H_h^1} + O(h) \Vert \widetilde{\chi} u \Vert_{L^2} \nonumber \\
&\hspace{0.3cm} \lesssim \Vert \widetilde{\chi} u \Vert_{L^2}. \nonumber
\end{align}
Hence, by \eqref{eq:upper_commu}, for $0< h \leq h_0$,
\begin{equation} \label{eq:upper_e2}
  h^{-1} \Big\Vert  \big(\widetilde{\bcN}hD_{\widetilde{n}}-\widetilde{\bcD}\big)\big[-h^2 \Delta -1,\chi\big]  u  \Big\Vert_{L^2}    \lesssim \Upsilon(R)  \Vert \widetilde{\chi} u \Vert_{L^2}.
\end{equation}
 
 Finally, observe that $h^{-1} [-h^2 \Delta - 1, \widetilde{\bcN}hD_{\widetilde{n}}-\widetilde{\bcD}]$ has principal $h-$symbol
\begin{align*}
 \sigma\Big(h^{-1}\big[-h^2 \Delta - 1,  \big(\widetilde{\bcN}hD_{\widetilde{n}}-\widetilde{\bcD}\big)\big]\Big) 
& = \frac{1}{i} \Big\{ |\xi|^2 - 1, \widetilde{\bcN}(x,\xi)\xi \cdot \widetilde{n}(x) - \widetilde{\bcD}(x,\xi)\Big\} \\
 &= \frac{1}{i} H_p \Big( \widetilde{\bcN}(x,\xi)\xi \cdot \widetilde{n}(x) - \widetilde{\bcD}(x,\xi) \Big),
\end{align*}
 therefore, using Lemma \ref{lem:WFu} in the same way as before, we obtain
$$
 h^{-1} \Big\Vert \big[h^2 \Delta + 1, \widetilde{\bcN}hD_{\widetilde{n}}-\widetilde{\bcD}\big] \chi u \Big\Vert_{L^2}  \lesssim  \sup \Big| \widetilde \chi \psi H_p \Big( \widetilde{\bcN}(x,\xi)\xi \cdot \widetilde{n}(x) - \widetilde{\bcD}(x,\xi) \Big) \Big|  \Vert {\chi} u \Vert_{L^2} + O(h^{1/2})\Vert {\chi} u \Vert_{L^2} .
$$
By the support properties of $\psi$ and $\widetilde \chi$ 
 $$
\sup \Big| \widetilde \chi \psi H_p \Big( \widetilde{\bcN}(x,\xi)\xi \cdot \widetilde{n}(x) - \widetilde{\bcD}(x,\xi) \Big) \Big|\lesssim \Upsilon(R).
 $$
Reducing $h_0>0$ depending on $R$ if necessary,  we obtain that for $0< h \leq h_0$
 \begin{equation} \label{eq:upper_e3}
  h^{-1} \Big\| [-h^2 \Delta - 1, \widetilde{\bcN}(x,\xi)\xi \cdot \widetilde{n}(x) - \widetilde{\bcD}(x,\xi) ] \chi u \Big\|_{L^2}  \lesssim \Upsilon(R) \Vert \chi u \Vert_{L^2}.
\end{equation}
 
Combining \eqref{eq:upper_dec} with \eqref{eq:upper_e1}, \eqref{eq:upper_e2}, and \eqref{eq:upper_e3}, we have, for $0< h \leq h_0(R)$,
 $$
 \Vert (\bcN hD_{n}-\bcD) u \Vert_{L^2(\GammaI)} \lesssim \Upsilon(R)  \Vert \widetilde \chi u \Vert_{L^2},
 $$
 and then Lemma \ref{lem:crown} 
implies that 
 $$
 \N{(\bcN hD_{n}-\bcD) u}_{L^2(\GammaI)}  \leq C \frac{\Upsilon(R)}{R^{1/2}}\N{u}_{L^2(\domain)}.
 $$
 
 To obtain the bound on $Au$, we observe that, by Lemma \ref{lem:WFu}, $S^*\GammaI\subset \mathcal{I}$, and, by Lemma~\ref{l:traceEst},
 $$
 \Vert A u \Vert_{L^2(\GammaI)} \leq \Vert A \chi u \Vert_{L^2} + h^{-1} \Vert (-h^2\Delta - 1) A \chi u \Vert_{L^2} + O(h^\infty)\Vert \chi u \Vert_{L^2}.
 $$
However, in the same way as we obtained (\ref{eq:eqnorms}), the fact that $u$ has $h$-wavefront set in $\{|\xi|^2 = 1\}$ implies that
 $$
 \Vert A \chi u \Vert_{L^2} + h^{-1} \Vert (-h^2\Delta - 1) A \chi u \Vert_{L^2} \lesssim  \Vert \widetilde \chi u \Vert_{L^2},
 $$
 and the bound on $Au$ follows by reducing $h_0(R)>0$ again if necessary.
 \end{proof}

\begin{proof}[Proof of Lemma \ref{lem:impedancetrace2}]

\

\noi \emph{Proof of (i).} 
First observe that if $\GammaIR = \partial B(0,R)$, then for $x\in \GammaIR$, $n(x)=x/|x|$. Therefore, on
$$
\mathcal{O}:=\left\{(x,\xi)\,:\,  x\in \GammaIR, \,\,\bigg|\xi \cdot \frac {x}{|x|} - 1\bigg| \leq \frac{C}{R^2},\,\, |\xi|=1\right\}.
$$
since $n(x)\cdot \xi=\sqrt{1-|\xi'|_g^2}$, we have
$$
|\xi'|^2_g=1-|n(x)\cdot \xi|^2\leq\frac{C}{R^{2}}.
$$ 
We now claim that
\beq\label{eq:decomp}
\sigma(\bcN)(x',\xi')n(x)\cdot\xi-\sigma(\bcD)(x',\xi') = e(x',\xi') |\xi'|_g^{2\morderzero} \quad \ton \mathcal{O},
\eeq
where $e(x',\xi')$ is smooth on $\mathcal{O}$. 
 Indeed, the existence of $e(x',\xi')$ follows from the definition of $\morderzero$ \eqref{eq:pq} and that $n(x)\cdot \xi=\sqrt{1-|\xi'|_g^2}$ on $\mathcal{O}$. 

Therefore
\begin{align}
\sup_{\mathcal{O}}\big|\sigma(\bcN)(x',\xi')n(x)\cdot\xi-\sigma(\bcD)(x',\xi')\big|
\leq C|\xi'|_g^{2\morderzero}\leq CR^{-2\morderzero}.\label{e:vanishingRabbit}
\end{align}

Next, we bound the terms in $\Upsilon(R)$ \eqref{eq:Upsilon} involving the Hamiltonian vector field $H_p=2\langle \xi,\partial_x\rangle$.
First, using again that $\xi=(n(x)\cdot \xi) n(x)+\xi'$ (where we abuse notation slightly to identify vectors and covectors), we have $H_p=2 n(x)\cdot \xi\,\partial_n  +2\langle \xi',\partial_{x'}\rangle$. Thus, on $\mathcal{O}$,
\begin{align}\nonumber
H_p\Big(\sigma(\bcN)n(x)\cdot\xi-\sigma(\bcD)\Big)
&=\sigma(\bcN)2 \left( \frac{x}{|x|}\cdot \xi \right)\left\langle \frac{x}{|x|},\partial_x\right\rangle\left( \frac{x}{|x|}\cdot \xi \right)+ 2\big\langle\xi',\partial_{x'}\big\rangle\big(\sigma(\bcN)n(x)\cdot\xi-\sigma(\bcD)\big)\\
&=2\big\langle\xi',\partial_{x'}\big\rangle\Big(\sigma(\bcN)\sqrt{1-|\xi'|_g^2}-\sigma(\bcD)\Big)\label{eq:rabbit2}
\end{align}
where we have used that $\partial_{x'}$ is tangent to $\GammaIR\cap \{|\xi|=1\}$ to write $n(x)\cdot \xi=\sqrt{1-|\xi'|_g^2}$ in the last line. Now, 
by \eqref{eq:decomp},
$$
\partial_{x'}\left(\sigma(\bcN)\sqrt{1-|\xi'|_g^2}-\sigma(\bcD)\right)=O\big(|\xi'|_g^{2\morderzero}\big).
$$
In particular, 
\beq\label{eq:rabbit3}
2\langle\xi',\partial_{x'}\rangle\left(\sigma(\bcN)\sqrt{1-|\xi'|_g^2}-\sigma(\bcD)\right)=O\big(|\xi'|_g^{2\morderzero+1}\big)=O(R^{-2\morderzero-1}).
\eeq
The required bound on $\Upsilon(R)$ follows by combining \eqref{e:vanishingRabbit}, \eqref{eq:rabbit2}, and \eqref{eq:rabbit3}.

\noi\emph{Proof of (ii).} This follows from the fact that $\sigma(\bcN)$ and $\sigma(\bcD)$ have uniformly bounded $C^1$ norms in $R$.
\end{proof}


\section{Proof of wellposedness of the truncated problem (Theorem \ref{th:uniformestimates})}
\label{sec:resolventestimates}

\subsection{Trace bounds for higher order boundary conditions}

In this section, we consider the solution to 
\begin{equation}
\label{eq:genBC}
\begin{cases}
(-h^2\Delta_g-1)u=hf&\text{ in }M,\\
\bcN_{i}hD_n u- \bcD_{i}u =g_i&\text{ on }\Gamma_i\subset \boundary,
\end{cases}
\end{equation}
where $(M,g)$ is a Riemannian manifold with smooth boundary $\partial M=\cup_{i=1}^N\Gamma_i$
such that $\Gamma_i$ are the connected components of $\boundary $, and 
$\bcN_{i}\in \Psi^{m_{1,i}}(\Gamma_i)$, and $\bcD_i\in \Psi^{m_{0,i}}(\Gamma_i)$  have real-valued principal symbols. We further assume that for all $i=1,\dots,N$, 
\begin{equation}
\label{eq:conditions}
\begin{gathered}
|\sigma(\bcN_{i})|^2\langle \xi '\rangle^{-2m_{1,i}}+|\sigma(\bcD_{i})|^2\langle \xi '\rangle^{-2m_{0,i}}\geq c>0\quad  \ton T^{*}\Gamma_i,\\ |\sigma(\bcD_{i})|>0\quad\text{ on }S^*\Gamma_i,\\
\end{gathered}
\end{equation}
and for each $i$ one of the following holds:
\begin{align}
&m_{0,i}=m_{1,i}+1,\qquad\text{ or }\label{eq:conditions2a}\\
&|\sigma(\bcN_{i})|^2\langle \xi'\rangle^{-2m_{1,i}}\geq c>0,&& |\xi'|\geq C,&&\text{and }\qquad m_{0,i}\leq m_{1,i}+1,&\text{ or}\label{eq:conditions2c}\\
&|\sigma(\bcD_{i})|^2\langle \xi'\rangle^{-2m_{0,i}}\geq c>0,&& |\xi'|\geq C,&&\text{and }\qquad m_{1,i}+1\leq m_{0,i}.\label{eq:conditions2d}
\end{align}
The first condition in \eqref{eq:conditions} ensures non-degeneracy at infinity in $\xi$ (with \eqref{eq:conditions2a}, \eqref{eq:conditions2c} and \eqref{eq:conditions2d} the different options for which term in the boundary condition is dominant), and the second condition in \eqref{eq:conditions} ensures that the Dirichlet trace is bounded.

\begin{theorem}
\label{th:higherOrderBound}
Suppose that $u$ solves~(\ref{eq:genBC}) where $\bcN_{i}\in \Psi^{m_{1,i}}(\Gamma_i)$, $\bcD_{i}\in \Psi^{m_{0,i}}(\Gamma_i)$ have real-valued principal symbols and satisfy~\eqref{eq:conditions} and one of~\eqref{eq:conditions2a}-~\eqref{eq:conditions2d}. Then, there exist $C>0$ and $h_0>0$ such that for $0<h<h_0$, and $i$ and all $\ell_i$ satisfying
\beq\label{eq:ellineq}
-\frac{m_{0,i}+m_{1,i}}{2}\leq \ell_i\leq \frac{1}{2}-\frac{m_{0,i}+m_{1,i}}{2},
\eeq
\begin{align}\label{eq:trace1}
 \|u\|_{H_h^{\ell_i+m_{0,i}}(\Gamma_i)}+\|hD_{\nu}u\|_{H_h^{\ell_i+m_{1,i}}(\Gamma_i)}&\leq C\Big(\|u\|_{L^2(M)}+\|f\|_{H_h^{\ell_i+\frac{m_{1,i}+m_{0,i}-1}{2}}(M)} +\|g_i\|_{H_h^{\ell_i}(\Gamma_i)}\Big),
\end{align}
\beq\label{eq:trace2}
\|u\|_{H_h^1(M)}\leq C\Big(\|u\|_{L^2(M)}+h\|f\|_{L^2(M)}+\sum_i\|g_i\|_{H_h^{\ell_i}(\Gamma_i)}\Big),
\eeq
and for $s\leq 0$,
\beq\label{eq:trace3}
\|hD_{\nu}u\|_{H_h^s(\Gamma_i)}\leq C\Big(\|u\|_{H_h^{s+1}(\Gamma_i)}+\|u\|_{L^2(M)}+\|f\|_{L^2(M)}+\sum_i\|g_i\|_{H_h^{\ell_i}(\Gamma_i)}\Big) .
\eeq

\end{theorem}
The proof of Theorem~\ref{th:higherOrderBound} is postponed until Section~\ref{sec:traceHO}. Here we proceed directly to its application.

\subsubsection{Application of  Theorem~\ref{th:higherOrderBound} with $L^2$ right hand sides}

\begin{cor}\label{cor:L2RHS}
Suppose that 
\beq\label{eq:mcond0}
m_0\geq 0, \quad m_0+m_1\geq 0, \quad m_1\leq m_0+1,
\eeq
and \emph{either}
\beq\label{eq:mcond1}
m_0\leq m_1 + \min\{1, m_0+ m_1\},
\eeq
\emph{or}
\beq\label{eq:mcond2}
m_0\geq m_1 + 1 \quad\tand\quad m_0\geq 1.
\eeq
Then there exists $C>0$ and $h_0>0$ such that, for $0<h\leq h_0$, the solution to 
$$
\begin{cases}
(-h^2\Delta-1)u=hf&\text{in }\domaintwo,\\
\big(\bcN hD_n -\bcD\big)u=g&\text{on }\Gamma,
\end{cases}
$$
with $f\in L^2(\Omega)$ and $g\in L^2(\Gamma)$ 
satisfies
\beq\label{eq:trace4}
\|u\|_{L^2(\Gamma)}+\|hD_n u\|_{L^2(\Gamma)}+\|u\|_{H_h^1(\domain)}\leq C\Big(\|u\|_{L^2(\domain)}+\|f\|_{L^2(\domain)}+\|g\|_{L^2(\Gamma)}\Big).
\eeq
\end{cor}

\begin{proof}
Let 
$$
\ell=r-\frac{m_0+m_1}{2}.
$$
If $0\leq r\leq \frac{1}{2}$, then Theorem~\ref{th:higherOrderBound} holds and \eqref{eq:trace1} and \eqref{eq:trace2} become
\beq\label{eq:trace1a}
 \|u\|_{H_h^{r+\frac{m_{0}-m_1}{2}}(\Gamma)}+\|hD_{\nu}u\|_{H_h^{r+\frac{m_{1}-m_0}{2}}(\Gamma)}\leq C\Big(\|u\|_{L^2(\domain)}+\|f\|_{H_h^{r-\frac{1}{2}}(\domain)} +\|g\|_{H_h^{r-\frac{m_1+m_0}{2}}(\Gamma)}\Big)
 \eeq
 and
 \beq\label{eq:trace2a}
\|u\|_{H_h^1(M)}\leq C\Big(\|u\|_{L^2(M)}+h\|f\|_{L^2(M)}+\|g\|_{H_h^{r-\frac{m_1+m_0}{2}}(\Gamma)}\Big),
\eeq 
respectively. 
Focusing on \eqref{eq:trace1a}, we therefore impose the conditions that  
\beqs
 r\geq \frac{m_1-m_0}{2},\qquad 0\leq r\leq \frac{1}{2}, \qquad r\leq \frac{m_1+m_0}{2}, 
 \eeqs
 i.e.,
\beqs
  \max\left(0, \frac{m_1-m_0}{2}\right) \leq r\leq \min\left(\frac12, \frac{m_1+m_0}{2}\right)
\eeqs
 (observe that this range of $r$ is nonempty since $m_0\geq 0$, $m_1-m_0\leq 1$, and $m_1+m_0\geq 0$).
Choosing $r = \min\{1/2, (m_1+m_0)/2\}$, we have 
 \beq\label{eq:trace4a}
\N{u}_{L^2(\Gamma)}+\N{hD_{\nu}u}_{H_h^{s^*}(\Gamma)}
 \leq C\Big(\|u\|_{L^2(\domain)}+\|f\|_{L^2(\domain)} +\|g\|_{L^2(\Gamma)}\Big),
 \eeq
where 
\beqs
s^*:= \min\left(\frac12, \frac{m_1+m_0}{2}\right)+ \frac{m_1-m_0}{2}.
\eeqs 
If  $s^*\geq 0$, i.e., if  \eqref{eq:mcond1} holds, then the result \eqref{eq:trace4} follows from combining \eqref{eq:trace4a} with \eqref{eq:trace2a}.

If \eqref{eq:mcond1} doesn't hold, we seek control of $\|hD_n u\|_{L^2(\Gamma)}$ via the bound \eqref{eq:trace3} with $s=0$, i.e.
\beqs
\|hD_{\nu}u\|_{L^2(\Gamma)}\leq C\Big(\|u\|_{H_h^{1}(\Gamma)}+\|u\|_{L^2(M)}+\|f\|_{L^2(M)}+\|g_i\|_{H_h^{r-\frac{m_1+m_0}{2}}(\Gamma)}\Big).
\eeqs
To prove \eqref{eq:trace4}, therefore, we only need to bound $\|u\|_{H_h^{1}(\Gamma)}$ in terms of the right-hand side of \eqref{eq:trace4}.
This follows from \eqref{eq:trace1a} if 
\beqs
  \max\left(0, 1+\frac{m_1-m_0}{2}\right) \leq r\leq \min\left(\frac12, \frac{m_1+m_0}{2}\right),
\eeqs
which is ensured if \eqref{eq:mcond2} holds.
\end{proof}

\subsubsection{Application of Theorem \ref{th:higherOrderBound} to Dirichlet boundary conditions}
\begin{cor}
\label{cor:Dir}
There exist $C>0$ and $h_0>0$ such that if $0\leq h\leq h_0$, then the solution of 
\beqs
\begin{cases}
(-h^2\Delta-1)u=hf&\text{ in }\domaintwo\\
u=g&\text{ on } \Gamma.
\end{cases}
\eeqs
with $f\in L^2(\Omega)$ and $g\in H^1_h(\Gamma)$ satisfies 
$$
\|u\|_{H_h^1(\Gamma)}+\|hD_{\nu}u\|_{L^2(\Gamma)}+\|u\|_{H_h^1(\domain)}\leq C\Big(\|u\|_{L^2(\domain)}+\|f\|_{L^2(\domain)}+\|g\|_{H_h^1(\Gamma)}\Big)
$$
\end{cor}

\bpf[Proof of Lemma \ref{lem:impedancetrace2}]
The Dirichlet boundary condition corresponds to  $\bcD=I, \bcN=0$, and so satisfies the assumptions of Theorem~\ref{th:higherOrderBound} with $m_0=0$ and $m_1=-1$, say. 
The result follows by choosing $\ell=1$ and combining \eqref{eq:trace1} and \eqref{eq:trace2}.
\epf

\subsection{Recap of results of \cite{TrHa:86} about Pad\'e approximants}\label{sec:TH}

We now recall results of \cite{TrHa:86} about Pad\'e approximants. These results consider a larger class of approximants than covered in our Assumption \ref{ass:Pade}; before stating these results, we explain this difference.

With $p(t)$ and $q(t)$ defined by \eqref{eq:pq}, by Assumption \ref{ass:Pade}, 
\beq\label{eq:prinsym}
\sigma(\bcD)(x',\xi')= \PadeP(x',\xi')= p(|\xi'|_g^2) \quad \tand\quad \sigma(\bcN)(x',\xi')= \PadeQ(x',\xi')= q(|\xi'|_g^2).
\eeq
As described in \S\ref{sec:setup}, this choice of $\bcD$ and $\bcN$ is based on approximating $\sqrt{1-|\xi'|^2_g}$ with a rational function in $|\xi'|_g^2$.

The boundary conditions in \cite{TrHa:86} are based on approximating $\sqrt{1-|\xi'|^2_g}$ with a rational function in $|\xi'|_g$, i.e.~\cite{TrHa:86} consider Pad\'e approximants with polynomials $\widetilde{p}(s)$ and $\widetilde{q}(s)$, where the degrees $\widetilde{p}(s)$ and $\widetilde{q}(s)$ allowed to be either even or odd. 
Our polynomials $p,q$ fit into the framework of \cite{TrHa:86} with 
\beq\label{eq:widetildep}
\widetilde{p}(s):= p(s^2) \quad\tand\quad \widetilde{q}(s):=q(s^2),
\eeq
and then $\widetilde{p}$ has degree $2\MPade $ 
and $\widetilde{q}$ has degree $2\NPade $. 
For $d-1\geq 2$ (i.e.~when the boundary dimension is $\geq 2$), polynomials with odd powers of $|\xi'|_g$ do not lead to $\bcN$ and $\bcD$ being local differential operators, 
but for $d-1=1$ (i.e.~$d=2$) they do, since in this case $\sqrt{|\xi'|^2_g}=\sqrt{g(x'})\xi'$, i.e., a polynomial in $\xi'$. Our arguments also apply to polynomials with odd powers of  $|\xi'|_g$ in $d=2$, but we do not analyze them specifically, instead leaving this to the interested reader.

To state the results of \cite{TrHa:86}, we let $\widetilde{p}(s)$ and $\widetilde{q}(s)$ be polynomials of degree $m_0$ and $m_1$ respectively; this notation is chosen so that, when we specialise the results to our case with \eqref{eq:widetildep}, these $m_0$ and $m_1$ are the same as in Theorem \ref{th:higherOrderBound}/Corollary \ref{cor:L2RHS}, i.e., $m_0=2\MPade $ and $m_1=2\NPade $.
Finally, we let 
\beqs
 \widetilde{r}(s):= \frac{\widetilde{p}(s)}{\widetilde{q}(s)}.
\eeqs

\ble\mythmname{\cite[Theorems 2 and 4]{TrHa:86}}\label{lem:TH}
If, and only if, $m_0=m_1$ or $m_0=m_1+2$, then 

(a) $\widetilde{r}(s)>0$ for $s\in [-1,1]$, and

(b) the zeros and poles of $\widetilde{r}(s)/s$ are real and simple and interlace along the real axis.
\ele

\begin{cor}\label{cor:TH}
If $m_0=m_1$ or $m_0=m_1+2$, then neither $\widetilde{p}(s)$ nor $\widetilde{q}(s)$ has any zeros in $[-1,1]$. 
\end{cor}

\bpf
For $\widetilde{p}(s)$, this property follows directly from Part (a) of Lemma \ref{lem:TH}. For $\widetilde{q}(s)$, this property follows from Parts (a) and (b) of Lemma \ref{lem:TH}; indeed, if there were a zero of $\widetilde{q}(s)$ (i.e.~a pole of $\widetilde{r}(s)$) in $[-1,1]$, since 
the zeros of $\widetilde{q}(s)$ are simple and interlace with the zeros of $\widetilde{p}(s)$ (by Part (b)), $\widetilde{r}(s)$ would change sign in $[-1,1]$, contradicting Part (a).
\epf

\subsection{Proof of Theorem \ref{th:uniformestimates}}

Throughout this section, we let $\tildedomain$ be a smooth family of domains depending on $R$ and assume that there is $M>0$ such that
\begin{equation}
\label{e:doughnut}
\begin{gathered}
 B(0,1)\subset \tildedomain\subset B(0,MR),\\
\tildedomain\text{ is convex with smooth boundary, $\GammaI$, that is nowhere flat to infinite order}
\end{gathered}
\end{equation}
Furthermore, we assume that 
\begin{equation*}
\begin{gathered}
\tildedomain/R\to \Omega_\infty
\end{gathered}
\end{equation*}
in the sense that $\partial{\tildedomain}/R\to\partial  \Omega_\infty$ in $C^\infty$. 

We prove below that Theorem~\ref{th:uniformestimates} is a consequence of the following result, combined with the results from \cite{TrHa:86} in \S\ref{sec:TH}.

\begin{theorem}\label{th:uniformestimates2}
Let $\tildedomain$ be as in~\eqref{e:doughnut} and $\Omegaminus\Subset B(0,1)$ with $\Omegaminus$ non-trapping. Let 
$\bcN\in \Psi^{m_{1}}(\GammaIR)$, $\bcD\in \Psi^{m_{0}}(\GammaIR)$ have real-valued principal symbols and satisfy~\eqref{eq:conditions} and one of~\eqref{eq:conditions2a}-\eqref{eq:conditions2d}. Let $m_0$ and $m_1$ satisfy the assumptions of Corollary \ref{cor:L2RHS}, and furthermore let $\bcN$ and $\bcD$ satisfy 
\beq\label{eq:mainevent}
\sigma(\bcN)\sigma(\bcD)>0\text{ on }\overline{B^*\GammaI}.
\eeq
Let
$$
G^R_h{: L^2(\GammaI)\oplus H_h^1(\Gamma_D)\oplus L^2(\tildedomain\setminus \Omegaminus) \rightarrow H_h^1(\tildedomain\setminus \Omegaminus)}
$$ 
satisfy
\beqs
\begin{cases}
(-h^2\Delta-1)G^R_h(g_I,g_D,f)=hf&\text{on }\tildedomain\setminus \Omegaminus\\
(\bcN hD_n -\bcD\big) G^R_h(g_I,g_D,f)=g_I&\text{on }\GammaI\\
G^R_h (g_I,g_D,f)=g_D&\text{on }\Gamma_D.
\end{cases}
\eeqs
Then there exists $C>0$ such that for $R\geq 1$, there is $h_0=h_0(R)>0$ such that for $0<h<h_0$, 
$G_h^R$ is well defined and satisfies
\beq\label{eq:trunc_estimate}
\|G^R_h(g_I,g_D,f)\|_{H_h^1(\tildedomain\setminus \Omegaminus)}\leq CR^{1/2}\Big(\|g_I\|_{L^2(\GammaI)}+\|g_D\|_{H_h^1(\Gamma_D)}\Big)+CR\|f\|_{L^2(\tildedomain\setminus \Omegaminus)}.
\eeq
\end{theorem}

\bpf[Proof of Theorem \ref{th:uniformestimates} using Theorem \ref{th:uniformestimates2}]
Theorem \ref{th:uniformestimates} will follow from Theorem \ref{th:uniformestimates2} (translating between the $h$- and $k$-notations using \S\ref{sec:conventions}) if we can show that the boundary conditions in Assumption \ref{ass:Pade} with either $\MPade=\NPade$ or $\MPade=\NPade+1$, with $\MPade, \NPade\geq 0$, satisfy

(i)  \eqref{eq:conditions}, 

(ii) one of \eqref{eq:conditions2a}-\eqref{eq:conditions2d}, 

(iii) the assumptions of Corollary \ref{cor:L2RHS}, and 

(iv) \eqref{eq:mainevent},

\noi where $m_0= 2\MPade $ and $m_1=2\NPade $.

Regarding (iii): the first two inequalities in \eqref{eq:mcond0} are satisfied since $m_0, m_1\geq 0$, and the third inequality is satisfied both when $m_0=m_1$ and when $m_0=m_1+2$. If $m_0=m_1$, then \eqref{eq:mcond1} is satisfied, and if $m_0=m_1+2$ then \eqref{eq:mcond2} is satisfied (since $m_1\geq 0$ and thus $m_0\geq 2$).

Regarding (ii): if $m_0= m_1$, then \eqref{eq:conditions2c} holds since $q_{\MPade,\NPade}^{\NPade}\neq 0$ by definition. If $m_0=m_1+2$, then \eqref{eq:conditions2d} holds since $p_{\MPade,\NPade}^{\MPade}\neq 0$ by definition.

Regarding (i) and (iv): using \eqref{eq:prinsym}, the conditions \eqref{eq:conditions} and \eqref{eq:mainevent} become (with $t=|\xi'|^2_g$)
\beq\label{eq:conditions_alt}
\big|q(t)\big|^2 t^{-2\NPade } + \big|p(t)\big|^2 t^{-2\MPade }>0 \,\,\tfa t \quad\tand\quad |p(\pm 1)|>0,
\eeq
and
\beq\label{eq:mainevent2}
|p(t)q(t)| >0 \quad\ton -1\leq t\leq 1,
\eeq
respectively 

If $\widetilde{p}(s)$ and $\widetilde{q}(s)$ are defined by \eqref{eq:widetildep}, then \eqref{eq:conditions_alt} and \eqref{eq:mainevent2} become
\beq\label{eq:conditions_alt2}
\big|\widetilde{q}(s)\big|^2 s^{-2m_1} + \big|\widetilde{p}(s)\big|^2 s^{-2m_0}>0 \,\,\tfa s \quad\tand\quad |\widetilde{p}(\pm 1)|>0,
\eeq
and
\beq\label{eq:mainevent3}
|\widetilde{p}(s)\widetilde{q}(s)| >0 \quad\ton -1\leq s\leq 1.
\eeq

The first condition in \eqref{eq:conditions_alt2} holds since, by Part (a) of Lemma \ref{lem:TH}, $\widetilde{p}(s)$ and $\widetilde{q}(s)$ have no common zeros. Both the second condition in \eqref{eq:conditions_alt2} and the condition in \eqref{eq:mainevent3} hold by Corollary \ref{cor:TH}.
\epf

We now prove Theorem \ref{th:uniformestimates2}. We first show that, for each $z \in \mathbb{C}$ and $s\geq 0$ the operator 
\begin{multline*}
\widetilde{P}(z):H^{2+s}(\tildedomain\setminus \Omegaminus)\ni u\mapsto (-h^2\Delta-z,\big(\bcN hD_n-\bcD\big)u|_{\GammaI}, u|_{\Gamma_D})\\
\in H^s(\tildedomain\setminus\Omegaminus)\oplus H^{3/2+s-m}(\GammaI)\oplus H^{3/2+s}(\Gamma_D)
\end{multline*}
is Fredholm with $m=\max(m_0,m_1+1)$; we do this by checking the conditions of~\cite[Theorem 20.1.8', Page 249]{Ho:85}. Observe that, for fixed $h>0$, as a homogeneous pseudodifferential operator, $(-h^2\Delta-z^2)$ has symbol $p(x,\xi)=|\xi|^2$. Therefore, in Fermi normal coordinates at $\GammaI$, we need to check that the map
$$
M_{x,\xi'}\ni u\to \big(b\big(x,(D_t,\xi')\big)u\big)(0)
$$
is bijective, where $M_{x,\xi'}$ denotes the solutions to $(D_t^2+|\xi'|_g^2)u(t)=0$ with $u$ is bounded on $\mathbb{R}_+$, and 
$$
b(x,\xi)=\lim_{\lambda \to \infty}\Big(-\sigma(\bcN)(x,\lambda \xi')\lambda \xi_1-\sigma(\bcD)(x,\lambda\xi')\Big)\lambda^{-m}.
$$
Since $u=Ae^{-t|\xi'|_g}$, 
\beqs
\big(b\big(x,(D_t,\xi')\big)u\big)(0)
=A\lim_{\lambda \to \infty}\Big(-\sigma(\bcN)(x,\lambda \xi')\lambda i|\xi'|-\sigma(\bcD)(x,\lambda\xi')\Big)\lambda^{-m},
\eeqs
and bijectivity follows if the limit on the right-hand side is non-zero. 
Since $\bcN$ and $\bcD$ are both real, this is ensured by~\eqref{eq:conditions} and any of~\eqref{eq:conditions2a}-\eqref{eq:conditions2d}.

Now, to see that $\widetilde{P}$ is invertible somewhere, consider $z=-1$. First, note that for $s\geq 0$ the map
$$
P_D:(H^{2+s}(\tildedomain\setminus \Omegaminus)\ni u\mapsto \to (-h^2\Delta+1)u,u|_{\GammaI},u|_{\Gamma_D})\in H^s(\tildedomain\setminus \Omegaminus)\oplus H^{s-\frac{1}{2}}(\Gamma_I)\oplus H^{s-\frac{1}{2}}(\Gamma_D)
$$
is invertible with inverse $G_D:H_h^s(\tildedomain\setminus \Omegaminus)\oplus H_h^{s-\frac{1}{2}}(\GammaI)\oplus H_h^{s-\frac{1}{2}}(\Gamma_D)\to H_h^{2+s}(\tildedomain\setminus \Omegaminus)$ (see e.g.~\cite[Chapter 6]{Eva}). In particular, the Dirichlet to Neumann map
$$
\Lambda:g_1\mapsto hD_n u|_{\Gamma_I},\quad 
\text{ where }\quad
\begin{cases}
(-h^2\Delta +1)u=0&\text{on }\tildedomain\setminus \Omegaminus,\\
u=g_1&\text{on }\GammaI,\\
u=0&\text{on }\Gamma_D,
\end{cases}
$$
is well defined. Furthermore, $\Lambda\in \Psi^1(\GammaI)$ is a semiclassical pseudodifferential operator with symbol $\sigma(\Lambda)=-i\sqrt{|\xi'|_g+1}$ (see, e.g., \cite[Proposition 4.1.1, Lemma 4.27]{Ga:19b}). In particular, by~\eqref{eq:conditions} and~\eqref{eq:conditions2a}-\eqref{eq:conditions2d}, $(-i\bcN\Lambda-\bcD)^{-1}$ exists, and hence
$$
[\widetilde{P}(-1)]^{-1}(f,g_I,g_D)=G_D(f,(-i\bcN \Lambda -\bcD)^{-1}g_I,g_D)
$$

Therefore, since for $z=-1$, the operator is invertible, by the analytic Fredholm Theorem (see e.g.~\cite[Theorem C.8]{DyZw:19}) the family $G^R_h(z)$ of operators solving
$$
\begin{cases}
(-h^2\Delta-z)G^R_h(z)(g_I,0,f)=hf&\text{on }\tildedomain\setminus \Omegaminus\\
\big(\bcN hD_{\nu}-\bcD\big) G_h^R(z)(g_I,0,f)=g_I&\text{on }\GammaI\\
G^R_h(z) (g_I,0,f)=0&\text{on }\Gamma_D
\end{cases}
$$
is a meromorphic family of operators with finite rank poles. To include the Dirichlet boundary values, we observe that by standard elliptic theory, the operator $\widetilde{G}_h(z):H_h^1(\Gamma_D)\to H^{3/2}(B(0,1)\setminus \overline{\Omegaminus})$ solving
$$
\begin{cases}
(-h^2\Delta-z)\widetilde{G}(z)g=0& \text{on}B(0,1)\setminus \overline{\Omegaminus}\\
\widetilde{G}_h(z)g=g&\text{ on }\Gamma_D\\
(hD_n -1)\widetilde{G}_h(z)g=0&\text{ on }\partial B(0,1)
\end{cases}
$$
is a meromorphic family of operators with finite rank poles. With $\chi \in C_c^\infty(B(0,1))$ with $\chi\equiv 1$ near $\Omegaminus$, 
$$
G_h^R(g_I,g_D,f)=G_h^R\big(g_I,0,f-h^{-1}[-h^2\Delta,\chi]\widetilde{G}_hg_D\big) +\chi\widetilde{G}_h g_D,
$$
and thus the operator $G^{R}_h$ is well defined.

We start by studying $G_h^R(0,g,0)$.

\begin{lem}\label{lem:4.5}
Let $R>0$ and assume that $\bcN$ and $\bcD$ satisfy the assumptions of Theorem~\ref{th:higherOrderBound}. Then there exist $C, h_0>0$ such that $u=G_h^R(0,g,0)$, the solution to
$$
\begin{cases}
(-h^2\Delta-1)u=0 & \text{ in } \domain,\\
u=g& \ton \Gamma_D, \\
\big(\bcN hD_{\nu}-\bcD\big)u=0 & \ton \GammaI,
\end{cases}
$$
satisfies
$$
\|u\|_{H_h^1(\tildedomain\setminus \Omegaminus)}\leq C\|g\|_{H_h^1(\Gamma_D)}
$$
\end{lem}
\begin{proof}
Suppose the lemma fails. Then there exist $(h_n,g_n)$ with $h_n\to 0$ such that $u_n=G_{h_n}^R(0,g_n,0)$, 
$$
\|u_n\|_{H_{h_n}^1(\tildedomain\setminus \Omegaminus)}=1,\qquad \|g_n\|_{H_{h_n}^1(\Gamma_D)}=n^{-1}
$$
Extracting subsequences, we can assume that $u_n$ has defect measure $\mu$. Moreover, by Corollaries~\ref{cor:L2RHS} and \ref{cor:Dir}, we can assume that the trace measures $\nu_d^{D/\tr},\nu_j^{D/\tr },$ and $\nu_n^{D/\tr}$ exist. 
In particular, since $g_n\to 0$ in $H_h^1$, $\nu_d^D=0$. Let $\varphi_t$ denote the billiard flow outside $\Omegaminus$. Then by Lemma \ref{lem:key_Miller} together with~\cite[Section 4]{GaSpWu:20}, 
\begin{gather}
\label{eq:buzz}
\mu(\varphi_t(A))=\mu(A)\qquad \text{ if }\qquad \bigcup_{0\leq t\leq T}\varphi_t(A) \cap \GammaI=\emptyset.
\end{gather}
Furthermore, using again Corollaries~\ref{cor:L2RHS} and \ref{cor:Dir}, we find that
\begin{equation*}
1=\limsup_{n}\|u_n\|_{H^1_{h_{n}}}^2\geq \mu(T^*\mathbb{R}^d)\geq \liminf_{n}\|v_{n}\|_{L^2}^2  \geq c\liminf_{n}\|\widetilde{v}_{n}\|_{H_{h_{n}}^1}^2=c>0.
\end{equation*}
Note also that $\mu^{\rm in/out,\tr}$, $\nudtr$, $\nujtr$, and $\nuntr$ satisfy the relations in Lemma~\ref{lem:key_Miller}. Next, by Lemma~\ref{lem:hierarchy_coef},
\begin{equation}
\label{eq:reflect}
\mu^{\rm out,\tr}=\alpharef \mu^{\rm in,\tr} \quad\text{where}\quad
\alpharef =\Bigg|\frac{\sqrt{r}\bcN-\bcD}{\sqrt{r}\bcN+\bcD}\Bigg|^{2}\in C^\infty(\{r> 0\});
\end{equation} 
Here, we abuse notation slightly, since when $\sigma(\bcN)\sigma(\bcD)<0$, $\sqrt{r}\bcN+\bcD$ may take the value $0$. In that case, the first equation in \eqref{eq:reflect} is replaced by $(\alpharef )^{-1}\mu^{\rm out,\tr}=\mu^{\rm in,\tr}$.

Finally, these measures satisfy Theorem~\ref{t:propagate} with $\dot{n}^j=\sigma(\bcN)/\sigma(\bcD)$ which is well defined and satisfies $\pm\dot{n}^j\geq m>0$ since $\pm \sigma(\bcN)\sigma(\bcD)>0$ on $\overline{B^*\GammaI}$.

The proof of Lemma \ref{lem:4.5} is completed by the following lemma.
\begin{lem}
\label{lem:contradictMe}
Suppose that $\Omegaminus$ is non-trapping, and let $M>0$. Then there exist $T_0,\delta_0>0$ such that the following holds for all $R\geq 1$. Suppose
 $\Omegaminus\Subset B(0,1)\subset\tildedomain\subset B(0,MR)$ has smooth boundary and is convex and that $\mu$ is a finite measure supported in $S^*_{\overline{\tildedomain}\setminus \Omegaminus}\mathbb{R}^d$ 
satisfying~\eqref{eq:buzz},~\eqref{eq:reflect} and Theorem~\ref{t:propagate} with $ \Re\dot{n}^j=\frac{\sigma(\bcN)}{\sigma(\bcD)}$ 
with $0<\pm\sigma(\bcN)\sigma(\bcD)$ on $\overline{B^*\GammaI}$
 Then, for all $A\subset S^*_{\overline{\tildedomain}\setminus \Omegaminus}\mathbb{R}^d$,
 $$
\mu\big(\varphi_{\mp T_0R}(A)\big)\geq (1+\delta_0)\mu(A).
$$ 
\end{lem}
To see that Lemma~\ref{lem:contradictMe} completes the proof of Lemma \ref{lem:4.5} observe that our defect measure $\mu$ has $\mu(T^*\mathbb{R}^d)\neq 0$, is finite, and is supported in $S^*_{\overline{\tildedomain}\setminus \Omegaminus}\mathbb{R}^d$. Therefore, there is $A\subset S^*_{\overline{\tildedomain}\setminus \Omegaminus}\mathbb{R}^d$ such that $\mu(A)>0$. But then
$$
\mu\big(\varphi_{\mp NRT_0}(A)\big)=(1+\delta_0)^N\mu(A)\to \infty,
$$
which is a contradiction.
 \epf

\begin{proof}[Proof of Lemma~\ref{lem:contradictMe}]

We consider only  the case where $\sigma(\bcN)\sigma(\bcD)>0$. The other case follows from an identical argument but reversing the time direction.

 By~\eqref{eq:buzz}, $\mu$ is invariant under $\varphi_t$ away from $\GammaI$. We first study the glancing set, $\mathcal{G}=T^{*}\GammaIR\cap \{r=0\}$. Note that since $\GammaI$ is convex, $\mathcal{G}\subset \{H_p^2x_1\leq 0\}$ where $x_1$ is a boundary defining function for $\GammaI$. Note that for $\rho \in\mathcal{G}$, since $\tildedomain(R)$ is convex and $\tildedomain(R)\subset B(0,MR)$, there exist $c>0$ and $T_0>0$ independent of $R$ such that
$$
\int_{-T_0R}^0 -H_p^2x_1(\varphi_s(\rho))ds \geq c>0
$$
In particular, since $\sigma(\bcN)\sigma(\bcD)>m>0$ on $S^*\GammaI$ (by \eqref{eq:mainevent}), $\Re \dot{n}^j\geq m>0$ and hence by Theorem~\ref{t:propagate}, for $A\subset \mathcal{G}$,
\begin{equation*}
\mu(\varphi_{-T_0R}(A))\geq e^{mc}\mu(A), 
\end{equation*}

Next, we study the case where $A\subset S^*_{\overline{\tildedomain}\setminus \Omegaminus}\mathbb{R}^d\setminus \mathcal{G}$. Let $\beta^{-1}:B^*\GammaI\to B^*\GammaI$ be the reversed billiard ball map induced by $\varphi_t$. That is, let $\pi:S^*_{\GammaI}\mathbb{R}^d\to B^*\GammaI$ be the natural projection map and $\pi_{\pm}^{-1}:B^*\GammaI\to S^*_{\GammaI}\mathbb{R}^d$ the inward- and outward-pointing inverse maps. Next, for $(x,\xi)\in S^*_{\GammaI}\mathbb{R}^d$ define
$$
T_-(x,\xi)=\inf\{ t>0\,:\, \varphi_{-t}(x,\xi)\in S^*_{\GammaI}\mathbb{R}^d\}.
$$
Since $\Omegaminus$ is nontrapping, there is $T_0>0$ such that for all $(x,\xi)\in S^*_{\tildedomain\setminus \Omegaminus\mathbb{R}^d \cup \pi_-^{-1}(B^*\GammaI)}$, $T_-(x,\xi)\leq T_0 R$. In particular every trajectory intersects the boundary in time $T_0R$.

The reversed billiard map is then given by 
$$
\beta^{-1}(q):\pi\big( \varphi_{-T_-(\pi_-^{-1}(q)}(\pi^{-1}_-(q))\big).
$$
Since $\GammaI$ is convex $\beta:B^* \GammaI\to B^*\GammaI $ is well defined and, since $\mu$ is invariant under $\varphi_t$, $\beta_*\mu^{\rm out,\tr}=\mu^{\rm in,\tr}.$ Then, using~\eqref{eq:reflect}, we have 
\begin{equation}
\label{eq:reflectMeasure}
\mu^{\rm out,\tr}=\alpharef \mu^{\rm in,\tr}=\alpharef \beta_*\mu^{\rm out,\tr}.
\end{equation}

Fix $0<c<1$ and for $\rho \in B^*\GammaI$, let
$$
N(\rho,c):=\inf\Big\{ N\geq 0\,:\, \sum_{j=0}^{N}\log(r(\beta^{-j}(\rho)))<-c\Big\}
$$
We claim that there exist $c_0,T_0>0$ such that for all $\rho\in B^*\GammaI$
\begin{equation}
\label{e:claimMe}
\sum_{j=0}^{N(\rho, c_0)} T_-(\beta^{-j}(\rho))< T_0R
\end{equation}
Once we prove this claim, using~\eqref{eq:reflectMeasure} together with the definition of $\mu^{\rm out,\tr}$ as the derivative along the flow of $\mu$, we see that if $A\subset S^*_{\overline{\tildedomain}\setminus \Omegaminus}\mathbb{R}^d\setminus \mathcal{G}$, then
\begin{equation*}
\mu\big(\varphi_{-T_0R}(A)\big)\geq e^{-c_0}\mu(A).
\end{equation*}
and hence the proof will be complete.

We now prove \eqref{e:claimMe}. If the claim fails then there is a sequence 
$$
(R_n,\rho_n,M_n)\in [1,\infty)\times B^*\GammaI(R_n)\times \mathbb{Z}
$$
such that 
\begin{equation}
\label{e:longTime}
\sum_{j=0}^{M_n}T_-(\beta^{-j}(\rho_n))\geq nR_n,\qquad \sum_{j=0}^{M_n}\log \alpharef (r(\beta^{-j}\rho_n))>-\frac{1}{n}.
\end{equation}
Without loss of generality, we can assume that $R_n\to R_\infty\in [1,\infty]$. 
Note that 
$$
\log \alpharef (\rho)= -\frac{4\sigma(\bcN)}{\sigma(\bcD)}\sqrt{r(\rho)}+O(r(\rho))
$$
By \eqref{eq:mainevent}, since $\frac{\sigma(\bcN)}{\sigma(\bcD)}>m>0$ on $S^*\GammaI$, 
\begin{equation}
\label{e:rIsSmall}
\sum_{j=0}^{M_n} \sqrt{(r(\beta^{-j}\rho_n))}\leq \frac{1}{4mn}.
\end{equation}
and in particular,
\begin{equation}
\label{e:supRIsSmall}
\sup_{0\leq j\leq M_n}r(\beta^{-j}(\rho_n))\leq \frac{1}{16m^2n^2}.
\end{equation}

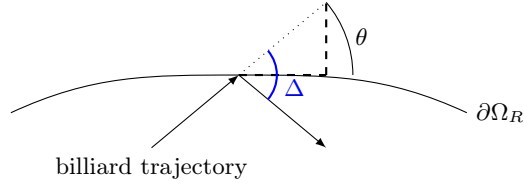
\begin{figure}
\begin{tikzpicture}
\def \a{40};
\def \r{1.5}
\def \l{.2};'
\draw (-3,-.5) to [out=25, in =180] (0,0) to [out=0, in =155](3,-.5) node[right]{$\partial\domain$};
\draw[->] (\a:-\r)node[below]{billiard trajectory}--(\a:0);
\draw [dotted](\a:0)--(\a:\r);
\draw[->](-\a:0)--(-\a:\r);
\draw[thick,dashed] (0,0)--({\r*cos (\a)},0);
\draw[thick,dashed] ({\r*cos(\a)},0)--(\a:\r);
\draw ({\r},0)arc (0:\a:{\r});
\draw node[right] at ({\a/2}:\r){$\theta$};
\draw[thick,blue] (-\a:{\r/3}) arc (-\a:\a:\r/3);
\draw node[right] at (-\a/2:\r/3){$\textcolor{blue}{\Delta}$};
\end{tikzpicture}
\caption{\label{f:ray}Ray construction showing the change, $\Delta$, in the angle of a ray when hitting the boundary at angle $\theta$. Note that $r=\sin^2\theta $.}
\end{figure}

Now, let $\pi_M:T^*M\to M$ and $\rho \in B^*\partial\tildedomain$. We consider the angle between the two vectors
$$
V_\pm(\rho):=d\pi_M(\partial_t \varphi_t (\pi_{\pm}^{-1}(\rho)))=2\xi(\pi_\pm^{-1}(\rho)).
$$
Note that $V_\pm$ are the tangent vectors to the billiard trajectory just before $(-)$ and after $(+)$ reflection. We define the angle accumulated at $\rho$, $\Delta(\rho)\in [0,\pi]$ by
$$
\langle V_{+}(\rho),V_-(\rho)\rangle =4 \cos \Delta(\rho).
$$ 
As can be seen, e.g., in Figure~\ref{f:ray}, 
$$
\sin (\Delta(\rho)/2)=\sqrt{r(\rho)},\qquad \cos(\Delta(\rho)/2)=\sqrt{1-r(\rho)}.
$$
In particular, 
$$
\sin(\Delta(\rho))=2\sqrt{r(\rho)}\sqrt{1-r(\rho)}.
$$
Therefore, 
$$
\Delta(\rho)=2\sqrt{r(\rho)}+O(r(\rho)^{3/2}).
$$
Now, note that if 
$$
\sum_{j=0}^k\Delta(\beta^{-j}(\rho))<\frac{\pi}{4},
$$
then
\begin{equation}
\label{e:travel}
|\pi_M(\rho)-\pi_M(\beta^{-k}(\rho))|\geq \frac{1}{\sqrt{2}}\sum_{j=0}^k T_-(\beta^{-j}(\rho)).
\end{equation}
By~\eqref{e:rIsSmall} and~\eqref{e:supRIsSmall}, 
$$
\sum_{j=0}^{M_n}\Delta(\beta^{-j}(\rho_n)) = \sum_{j=0}^{M_n}2\sqrt{r(\beta^{-j}(\rho_n))}+O(r(\beta^{-j}(\rho_n))^{3/2}\leq \frac{1}{2mn}+O(n^{-3})<\frac{\pi}{4}
$$
for $n$ large enough. In particular,~\eqref{e:longTime} and~\eqref{e:travel} imply that
$$
|\pi_M(\rho_n)-\pi_M(\beta^{-k}(\rho_n))|\geq \frac{1}{\sqrt{2}}\sum_{j=0}^k T_-(\beta^{-j}(\rho_n))\geq \frac{1}{\sqrt{2}}nR_n
$$
which, for $n$ large enough, is impossible since $\tildedomain\subset B(0,MR)$. 
\end{proof}

We now set up our contradiction argument to prove the 
bound \eqref{eq:trunc_estimate}.
Suppose there is no constant $C>0$ such that for all $R\geq 1$ the estimate fails. Then, there exists $\{R_\ell\}_{\ell=1}^\infty\subset [1,\infty)$, $\{h_{k,\ell}\}_{k,\ell=1}^\infty$, with $\lim_{k
\to \infty} h_{k,\ell}=0$, $u_{k,\ell}$, and $g_{k,\ell,\tr/D}$, $f_{k,\ell}$ such that  $\|u_{k,\ell}\|_{H_h^1(\widetilde{\Omega}_{R_\ell}\setminus \Omegaminus)}=1$, 
$$
\Big(\|g_{k,\ell,I}\|_{L^2(\GammaI(R_\ell))}+\|g_{k,\ell,D}\|_{H_{h_{k,\ell}}^1(\Gamma_D)}\Big)\leq R_\ell^{-1/2}\ell^{-1},\qquad \|f_{k,\ell}\|_{L^2(\widetilde{\Omega}(R_\ell)\setminus \Omegaminus)}\leq R_\ell^{-1}\ell^{-1},
$$ 
and such that
$$
\begin{cases}(-h_{k,\ell}^2\Delta-1)u_{k,\ell}=h_{k,\ell}f_{k,\ell}&\text{on }\widetilde{\Omega}_{R_\ell}\setminus \Omegaminus\\
(\bcN h_{k,\ell}D_n-\bcD)u_{k,\ell}=g_{k,\ell,I}&\text{on }\Gamma_{\tr, R_\ell}\\
u_{k,\ell}=g_{k,\ell,D}&\text{on }\Gsc.
\end{cases}
$$
Rescaling, we define
\begin{gather*}
\widetilde{u}_{k,\ell}(x)=R_\ell^{\frac{n}{2}}v_{k,\ell}(xR_\ell),\quad \widetilde{g}_{k,\ell,I}(x)=R_\ell^{\frac{n}{2}}g_{k,\ell,I}(xR_\ell),\\ \widetilde{f}_{k,\ell}(x)=R_\ell^{\frac{n+2}{2}}f_{k,\ell}(xR_\ell),\quad \widetilde{G}_{k,\ell,D}=R_{\ell}^{\frac{n}{2}}g_{l,\ell,D}(xR_\ell). 
\end{gather*}
Then,
\begin{gather*}
\|\widetilde{g}_{k,\ell,I}\|_{L^2(\Gamma_{\tr, R_\ell}/R_\ell)}+\|\widetilde{g}_{k,\ell,D}\|_{L^2(\Gsc/R_\ell)}\leq \frac{1}{\ell}, \qquad \|\widetilde{u}_{k,\ell}\|_{H_{h_{k,\ell}(\widetilde{\Omega}_{R_\ell})}^1}\geq 1-\frac{C}{R_\ell^{1/2}\ell},\qquad \|\widetilde{f}_{k,\ell}\|_{L^2}\leq \frac{1}{\ell},
\end{gather*}
and, with $U_\ell=(\widetilde{\Omega}_{R_\ell}/R_\ell)\setminus \overline{(\Omegaminus/R_\ell)}$, $\widetilde{\Gamma}_{D,\ell}=\Gamma_D/R_\ell$, $\widetilde{\Gamma}_{I,\ell}=\Gamma_{\tr, R_\ell}/R_\ell$, 
$$
\begin{cases}(-(h_{k,\ell}R_\ell^{-1})^2\Delta-1)\widetilde{u}_{k,\ell}=(h_{k,\ell}R^{-1}_{\ell})\widetilde{f}_{k,\ell}&\text{on }U_\ell\\
(\widetilde{\bcN} h_{k,\ell}R_{k,\ell}^{-1}D_n  -\widetilde{\bcD})\tilde{u}_{k,\ell}=\widetilde{g}_{k,\ell,I}&\text{on }\widetilde{\Gamma}_{I,\ell}\\
\tilde{u}_{k,\ell}|_{\tilde{\Gamma}_{D,\ell}}=\widetilde{G}_{k,\ell,D},
\end{cases}
$$
where, if a pseudodifferential operator $B$ on $\GammaI$ is given by 
$$
B={\rm Op}_h(b),\qquad b\sim \sum_j h^j b_j,
$$
then 
$$
\tilde{B}={\rm Op}_{hR^{-1}}(\tilde{b}),\qquad \tilde{b}\sim \sum_j (hR^{-1})^j R^jb_j.
$$
Putting $\widetilde{h}_{k,\ell}=h_{k,\ell}R_\ell^{-1}$, we have $\widetilde{h}_{k,\ell}\underset{k\to \infty}{\rightarrow} 0$ hence, extracting subsequences if necessary, we can assume that $u_{k,\ell}$ ($k\to \infty$) has a defect measure $\mu_\ell$  and by Corollaries~\ref{cor:L2RHS} and~\ref{cor:Dir} we can assume that the trace measures for $u_{k,\ell}$, $\nu_{d,\ell}^{I/D}$, $\nu_{n,\ell}^{I/D}$, and $\nu_{j,\ell}^{I/D}$ exist. Moreover, $\mu_{\ell}$ satisfies the relations from Proposition~\ref{lem:key_Miller} where $\mu^{\rm in/out}$. 
Finally, extracting even further subsequences, we can assume $\widetilde{g}_{k,\ell,I/D}$ have defect measures $\omega_{\ell,I/D}$, $\widetilde{f}_{k,\ell}$ has defect measure $\alpha_\ell$, 
and the joint measure of $\tilde{u}_{k,\ell}$ and $\tilde{f}_{k,\ell}$ is $\mu_\ell^j$ with 
\begin{gather*}
\omega_{\ell,I}(T^*\widetilde{\Gamma}_{I,\ell})\leq  \frac{1}{\ell^2},\qquad \omega_{\ell,D}(T^*\widetilde{\Gamma}_{D,\ell})\leq  \frac{1}{\ell^2},\qquad \alpha_\ell(T^*U_\ell)\leq \frac{1}{\ell^2},\\
| \mu^j_\ell(A)|\leq \sqrt{\mu_\ell(A)\alpha_\ell(A)}.
\end{gather*}
and $R_\ell\to R\in [1,\infty]$. Therefore, using e.g.~\cite[Lemma 4.2]{GaSpWu:20} together with Corollaries~\ref{cor:L2RHS} and~\ref{cor:Dir} to estimate the $H_{h_\ell/R_\ell}^1$ norm of $\widetilde{v}$ by its $L^2$ norm, 
$$
1=\limsup_{k}\|\widetilde{v}_{k,\ell}\|_{H_{h_{k,\ell}/R_\ell}^1}^2\geq \mu_\ell(T^*\mathbb{R}^d)\geq \liminf_{k}\|\widetilde{v}_{k,\ell}\|_{L^2}^2  \geq c\liminf_{k}\|\widetilde{v}_{k,\ell}\|_{H_{h_{k,\ell}/R_\ell}^1}^2\geq \frac{c}{2}>0.
$$

Note that each $\mu_\ell$ is a finite measure satisfying $\supp \mu_\ell\subset S^*_{B(0,M)}\mathbb{R}^d$. Therefore, the sequence $\mu_\ell$ is tight and bounded and hence by Prokhorov's theorem (see, e.g., \cite[Theorem 5.1, Page 59]{Bi:99} 
we can assume that $\mu_\ell\rightharpoonup \mu$ for some measure $\mu$. Moreover, $\supp \mu\subset S^*_{\overline{U_\infty}}\mathbb{R}^d$ and 
\begin{equation}
\label{e:ImABigMeasure}
1\geq \mu(S^*\mathbb{R}^d)>c>0.
\end{equation} 

\begin{lem}\label{lem:tight}
The sequences of boundary measures $\nu_{d,\ell}^{\tr}$, $\nu_{n,\ell}^{\tr}$, and $\nu_{j,\ell}^{\tr}$, and $\nu_{n,\ell}^D$ are tight.
\end{lem}
\begin{proof}
Since $\{r\geq 0\}\subset T^{*}\boundary _{I,\ell}$ is a compact set, we need only consider $r<0$. By Lemma~\ref{l:elliptic}, 
\begin{equation} 
\label{e:measureElliptic}
\Re \nu_{j,\ell}^{I/D}1_{r<0}=0,\qquad \nu_{n,\ell}^{I/D}1_{r<0}=-r\nu_{d,\ell}^{I/D}1_{r<0}.
\end{equation}
On the other hand, the boundary condition on $\GammaI$ gives for $a\in C_c^\infty(\{r<0\}),$
\begin{align*}
\big\langle a(x,\tilde{h}D)\widetilde{\bcN}\tilde{h}D_n u, u\big\rangle &=\big\langle a(x,\tilde{h}D)\widetilde{\bcD} u, u\big\rangle  +O(\ell^{-1})+o(1)_{\tilde{h}\to 0}.
\end{align*}
Sending $\tilde{h}\to0$, we obtain 
$$
\nu^{\tr}_{j,\ell}(\sigma(\bcN)a)=\nu^{\tr}_{d,\ell}(\sigma(\bcD)a)+O(\ell^{-1}).
$$
In particular, 
$$
\|\nu^{\tr}_{j,\ell}(\sigma(\bcN))1_{r<0}-\nu_{d,\ell}^{\tr}(\sigma(\bcD))1_{r<0}\|=O(\ell^{-1}).
$$
Now, since $\Re \nu^{\tr}_{j,\ell}=0$ and $\nu_{d,\ell}^{\tr}$, $\sigma(\bcD)$ are real, 
$$
\|\nu_{d,\ell}^{\tr}(\sigma(\bcD))1_{r<0}\|=O(\ell^{-1}).
$$
Similarly, for $a\in C_c^\infty(\{r<0\}),$
\begin{align*}
\big\langle a(x,\tilde{h}D)\tilde{h}D_n u, \bcD u\big\rangle &=\big\langle a(x,\tilde{h}D) u, \widetilde{\bcN}h D_{\nu} u\big\rangle  +O(\ell^{-1})+o(1)_{\tilde{h}\to 0},
\end{align*}
so that, since $\sigma(\bcN)$ and $\sigma(\bcD)$ are both  real,
$$
\|\nu^{\tr}_{j,\ell}(\sigma(\bcD))1_{r<0}-\nu^{\tr}_{n,\ell}(\sigma(\bcN))1_{r<0}\|=O(\ell^{-1}).
$$
and hence
$$
\|\nu^{\tr}_{j,\ell}(\sigma(\bcD))1_{r<0}+r\nu^{\tr}_{d,\ell}(\sigma(\bcN))1_{r<0}\|=O(\ell^{-1}),
$$
which again implies
$$
\|r\nu_{d,\ell}^{\tr}(\sigma(\bcN))1_{r<0}\|=O(\ell^{-1}).
$$

We now claim that 
\beq\label{eq:compact}
\text{ there exists } \e>0 \tst \big\{r|\sigma(\bcN)|\leq\e\big\}\cap \big\{|\sigma(\bcD)|\leq\e\big\}\text{ is compact},
\eeq
which then implies that
$\nu_{d,\ell}^{\tr}$ is tight. 
We now show that \eqref{eq:compact} holds in each of the three cases: $m_0>m_1+1, m_0<m_1+1$, and $m_0=m_1+1$.
If $m_0>m_1+1$, then  $ \big\{|\sigma(\bcD)|\leq c/2\big\}$ is compact by \eqref{eq:conditions2d} since $m_0\geq 0$ by \eqref{eq:mcond0}.
If $m_0<m_1+1$ and $m_1\geq -2$ then $\big\{r|\sigma(\bcN)|\leq c/2\big\}$ is compact by \eqref{eq:conditions2c}; observe that the inequality $m_1\geq -2$ follows from $m_0<m_1+1$ since $m_0\geq 0$ by \eqref{eq:mcond0}.
We now show that if $m_0=m_1+1$ then the first inequality in \eqref{eq:conditions} implies that there exists $C>0$ such that if $|\xi'|\geq C$ then the intersection \eqref{eq:compact} with $\e=\sqrt{c/2}$ (with $c$ the constant in \eqref{eq:conditions2a} is empty (and hence compact). Indeed, since $m_0\geq 0$ and $\langle \xi\rangle\geq 1$,
\beqs
\text{ if }  |\sigma(\bcD)|^2 \leq (c/2) \quad\text{ then }\quad |\sigma(\bcD)|^2 \leq (c/2) \langle\xi\rangle^{2m_0}.
\eeqs
Now, by the first inequality in \eqref{eq:conditions}
\beqs
\text{ if } |\sigma(\bcD)|^2 \leq( c/2) \langle\xi\rangle^{2m_0} \quad\text{ then }\quad |\sigma(\bcN)|^2 \leq (c/2) \langle\xi\rangle^{2m_1}.
\eeqs
If  $|\sigma(\bcN)|^2 \leq (c/2) \langle\xi\rangle^{2m_1}$ then, since $m_1 >-2$, $r^2 |\sigma(\bcN)|^2 \geq c/2$ for sufficiently large $\xi$,
and thus \eqref{eq:compact} indeed holds with $\e=\sqrt{c/2}$.

The tightness of $\nu_{d,\ell}^{\tr}$ and \eqref{e:measureElliptic} then imply that $\nu_{n,\ell}^{\tr}$ is tight and $|\nu_{j,\ell}^{\tr}|\leq \sqrt{\nu_{d,\ell}^{\tr}\nu_{n,\ell}^{\tr}}$ implies that $\nu_{j,\ell}^{\tr}$  is tight.
Next, the boundary condition on $\Gsc$ gives that 
$$
\nu_{d,\ell}^D=\omega_{\ell,D}\leq \frac{1}{\ell^2}.
$$
Hence, $\nu_{n,\ell}^D$ and $\nu_{j,\ell}^D$ are tight as above.
\end{proof}

Since the boundary measures form tight sequences, extracting subsequences if necessary, we can assume $\nu_{d,\ell}^{I/D}\rightharpoonup \nu_d^{I/D}$, $\nu_{n,\ell}^{I/D}\rightharpoonup \nu_n^{I/D}$, and $\nu_{j,\ell}^{I/D}\rightharpoonup \nu_j^{I/D}$ for some measures $\nu_d^{I/D}$, and $\nu_n^{I/D}$, and a complex measure $\nu_j^{I/D}$. Furthermore, $\nu_{d,\ell}^D=\omega_{\ell,D}\to 0$, and hence $\nu_{j,\ell}^D\to 0$. We also have $\alpha_\ell\to 0$.

Since these measures converge as distributions and $\tilde{\Gamma}_{I,\ell}\to \GammaIinf$ in $C^\infty$, the equations from Lemma~\ref{lem:key_Miller} and Theorem~\ref{t:propagate} hold for the limiting measures on $\GammaIinf$. (Here, we think of $\tilde{\Gamma}_{I,\ell}$ as a $C^\infty$ graph over $\GammaIinf$.) In addition, since $\alpha_\ell \to 0$, 
\begin{equation*}
\mu(H_pa)=\lim_{\ell\to \infty}\mu_\ell(H_p a)=0,\qquad a\in C_c^\infty( T^* U_\infty\setminus B(0,R^{-1})).
\end{equation*}
In addition, \eqref{eq:reflect} holds by Lemma~\ref{lem:key_Miller}.

We now introduce notation for various billiard flows in the next section. First, let $\varphi^{\ell}_t$ denote the billiard flow on $\mathbb{R}^d\setminus (\Omegaminus/R_{\ell})$. Then, define 
$$
\varphi_t^{\infty}(x,\xi)=\lim_{\ell\to \infty}\varphi^{\ell}_t(x,\xi),\qquad (x,\xi)\in S^*\big(\mathbb{R}^d\setminus (\Omegaminus/R)\big).
$$
Note that, the convergence to $\varphi_t^\infty$ is uniform and, in the case $R<\infty$, $\varphi_t^{\infty}(x,\xi)$ agrees with the billiard flow on $\mathbb{R}^d\setminus (\Omegaminus/R)$ and we identify the two flows.

\begin{prop}
Suppose that $T<\infty$ and $A\subset S^*_{U_\ell}\mathbb{R}^d$with
\begin{equation*}
\bigcup_{0\leq t\leq T}\varphi^{\ell}_t(A)\cap \tilde{\Gamma}_{I,\ell}=\emptyset.
\end{equation*}
Then,
$$
\lim_{\ell\to \infty}\sup_{t\in [0,T]}\big|\mu_\ell(\varphi^{\ell}_t(A))-\mu_\ell(A)\big|=0
$$
\end{prop}
\begin{proof}
This follow from Theorem~\ref{t:propagate} since 
$$\|\muin_{D,\ell}-\muout_{D,\ell}\|=2\|\Re \nu_{j,\ell}^D\|\leq C \sqrt{\|\omega_{D,\ell}\|}=O(\ell^{-1})
$$
and
$$
\|\mu^j_\ell\|\leq C\sqrt{\alpha_\ell}=O(\ell^{-1}).
$$
\end{proof}

Next, we show that $\mu_\infty$ is invariant under $\varphi_t^{\infty}$ when $R<\infty$. 
\begin{lem}
\label{lem:invarianceLimit}
Suppose that $R<\infty$ and that $A\subset S^*_{U_\infty}\mathbb{R}^d$ is closed and
$$
\bigcup_{0\leq t\leq T}\varphi_t^{\infty}(A)\cap \GammaIinf=\emptyset.
$$
Then, 
$$
\mu(\varphi_t^{\infty}(A))=\mu(A).
$$
\end{lem}
\begin{proof}
First, note that since the convergence of $\varphi^{\ell}_t$ to $\varphi_t^{\infty}$ is uniform, 
$$
\lim_{\ell\to \infty}d(\varphi_t^{\infty}(A),\varphi^{\ell}_t(A))=0.
$$
Therefore, fixing $\e>0$, for $\ell$ large enough, 
$$
\varphi^{\ell}_t(A)\subset \{(x,\xi)\,:\, d(\varphi_T(A),(x,\xi))<\e\}
$$
and 
$$
\varphi_\ell^{-T}(\varphi_t^{\infty}A)\subset \{(x,\xi)\,:\, d(A,(x,\xi))<\e\}
$$
Now, for finite times $T$, $\mu_\ell$ is invariant under $\varphi^{\ell}_t$ up to $o(1)_{\ell\to\infty}$. Combining this with the fact that our assumption on $A$ implies that, for $\ell$ large enough, $\varphi^{\ell}_t$ does not intersect $\GammaI$ in $[0,T]$, we have
$$
\mu_\ell(\varphi_t^{\infty}(A))=\mu_\ell(\varphi_\ell^{-T}\varphi_t^{\infty}(A))+o(1)_{\ell\to \infty}\leq \mu_\ell(\{(x,\xi)\,:\, \dist\big((x,\xi),A\big)<\e\}+o(1)_{\ell\to \infty}
$$
and 
$$
\mu_\ell(A)=\mu_\ell(\varphi^{\ell}_t(A))+o(1)_{\ell\to\infty}\leq \mu_\ell(\{(x,\xi)\,:\, \dist\big((x,\xi),\varphi_t^{\infty}(A)\big)<\e\}+o(1)_{\ell\to \infty}.
$$
Sending $\ell\to \infty$ and then $\e\to 0$, we obtain
$$
\mu(A)=\mu(\varphi_t^{\infty}(A))
$$
as claimed.
\end{proof}

\begin{rem*} Note that when $R=\infty$, the analogue of Lemma~\ref{lem:invarianceLimit} is obvious except on the sets $\{\xi=\pm \tfrac{x}{|x|}\}$ and $\{x=0\}$ since we can test $\mu$ against $H_pa$ away from these sets.
\end{rem*}

In the case $R=\infty$, we use the following lemmas.
\begin{lem}
\label{lem:zeroIsZero}
If $R=\infty$, then $\mu(\{x=0\}=0\}$.
\end{lem}
\begin{proof}
Fix $\e>0$.  Since $\Omegaminus$ is nontrapping and $\Gsc\Subset B(0,1)$, there is $T>0$ and $c>0$ such that 
$$
\bigcup_{\pm t\geq TR_\ell^{-1}}\varphi^{\ell}_t\big(|x|\leq 2R_{\ell}^{-1}\big)\cap \Big(\big\{|x|\leq 3R_{\ell}^{-1}\big\}\cup \big\{| \langle \tfrac{x}{|x|},\xi\rangle| \leq c\big\}\Big)=\emptyset. 
$$ 
Thus, for $\ell$ large enough 
$$
\varphi^\ell_{4\e}(|x|\leq \e)\subset \{ 2\e\leq |x|\leq 6\e, \,\xi\cdot \tfrac{x}{|x|}>c\}.
$$
In particular, there is $c>0$ such that  for $j\neq k$, $0\leq j<k<c\e^{-1}$
$$
\varphi^{\ell}_{4\e +c^{-1}k}(\{|x|\leq \e\})\cap \varphi^{\ell}_{4\e+c^{-1}j }(\{|x|\leq \e\})\Big)=\emptyset.
$$
Since $\mu_\ell(T^*\mathbb{R}^d)\leq 1$, this implies that 
$$
\mu_\ell(\{|x|\leq \e\})\leq C\e +o_{\ell\to \infty}(1)
$$
and hence, sending $\ell\to \infty$, 
$$
\mu(\{|x|\leq \e\})\leq C\e.
$$
Finally, sending $\e\to 0$ proves the claim.
\end{proof}

\begin{lem}
\label{lem:singularReflect}
If $R=\infty$ then $\mu_\infty$ is invariant under $\varphi_t^\infty$ away from $\GammaIinf$.
\end{lem}
\begin{proof}
Let 
$$
A_{\pm}:=\big\{\pm\xi=\tfrac{x}{|x|}\big\}\cap \big\{ |x|=\tfrac{1}{2\MPade } \big\}.
$$
 Note that $\mu_\ell$ is invariant under $\varphi_t^{\ell}$ modulo $o_{\ell\to \infty}(1)$. Now, $\widetilde{\Gamma}_{D,\ell}\subset B(0,R^{-1}_\ell)$. Since $R_\ell\to \infty$, and $\Omegaminus$ is nontrapping for $(x,\xi)\in A_-$,
$$
\lim_{\ell\to \infty}\sup_{(x,\xi)\in A_-}\dist(\varphi^\ell_{1/M}(x,\xi),A_+)=0.
$$ 
Similarly,
$$
\lim_{\ell\to \infty}\sup_{(x,\xi)\in A_+}\dist(\varphi^\ell_{-1/M}(x,\xi),A_-)=0.
$$ 
Now, for $\delta>0$ small enough, $-\delta\leq t\leq \delta$ and $\dist\big((x,\xi), A_{\pm}\big)\leq \delta$, $\varphi^{\ell}_t(x,\xi)=\varphi_t^{\infty}(x,\xi)$. In particular, for $B_0\subset A_-$,
$$
\mu_\ell\Big(\bigcup_{-\delta\leq t\leq \delta}\varphi_t^{\infty}(B_-)\Big)=\mu_\ell\Big(\bigcup_{-\delta\leq t\leq \delta}\varphi_t^{\ell}(B_-)\Big)=\mu_\ell\Big(\bigcup_{\frac{1}{M}-\delta\leq t\leq \frac{1}{M}+\delta}\varphi_t^{\ell}( B_-)\Big)+o_{\ell\to \infty}(1).
$$
Fix $\e>0$. Then for $\ell$ large enough,
$$
\bigcup_{1/M-\delta\leq t\leq 1/M+\delta}\varphi_t^{\ell}( B_-)\subset \bigcup_{-\delta\leq t\leq \delta}\varphi_t^{\ell}( \{(x,\xi)\,:\, \dist\big((x,\xi),\varphi_{1/M}^\infty(B_-)\big)\leq \e\}).
$$
In particular
\begin{align*}
\mu_\ell\Big(\bigcup_{-\delta\leq t\leq \delta}\varphi_t^{\infty}(B_-)\Big)&\leq \mu_\ell\Big( \bigcup_{-\delta\leq t\leq \delta}\varphi_t^{\ell}( \{(x,\xi)\mid \dist\big((x,\xi),\varphi_{1/M}^{\infty}(B_-)\big)\leq \e\})
\Big)+o(1)_{\ell\to\infty},\\
&= \mu_\ell\Big( \bigcup_{-\delta\leq t\leq \delta}\varphi_t^{\infty}( \{(x,\xi)\,:\, \dist\big((x,\xi),\varphi_{1/M}^{\infty}(B_-)\big)\leq \e\})
\Big)+o(1)_{\ell\to\infty},
\end{align*}
where in the last line we use that $\varphi^{\ell}_t=\varphi_t^{\infty}$ on the relevant set.
Similarly,  for $\ell$ large enough  (depending on $\e$), and $B_+\subset A_+$
$$
\mu_\ell\Big(\bigcup_{-\delta\leq t\leq \delta}\varphi_t^{\infty}(B_+)\Big)\leq \mu_\ell\Big( \bigcup_{-\delta\leq t\leq \delta}\varphi_t^{\infty}( \{(x,\xi)\,:\, \dist\big((x,\xi),\varphi_{-1/M}^{\infty}(B_+)\big)\leq \e\})
\Big)+o(1)_{\ell\to\infty}.
$$
Putting $B_+=\varphi_{1/M}^\infty(B_-)$, sending $\ell \to \infty$ and then $\e \to 0$, we obtain
\begin{align*}
\mu\Big(\bigcup_{-\delta\leq t\leq \delta}\varphi_t^{\infty}(B_+)\Big)\leq \mu\Big( \bigcup_{-\delta\leq t\leq \delta}\varphi^\infty_{t-1/M}( B_+)\Big)&= \mu\Big( \bigcup_{-\delta\leq t\leq \delta}\varphi^\infty_{t}( B_-)\Big)\\&\leq \mu\Big(\bigcup_{-\delta\leq t\leq \delta}\varphi^\infty_{t+1/M}(B_-)\Big)= \mu\Big( \bigcup_{-\delta\leq t\leq \delta}\varphi^\infty_{t}( B_+)\Big),
\end{align*}
and the  claim then follows from the fact that 
$$
\mu(H_pa)=0
$$
for all $a\in C_c^\infty( T^*_{U_\infty\setminus \{0\}}\mathbb{R}^d).$
\end{proof}

We now derive our contradiction to prove the bound \eqref{eq:trunc_estimate} and thus complete the proof of Theorem~\ref{th:uniformestimates2}. By~Lemmas \ref{lem:invarianceLimit},~\ref{lem:zeroIsZero}, and~\ref{lem:singularReflect}, $\mu$ is invariant under $\varphi_t^\infty$ away from $\GammaIinf$.  In particular, Lemma~\ref{lem:contradictMe} applies and we obtain that $\mu=0$, which is a contradiction to~\eqref{e:ImABigMeasure}.


\section{Proofs of the bounds on the relative error (Theorems \ref{th:lower}-\ref{th:local_square})}\label{sec:mainproofs}

As discussed in \S\ref{sec:outgoing}, the upper bounds in Theorem \ref{th:quant} and in Theorem \ref{th:quant3} follow from 
applying Theorem \ref{th:uniformestimates} to $u-v$ and then using Lemma \ref{lem:impedancetrace}. It therefore remains to prove the 
lower bounds in Theorems \ref{th:lower}, \ref{th:quant},  \ref{th:quant2}, \ref{th:local_ball}, and \ref{th:local_square}.

\subsection{Existence of defect measures}

\begin{lem}
\label{l:defectExist}
If $\Omegaminus$ is nontrapping, then 
Assumption \ref{ass:1} holds for $u$ and $v$ the solutions of \eqref{eq:BVP2} and \eqref{eq:BVPimp2}, respectively.
\end{lem}

\bpf
The bound on $\|\chi u\|_{L^2}$ follows from Lemma~\ref{lem:mass_u}; the bound on $\|hD_n u\|_{L^2(\Gsc)}$ follows from Corollary~\ref{cor:Dir} and that on $\|u\|_{L^2(\Gsc)}$ follows from the condition \eqref{eq:BC2} that $u|_{\Gamma_D}=\exp(i x\cdot a/h)$. 
The bound on $\|v\|_{L^2}$ follows from Theorem~\ref{th:uniformestimates}. The bounds on $\|v\|_{L^2(\GammaI)}$ and $\|hD_n v\|_{L^2(\GammaI)}$ follow from Corollary~\ref{cor:L2RHS}, and those for $\|hD_n u\|_{L^2(\Gsc)}$ from Corollary~\ref{cor:Dir}. The bound on $\|v\|_{L^2(\Gsc)}$ follows from the condition \eqref{eq:BCimp2} that $v|_{\Gamma_D}=\exp(i x\cdot a/h)$.
\epf

\bre[Neumann boundary conditions]\label{rem:Neumann}
We do not consider Neumann boundary conditions on $\Gamma_D$ because, as far as we know, propagation of measures for Neumann boundary conditions is not available. Indeed, the Neumann boundary condition does not satisfy the uniform Lopatinski--Shapiro condition (see, e.g., \cite[Part (ii) of Definition 20.1.1, Page 233]{Ho:85}) and, under Neumann boundary conditions, 
if $u$ is normalised so that $\|h \partial_n u\|_{L^2(\Gamma_D)}$ is bounded, then $\|u\|_{L^2(\Gamma_D)}$ is typically not uniformly bounded as $h\to 0$ (for example, when $\Gamma_D$ is the boundary of a ball; see, e.g., \cite[Equation 3.31]{Sp:14}); therefore Assumption \ref{ass:1} does not hold.
\ere

\subsection{Reduction to a lower bound on the measure of the incoming set}

\begin{lem}
There exists $C_{1}>0$ such that if $\{u_\ell\}_{\ell=1}^\infty$ and $\{v_\ell\}_{\ell=1}^\infty$ are sequences of solutions to ~\eqref{eq:BVP2} and~\eqref{eq:BVPimp2}, respectively, such that $u_\ell$ has a defect measure and $v_\ell$ has defect measure $\mu$, then
\begin{equation}\label{eq:reduction_to_lower}
\liminf_{\ell \to \infty}\frac{\Vert u_\ell- v_\ell \Vert_{L^2(\domain)}}{\Vert u_\ell  \Vert_{L^2(\domain)}} \geq C_1 \sqrt{\frac{\mu(\mathcal I)}{R}},
\end{equation}
and

\begin{equation}\label{eq:reduction_to_lower_local}
\liminf_{\ell \to \infty}\frac{\Vert u_\ell- v_\ell \Vert_{L^2(B(0,2)\setminus\Omegaminus)}}{\Vert u_\ell  \Vert_{L^2(B(0,2)\setminus\Omegaminus)}} \geq C_1 \sqrt{\mu\big(\mathcal I \cap(S_{B(0,3/2)}^*\mathbb R ^d)\big)},
\end{equation}
where $\mathcal I$ is the directly-incoming set defined by \eqref{eq:incoming}.
\end{lem}

\begin{proof}
Let $b\in C_{c}^{\infty}(S^{*}\domain)$ be supported in $\mathcal{I}$ and such that
$$
\int |b|^2 \ d\mu \geq \mu(\mathcal I)/2.
$$
If $\widetilde{\mu}$ is a defect measure of $u$, then $\widetilde{\mu}(\mathcal{I})=0$
by Lemma \ref{lem:null_incom}. By the definition of defect measures, \
$$\lim_{\ell \to \infty} \big\langle b(x,h_{\ell}D) u_{\ell}, b(x,h_{\ell}D) u_{\ell}\big\rangle=0,$$
and therefore
\begin{align*}
\mu(\mathcal I)/2 & \leq \lim_{\ell\to \infty}\big\langle b(x,h_{\ell}D)v_{\ell},b(x,h_{\ell}D)v_{\ell}\big\rangle\\
& = \lim_{\ell \to \infty}\Big(\big\langle b(x,h_{\ell}D)v_{\ell},b(x,h_{\ell}D)v_{\ell}\big\rangle
+\big\langle  b(x,h_{\ell}D)u_{\ell},b(x,h_{\ell}D)u_{\ell}\big\rangle\Big)\\
&\qquad-2\lim_{\ell \to \infty}\Re \langle b(x,h_{\ell}D)u_{\ell},b(x,h_{\ell}D)v_{\ell}\rangle
\\
& = \lim_{\ell \to \infty}\big\langle b(x,h_{\ell}D)(v_{\ell}-u_{\ell}),b(x,h_{\ell}D)(v_{\ell}-u_{\ell})\big\rangle\\
 & \lesssim\Vert u_{\ell}-v_{\ell}\Vert_{L^{2}(\domain)}^2
\end{align*}
(where the upper bound on $b(x, h_\ell D)$ is independent of $\cI$ by \cite[Theorem 5.1]{Zworski_semi}).
The bound \eqref{eq:reduction_to_lower} then follows from the upper bound on $\|u_{\ell}\|_{L^2(\domain)}$ in Lemma \ref{lem:mass_u}.
The estimate \eqref{eq:reduction_to_lower_local} is proved in the same way by taking $b$ supported in $S_{B(0,3)}^*\mathbb R ^d$ and such that 
$\int |b|^2 \ d\mu \geq \mu(\mathcal I \cap S_{B(0,3/2)}^* \mathbb{R}^d )/2$.
\end{proof}

\begin{cor}\label{cor:reduction}
Let $\{v_\ell\}_{\ell=1}^\infty$, $\{h_\ell\}_{\ell=1}^\infty$, and $\{a_\ell\}_{\ell=1}^\infty$ be sequences such that $v_\ell$ satisfies \eqref{eq:BVPimp2} with $a=a_\ell$ and $\{v_\ell\}_{\ell=1}^\infty$ has defect measure $\mu$.

(i) To prove Theorem \ref{th:lower} it is sufficient to prove that there exists $c_0>0$ 
that depends continuously on $\GammaIR$
such that
\beqs
\mu(\mathcal I)\geq c_0.
\eeqs

(ii) Having proved Theorem \ref{th:lower}, to prove the lower bound in Theorem \ref{th:quant} it is sufficient to prove that there exists $c_1>0$ (independent of $R$) and $R_0$ such that, for all $R\geq R_0$, 
\beq\label{eq:lowerboundmu}
\mu(\mathcal I)\geq
\frac{c_1}{R^{4\morderzero-1}}.
\eeq

(iii) Having proved Theorem \ref{th:lower}, to prove Theorem \ref{th:quant2} it is sufficient to prove that there exists $c_2>0$ and $R_0>0$ (independent of $R$) such that, for all $R\geq R_0$,
\beq\label{eq:lowerboundmu2}
\mu(\mathcal I)\geq 
c_2 R.
\eeq

(iv) To prove Theorem \ref{th:local_ball} it is sufficient to prove that there exists $c_3>0$ (independent of $R$) and $R_0\geq 2$ such that, for all $R\geq R_0$, 
\beq\label{eq:lowerboundmu3}
\mu\big(\mathcal I \cap(S_{B(0,3/2)}^*\mathbb R ^d)\big) \geq \frac{c_3}{R^{4\morderzero}}.
\eeq

(v) To prove Theorem \ref{th:local_square} it is sufficient to prove that there exists $c_4>0$ (independent of $R$) such that, 
for all $R\geq 2$,
\beq\label{eq:lowerboundmu4}
\mu\big(\mathcal I \cap(S_{B(0,3/2)}^*\mathbb R ^d)\big) \geq \frac{c_4}{R^{d-1}}.
\eeq
\end{cor}

\bpf
We prove Part (ii), i.e.~the lower bound in \eqref{eq:mainbound} in Theorem \ref{th:quant}; the proofs of the other parts are essentially identical and/or simpler.

We first show that it is sufficient to prove that there exists $\widetilde{C}_1=\widetilde{C}_1(\Omegaminus, \MPade,\NPade )$ and $R_0= R_0(\Omegaminus, \MPade,\NPade)>0$ such that for any $R\geq R_0 $, there exists $\widetilde{k}_0(R)>0$ such that, for any direction $a$, 
\beq\label{eq:mainbound_proof}
\frac{\Vert u-v\Vert_{L^{2}(\domain)}}{\Vert u\Vert_{L^{2}(\domain)}}\geq \frac{\widetilde{C}_1}{R^{2\morderzero}}\qquad\tfa k \geq k_0.
\eeq
Indeed, having proved \eqref{eq:mainbound_proof}, we let 
\beqs
C_1 := \min \left( \widetilde{C}_1, \, \min_{1\leq R\leq R_0} \frac{\|u-v\|_{L^2(\domain)}}{\|u\|_{L^2(\domain)}}\right).
\eeqs
By Theorem \ref{th:lower} and the fact that the constant $C$ in this theorem depends continuously on $R$, $C_1$ exists, is $>0$, and is independent of $k$. With this definition of $C_1$, \eqref{eq:mainbound_proof} implies that the lower bound in \eqref{eq:mainbound} holds with $k_0(R):= \widetilde{k}_0(R)$ for $R\geq R_0$, and $k_0(R)$ equal to the respective $k_0$ from Theorem \ref{th:lower} for $1\leq R\leq R_0$.

We now prove \eqref{eq:mainbound_proof}; seeking a contradiction, suppose that the converse of 
\eqref{eq:mainbound_proof} is true; that is, 
given $C_0>0$, for any $\widetilde{R}_0>0$ there exists $R\geq \widetilde{R}_0$ and sequences $\{h_\ell\}_{\ell=1}^\infty$, $\{a_\ell\}_{\ell=1}^\infty$ with $h_{\ell} \to 0$, $|a_{\ell}|=1$ such that the solutions $u_\ell$ and $v_\ell$ to \eqref{eq:BVP2} and \eqref{eq:BVPimp2} satisfy
\begin{equation}
\label{e:contradictMeAgain}
\frac{\|u_{\ell} -v_{\ell}\|_{L^2(\domain)}}{\|u_{\ell}\|_{L^2(\domain)}}\leq \frac{C_0}{R^{2\morderzero}}.
\end{equation}
By extracting subsequences, we can assume that $u_\ell$ has defect measure $\widetilde{\mu}$ and $v_\ell$ has defect measure $\mu$ by Lemma~\ref{l:defectExist}

Setting $\widetilde{R}_0:= R_0$, with $R_0$ such that \eqref{eq:lowerboundmu} holds for $R\geq R_0$, and using this lower bound on $\mu(\mathcal I)$ in 
\eqref{eq:reduction_to_lower}, we have
\beqs
\liminf_{\ell \to \infty}\frac{\Vert u_\ell- v_\ell \Vert_{L^2(\domain)}}{\Vert u_\ell  \Vert_{L^2(\domain)}} \geq \frac{ C_1 \sqrt{c_1}}{R^{2\morderzero}},
\eeqs
for all $R\geq \widetilde{R}_0$, which contradicts \eqref{e:contradictMeAgain} for $C_0< C_1 \sqrt{c_1}$, thus proving the lower bound in Theorem \ref{th:quant}.
\epf

\subsection{Outline of the ideas behind rest of the proofs, and the structure of the rest of this section}\label{sec:outlinerays}

By Corollary \ref{cor:reduction}, we need to prove lower bounds on the measure of the incoming set $\mu(\mathcal{I})$.
We argue by contradiction and assume that $\mu(\mathcal{I})$ is small. The overall plan is the following.

(i) Show that, since $\mu(\cI)$ is small, mass is created when incoming rays reflect off $\Gamma_D$ using Lemma \ref{lem:nuD} above.

(ii) Show that there exists a neighbourhood of rays starting from $\Gamma_D$ that hit $\GammaIR$ directly (i.e. without hitting $\Gamma_D$ in the meantime) and hit $\GammaIR$ at angles to the normal that are not zero, and not one of the special angles corresponding to the non-zero zeros $\{t_j\}_{j=1}^{\mvanish}$ of $q(t)\sqrt{1-t}-p(t)$ (these conditions are made more precise in Condition \ref{cond:ray} below).

(iii) Propagate the mass created in Point (i) on the rays constructed in Point (ii) using Part (i) of Corollary \ref{cor:explain} (to go from mass on $\Gamma_D$ to mass on $\GammaIR$). 

(iv) Show that mass is reflected on $\GammaIR$ using the expression for the reflection coefficient in Corollary \ref{cor:reflection} and the fact that the rays hit $\GammaIR$ away from angles where the reflection coefficient vanishes.

(v) Show that this reflected mass produces mass on $\cI$ using Part (ii) of Corollary \ref{cor:explain} (to go from mass on $\GammaIR$ to mass in $\domain$), contradicting the assumption that $\mu(\mathcal{I})$ is small.

For the quantitative (i.e.~explicit-in-$R$) bounds
the goal is to prove a lower bound on $\mu(\cI)$ that is explicit in $R$. 
Therefore, on top of the requirements on the rays in Point (ii) above, we need (a) the angles the rays hit $\GammaIR$ to have certain $R$-dependence (since this will affect the $R$-dependence of the reflection coefficient in Point (iv)), and (b) information about when the reflected rays next hit $\Gamma_D$.

For the bounds on the relative error in subsets of $\domain$ (Theorems \ref{th:local_ball} and \ref{th:local_square}), we also require information about when the rays return to a neighbourhood of $\Omegaminus$, since we need information about the defect-measure mass here (more specifically, $\mu(\cI\cap S^*_{B(0,3/2)}\Rea^d)$).

\noi\textbf{Outline of the rest of \S\ref{sec:mainproofs}.}

\S\ref{sec:prelimray} contains preliminary results required for the ray arguments.
\S\ref{sec:raystatement} states the condition the rays must satisfy (Condition \ref{cond:ray}) and 
results constructing rays satisfying this condition (Lemmas \ref{lem:ray1}-\ref{lem:ray4}).
\S\ref{sec:rayconstruction} proves Lemmas \ref{lem:ray1}-\ref{lem:ray4}.
\S\ref{sec:reflectionbound} bounds the reflection coefficient \eqref{eq:reflectioncoefficient} for rays satisfying Condition \ref{cond:ray}.
\S\ref{sec:lowerboundproof1} proves the qualitative (i.e.~not explicit in $R$) lower bound in Theorem \ref{th:lower}; the steps (i)-(v) above therefore appear in their simplest form in this proof.
\S\ref{sec:lowerboundproof2} proves the quantitative (i.e.~explicit in $R$) lower bounds in Theorem \ref{th:quant}, \ref{th:quant2}, \ref{th:local_ball}, \ref{th:local_square}.

\subsection{Preliminary results required for the ray arguments}
\label{sec:prelimray}

Recall that $\mathcal S^{d-1}$ denotes the $d$-dimensional unit sphere.
Given $a\in \mathbb{R}^d$ with $|a|=1$, let $\mathfrak R_a: \Gamma_D \rightarrow \mathcal{S}^{d-1}$ be defined by
$$
\mathfrak R_a (x')= \Big(\xi_1 = \sqrt{r(x', (a_T(x'))^{\flat})}, \xi' = (a_T(x'))^{\flat}\Big).
$$
The definition of the local coordinates in \S\ref{subsec:geo} and the fact that $\xi_1>0$ imply that
\beq\label{eq:mathfrakR}
\mathfrak R_a (x') = 
\begin{cases}
a - 2(n(x') \cdot a)n(x') & \text{if } a\cdot n(x') \leq 0, \\
a & \text{if } a\cdot n(x') \geq 0,
\end{cases}
\eeq
i.e.,~$\mathfrak R_a(x')$ is the reflection of $a$ from $\Gamma_D$ if $x'$ is in the illuminated part of $\Gamma_D$ and 
$\mathfrak R_a(x')$ is just $a$ if $x'$ is in the shadow part of $\Gamma_D$.

\begin{definition} \label{def:emanating}
Given $x'\in\Gamma_D$ and $a\in\mathbb R ^d$ with $|a|=1$,
\emph{the ray emanating from $x'$} is the ray 
starting from $\big(x = x', \xi=\mathfrak R_a(x'))$.
\end{definition}

\begin{definition} The ray emanating from $x'\in \Gsc$ is \emph{direct} if the flow along the ray, starting at $x'$, hits $\GammaI$ before hitting $\Gsc$.
\end{definition}

We now show that there are direct rays emanating from $\Gamma_D$ in every direction.

\begin{lem} \label{lem:all_directions}
Given $a\in \mathbb{R}^d$ with $|a|=1$. Let $\Gamma^{+,a}_D \subset \Gamma_D$ denote the set of points $x'$ of $\Gamma_D$ such that both $a\cdot n(x') \neq 0$ and the ray emanating from $x'$ is direct. Then, 
$$
\mathfrak R_a(\Gamma^{+,a}_D) = \mathcal S^{d-1}.
$$
\end{lem}

\begin{figure} 
\includegraphics[scale=0.7]{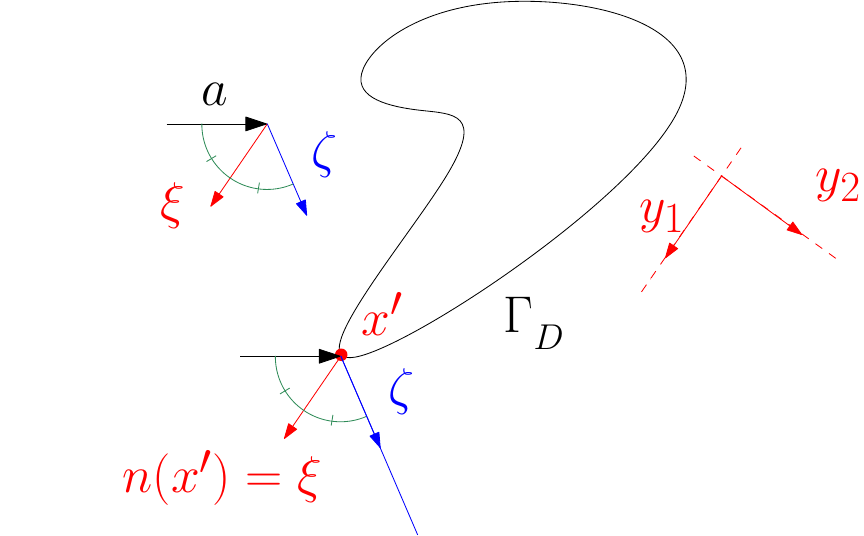}
\caption{Illustration of the proof of Lemma \ref{lem:all_directions} in the 
two-dimensional case; i.e., construction of a ray reflecting from $\Gamma_D$ in an arbitrary direction $\zeta$.
The point $x'$ has maximal $y_1$ coordinate, where the vector $\xi$ defines the $y_1$ axis, and $\xi$ is 
defined by $\zeta = a - 2(a\cdot\xi)\xi$.}   \label{fig:ray_construct}
\end{figure}

\begin{proof}
We first prove that  $a\in \mathfrak R_a(\Gamma^{+,a}_D)$. Without loss of generality $a=(1,0,\ldots,0)$. Let $x'_0 \in \Gamma_D$ be the point with maximal $x_1$ coordinate. Then $\mathfrak R_a(x_0')= a$ by \eqref{eq:mathfrakR}, $x'_0 \in \Gamma^{+,a}_D$ by the fact it has maximal $x_1$ coordinate, and so $a\in \mathfrak R_a(\Gamma^{+,a}_D)$.

We now need to show that, given $\zeta \in \mathcal S^{d -1}  \setminus  \{ a\}$, $\zeta \in \mathfrak R_a(\Gamma^{+,a}_D)$. Let $\mathcal P$ be the plane defined by $\mathcal P := \text{Span }(a,\zeta)$. Choose a cartesian
system of coordinates in which $\mathcal P = \{ x_3 = \cdots = x_n = 0 \}$, $a=(1,0,\cdots,0)$, and $(x_1, x_2)$ is right-handed oriented in $\mathcal P$. 
For $\xi \in \mathcal S^{d -1}$, let $r_a(\xi) := a  - 2(\xi\cdot a) \xi$;
i.e.~$r_a(\xi)$ is the reflection of $a$ from a boundary with normal $\xi$.
This definition implies that 
$$
r_a((\cos \omega, \sin \omega, 0, \cdots, 0)) = (\cos (2\omega-\pi), \sin (2\omega-\pi), 0, \cdots, 0),
$$
so that
$$
r_a(\mathcal D) = \big(\mathcal S^{d -1}\cap\mathcal P\big)  \setminus   \{ a\}, \; \text{where }\mathcal D := \left\{ (\cos \omega, \sin \omega, 0, \cdots, 0), \; \omega\in\left(\frac{\pi}{2}, \frac{3\pi}{2}\right)\right\}.
$$
Hence, there exists $\xi \in\mathcal D$ such that $r_a(\xi)=\zeta$. 

Finally, to show $\zeta \in \mathfrak R_a(\Gamma^{+,a}_D)$, we need to find $x'\in \Gamma_D^{+,a}$ such that $\mathfrak R_a (x') = r_a(\xi)$. 
Let $(y_1, \cdots, y_d )$ be a cartesian system of coordinates such that $\xi = (y_1 = 1, y_2 = 0, \cdots, y_d  = 0)$; see Figure \ref{fig:ray_construct}; let $x'$ be a point of $\Gamma_D$ with maximal $y_1$ coordinate. By definition, $n(x') = \xi$, and, since $\xi \in \mathcal D$,
$a\cdot n(x') < 0$. Therefore, $\mathfrak R_a (x') = r_a(n(x')) = r_a(\xi) = \zeta$. 
Since $x'$ has maximal $y_1$ coordinate in $\Gamma_D$, the ray emanating from $x'$ only intersects $\Gamma_D$ at $x'$, and thus $x' \in \Gamma^{+,a}_D$.
\end{proof}

The following dilation property $\mathfrak R_a(x')$ is need\dal{ed} for one of the proofs below (the proof of Lemma \ref{lem:ray4}).

\begin{lem} \label{lem:raydil}
Let $0<\delta<1$ and let $\mathcal C \subset \Gamma_D^{+,a}$ be uniformly convex (i.e.~the second fundamental form is positive definite)
 and such that, for any $x' \in \mathcal C$, $\delta \leq |n(x')\cdot a| \leq 1 - \delta$. Then,
there exists $\Cfrak>0$ and $\alpha_0 > 0$ such that,
for any $x' \in \mathcal C$ and any $0<\alpha\leq \alpha_0$, if $\partial B(x',\alpha) \cap \mathcal C \neq \emptyset$ and $\partial B(x',\alpha) \cap \mathcal \partial \mathcal C = \emptyset$, there exists $y' \in \partial B(x',\alpha) \cap \mathcal C $ so that
$$
| \mathfrak R_a(x') - \mathfrak R_a(y') | \geq \Cfrak |x'-y'| = \Cfrak\alpha.
$$
\end{lem}

\bpf[Proof  of Lemma \ref{lem:raydil}]

Let $({\tt x}_1, \cdots, \tt{x_d}) =: ({\tt x}_1, {\tt x}')$ be an Euclidian system of coordinates in which $a = (1,0,\cdots,0)$.
Since $\mathcal C$ is included in $\{ \delta \leq |n(x')\cdot a| \leq 1 - \delta \}$, there exists ${\tt X} \subset \{ { \tt x}_1 = 0\}$ and
a smooth map
$\gamma_{\mathcal C} : {\tt X} \longrightarrow \mathbb R$ such that $\mathcal C$ is given by, in this Euclidian system of
coordinates
$$
\mathcal C = \big\{(\gamma_{\mathcal C} ({\tt x}'),{\tt x}') \,:\, {\tt x}'\in {\tt X}\big\}.
$$
First observe that, for $x'=(\gamma_D({\tt x}'), {\tt x}') \in \mathcal C$ and $y'=(\gamma_D({\tt y}'), {\tt y}') \in \mathcal C$
\beqs
|x'-y'| \leq |{\tt x}'-{\tt y}'| + |\gamma_D({\tt x}') - \gamma_D({\tt y}')| \leq \Big(1+ \sup_{\tt X} |\nabla \gamma_{\mathcal C}|\Big)  |{\tt x}'-{\tt y}'|,
\eeqs
and hence
\beq \label{eq:raydil1:eq}
C_0|x'-y'|\leq |{\tt x}'-{\tt y}'| \leq  |x'-y'|, \quad\text{ where } \quad  C_0 :=  (1+ \sup_{\tt X} |\nabla \gamma_D|)^{-1}.
\eeq
By the definition of $\mathfrak R_a$ \eqref{eq:mathfrakR},
\beq \label{eq:raydil1:dec}
\mathfrak R_a(x') - \mathfrak R_a(y') = 2\big(H({\tt x}') - H({\tt y}')\big),
\eeq
\beqs
H({\tt x}') := (n({\tt x}')\cdot a)n({\tt x}')\quad\tand\quad 
n({\tt x}'):=(1,-\nabla \gamma_D({\tt x}'))/\sqrt{1+|\nabla \gamma_D({\tt x}')|^2}.
\eeqs
i.e., $n({\tt x}')$ is the outward-pointing normal to $\Gamma_D$ at $x'=(\gamma_D({\tt x}'),{\tt x}')\in\mathcal C$. 

Given $x'$, our plan is to use Taylor's theorem on $H$ to bound $|\mathfrak R_a(x') - \mathfrak R_a(y')|$ below, and then choose $y'$ appropriately so that this lower bound is $\geq  \Cfrak |x'-y'|$.
We first record that, since $|n({\tt x}')\cdot a| \leq 1-\delta$ and $a=(1,0,\ldots,0)$,
\beq \label{eq:raydil1:beta}
|\nabla \gamma_{\mathcal{C}}({\tt x}')| \geq (1-\delta)^{-2} - 1 =: \beta > 0.
\eeq
Let $H_1$ be the component of $H$ in the ${\tt x}_1$ direction (i.e., the direction of $a$), i.e.
\beq\label{eq:H1}
H_1({\tt x}') = \frac{1}{1+|\nabla \gamma_{\mathcal{C}}({\tt x}')|^2}.
\eeq
Then, using (\ref{eq:raydil1:dec}), Taylor's theorem, \eqref{eq:H1}, (\ref{eq:raydil1:eq}), and (\ref{eq:raydil1:beta}), we obtain
\begin{align}\nonumber
&\frac 12 \big|\mathfrak R_a(x') - \mathfrak R_a(y')\big| \geq \big|H_1({\tt x}') - H_1({\tt y}')\big| \geq  \big|\nabla H_1({\tt x}') \cdot({\tt x}'-{\tt y}')\big| - \sup_{\tt X} \big|\partial^2 H_1\big|\big|{\tt x}'-{\tt y}'\big|^2 \\ \nonumber
&\qquad= \left|\left\langle \frac{2 \partial^2 \gamma_{\mathcal C} ({\tt x}') \nabla \gamma_{\mathcal C} ({\tt x}')}{(1+|\nabla \gamma_{\mathcal C} ({\tt x}')|^2)^2}, \frac{{\tt x}'-{\tt y}'}{|{\tt x}'-{\tt y}'|}  \right\rangle\right| |{\tt x}'-{\tt y}'| - \sup_{\tt X}  \big|\partial^2 H_1\big|\big|{\tt x}'-{\tt y}'\big|^2 \\ \nonumber 
&\qquad=\frac{2|\nabla \gamma_{\mathcal C} ({\tt x}')|}{(1+|\nabla \gamma_{\mathcal C} ({\tt x}')|^2)^2} 
\left|\left\langle  \partial^2 \gamma_{\mathcal C} ({\tt x}') \frac{\nabla \gamma_{\mathcal C} ({\tt x}')}{|\nabla \gamma_{\mathcal C} ({\tt x}')|}, \frac{{\tt x}'-{\tt y}'}{|{\tt x}'-{\tt y}'|}  \right\rangle\right| |{\tt x}'-{\tt y}'| - \sup_{\tt X}  \big|\partial^2 H_1\big|\big|{\tt x}'-{\tt y}'\big|^2 \\
&\qquad\geq 2C_1 \beta C_0 Q_{\mathcal C}  \left| \left\langle v, \frac{{\tt x}'-{\tt y}'}{|{\tt x}'-{\tt y}'|}  \right\rangle\right| \big|x'-y'\big| - C_2\big|x'-y'\big|^2,
 \label{eq:raydil1:dil}
\end{align}
where
$$
v:= \left(\partial^2 \gamma_{\mathcal C} ({\tt x}') \frac{\nabla \gamma_{\mathcal C} ({\tt x}')}{|\nabla \gamma_{\mathcal C} ({\tt x}')|}
\right)
 \left| \partial^2 \gamma_{\mathcal C} ({\tt x}') \frac{\nabla \gamma_{\mathcal{C}}({\tt x}')}{|\nabla \gamma_{\mathcal C} ({\tt x}')|} \right|^{-1},
$$
and
$$
C_1 := \big(1+\sup_{\tt X} |\nabla \gamma_{\mathcal C} |^2\big)^{-2} > 0, \quad C_2 := \sup_{\tt X} |\partial^2 H_1| < \infty, \quad Q_{\mathcal C} := \inf_{{\tt x'}\in{\tt X}, |e| = 1} |\partial^2 \gamma_{\mathcal C} ({\tt x'}) e| > 0,
$$
where $Q_{\mathcal C} >0$ because $\mathcal C$ is uniformly convex.

We now claim that, under the assumption that $\partial B(x',\alpha) \cap \mathcal C \neq \emptyset$ and $\partial B(x',\alpha) \cap \mathcal \partial C = \emptyset$, it is always possible to choose $y' \in \mathcal C$  so that 
\beq\label{eq:x'y'}
|x' - y'|=\alpha
\quad\tand\quad \frac{{\tt x}'-{\tt y}'}{|{\tt x}'-{\tt y}'|} = v.
\eeq
Indeed, for $d\geq 3$, the projection of $\partial B(x',\alpha) \cap \mathcal C$ on the hyperplane $\{ {\tt x_1} = 0 \}$ is a closed hypersurface of $\mathbb R^{d-1}$ (e.g., for $d=3$ it is a closed curve). Since $\tt x'$ is in the geometrical interior of this hypersurface, for any $v\in \Rea^{d-1}$, there exists $y'$ satisfying \eqref{eq:x'y'}. 
For $d=2$, the projection of $\partial B(x',\alpha) \cap \mathcal C$ on the hyperplane $\{ {\tt x_1} = 0 \}$ equals two points (one on either side of ${\tt x'}$); since $v= \pm 1$ in this case, there exists $y'$ satisfying \eqref{eq:x'y'}.

For such a $y'\in \mathcal{C}$ satisfying \eqref{eq:x'y'}, by (\ref{eq:raydil1:dil}),
$$
|\mathfrak R_a({x}')- \mathfrak R_a({ y}')| \geq \Big( 2C_1 \beta C_0 Q_{\mathcal C} - C_2 \alpha\Big) \alpha;
$$
taking $\alpha_0 := C_0C_1 \beta Q_{\mathcal C}/C_2$ gives the result with $\Cfrak:= C_1\beta C_0 Q_{\mathcal C}$.
\epf 

\subsection{Statement of the lemmas constructing the rays}
\label{sec:raystatement}

\begin{condition}\label{cond:ray}
Given $\{\psi_j\}_{j=1}^m \in (0,\pi/2]$, 
there exist $c_{{\rm ray}, j}, j=1,\ldots,5$, such that, given $a\in\mathbb R ^d$ with $|a|=1$, there exists  $V_D\subset \Gamma_D$ such that 

(i) $\vol(V_D)\geq \crayone$,

(ii) $|n(x') \cdot a|\geq \craytwo$  for all $x'\in V_D$, 

(iii) the emanating rays from $V_D$ hit $\GammaIR$ directly and, for each ray, the angle $\theta$ the ray makes with the normal  satisfies
\beq\label{eq:anglebounds}
\theta \geq \craythree \quad\tand\quad
\min_{j=1,\ldots,m} | \theta - \psi_j| \geq \crayfour,
\eeq

(iv) after hitting $\GammaIR$, the rays travel a distance $\geq \crayfive$ before hitting either $\GammaIR$ or $\Gamma_D$ again.
\end{condition}

The $\{\psi_j\}_{j=1}^m$ in Condition \ref{cond:ray} are arbitrary angles, but in the proofs below we choose them to 
be the angles at which the reflection coefficient on $\GammaIR$ (i.e. \eqref{eq:reflectioncoefficient}) vanishes, i.e., the angles 
corresponding to the 
zeros of $q(t)\sqrt{1-t}-p(t)$ in $(0,1]$. We set 
\beq\label{eq:psi_j}
\psi_j := \sin^{-1} \sqrt{t_j} \in (0,\pi/2], \qquad j=1,\ldots,\mvanish,
\eeq
where $\{t_j\}_{j=1}^{\mvanish}$ are defined at the end of \S\ref{sec:setup}.
Then, when $|\xi'|_g= \sin \psi_j$ for some $j=1,\ldots,\mvanish$,
$\sigma(\bcN)\sqrt{r}- \sigma(\bcD)= q(t_j)\sqrt{1-t_j} - p(t_j)=0.$

We now state four lemmas constructing the rays used to prove the different lower bounds on $\mu(\mathcal{I})$ required by Corollary \ref{cor:reduction}.

\begin{lem}[The rays for general strictly-convex $\GammaIR$] \label{lem:ray1}
If $\GammaIR$ is strictly convex, then Condition \ref{cond:ray} holds with 
$c_{{\rm ray}, j}= c_{{\rm ray}, j}(\Gamma_D, \GammaIR)$ for $j=1,3,4,5,$ and $\craytwo= \craytwo(\Gamma_D)$. 
Furthermore $c_{{\rm ray},j}$, $j=1,3,4,5,$ are continuous in $R$.
\end{lem}

\begin{lem}[The rays for $\GammaIR= \partial B (0,R)$] \label{lem:ray2}
If $\GammaIR= \partial B (0,R)$ then there exists $R_0>0$ such that Condition \ref{cond:ray} holds for all $R\geq R_0$ with $\crayone, \craytwo, \crayfour$ independent of $R$, $\craythree =\widetilde{c}_3/R$ and $\crayfive=\widetilde{c}_5 R$
with $\widetilde{c}_3, \widetilde{c}_5>0$ independent of $R$. Furthermore,

(iv)$^\prime$ after their first reflection from $\GammaIR$, all of the rays hit $B(0,1)$.
\end{lem}

\begin{lem}[The rays for generic $\GammaIR$] \label{lem:ray3}
If  $\GammaIR$ satisfies the assumptions of Theorem \ref{th:quant2}, then Condition \ref{cond:ray} holds for $R$ sufficiently large with 
$c_{{\rm ray},j}, j=1,\ldots,4$, independent of $R$  and $\crayfive=\widetilde{c}_5 R$
with $\widetilde{c}_5>0$ independent of $R$.

\end{lem}

\begin{lem}[The rays for when $\GammaIR$ is a smoothed hypercube] \label{lem:ray4}
Let $\GammaIR$ coincide with the boundary of the hypercube $[-R/2, R/2]^d$ at distance more than $\epsilon$ from the corners (as described in the statement of Theorem \ref{th:local_square}).

There exists $\epsilon_0>0$ and $M\in \mathbb{Z}^+$ (both dependent on $\Gsc$ but not on $R$)  such that, 
if $0<\epsilon\leq\epsilon_0$ and $R\geq 4$, then Condition \ref{cond:ray} holds with 
$\craytwo$, $\craythree$, and $\crayfour$ independent of $R$, 
$\crayone =  \widetilde{c}_{\rm ray,1}/R^{d-1}$ and $\crayfive=\widetilde{c}_5 R$ with $\widetilde{c}_1, \widetilde{c}_5>0$ independent of $R$, and 

(iv)$^\prime$ the emanating rays from $V_D$ hit $\GammaIR$ $N(R)\leq M$ times, each time with an angle $\theta$ to the normal satisfying 
\eqref{eq:anglebounds} without hitting 
$\Gamma_D$ in between, and then, after their $N(R)$th reflection, the rays intersect $B(0,  3/2)  \setminus  B(0, 5/4)$ 
before hitting either $\Gamma_D$ or $\GammaIR$ again.
\end{lem}

\subsection{The ideas used in the proofs of the lemmas constructing the rays}\label{sec:idea_rays}

In this subsection, we outline the ideas used in the proofs of Lemmas \ref{lem:ray1}-\ref{lem:ray3}, in the simplest possible case when $\MPade=\NPade=0$ (i.e., the boundary condition on $\GammaIR$ is the impedance boundary condition \eqref{eq:impedance1}). In this case $\mvanish=0$ and there are no non-zero angles $\psi_j$; when such angles exist, mass needs to be excluded in a careful way from the neighbourhoods described below so that the rays avoid these angles. The proof of Lemma \ref{lem:ray4} has a different character to the proofs of Lemmas \ref{lem:ray1}-\ref{lem:ray3}, and so we postpone discussion of the ideas of that proof until the start of the proof itself.

The idea behind the ray construction for general strictly-convex $\GammaI$ in Lemma \ref{lem:ray1} is as follows.
We consider a point $x'_0$ in $\Gamma_D$ that
is the extremum point on $\Gamma_D$ in the direction of $a$. The rays emanating from a neighbourhood of this point are rays in the direction $a$, and thus hit  $\GammaIR$ directly.  Since  $\GammaIR$ is 
strictly-convex, these rays cannot be normal to $\GammaIR$ at more than one point, see Figure \ref{fig:ray_casgen}, and thus the required neighbourhood exists.

For the proof of Lemma \ref{lem:ray2}, we need in addition to quantify how far from the normal the ray described in the last paragraph hits $\GammaIR$. When $\GammaIR = \partial B(0,R)$, we show that a set of points of volume $c>0$ can reach $\GammaIR$  with an angle $|\theta| \gtrsim R^{-1}$; see Figure \ref{fig:ray_circle}.

For the proof of Lemma \ref{lem:ray3}, i.e.~when $\GammaIinf := \lim_{R\tendi} (\GammaIR/R)$ is \emph{not} a sphere centred at zero, we  recall from Lemma \ref{lem:all_directions} that,
given any direction, there exists a direct ray emanating from $\Gamma_D$ in that direction. We need to show that at least one of these rays hits $\GammaIinf$ non-normally.
Since $\GammaIinf$ is not a sphere centred at the origin, there exists $ x ^\infty_0 \in \GammaIinf$ with $n_{\GammaIinf}( x ^\infty_0) \neq x ^\infty_0/| x ^\infty_0|$. We use Lemma \ref{lem:all_directions} to identify a point $x'_0 \in \Gamma_D$ such that the ray emanating from $x'_0$ is 
in the direction 
$x ^\infty_0/| x ^\infty_0|$ and
does not hit $\Gamma_D$ again. In the limit $R\rightarrow \infty$, the rescaled obstacle $\Omegaminus/R$ shrinks to the origin; therefore the rays emanating from a neighbourhood of $x'_0$ hit $\GammaIR$ with an angle close to the angle between $ x ^\infty_0/| x ^\infty_0|$ and $n_{\GammaIinf}(x_0^\infty)$; this angle is $\geq c >0$, with $c$ independent of $R$; see Figure \ref{fig:ray_noncircle}.

\begin{figure}
\includegraphics[scale=0.4]{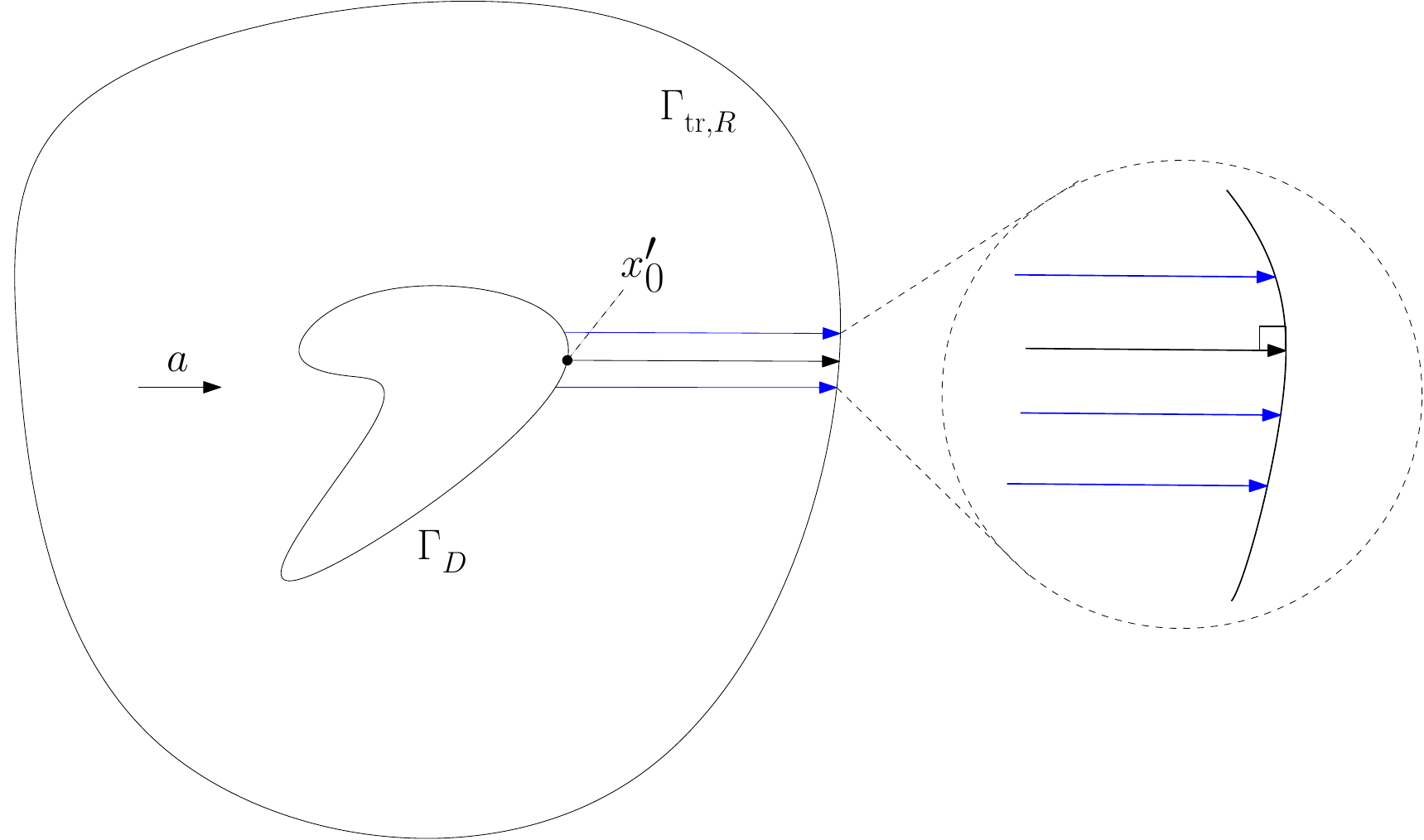} 
\caption{The rays in Lemma \ref{lem:ray1} (i.e., for general 
strictly convex $\GammaI$). 
Neighbourhoods on $\Gamma_D$ from which any of the blue rays emanate satisfy Condition \ref{cond:ray}.}  \label{fig:ray_casgen}
\end{figure}

\begin{figure} 
\includegraphics[scale=0.4]{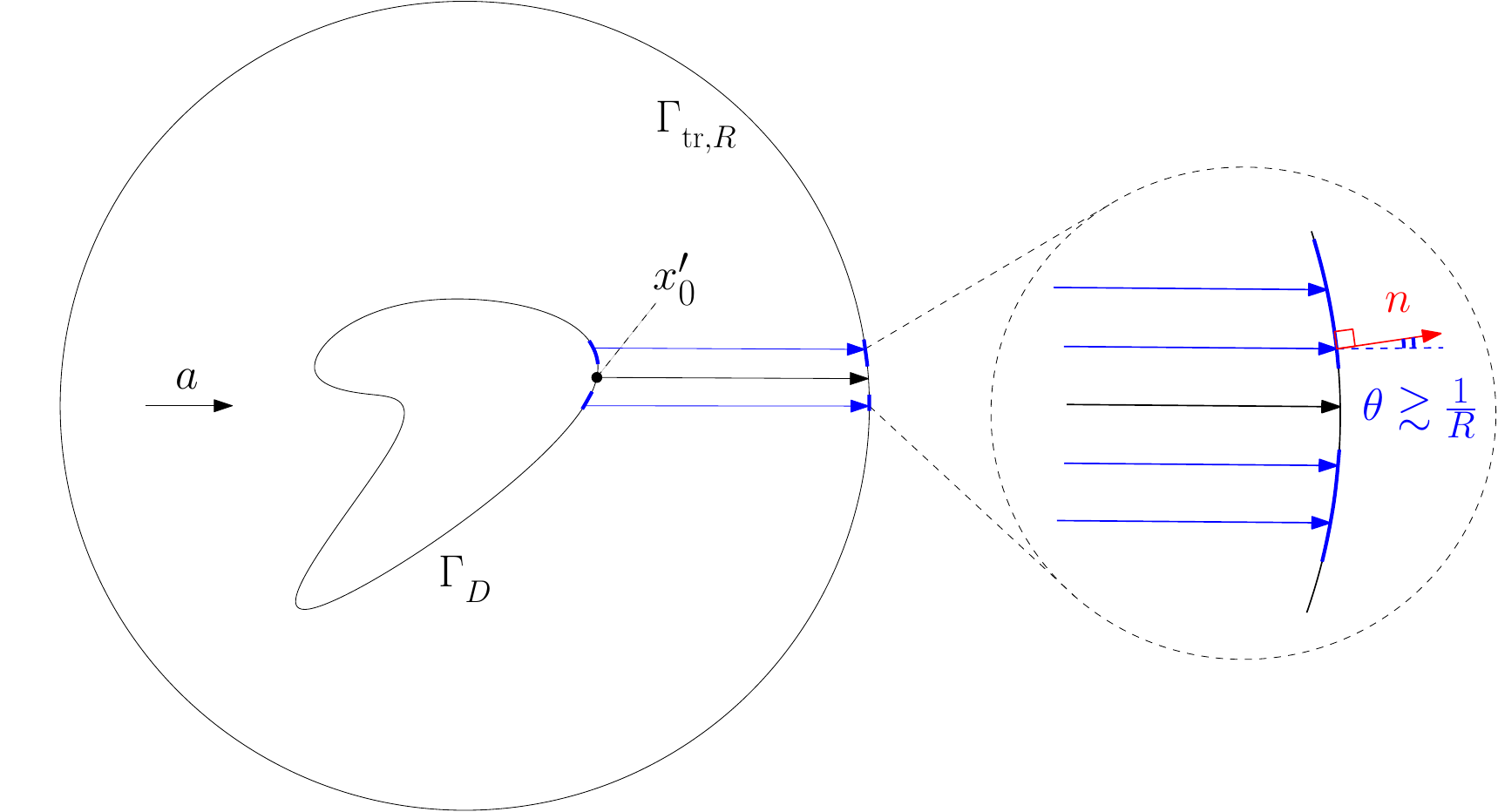}
\caption{The rays in Lemma \ref{lem:ray2} (i.e., for $\GammaIR = \partial B(0,R)$).
Neighbourhoods on $\Gamma_D$ from which any of the blue rays emanate satisfy Condition \ref{cond:ray}.}   \label{fig:ray_circle}
\end{figure}

\begin{figure}  
\includegraphics[scale=0.5]{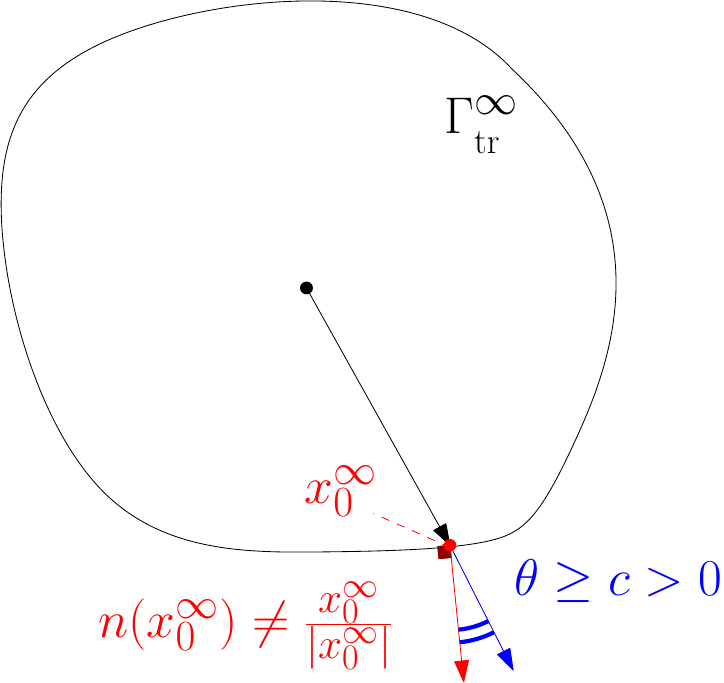}
\caption{The rays in Lemma \ref{lem:ray3}, i.e., when $\GammaIinf$ is not a ball centred at the origin. The figure shows the rescaled domain in the limit $R\rightarrow \infty$ (recall that in this limit the obstacle shrinks to the origin).} \label{fig:ray_noncircle}
\end{figure}

\subsection{Proofs of Lemmas \ref{lem:ray1}-\ref{lem:ray4}}\label{sec:rayconstruction}

In the proofs of these lemmas we use the notation that $\reallywidehat{(b_1,b_2)}$ is the angle between vectors $b_1$ and $b_2$; i.e.
\beqs
\reallywidehat{(b_1,b_2)}:= \cos^{-1}\left( \frac{b_1 \cdot b_2}{|b_1||b_2|}\right),
\eeqs
where the range of $\cos^{-1}$ is $[0,\pi]$.

\begin{proof}[Proof of Lemma \ref{lem:ray1}]

\

\emph{Step 1. Construction of direct emanating rays in the direction of $a$.}

\noindent Without loss of generality, we assume that $a=(1,0,\dots,0)$. Let $x_0'\in \Gsc$ be the point on $\Gamma_D$ with maximal $x_1$ coordinate. By translating the obstacle $\Omegaminus$, we can assume that $x_0'=0$. Then, locally near $0$, for any $0<\epsilon\leq\epsilon_0(\Gamma_D)$, where $\epsilon_0$ is small enough
\begin{equation}
\label{eq:graph}
\Gsc\cap B(0,\epsilon)\subset \big\{(\gamma_D(x'),x') \,:\, x'\in B(0,\epsilon)\subset \mathbb{R}^{d-1}\big\}
\end{equation}
where $\gamma_D \in C^\infty(\mathbb{R}^{d-1})$ and $\partial \gamma_D(0)=0$, and $\gamma_D(x')\leq 0$. Moreover, 
for $\epsilon_0>0$ small enough and $0<\epsilon\leq\epsilon_0$
\begin{equation}
\label{eq:above}
\Gsc\cap \big\{(x_1,x')\,:\, x_1> \gamma_D(x') \,\tand\, x'\in B(0,\epsilon)\big\}=\emptyset.
\end{equation}
Indeed, if not then there exist $x_n'\to 0$, $(y_n,x_n')\in \Gsc$ such that $y_n>\gamma_D(x_n')$. But then, extracting subsequences if necessary, $(y_n,x_n')\to (y,0)\in \Gsc$ and $y\geq \gamma_D(0)$. In particular, by maximality of the $x_1$ coordinate at $x_0'$, $y=0$. But, near $x_0'$~(\ref{eq:graph}) holds and in particular, for $n$ large enough, $y_n=\gamma_D(x_n')$, which is a contradiction.

Observe that, shrinking $\epsilon_0>0$ if necessary, $a$ is outward-pointing along $\Gsc\cap B(0,\epsilon_0)$, and
\beq \label{eq:ray1:tang}
|n(x') \cdot a|\geq \craytwo(\epsilon_0), \; \tfa x' \in B(0,\epsilon_0),
\eeq
where $\craytwo(\epsilon_0)>0$ depends only on $\epsilon_0$ and hence $\Gamma_D$.
By \eqref{eq:mathfrakR}, $\mathfrak R_a(x')=a$ for all $x' \in \Gsc\cap B(0,\epsilon_0)$, and thus the rays emanating from $\Gsc\cap B(0,\epsilon_0)$ are
 the rays in the $x_1$ direction; see Figure \ref{fig:ray_casgen}. By~\eqref{eq:above}, these rays 
 hit $\GammaI$ before hitting $\Gsc$ again.
 The neighbourhood $V_D$ will be a subset of $B(0,\epsilon_0)$, and thus Point (ii) in Condition \ref{cond:ray} follows.

\emph{Step 2. Parametrisation of $\GammaI$.}

\noindent Let $\gamma_\tr:B(0,\epsilon_0) \subset \mathbb R ^{d-1}\to \mathbb{R}_+$ be such that 
$$
\GammaI\cap \{x_1>0,\,|x'|<\epsilon_0\}:=\big\{(\gamma_{\tr}(x'),x')\,:\, |x'|<\epsilon_0\big\};
$$
since $\GammaI$ is strictly convex, this property holds without needing to reduce $\epsilon_0$ and thus $\epsilon_0$ still only depends on $\Gamma_D$.
The outward-pointing normal to $\GammaI$ is given by 
$$
n_{\rm tr}(x'):=\frac{(1,-\nabla \gamma_{\tr}(x'))}{\sqrt{1+|\nabla \gamma_{\tr}(x')|^2}}.
$$
For $x' \in B(0,\epsilon_0) \subset \mathbb R^{d-1}$, let $\theta(x')\in [0,\pi/2)$ be the angle between the ray emanating from $(\gamma_D(x'),x')$ and the normal to $\GammaI$; since $\cos \theta(x')=\big(1,0,\dots 0\big)\cdot n_{\rm tr}(x')$,
\beq \label{eq:ray1:theta}
\theta(x') = \cos^{-1}\left( \frac{1}{\sqrt{1+|\nabla \gamma_{\tr}(x')|^2}}\right)\in [0,\pi/2).
\eeq
We use later the facts, obtained from 
from \eqref{eq:ray1:theta} by direct calculation, that,
\beq\label{eq:tantheta}
\tan \theta(x') = |\nabla \gamma_\tr(x')|,
\eeq
and, in $\{ \nabla \gamma_{\tr}(x') \neq 0\}$,
\beq \label{eq:ray1:dtheta}
\nabla \theta(x') = \frac{1}{1+|\nabla \gamma_{\tr}(x')|^2} \partial^2 \gamma_{\tr} (x') \frac{\nabla \gamma_{\tr}(x')}{| \nabla \gamma_{\tr}(x') |}.
\eeq
We also use the following quantities,
\beq \label{eq:ray1:quant}
Q:=\inf_{x',|v|=1} | \partial^2\gamma_{\tr}(x')v| \qquad \tand\qquad
C_k:=\sup_{x'} \max_{|\mathbf{k}|=k} |\partial^{\mathbf{k}}\gamma_{\tr}(x')|, \; k=1,2,3.
\eeq

\emph{Step 3. Avoiding the angle $\psi_i = 0$.}

\noi Recall that our goal is to construct  $V_D \subset \Gsc\cap B(0,\epsilon)$ so that
$$
\min_{i=1,\ldots,m} | \theta(x') - \psi_i | \geq C > 0 \quad\tfa x'\in V_D,
$$
where $\vol(V_D)$ and $C$ depend only on $\GammaIR$. 
Our plan is to exclude mass from $B(0,\epsilon)$ for each $i$, taking care that the volume is still bounded below to give Point (i) of Condition \ref{cond:ray}.

Avoiding the angle zero corresponds to obtaining a lower bound on $|\theta(x')|$. By Taylor's theorem,
$$
|\nabla\gamma_{\tr}(x')|\geq  |\nabla \gamma_{\tr}(0)+\partial^2\gamma_{\tr}(0)x'|- \widetilde{C}_d C_3|x'|^2,
$$
where $C_3$ is defined by \eqref{eq:ray1:quant}, and $\widetilde{C}_d$ depends only on $d$.
By the definition of $Q$ in \eqref{eq:ray1:quant},
$$
\big|\nabla \gamma_{\tr}(0)+\partial^2\gamma_{\tr}(0)x'\big|
=\Big|\partial^2\gamma_{\tr}(0)\Big(\big(\partial^2\gamma_{\tr}(0)\big)^{-1}\nabla\gamma_{\tr}(0)+x'\Big)\Big|
\geq Q\Big|\big(\partial^2\gamma_{\tr}(0)\big)^{-1}\nabla\gamma_{\tr}(0)+x'\Big|.
$$
Suppose  that $|(\partial^2\gamma_{\tr}(0))^{-1}\nabla\gamma_{\tr}(0)|\leq \e/3$. Then
$$
|\nabla \gamma_{\tr}(x')|\geq \frac{Q\e}{6}-\widetilde{C}_dC_3\e^2 \quad\tfor x' \in B(0,\epsilon) \setminus B(0,\epsilon/2).
$$
On the other hand, if $|\big(\partial^2\gamma_{\tr}(0)\big)^{-1}\nabla\gamma_{\tr}(0)|\geq \e/3$, then
$$
|\nabla \gamma_\tr(x')|\geq \frac{Q\e}{6}-\frac{\widetilde{C}_dC_3\e^2}{36} \quad\tfor x' \in B(0,\epsilon/6).
$$
Therefore, in both cases, if $\e\leq Q/(12\widetilde{C}_dC_3)$, then there exists a set $W$ with 
\beq \label{eq:ray1:contrW}
\vol (W) \leq \max(2^{-d}, 1-6^{-d})\vol (B(0,\epsilon)) =  (1 - 6^{-d})\vol (B(0,\epsilon)).
\eeq
such that
$$
|\nabla \gamma_\tr(x')|\geq \frac{Q\e}{12}\quad \tfa  x' \in B(0,\epsilon)\setminus W.
$$
Therefore, for $x' \in B(0,\epsilon)\setminus W$, by (\ref{eq:ray1:theta})
$$
1-\frac{\theta(x')^2}{2}\leq \cos \theta(x') \leq 1 - \frac{|\nabla\gamma_{\tr}(x')|^2}{2}\leq 1- \frac{Q^2\epsilon^2}{288},
$$
and we conclude that
\beq \label{eq:ray1:step4}
\text{ if } 0<\epsilon \leq \min\left(\frac{Q}{12\widetilde{C}_dC_3}, \epsilon_0\right), \quad\text{ then }
\theta(x') \geq \frac{Q\epsilon}{12} \hspace{0.3cm} \tfa  x' \in B(0,\epsilon)  \setminus  W.
\eeq

\emph{Step 4. Avoiding the angles $\psi_i$.}

Given $\psi_i$, let $x'_i \in \overline{B(0,\epsilon)}\subset \mathbb R^{d-1}$ be such that
\beq \label{eq:ray1:defxi}
|\theta(x_i') - \psi_i| = \min_{x'\in \overline{B(0,\epsilon)}} |\theta(x') - \psi_i|,
\eeq
i.e.,~$x'_i$ is the point in $\overline{B(0,\epsilon)}$ where $\theta(x')$ is closest to $\psi_i$.
Let
\beq \label{eq:ray1:beta}
\psi_{\rm min} := \min_{j=1,\ldots,m} \psi_j > 0,
\eeq
In the following we use the notation $[a,b]$ for the line segment between $a$ and $b$, i.e.
$$
[a,b] := \big\{ ta+(1-t)b, \; t\in[0,1]\big\},
$$
and $\langle\cdot,\cdot\rangle$ denotes the Euclidean inner product on $\Rea^d$.

The main idea of the rest of this step is the following: $|\theta(x')- \psi_i|$ is, by definition, smallest at $x'_i$, and will be smallest when the minimum in \eqref{eq:ray1:beta} is attained, i.e. $\theta(x_i')=\psi_i$; in this case, the idea is for the size of the neighbourhood of $x_i'$ that we exclude to be dictated by using Taylor's theorem
\begin{align} \nonumber 
| \theta(x') - \theta(x_i') |
&\geq | \nabla \theta (x'_i) \cdot (x'-x'_i) | - \sup_{y' \in [x',x'_i]}\max_{|\mathbf{k}| = 2} \big| \partial^{\mathbf{k}} \theta(y')\big| |x' - x'_i|^2  \nonumber \\
&= \frac{1}{1+|\nabla \gamma_{\tr}(x_i')|^2} \left| \left\langle \partial^2 \gamma_{\tr} (x_i') \frac{\nabla \gamma_{\tr}(x_i')}{| \nabla \gamma_{\tr}(x_i') |}, x'-x'_i \right\rangle \right| -\sup_{y' \in [x',x'_i]} \max_{|\mathbf{k}| = 2} \big| \partial^{\mathbf{k}} \theta(y')\big|  | x' - x'_i|^2, \label{eq:ray1:alpha_explain}
\end{align}
where the requirement that the right-hand side is bounded below determines the size of the excluded neighbourhood.
The issues we then have to deal with are (a) $\theta(x_i')$ is not necessarily equal to $\psi_i$, and (b) $|\gamma_\tr(x')| = \tan \theta(x')$ is zero when $\theta(x')=0$, and then the second-order term in \eqref{eq:ray1:alpha_explain} blows up.

To deal with Point (b), we first consider points in $B(0,\epsilon)$ where the second-order term in \eqref{eq:ray1:alpha_explain} does not blow up. Let
\beq \label{eq:ray1:defZ}
Z_i := \Big\{ x'\in B(0,\epsilon)\,:\, \theta(y') \geq \theta_0 \,\tfa y' \in [x',x'_i] \Big\}
\eeq
where $\theta_0$ will be chosen later in the proof (when dealing with the points not in $Z_i$).
By \eqref{eq:tantheta}, for any $x' \in B(0,\epsilon) \cap Z_i$, $ |\nabla \gamma_{\tr}(y')| \geq \tan ( \theta_0) >0$ for $y' \in [x',x'_i]$. Recalling the definitions \eqref{eq:ray1:quant}, and using \eqref{eq:ray1:alpha_explain} and \eqref{eq:ray1:dtheta}, we have
\beq\label{eq:ray1:alpha}
| \theta(x') - \theta(x_i') |
\geq D_1  Q \left| \left\langle v_i,\frac{x'-x'_i}{ | x' - x'_i|} \right\rangle\right|  | x' - x'_i| - D_3 | x' - x'_i|^2,
\eeq
where
\beq \label{eq:ray1:D}
D_1 := (1+C_1^2)^{-1}, \qquad D_3 := C_3 + C_1C_2^2 + \frac{C_2^2}{|\tan (\theta_0)|},
\eeq
and the unit vector $v_i$ is defined by 
\beq\label{eq:vi}
v_i := \left(\partial^2 \gamma_{\tr} (x_i') \frac{\nabla \gamma_{\tr}(x_i')}{|\nabla \gamma_{\tr}(x_i')|}\right)\,\, \left| \partial^2 \gamma_{\tr} (x_i') \frac{\nabla \gamma_{\tr}(x_i')}{|\nabla \gamma_{\tr}(x_i')|} \right| ^ {-1}.
\eeq
Let 
\beqs
 W_i(\eta, \delta) := B(x'_i, \eta \epsilon) \cup \left\{ \left| \left\langle \frac{x'-x'_i}{ | x' - x'_i|},  v_i \right\rangle \right| \leq \delta \right\},
 \eeqs
where $\eta<1$; then \eqref{eq:ray1:alpha} implies that
\beqs
| \theta(x') - \theta(x'_i) | \geq \big(D_1    Q \delta \eta - 4 D_3 \epsilon\big) \epsilon  \quad \tfa
 x' \in \big( B(0,\epsilon) \cap Z_i \big)  \setminus  W_i.
\eeqs
We now deal with Point (a) above (i.e.~that $\theta(x_i')$ is not necessarily equal to $\psi_i$). If $|\theta(x_i') - \psi_i| > \alpha$, for $\alpha$ to be fixed later, then, by \eqref{eq:ray1:defxi},
\beq \label{eq:ray1:galpha}
 |\theta(x') - \psi_i| \geq |\theta(x_i') - \psi_i| > \alpha \quad \tfa x' \in B(0,\epsilon).
\eeq
If $|\theta(x_i') - \psi_i| \leq \alpha$, then 
\beqs
| \theta(x') - \psi_i|\geq | \theta(x') - \theta(x_i') |-\alpha
\eeqs
and then 
\beq \label{eq:ray1:galphafin}
| \theta(x') - \psi_i | \geq \big(D_1    Q \delta \eta - 4 D_3 \epsilon\big) \epsilon   - \alpha \; \tfa
 x' \in \big( B(0,\epsilon) \cap Z_i \big)  \setminus  W_i.
\eeq
Combining \eqref{eq:ray1:galpha} and \eqref{eq:ray1:galphafin}, we have
\beq \label{eq:ray1:noalpha_1}
\min_{i=1,\ldots,m} | \theta(x') - \psi_i | \geq  \min \Big( \big(D_1    Q \delta \eta - 4 D_3 \epsilon\big) \epsilon   - \alpha, \alpha \Big) 
 \tfa x'\in \big( B(0,\epsilon) \cap Z_i \big)\setminus \bigcup_{i=1}^m W_i(\eta, \delta);
\eeq
recall that we still have the freedom to choose $\theta_0, \eta, \delta,$ and $\alpha$.

We now deal with the case $x' \in B(0,\epsilon) \setminus Z_i$; the idea here is the following: $Z_i$ consists of points $x'$ such that every point on $[x',x_i']$ has $\theta\geq \theta_0$, i.e.~$\theta$ bounded below.
If $\theta(x')< \theta_0$, and we chose $\theta_0$ appropriately, then $|\theta(x')|$ can be small compared to $|\psi_i|$, and thus 
$|\theta(x') - \psi_i|$. can be bounded below. Indeed, let $\theta_0:=\psi_{\rm min}/2$; if $\theta(x') < \psi_{\rm min}/2$, then 
\beq \label{eq:ray1:Z1}
|\theta(x') - \psi_i| \geq |\psi_i| - |\theta(x')| \geq \frac 12 \psi_{\rm min}.
\eeq
We now need to consider $x' \in B(0,\epsilon) \setminus Z_i$ with $\theta(x') \geq \psi_{\rm min}/2$.
The sequence of  ideas here is that (i) by the definition of $Z_i$, there is a point, $x'_t$, in $[x',x'_i]$ with $\theta(x'_t) <\psi_{\min}/2$, 
(ii) the argument in 
\eqref{eq:ray1:Z1} applies at $x'_t$, (iii) $|x'-x'_t|\leq \epsilon$, which is small, (iv) $x'_t$ can be chosen so that $|\nabla\gamma_\tr|\neq 0$ on $[x',x'_t]$ and then $|\theta(x') - \theta(x_t')|$ can also be made small.
The detail is as follows: let
$$
t_i(x') := \inf \Big\{t\in [0,1]\,:\,  \big|\nabla \gamma_{\tr}((1-t)x' + tx'_i)\big| < \big|\tan (\psi_{\rm min}/2)\big| \Big\};
$$
the set on the right-hand side is not empty by \eqref{eq:tantheta} and the definition of $Z_i$ \eqref{eq:ray1:defZ}. Let $x'_t := (1-t_i(x'))x' + t_i(x')x'_i$. This definition implies that $\nabla\gamma_{\tr}(y')\neq 0$ for $ y' \in [x',x'_t]$. Therefore, 
using the mean-value theorem and (\ref{eq:ray1:dtheta}), we have
$$
|\theta(x') - \theta(x_t')| \leq \sup_{y' \in  [x',x'_t]} |\nabla \theta(y')| \,|x' - x_t'| \leq 2C_2\epsilon,
$$
Using this together with \eqref{eq:ray1:Z1}, 
we obtain
\beq \label{eq:ray1:Z2}
|\theta(x') - \psi_i| \geq |\theta(x_t') - \psi_i| - |\theta(x') - \theta(x_t')|  \geq \frac 12 \psi_{\rm min} - 2C_2\epsilon.
\eeq
Collecting both cases (\ref{eq:ray1:Z1}) and  (\ref{eq:ray1:Z2}), we obtain that
\beq \label{eq:ray1:Zconcl}
\tif 0<\epsilon\leq \min\left(\frac{\psi_{\rm min}}{4C_2}, \epsilon_0\right), \quad\text{ then }
|\theta(x') - \psi_i| \geq \frac 14 \psi_{\rm min} \quad \tfa x' \in B(0,\epsilon)  \setminus  Z_i.
\eeq

Putting \eqref{eq:ray1:noalpha_1} and \eqref{eq:ray1:Zconcl} together, we find that
if
\beq \label{eq:ray1:V0}
\widetilde{V}_D := B(0,\epsilon) \setminus \bigcup_{i=1}^m W_i(\eta, \delta),
\eeq
and 
\beqs
 0< \epsilon \leq \min\left(\frac{\psi_{\rm min}}{4C_2}, \epsilon_0\right),
\eeqs
then
\begin{align}
\min_{i=1,\ldots,m} | \theta(x') - \psi_i | \geq  \min \Big( \big(D_1    Q \delta \eta - 4 D_3 \epsilon\big) \epsilon   - \alpha, \alpha, \frac 14 \psi_{\rm min} \Big) \tfa x'\in \widetilde{V}_D.
 \label{eq:ray1:noalpha}
\end{align}

We now tune $\eta>0$ and $\delta>0$ to make the volume of $\widetilde{V}_D$ big enough, and conclude the step by selecting suitable $\epsilon>0$ and $\alpha>0$. 
From the definition \eqref{eq:ray1:V0},
\begin{align}\nonumber
\vol \big(\widetilde{V}_D\big) &\geq \vol \big(B(0,\epsilon)\big) - \sum_{i=1}^m \Big( \vol \big(B(x'_i,\eta\epsilon)\big) + \vol\big(\mathcal{C}_i \cap B(0,\epsilon)\big)\Big),\\
&\geq \vol \big(B(0,\epsilon)\big) - \sum_{i=1}^m \Big( \vol \big(B(x'_i,\eta\epsilon)\big) + \vol\big(\mathcal{C}_i \cap B(x'_i,2\epsilon)\big)\Big),
\label{eq:xmas1}
\end{align}
where
\beqs
\mathcal{C}_i:= \left\{ x' \,:\, \left|  \left\langle \frac{x'-x'_i}{ | x' - x'_i|},  v_i \right\rangle \right|\leq \delta \right\}
=  \left\{ x' \,:\,\cos^{-1} \delta \leq \reallywidehat{\left(\frac{x'-x'_i}{ | x' - x'_i|}, v_i\right)} \leq \pi - \cos^{-1} \delta  
  \right\}.
   \eeqs
Observe that $\mathcal{C}_i$ is the complement of a double cone, rotationally symmetric around the axis $v_i$ (recall that $v_i$ defined by \eqref{eq:vi} depends on $x_i'$ and not $x'$); therefore, $\vol(\mathcal{C}_i)$ decreases as $\delta\tendo$.
By integrating in hyperspherical coordinates centered at $x_i$ with axis $v_i$, and comparing $\vol(\mathcal{C}_i \cap B(x'_i,2\epsilon))$ to $\vol( B(x'_i,2\epsilon))$, we have 
\begin{align*}
\vol \Big( \mathcal{C}_i \cap B(x'_i, 2\epsilon) \Big)  
\leq \left( \frac{ \pi - 2 \cos^{-1} \delta }{2\pi}\right) \vol \big(B(x'_i, 2\epsilon)\big)= \frac{2^{d}}{\pi}\left(\frac \pi 2 - \cos^{-1} \delta\right) \vol \big(B(0,\epsilon)\big).
\end{align*}
Using this in \eqref{eq:xmas1}, we have
\begin{align}\nonumber 
\vol \big(\widetilde{V}_D\big) &\geq \vol \big(B(0,\epsilon)\big) - \sum_{i=1}^m \left( \vol \left(B(x'_i,\eta\epsilon)\right) + \frac{2^{d}}{\pi}\left(\frac \pi 2 - \cos^{-1} \delta\right) \vol \big(B(0,\epsilon\big) \right) \\
&\geq \left( 1 -  m  \eta^d -  m  \frac{2^{d}}{\pi}\left(\frac \pi 2 - \cos^{-1} \delta\right)\right) \vol \big(B(0,\epsilon)\big).
\label{eq:ray1:contrV0pre}
\end{align}
We now fix both $\delta > 0$ and $\eta>0$ to be sufficiently small such that
$$
0<\frac \pi 2 - \cos^{-1} \delta \leq \frac{\pi}{2^{d}m } \frac{10^{-d}}{2}, \qquad 0<\eta^d \leq  \frac{1}{m } \frac{10^{-d}}{2};
$$
then (\ref{eq:ray1:contrV0pre}) implies that
\beq \label{eq:ray1:contrV0}
\vol \big(\widetilde{V}_D\big) \geq (1-10^{-d}) \vol \big(B(0,\epsilon)\big) > 0.
\eeq
To conclude this step, we now restrict $\epsilon$ so that $0<\epsilon\leq(D_1  Q \delta\eta)/(8D_3)$ and then set $\alpha := D_1  Q \delta\eta \epsilon/4$; then (\ref{eq:ray1:noalpha}) implies that if
\beq\label{eq:ray1:step3a}
0 < \epsilon \leq \min \left(\frac{D_1 Q \delta\eta}{8D_3},\frac{\psi_{\rm min}}{4C_2},\epsilon_0\right)
\eeq
then
\beq\label{eq:ray1:step3}
\min_{i=1,\ldots,m} | \theta(x') - \psi_i | \geq  \frac{1}{4}  \min\big( D_1  Q \delta\eta \epsilon, \,\psi_{\rm min}\big)\quad
\tfa x' \in \widetilde{V}_D.
\eeq 

\emph{Step 5. Conclusion.}

\noindent 
Combining the result of Step 3 \eqref{eq:ray1:step4} and the result of Step 4 \eqref{eq:ray1:step3a}-\eqref{eq:ray1:step3}, we see that  
if
\beqs
0\leq \epsilon \leq \min\left(\frac{Q}{12\widetilde{C}_dC_3}, \epsilon_0 , \frac{D_1 Q\delta\eta}{8D_3},  \frac{\psi_{\rm min}}{4C_2}\right),
\eeqs
then
\beqs
\theta(x') \geq \frac{Q\epsilon}{12}\quad\tand\quad
\min_{i=1,\ldots,m} | \theta(x') - \psi_i | \geq \min \left( 
\frac{D_1  Q \delta\eta \epsilon}{4}, \,\frac{\psi_{\rm min}}{4}\right)\quad\tfa  x' \in \widetilde{V}_D \setminus  W.
\eeqs
We then let
\beq \label{eq:ray1:eps1}
\epsilon =\epsilon_1:=\min\left(\frac{D_1 Q\delta\eta}{8D_3}, \frac{Q}{12\widetilde{C}_dC_3}, \frac{\psi_{\rm min}}{4C_2}, \epsilon_0 \right),
\eeq
so that
\begin{align*} 
&\min_{i=1,\ldots,m} | \theta(x') - \psi_i | \geq  Q^2 \times  \min \left(
\frac{D_1 \delta\eta}{4}, \frac{ \psi_{\rm min}}{4Q \epsilon_0}\right)\times 
\min\left(\frac{D_1\delta\eta}{8D_3},\frac{1}{12\widetilde{C}_dC_3},  \frac{\psi_{\rm min}}{4QC_2} ,\frac{\epsilon_0}{Q} \right), \\
&\hspace{5cm}\tfa  x' \in \widetilde{V}_D \setminus  W, \nonumber
\end{align*}
where, by (\ref{eq:ray1:contrV0}) and (\ref{eq:ray1:contrW})
\beqs 
\vol \big( \widetilde{V}_D  \setminus  W \big) \geq (6^{-d} - 10^{-d}) \vol \big( B(0,\epsilon_1) \big).
\eeqs
Points (i) and (iii) in Condition \ref{cond:ray} then hold with 
\beqs
V_D :=  \widetilde{V}_D  \setminus  W, \qquad \crayone:= (6^{-d} - 10^{-d}) \vol \big( B(0,\epsilon_1) \big),
\eeqs
\beq\label{eq:c3}
\craythree:= \frac{Q}{12}\min\left(\frac{D_1 Q\delta\eta}{8D_3}, \frac{Q}{12\widetilde{C}_dC_3}, \frac{\psi_{\rm min}}{4C_2}, \epsilon_0 \right),
\eeq
and 
\beq\label{eq:c4}
\crayfour:= 
Q^2 \times  \min \left(
\frac{D_1 \delta\eta}{4}, \frac{ \psi_{\rm min}}{4Q \epsilon_0}\right)\times 
\min\left(\frac{D_1\delta\eta}{8D_3},\frac{1}{12\widetilde{C}_dC_3},   \frac{\psi_{\rm min}}{4QC_2},\frac{\epsilon_0}{Q},  \right).
\eeq
Since $Q, C_2, C_3, D_1$ and $D_3$ (defined by \eqref{eq:ray1:quant} and \eqref{eq:ray1:D}) all depend continuously on $\gamma_\tr$, and $\gamma_\tr$ depends continuously on $R$, 
$\crayone, \craythree,$ and $\crayfour$ depend continuously on $R$. 
The constant $\crayfive$ depends on $\craythree, \crayfour$, $\GammaIR$, and $\Gamma_D$, and thus also depends continuously on $R$.
\epf

Before proving Lemma \ref{lem:ray2}, we prove the following simple lemma.

\ble\label{lem:ballangle}
If $\GammaIR= \partial B(0,R)$, then the emanating rays from $\Gamma_D$ hit $\GammaIR$ directly with an angle to the normal $\theta$ satisfying $\theta< R^{-1}$.
\ele

\bpf
Since $\Omegaminus \subset B(0,1)$,
any ray starting from $\Omegaminus$ hits $\GammaIR = \partial B(0,R)$ with an angle to the normal $\theta$ satisfying $\tan \theta \leq 1/R$.
Since $\theta<\tan\theta$, the result follows.
\epf

\bpf[Proof of Lemma \ref{lem:ray2}]
We first observe that Point (iv)$^{\prime}$ follows from the same argument used to prove Lemma \ref{lem:ballangle}; this implies that $\crayfive =\widetilde{c}_5 R$ with $\widetilde{c}_5$ independent of $R$.

The fact that $\craytwo$ is independent of $R$ follows from the proof of Lemma \ref{lem:ray1}; see
\eqref{eq:ray1:tang}. 
By direct calculation from the definitions (\ref{eq:ray1:quant}), (\ref{eq:ray1:D})), using the fact that $\gamma_{\tr}(x') = \sqrt{R^2 - |x'|^2} + c$ where $c$ is a constant, we obtain that
\beqs
Q\sim R^{-1}, \; C_1\sim 1, \; C_2\sim R^{-1}, \; C_3 \sim R^{-2} ,\; \text{and thus } D_1\sim 1, \; D_3\sim R^{-2}.
\eeqs
Using these asymptotics in \eqref{eq:ray1:eps1}, \eqref{eq:c3}, and \eqref{eq:c4}, we find that 
$\crayone$ is independent of $R$ and $\craythree \sim R^{-1}$ (observe that the first minimum in \eqref{eq:c4} $\sim 1$ and the second minimum $\sim R$).

These arguments from the proof of Lemma \ref{lem:ray1} also show that $\crayfour\sim R^{-1}$, but we now show that in fact $\crayfour \sim 1$ for $R$ sufficiently large.
By Lemma \ref{lem:ballangle}, all the rays from $V_D$ hit $\GammaIR$ with angles $<1/R$. Therefore, if $R\geq 2/\psi_{\min}$, then $|\theta- \psi_j|\geq \psi_{\min}/2$ for all $j$. 
\epf

\bre[Lemma \ref{lem:ray2} when $\MPade=\NPade=0$]\label{rem:MN0}
Recall that when $\MPade=\NPade=0$, then $\mvanish=0$, inspecting the proof of Lemma \ref{lem:ray2}, we see that the result then holds with $\crayfour=0$ and $R_0=1$.
\ere

\bpf[Proof  of Lemma \ref{lem:ray3}]
For $0<\delta<1$, let $\Psi=\{0,\psi_1,\ldots,\psi_m\}$ and 
$$
V^\infty_{\rm tr}(\delta) := \Big\{  x^\infty \in \GammaIinf,\,\,:\,\, n(x^\infty) \text{ exists and } \min_{\psi\in\Psi} \Big|\reallywidehat{\left( n(x^\infty), \frac{ x^\infty}{| x^\infty|}\right)} - \psi \Big| > \delta \Big\}.
$$
We now claim that there exists $\delta_0<1$ such that $V^\infty_{\rm tr}(\delta_0)$ is non-empty.
Indeed, first observe that the map
\beqs
\big\{ x \in \GammaIinf \, :\, n(x) \text{ exists} \big\} \rightarrow \Rea 
\quad\text{ given by }\quad
x \mapsto \reallywidehat{\left( n(x), \frac{ x}{| x|}\right)}=\left\langle n(x) , \frac{x}{|x|} \right\rangle
\eeqs
is continuous.
The only way for this map to be constant is for $\Gamma^\infty_{\rm tr}$ to be a sphere centred at the origin, and this is ruled out by assumption. Since $\GammaIR/R \rightarrow \GammaIinf$ in $C^{0,1}$, $\GammaIinf$ is Lipschitz, and the set $\{x \in \GammaIinf : n(x) \text{ exists} \}$ has full $(d-1)$ dimensional (i.e.~surface) measure. 
Therefore, the image of the map contains an interval, and the claim follows. We note for later that $V^\infty_{\rm tr}(\delta_0)$ is open in $\GammaIinf$.

\begin{figure}  
\includegraphics[scale=0.5]{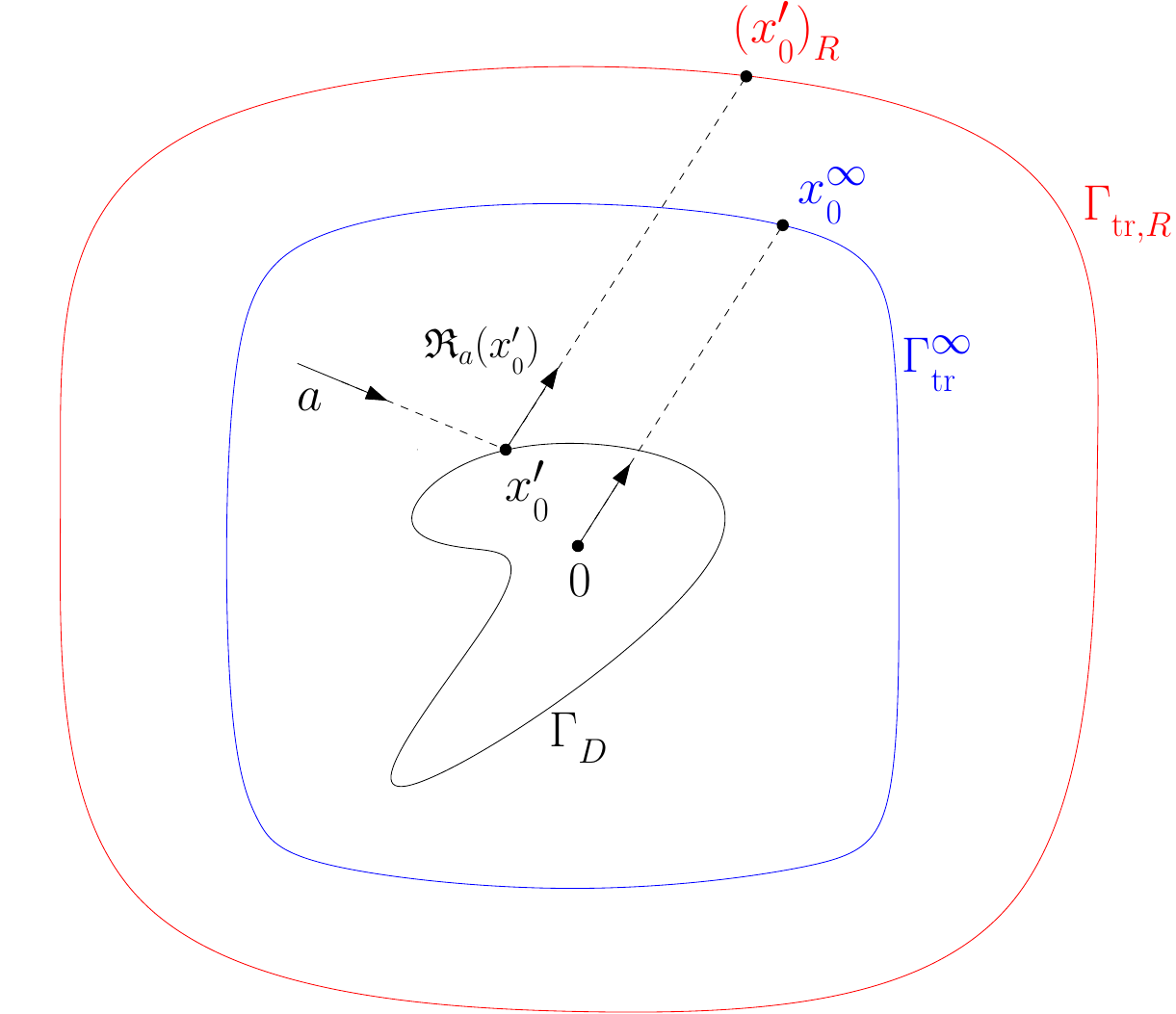}
\caption{The points and rays used in the proof of Lemma \ref{lem:ray3}.} \label{fig:ray_512}
\end{figure}

Let $x_0^\infty \in V^\infty_{\rm tr}(\delta_0)$. By Lemma \ref{lem:all_directions}, there exists $x'_0 \in \Gamma_D^{+,a}$ such that 
\beqs
\mathfrak{R}_a(x_0') = \frac{x_0^\infty}{|x_0^\infty|};
\eeqs
see Figure \ref{fig:ray_512}. 
For $x' \in \Gamma_D$, let $x'_R \in \GammaIR$ denote the point where the ray emanating from $x'$ first hits $\GammaIR$; we use later the fact that this definition implies that 
\beq\label{eq:comb2}
\frac{(x_0')_R-x_0'}{|(x_0')_R-x_0'|} = \frac{x_0^\infty}{|x_0^\infty|}.
\eeq

The neighbourhood $V_D$ in Condition \ref{cond:ray} will be $\Gsc\cap B(0,\epsilon)$ for $\epsilon$ sufficiently small, independent of $R$, and this ensures that Point (i) holds with $\crayone$ independent of $R$.
Let $\epsilon > 0$ be small enough so that $\Gamma_D \cap B(x'_0,\epsilon) \subset \Gamma_D^{+,a}$; this ensures that Point (ii) holds with $\craytwo$ independent of $R$.

We now show that Point (iii) of Condition \ref{cond:ray} holds with $\craythree$ and $\crayfour$ independent of $R$.
Let $W_{{\rm tr}, \epsilon}^\infty\subset \GammaIinf$ be defined by 
\beq\label{eq:Wtr}
W_{{\rm tr}, \epsilon}^\infty:= \lim_{R\tendi} \left\{
\frac{(x')_R}{R} \,: \, x' \in \Gamma_D \cap B(x_0',\epsilon)
\right\};
\eeq
this limit exists $W_{{\rm tr}, \epsilon}^\infty$ is the limit of subsets of  $\GammaIR/R$
and $\GammaIR/R\rightarrow \GammaIinf$ as $R\tendi$.
We claim that it is sufficient to prove that $W_{{\rm tr}, \epsilon}^\infty\subset V^\infty_{\rm tr}(\delta_0)$ for $\epsilon$ sufficiently small (independent of $R$). 
This shows the analogue of Point (iii) in Condition \ref{cond:ray} with $\GammaIR$ replaced by $\GammaIinf$; i.e., that the emanating rays from points in $V_D$ hit $\GammaIinf$ directly with an angle $\theta$ to the normal satisfying 
\eqref{eq:anglebounds} with $\craythree$ and $\crayfour$ independent of $R$.  Point (iii) for $\GammaIR$ with $R$ sufficiently large then follows since $W_{{\rm tr}, \epsilon}^\infty$ is the limit of subsets of $\GammaIR/R$, and $\GammaIR/R \rightarrow \GammaIinf$ as $R\tendi$.

We now claim that to prove that $W_{{\rm tr}, \epsilon}^\infty\subset V^\infty_{\rm tr}(\delta_0)$ for $\epsilon$ sufficiently small (independent of $R$) it is sufficient to show that 
$x_0^\infty \in W_{{\rm tr}, \epsilon}^\infty$ for all $\epsilon >0$. Indeed, if this is the case then $\cap_{\epsilon>0} W_{{\rm tr}, \epsilon}^\infty =  \{ x_0^\infty \}$. Then, since (i) $V_{\rm tr}^\infty(\delta_0)$ is open in $\GammaIinf$ and contains $x_0^\infty$, and (ii) $W_{{\rm tr}, \epsilon_1}^\infty \subseteq W_{{\rm tr}, \epsilon_2}^\infty$ for $\epsilon_1\leq \epsilon_2$, there exists $\epsilon_0>0$ such that $W_{{\rm tr}, \epsilon}^\infty \subset V_{\rm tr}^\infty(\delta_0)$ for all $\epsilon\leq \epsilon_0$. 

We now show that $x_0^\infty \in W_{{\rm tr}, \epsilon}^\infty$ for all $\epsilon >0$. 
We do this by showing that $(x'_0)_{R_k}/R_k \rightarrow x_0^\infty$ for a sequence $R_k\tendi$, 
and then the result follows from \eqref{eq:Wtr}.
Observe that the inclusions \eqref{eq:OmegaR} 
imply that $|x'_R|\leq M R$, for any $x' \in \Gamma_D$,
and thus $(x'_0)_R/R$ is bounded as $R\rightarrow \infty$. 
Therefore, there exists a sequence $R_k \rightarrow \infty$ and a $y\in \GammaIinf$ such that $(x'_0)_{R_k}/R_k \rightarrow y$, and thus also 
\beq\label{eq:xmas3}
\frac{(x'_0)_{R_k}}{|(x'_0)_{R_k}|}\rightarrow \frac{y}{|y|} \quad\tas R_k \tendi.
\eeq
By simple geometry, as $R\tendi$,
\beqs
\frac{ (x_0')_R}{|(x_0')_R|} = \frac{(x_0')_R-x_0'}{|(x_0')_R-x_0'|} + O(R^{-1})= \frac{x_0^\infty}{|x_0^\infty|}+ O(R^{-1}),
\eeqs
by \eqref{eq:comb2}. 
Comparing this to \eqref{eq:xmas3}, and using the uniqueness of the limit, we see that $y/|y| = x_0^\infty/|x_0^\infty|$. Since $\GammaIinf$ is convex, and thus star-shaped,
 $ y = x_0^\infty$, and the proof that $x_0^\infty \in W_{{\rm tr}, \epsilon}^\infty$ for all $\epsilon >0$ is complete; this completes the proof that Point (iii) of Condition \ref{cond:ray} holds with $\craythree$ and $\crayfour$ independent of $R$.

Finally, we show that Point (iv) of Condition \ref{cond:ray} holds for $R$ sufficiently large with $\crayfive= \widetilde{c}_5 R$ with $\widetilde{c}_5>0$ independent of $R$. 
Since $\Omegaminus\subset B(0,1)$ and $\domain$ satisfies the inclusions \eqref{eq:OmegaR}, after hitting $\GammaIR$, a ray must travel a distance $\sim R$ before hitting $\Gamma_D$. 
Therefore, we only need to show that, after hitting $\GammaIR$, a ray must travel a distance $\sim R$ before hitting $\GammaIR$ again.
Since $\GammaIR/R$ tends to a limit as $R\tendi$, this result follows if the rays first hit $\GammaIR$ with angle to the normal $\theta$ satisfying $|\theta - \pi/2|\geq c>0$, with $c$ independent of $R$, which is the case because $\Omegaminus\subset B(0,1)$ and $\domain$ satisfies the inclusions \eqref{eq:OmegaR}.
\epf

\bpf[Proof  of Lemma \ref{lem:ray4}]

The overall plan is to select a ray emanating from $\Gsc$ that returns to $B(0,1)$ after multiple reflections from the sides of the hypercube $[- \frac R2, \frac R2]^d$. 
We do this by identifying $\Rea^d$ with $[- \frac R2, \frac R2]^d$ by reflection through the lines 
\beqs
(x)_j = \frac{R}{2} + nR \quad \tfor n\in \mathbb{Z} \tand j=1,\ldots, d
\eeqs
(where $(x)_j$ denotes the $j$th component of the vector $x\in \Rea^d$);
under this identification the corners of the hypercube correspond to the points $(R/2 + R \mathbb{Z})^d$.
Since $\GammaIR$ coincides with the boundary of the hypercube $[-R/2, R/2]^d$ only at distance more than $\epsilon$ from the corners, we need to make sure that the selected ray avoids these neighbourhoods of the corners; hence the requirement that $\epsilon\leq \epsilon_0(\Omega_-)$ in the statement of the result. We highlight that the constant $C$ in the bound then depends only on the dynamics of the rays, and hence is independent of $\epsilon$.

\emph{Step 0: Preliminary notation and results.}
This argument involves three domains, and three associated flows. 
The first domain is $\domain$, with associated generalised bicharacteristic flow  $\varphi_t$ (as defined in \S\ref{subsec:geo}). The second domain is $\widehat{\Omega}_R := [- \frac R2, \frac R2]^d\setminus\overline{\domain}_-$, and we denote the 
generalised bicharacteristic flow on $\widehat{\Omega}_R$ by $\widehat{\varphi}_t$. The third domain is the hypercube $[- \frac R2, \frac R2]^d$, and we denote the generalised bicharacteristic flow on $[- \frac R2, \frac R2]^d$ by $\varphi^{[- \frac R2, \frac R2]^d}_t$.

By the definition \eqref{eq:mathfrakR} of $ \mathfrak R_a$, if both $x'$ and $y'$ are in the illuminated part of $\Gsc$ (i.e., $a\cdot n(x')<0$), then there exists $C_0>0$ (depending on the Lipschitz constant of $n$) such that
 \beq \label{eq:ref_lip}
| \mathfrak R_a(x') - \mathfrak R_a(y') | \leq C_0 |x'-y'|,
 \eeq
i.e. $\mathfrak R_a$ is Lipschitz.

We record for later use that, since $\Omegaminus \subset B(0,1)$ and $R\geq 4$,
\begin{equation} \label{eq:def_gamma}
\dist\left(\Gamma_D, \partial\left( \left[- \frac R2, \frac R2\right]^d\right)\right) \geq \frac{R}{2}-1\geq \frac{R}{4}.
\end{equation} 

Finally, let $\mathcal D$ be a non-empty, uniformly-convex open subset of $\Gamma^{+,a}_D$ in which $n(x')\cdot a < 0$
(such a $\mathcal D$ exists, since Lemma \ref{lem:all_directions} implies that $\Gamma^{+,a}_D \cap \{ n(x')\cdot a <0 \}$ is not
everywhere flat). 
Shrinking $\mathcal D$ if necessary, we can
assume that
\beq\label{eq:Dnu}
\text{ there exists} \quad 0<\nu<1 \tst  \quad\nu \leq |n(x')\cdot a| \leq 1 - \nu \quad\tfa x'\in \mathcal{D};
\eeq
this implies that the first assumption of Lemma \ref{lem:raydil} holds with $\mathcal{C}=\mathcal D$.
The neighbourhood $V_D$ we construct will be a subset of $\mathcal{D}$.

\emph{Step 1: Bounding the distance between projections of the flow on $[- \frac R2, \frac R2]^d$.}

For $(x_{j}, \xi_{j}) \in S^* B(0,1)$, $j=1,2$, since $\varphi^{\mathbb R^d}_t (x_j,\xi_j)= x_j+ 2t\xi_j$, 
\begin{equation}\label{eq:hyp0:1}
 \Big| \projx \varphi^{\mathbb R^d}_t (x_1, \xi_1) - \projx \varphi^{\mathbb R^d}_t (x_2, \xi_2)  \Big| \leq |x_1 - x_2| +  2t|\xi_1 - \xi_2|.
\end{equation}
We now show that the same inequality holds for the flow on $[- \frac R2, \frac R2]^d$; i.e., that for $(x_{j}, \xi_{j}) \in S^* B(0,1)$, $j=1,2$, 
\begin{equation}\label{eq:div_square}
\Big| \projx \varphi^{[- \frac R2, \frac R2]^d}_t (x_1, \xi_1) - \projx \varphi^{[- \frac R2, \frac R2]^d}_t (x_2, \xi_2)  \Big|  \leq |x_1 - x_2| +  2t|\xi_1 - \xi_2|.
 \end{equation}
To prove \eqref{eq:div_square}, we compare  $ \big| \projx \varphi^{[- \frac R2, \frac R2]^d}_t (x_1, \xi_1) - \projx \varphi^{[- \frac R2, \frac R2]^d}_t (x_2, \xi_2)  \big| $ with
$ \big| \projx \varphi^{\mathbb R^d}_t (x_1, \xi_1) - \projx \varphi^{\mathbb R^d}_t (x_2, \xi_2)  \big| $ by using the 
relationship between the two flows $\varphi^{[- \frac R2, \frac R2]^d}_t$ and $\varphi_t^{\mathbb R^d}$. 

First, observe that, since
$$
 \Big| \projx \varphi^{[- \frac R2, \frac R2]^d}_t (x_1, \xi_1) - \projx \varphi^{[- \frac R2, \frac R2]^d}_t (x_2, \xi_2)  \Big| \leq \operatorname{diam} \left[- \frac R2, \frac R2\right]^d = \sqrt{d} R,
$$
we can assume that 
$$
|x_1 - x_2| +  2t|\xi_1 - \xi_2|\leq \sqrt{d} R.
$$
Therefore, there exists $\ell = (\ell_1, \cdots \ell_d) \in \mathbb Z^d$ and
$\iota = (\iota_1, \cdots, \iota_d) \in \{-1, 0, 1\}^d$ such that
\begin{equation}\label{eq:hyp0:2}
\begin{cases}
\projx\varphi^{\mathbb R^d}_t (x_1, \xi_1) \in \Big( [- \frac R2, \frac R2] ^ d + \ell R \Big), \\
\projx\varphi^{\mathbb R^d}_t (x_2, \xi_2) \in \Big( [- \frac R2, \frac R2] ^ d + (\ell + \iota) R \Big);
\end{cases}
\end{equation}
i.e., after time $t$, the free-space rays from $(x_1, \xi_1) $ and $(x_2, \xi_2) $ are either in the same hypercube or in adjacent hypercubes. We use the following notation for the components of $\varphi^{\mathbb R^d}_t (x_j, \xi_j)$, $j=1,2$:
\begin{equation}\label{eq:hyp0:2b}
\projx \varphi^{\mathbb R^d}_t (x_j, \xi_j) := (z^{1}_j, \cdots, z^{d}_j) \in \mathbb R^d.
\end{equation}
Now, observe that by (\ref{eq:hyp0:2}) and the relationship between $\varphi^{[- \frac R2, \frac R2]^d}_t$ and $\varphi_t^{\mathbb R^d}$,
\begin{equation}\label{eq:hyp0:3}
\begin{cases}
\projx \varphi^{[- \frac R2, \frac R2]^d}_t (x_1, \xi_1) &=\Big(\operatorname{par}(\ell_1)\big(z^{1}_1 - \ell_1 R\big), \cdots, \operatorname{par}(\ell_d)\big(z^{d}_1- \ell_d  R\big)\Big), \\
\projx \varphi^{[- \frac R2, \frac R2]^d}_t (x_2, \xi_2) &=\Big(\operatorname{par}(\ell_1+\iota_1)\big(z^{1}_2 - (\ell_1 + \iota_1) R\big), \cdots, \operatorname{par}(\ell_d+\iota_d)\big(z^{d}_2- (\ell_d + \iota_d) R\big)\Big),
\end{cases}
\end{equation}
where
$$
\operatorname{par}(\ell) := 
\begin{cases}
1 &\text{ if } \ell \text{ is even}, \\
-1 &\text{ if } \ell \text{ is odd}.
\end{cases}
$$
Let $i \in \{1, \cdots, d\}$. We first assume that $\iota_i = 1$; then
\begin{equation}\label{eq:hyp0:4}
\big| \operatorname{par}(\ell_j)(z^{i}_1 - \ell_i R) - \operatorname{par}(\ell_i + \iota_i)(z^{i}_2  - (\ell_i + \iota_i) R) \big| = 
\big|(z^{i}_1-\ell_i R) + (z^{i}_2 - \ell_i R) - R \big|.
\end{equation}
Since $\iota=1$, $z^{i}_1-\ell_i R \in [- \frac R2, \frac R2]$,  $z^{i}_2-\ell_i R \in [\frac R2, \frac{3R}{2}]$, and hence $z^{i}_2 \geq z^{i}_1$. 
 Now, because $z^{i}_1-\ell_i R \leq R/2$,
\begin{equation}\label{eq:hyp0:5}
 (z^{i}_1-\ell_i R) + (z^{i}_2 - \ell_i R) - R \leq (z^{i}_2 - \ell_i R) - (z^{i}_1-\ell_i R) =z^{i}_2 - z^{i}_1 =|z^{i}_1 - z^{i}_2|.
\end{equation}
Similarly, since $z^{i}_2-\ell_i R \geq  R/2$, 
\begin{equation}\label{eq:hyp0:6}
 - (z^{i}_1-\ell_i R) - (z^{i}_2 - \ell_i R) + R \leq (z^{i}_2 - \ell_i R) - (z^{i}_1-\ell_i R) =z^{i}_2 - z^{i}_1 =|z^{i}_1 - z^{i}_2|.
\end{equation}
Then, combining (\ref{eq:hyp0:4}),  (\ref{eq:hyp0:5}),  and  (\ref{eq:hyp0:6}), we have that, for $i \in \{1, \cdots, d\}$ with $\iota_i = 1$,
\begin{equation}\label{eq:hyp0:7}
 \big| \operatorname{par}(\ell_j)\big(z^{i}_1 - \ell_i R\big) - \operatorname{par}(\ell_i + \iota_i)\big(z^{i}_2  - (\ell_i + \iota_i) R\big) \big|  \leq |z^{i}_1 - z^{i}_2|.
\end{equation}
 If $\iota_i = -1$, the prove of  (\ref{eq:hyp0:7}) follows in a very similar way; if $\iota_i = 0$, it is straightforward to check that  (\ref{eq:hyp0:7}) holds with equality. Hence
  (\ref{eq:hyp0:7}) holds for any $i \in \{1, \cdots, d\}$.
Recalling the notation (\ref{eq:hyp0:2b}), we therefore obtain from (\ref{eq:hyp0:3}) and   (\ref{eq:hyp0:7}) that
$$
\Big| \projx \varphi^{[- \frac R2, \frac R2]^d}_t (x_1, \xi_1)  - \projx \varphi^{[- \frac R2, \frac R2]^d}_t (x_2, \xi_2)  \Big| \leq \Big| \projx \varphi^{\mathbb R^d}_t (x_1, \xi_1)  - \projx \varphi^{\mathbb R^d}_t (x_2, \xi_2)  \Big|,
$$
and (\ref{eq:div_square}) follows from (\ref{eq:hyp0:1}).

 \emph{Step 2: Selecting a periodic ray.}
Let $\mathfrak F$ be the finite set of unit vectors forming an angle  belonging to $\Psi$ to one of the elements $(\pm e_i)_{1\leq i \leq d}$, where $(e_i)_{1\leq i \leq d}$ denote the unit vectors in cartesian coordinates.
With $\mathcal D$ as in Step 0, $\mathfrak R_a(\mathcal D)$ contains a non-empty open subset of $\mathcal S^{d-1}$ by Lemma \ref{lem:all_directions}, and therefore contains a vector of the form
$$
\xi_0 = \frac{(p_1, \cdots, p_d)}{|p|}, \quad p_i \in \mathbb Z, \quad \tand\quad \xi_0 \notin \mathfrak F
$$
(since vectors of this form are dense in $\mathcal S^{d-1}$).
Let $x'_0 \in \mathcal D$ be such that $\mathfrak R_a(x'_0) = \xi_0$.

We identify $\Rea^d$ with $[- \frac R2, \frac R2]^d$ as described above.
Then, given any $q_1,\ldots, q_d \in \mathbb{Z}$, 
\beq\label{eq:periodic}
(x_1,\ldots, x_d) + 2 R (q_1,\ldots, q_d) \equiv (x_1,\ldots, x_d);
\eeq
the factor of two is because one reflection changes the parity.

The trajectory starting from $(x_0',\xi_0)$ and evolving according to the flow $\varphi^{[- \frac R2, \frac R2]^d}_t$ can be identified with the trajectory in $\Rea^d$ 
\beqs
x_0' + 2t \xi_0= x_0' + 2t \frac{(p_1, \ldots, p_d)}{|p|};
\eeqs
therefore, by \eqref{eq:periodic}, the former trajectory is 
periodic, with period at most $R|p|$. 
Thus there exists $t>0$ such that $\varphi^{[- \frac R2, \frac R2]^d}_t \in B(0, 11/8)$; let $T(R)$ be the infimum of such $t$s.
Therefore \beq \label{eq:square_TisR}
T(R) \leq R|p|,
\eeq 
and
\beq\label{eq:square_firstray1}
\projx \big(\widehat\varphi_{T(R)}(x'_0, \xi_0)\big) \in  \partial B\left(0, \frac{11}{8}\right).
\eeq
Since $\Omegaminus \subset B(0,1)$, the flows $\widehat \varphi_t$ and $ \varphi^{[- \frac R2, \frac R2]^d}_t$ acting on $(x_0',\xi_0)$ agree up to (at least) time $T(R)$; i.e.
\begin{equation} \label{eq:square_firstray2}
\widehat \varphi_t(x'_0, \xi_0) = \varphi^{[- \frac R2, \frac R2]^d}_t (x'_0, \xi_0) \quad\tfa 0\leq t \leq T(R).
\end{equation} 
Furthermore, since $\xi_0\notin\mathfrak F$, 
the flows $\widehat \varphi_t$ and $ \varphi^{[- \frac R2, \frac R2]^d}_t$ acting on $(x_0',\xi_0)$ never hit $\partial \big([- \frac R2, \frac R2]^d\big)$ at an angle belonging to $\Psi$.

Finally, observe that a length $R$ of a ray can be reflected at most twice.
Therefore, since the length of $ \varphi^{[- \frac R2, \frac R2]^d}_t(x_0',\xi_0)$ for $t\in[0,T(R)]$ is at most $2R |p|$, 
if $M:= \lceil  4 |p| \rceil$, then the number of 
reflections of this ray for $t\in[0,T(R)]$, $N(R)$, is bounded by $M$, i.e.,
\beq\label{eq:Nbound}
N(R)\leq  \lceil  4 |p| \rceil.
\eeq

\emph{Step 3: The neighbourhood $V_D$ on $\Gsc$.}
The neighbourhood $V_D=V_D(R)$ is chosen later in the proof as a subset of 
\beq\label{eq:V1}
V_1(R) := \Gamma_D \cap B\left(x'_0, \frac{\delta_1}{R}\right)
\eeq
where $\delta_1>0$ (independent of $R$) is small enough so that, for all $R\geq 1$,
\beq\label{eq:sq_basicreq}
\begin{cases}
&V_1(R) \subset \mathcal D, \; \\
&\text{for all } x' \in V_1(R), \quad |n(x')\cdot a| \geq \frac 12 |n(x'_0)\cdot a|,\\
&\text{for all } x' \in V_1(R), \quad \min_{\mathfrak f \in \mathfrak F} |\mathfrak R_a(x') - \mathfrak f| \geq \frac 12 \min_{\mathfrak f \in \mathfrak F} | \xi_0 - \mathfrak f|.
\end{cases}
\eeq
Since the neighbourhood $V_D$ will be a subset of $V_1(R)$,
the second condition in \eqref{eq:sq_basicreq} implies that Part (ii) of Condition \ref{cond:ray} holds with $\craytwo:=  |n(x'_0)\cdot a|/2$, which is positive since $x'_0\in \mathcal{D}$, and 
the third condition in \eqref{eq:sq_basicreq} implies that Part (iii) of Condition \ref{cond:ray} holds with $\craythree>0$.

By \eqref{eq:div_square}, the fact that $\xi_0= \mathfrak{R}_a(x'_0)$, \eqref{eq:ref_lip}, and \eqref{eq:square_TisR}, we have, for any $x' \in V_1(R)$ and any $0 \leq t \leq T(R)$
\begin{align}\nonumber
 \big| \projx \varphi^{[- \frac R2, \frac R2]^d}_t (x'_0, \xi_0) - \projx \varphi^{[- \frac R2, \frac R2]^d}_t (x', \mathfrak R_a(x'))  \big| 
 &\leq |x_0' - x'| + 2 T(R) |\mathfrak{R}_a(x'_0) - \mathfrak{R}_a(x')|,\\ \nonumber
 &\leq  \big(1+ 2 R |p| C_0 \big)|x_0' - x'|,\\
 & \leq \big(1+ 2  |p| C_0 \big)R |x_0' - x'|.\label{eq:27_1}
 \end{align}
Therefore, if $\delta_1 \leq (16 (1+2C_0|p|))^{-1}$,
then 
\beq\label{eq:consequence1}
 \big| \projx \varphi^{[- \frac R2, \frac R2]^d}_t (x'_0, \xi_0) - \projx \varphi^{[- \frac R2, \frac R2]^d}_t (x', \mathfrak R_a(x'))  \big| \leq \frac{1}{16}
\eeq
for all $x' \in V_1(R)$ and for all $0\leq t\leq T(R)$. Combining \eqref{eq:consequence1}, \eqref{eq:square_firstray1}, and \eqref{eq:square_firstray2}, we have
\beq\label{eq:Step2Conc1}
\projx \big(\widehat\varphi_{T(R)}(x', \xi_0)\big) \in  B\left(0, \frac{23}{16}\right)\Big \backslash  B\left(0, \frac{21}{16}\right)
\quad\tfa x' \in V_1(R);
\eeq
and
\beq\label{eq:Step2Conc2}
\widehat \varphi_t(x', \mathfrak R_a(x')) = \varphi^{[- \frac R2, \frac R2]^d}_t (x', \mathfrak R_a(x')) \quad\tfa x' \in V_1(R) \text{ and for all } 0\leq t \leq T(R).
\eeq

\emph{Step 4: Avoiding the corners.}
Under the identification of $[- \frac R2, \frac R2]^d$ with $\Rea^d$, the corners of the hypersquare correspond to $(R/2 + R \mathbb Z)^d$.
Given $x'\in V_1(R)$, each point on the ray $x' + 2t \mathfrak R_a(x')$ for $0\leq t\leq T(R)$ has a corner that is closest; we let $Q_\alpha(x')$ denote the subset of these corners that are a distance $\leq \alpha$ away. More precisely,
\begin{align*}
Q_ \alpha(x') &:= \bigg\{ q \in (R/2+ R \mathbb Z)^d \,:\,\text{ there exists } 0\leq t \leq T \tst \\
&\hspace{3cm}\dist\left( x' +2t\mathfrak R_a(x'), \big(R/2 + R \mathbb Z\big)^d \right) = \dist\big( x' +2t\mathfrak R_a(x'), q \big) \leq \alpha \bigg\}.
\end{align*}
We then order the elements of $Q_\alpha(x')$ with the closest first; i.e., $Q_\alpha(x')= \{q_1(x'), \ldots, q_{m(x')}(x') \}$ with $\dist(x', q_i)$ non-decreasing with $i$.

We now prove that if $\delta_1\leq (4(1+2|p|C_0))^{-1}$, 
then
\beq\label{eq:Qinclusion}
 Q_{1/4}(x') \subset Q_{1/2} (x'_0)\quad\tfa x' \in V_1(R).
\eeq
To prove this, observe that, for $0\leq t\leq T(R)$, by \eqref{eq:square_TisR} and \eqref{eq:ref_lip} (in a similar way to as in \eqref{eq:27_1}),
\begin{align*}
\dist\big( x' +2t\mathfrak R_a(x'), x'_0 + 2t\mathfrak R_a(x_0')) &\leq  |x'-x'_0| + 2t|\mathfrak R_a(x') - \mathfrak R_a(x'_0)|,\\
&\leq \big(1 + 2|p|C_0\big)R|x'-x_0'|\leq  \delta_1  \big(1 + 2|p|C_0\big)
\end{align*}
if $x'\in V_1(R)$. 
Therefore, if $\delta_1 \leq (4(1+2|p|C_0))^{-1}$, 
 the distance between the rays is $<1/4$. If $q_i \in Q_{1/4}(x')$ then, since $R\geq 1$, $q_i$ is at most distance $1/2$ away from a point on the ray  $x'_0 + 2t\mathfrak R_a(x_0')$, and thus $q_i \in Q_{1/2}(x_0')$.

It turns out that we will not need to restrict $\delta_1$ further in the proof; we therefore set 
\beq\label{eq:delta1}
\delta_1 := \frac{1}{16 (1+2C_0|p|))},
\eeq
and observe that this satisfies the requirements imposed on $\delta_1$ earlier in the proof (to ensure that 
\eqref{eq:consequence1} and  \eqref{eq:Qinclusion} hold).

We now select one set of corners to work with for all $x' \in V_1(R)$.
Let $Q:= Q_{1/2} (x'_0)=(q_1, \ldots q_m)$. By \eqref{eq:Qinclusion}, 
\beqs
\Big((R/2 + R \mathbb{Z})^d \setminus Q\Big) \subset \Big((R/2 + R \mathbb{Z})^d \setminus Q_{1/4}(x')\Big) \quad\tfa x' \in V_1(R),
\eeqs
so that 
\beq\label{eq:othercorners}
\dist \left(x' +2t\mathfrak R_a(x'), (R/2 + R \mathbb{Z})^d \setminus Q\right) \geq 1/4 \quad\tfa x' \in V_1(R).
\eeq
Furthermore, since $R\geq 4$, the number of corners within distance $1/2$ of the ray is less than or equal to the number of reflections, i.e., 
\beq\label{eq:mbound}
m\leq N(R).
\eeq

We now iteratively construct $x_i' \in V_1(R)$, $i=1,\ldots, m$, such that the ray $x_i'  + 2t\mathfrak R_a(x_i' )$ for $0\leq t \leq T(R)$ is at least a distance $\eta_i$ from $(q_1,\ldots, q_i)$ where $\eta_i> 0$, $i=0,\ldots, m$, are defined below (see \eqref{eq:def_eta}) and, in particular, have the property that $\eta_i> \eta_{i+1}, i=0,\ldots m-1$.
Given $x'_i$, if $\dist(x_i'  + 2t\mathfrak R_a(x_i' ), q_{i+1})\geq \eta_{i+1}$, we set $x_{i+1}' := x_i'$. Otherwise, first observe that, 
for $0\leq t\leq R/16$,
\beq\label{eq:sq_it0}
\dist\big(x'  + 2t\mathfrak R_a(x' ), q_{i+1}\big)\geq R/8\geq 1/2,
\eeq
by \eqref{eq:def_gamma} and the fact that $R\geq 4$;
we can therefore restrict attention to $t\geq R/16$.
Let $\lambda_i>0$, to be fixed later. We first assume that there exists $x'_{i+1} \in V_1(R)$ so that,
with $\Cfrak$ the constant associated to $\mathcal D$ by  Lemma \ref{lem:raydil},
\beq \label{eq:raysq:alphai}
|x'_{i+1} - x'_i| = \lambda_i\quad \text{ and } \quad \big| \mathfrak R_a(x_{i+1}' ) - \mathfrak R_a(x_i' )\big| \geq \Cfrak \lambda_i;
\eeq
we later use Lemma \ref{lem:raydil} to show that such an $x'_{i+1}$ exists once the value of $\lambda_i$ has been fixed. 
By, respectively, the triangle inequality, the convexity of $V_1(R)\subset \mathcal D$, \eqref{eq:raysq:alphai}, 
and the fact that we're dealing with the case that $\dist(x_i'  + 2t\mathfrak R_a(x_i' ), q_{i+1})< \eta_{i+1}$,
we have that, for $R/16\leq t\leq T(R)$,
\begin{align} 
\dist\big(x_{i+1}'  + 2t\mathfrak R_a(x' ), q_{i+1}\big) &\geq  \dist\big(x_{i+1}'  + 2t\mathfrak R_a(x_{i+1}' ), x_i'  + 2t\mathfrak R_a(x_i' )\big) - \dist\big(x_i'  + 2t\mathfrak R_a(x'_i ), q_{i+1}\big), \nonumber  \\ 
& \geq  \dist\big(x_{i+1}'  + 2t\mathfrak R_a(x_{i+1}' ), x_{i+1}' + 2t\mathfrak R_a(x_i' )\big) - \dist\big(x_i'  + 2t\mathfrak R_a(x'_i ), q_{i+1}\big), \nonumber  \\
& =  2t  |\mathfrak R_a(x_{i+1}') - \mathfrak R_a(x_i')| - \dist\big(x_i'  + 2t\mathfrak R_a(x'_i ), q_{i+1}\big), \nonumber \\ \nonumber
& \geq 2t\Cfrak \lambda_i - \eta_{i+1},\\
& \geq \frac{\Cfrak R}{8}\lambda_i - \eta_{i+1}. \label{eq:sq_it1}
\end{align}
Having bounded the distance from the ray to $q_{i+1}$, we now bound the distance to $q_j$ for $j=0,\ldots,i$. 
By, respectively, the triangle inequality, \eqref{eq:ref_lip}, and \eqref{eq:square_TisR}, for $j=0,\ldots,i$ and $0\leq t \leq T(R)$, 
\begin{align} \nonumber
\dist\big(x_{i+1}'  + 2t\mathfrak R_a(x' ), q_{j}\big)&\geq \dist\big(x'_i  + 2t\mathfrak R_a(x'_i ), q_{j}\big) - \dist\big(x_{i+1}'  + 2t\mathfrak R_a(x_{i+1}' ), x_i'  + 2t\mathfrak R_a(x_i' )\big) \\ \nonumber
& \geq \eta_i - \big(1+2tC_0\big)|x'-x'_i|,\\
&\geq \eta_i - R\big(1+2C_0|p|\big)\lambda_i. \label{eq:sq_it2}
\end{align} 
The two inequalities \eqref{eq:sq_it1} and \eqref{eq:sq_it2} imply that if $\eta_i$ and $\eta_{i+1}$ satisfy 
\beq\label{eq:rain1}
\frac{16  \eta_{i+1}}{ \Cfrak} =   \frac{\eta_i - \eta_{i+1}}{( 1+ 2C_0|p|)},
\eeq
and  $\lambda_i$ is defined by 
\beq\label{eq:rain2}
 \lambda_i:= \frac{16  \eta_{i+1}}{R \Cfrak} =   \frac{\eta_i - \eta_{i+1}}{R( 1+ 2C_0|p|)},
\eeq 
then
\beqs
\dist\big(x'_{i+1}  + 2t\mathfrak R_a(x'_{i+1} ), q_{i+1}\big)\geq \eta_{i+1} \quad\tfa R/16\leq t\leq T(R)
\eeqs
and 
\beqs
\dist\big(x'_{i+1} + 2t\mathfrak R_a(x'_{i+1} ), q_{j}\big) \geq \eta_{i+1} \quad\tfor j=0,\ldots,i, \text{ and for all }0\leq t\leq T(R).
\eeqs
This last two inequalities, combined with \eqref{eq:sq_it0}, imply that 
\beqs
\dist\big(x'_{i+1}  + 2t\mathfrak R_a(x_{i+1} ), q_{j}\big) \geq \eta_{i+1} \quad\tfor j=0,\ldots,i+1, \text{ and for all }0\leq t\leq T(R)
\eeqs
as required. We observe for use later that \eqref{eq:rain1} implies that
\beq\label{eq:def_eta}
\eta_{i+1}= \frac{\eta_i }{1+ \frac{16}{\Cfrak}\big(1+ 2 C_0|p|\big)} \quad\text{ so that } \quad 
\eta_j := \eta_0 \left( \frac{1}{1+ \frac{16}{\Cfrak}\big(1+ 2 C_0|p|\big)}\right)^j, \quad j=0,\ldots,m.
\eeq

Since the value of $\lambda_i>0$ has been fixed by \eqref{eq:rain2},  it remains to show that there exists $x'_{i+1} \in V_1(R)$ satisfying \eqref{eq:raysq:alphai}. 
We now use the freedom we have in choosing $\eta_0$ to ensure that the can use Lemma \ref{lem:raydil}  to construct such an $x'_{i+1}$. Recall that we chose $\mathcal D$ so that the assumptions of Lemma \ref{lem:raydil} hold; let $\alpha_0$ be the associated constant.
We impose the condition that 
\beq\label{eq:boundeta0}
\sum_{j=0}^{m-1} \lambda_j \leq \min \left( \frac{\delta_1}{2R}, \alpha_0 \right), 
\qquad \text{i.e.}, \quad \eta_0\frac{16}{\Cfrak} \sum_{j=1}^{m-2} 
\left(\frac{1}{1+ \frac{16}{\Cfrak}\big(1+ 2 C_0|p|\big)}\right)^j\leq \min \left( \frac{\delta_1}{2}, 4\alpha_0 \right),
\eeq
where we have used the definitions of $\lambda_j$ \eqref{eq:rain2} and $\eta_j$ \eqref{eq:def_eta} and the fact that $R\geq 4$.
Observe that \eqref{eq:boundeta0} is a condition that $\eta_0$ is sufficiently small (recall that $\delta_1$ has been fixed by \eqref{eq:delta1}).

The rationale behind imposing \eqref{eq:boundeta0} is as follows;
recalling the definition of $V_1(R)$ \eqref{eq:V1}, we see that $\sum_{j=0}^{m-1} \lambda_j \leq \delta_1/2$ implies that $x_i' \in V_1(R)$ for $i=1,\ldots,m$.
The first inequality in \eqref{eq:boundeta0} implies that $\lambda_i \leq \alpha_0$, for all $i$, and, since $V_1(R)\subset \mathcal D$ (by \eqref{eq:sq_basicreq}),
$$
\partial B(x'_i,\lambda_i) \cap \mathcal D \neq \emptyset \quad\tand\quad\partial B(x'_i,\lambda_i) \cap \partial \mathcal D = \emptyset.
$$
These relations combined with \eqref{eq:Dnu} imply that the assumptions of Lemma \ref{lem:raydil} are satisfied with $\mathcal{D}= \mathcal{C}$. This lemma therefore implies that there exists $x'_{i+1} \in \mathcal D$ satisfying \eqref{eq:raysq:alphai}, for all $i=1,\ldots,m$.

In summary, we have proved that the ray $x'_m + t \mathfrak{R}_a(x'_m)$, $0\leq t\leq T(R)$, is a distance at least $\eta_m$ from any of the corners $q_1,\ldots, q_m$, and a distance at least $1/4$ from any of the other corners by \eqref{eq:othercorners}. 

Let $\eta_{\lceil 4|p|\rceil}$ be defined by the second equation in \eqref{eq:def_eta} with $j= \lceil 4|p|\rceil$
and with $\eta_0$ fixed to satisfy \eqref{eq:boundeta0}. By \eqref{eq:mbound} and \eqref{eq:Nbound}, $m \leq N(R)\leq \lceil 4|p|\rceil$ so that $\eta_m \geq \eta_{\lceil 4|p|\rceil}$.
Therefore, with
\beqs
\epsilon_0 := \half \min\left(\eta_{\lceil 4|p|\rceil}, \frac14\right),
\eeqs
the ray $x'_m + t \mathfrak{R}_a(x'_m)$, $0\leq t\leq T(R)$ is a distance at least $2\epsilon_0>0$ from any corner.
By \eqref{eq:def_eta} and \eqref{eq:boundeta0}, $\eta_{\lceil 4|p|\rceil}$ (and hence 
$\epsilon_0$) depends on  $C_0$, $\Cfrak$, $\alpha_0$, and $|p|$, and hence only on $\Gsc$.

\emph{Step 5: Putting everything together.}
By combining the results of Step 4 with the results \eqref{eq:Step2Conc1} and \eqref{eq:Step2Conc2} of Step 3, we have
\begin{equation}\label{eq:sq_xpm}
\begin{cases}
\widehat \varphi_t(x'_m, \mathfrak R_a(x'_m)) = \varphi^{[- \frac R2, \frac R2]^d}_t (x'_m, \mathfrak R_a(x'_m)) \quad\tfa 0\leq t \leq T(R),\\
\dist\Big(\widehat \varphi_t(x_m', \mathfrak R_a(x_m')), (\frac R2 + R \mathbb Z)^d \Big) \geq 2\epsilon_0 \quad\tfa 0\leq t 
\leq T(R), \quad\tand\\
\projx \big(\widehat\varphi_{T(R)}(x'_m, \xi_0)\big) \in  B\left(0, \frac{23}{16}\right)\Big \backslash    B\left(0, \frac{21}{16}\right).
\end{cases}
\end{equation}
We now define the neighbourhood $V_D$ (the neighbourhood of rays in the statement of the lemma) as a neighbourhood of $x'_m$.
Indeed, we let
\beqs
V_D := \Gamma_D \cap B\left(x'_m, \frac{\delta}{R}\right)
\eeqs
with $\delta>0$ chosen sufficiently small; if $\delta>0$ is independent of $R$, then this implies that $\vol(V_D)\geq  \widetilde{c}_{\rm ray,1}/R^{d-1}$ for some $\widetilde{c}_{\rm ray,1}>0$ independent of $R$; i.e., that Point (i) of Condition \ref{cond:ray} holds.

We first choose $\delta>0$ sufficiently small so that $V_D \subset V_1(R)$; since $\delta_1$ \eqref{eq:delta1} is independent of $R$, $\delta$ can be chosen to be independent of $R$.
As discussed below \eqref{eq:sq_basicreq}, the inclusion $V_D \subset V_1(R)$ ensures that Points (ii) and (iii) of Condition \ref{cond:ray} hold.

Point (iv) in the statement of the result will follow if we can show that, for all $x'\in V_D$,
\begin{equation}\label{eq:Step41}
\begin{cases}
\widehat \varphi_t(x_m, \mathfrak R_a(x_m)) = \varphi^{[- \frac R2, \frac R2]^d}_t (x', \mathfrak R_a(x')) \quad\tfa 0\leq t \leq T(R),\\
\dist\Big(\widehat \varphi_t(x', \mathfrak R_a(x')), (\frac R2 + R \mathbb Z)^d \Big) \geq \epsilon_0 \quad\tfa 0\leq t 
\leq T(R), \quad\tand\\
\projx \big(\widehat\varphi_{T(R)}(x', \xi_0)\big) \in  B\left(0, \frac{47}{32}\right)\Big \backslash    B\left(0, \frac{41}{32}\right).
\end{cases}
\eeq
Indeed, the second property in \eqref{eq:Step41} (missing the corners) implies that all three flows are the same when applied to $(x_m, \mathfrak R_a(x_m))$ for $0\leq t\leq T(R)$, i.e. 
\beqs
\varphi_t(x_m, \mathfrak R_a(x_m))=\widehat \varphi_t(x_m, \mathfrak R_a(x_m)) = \varphi^{[- \frac R2, \frac R2]^d}_t (x', \mathfrak R_a(x')) \quad\tfa 0\leq t \leq T(R).
\eeqs

We now obtain \eqref{eq:Step41} from \eqref{eq:sq_xpm}. 
By \eqref{eq:div_square}, \eqref{eq:ref_lip}, and \eqref{eq:square_TisR} (in a similar way to as in \eqref{eq:27_1}), for any $x' \in V(R)$ and any $0 \leq t \leq T(R)$,
\beqs
 \Big| \projx \varphi^{[- \frac R2, \frac R2]^d}_t (x'_m, \mathfrak{R}_a(x_m)) - \projx \varphi^{[- \frac R2, \frac R2]^d}_t (x', \mathfrak R_a(x'))  \Big|  \leq  \big(1+2C_0|p|\big)R |x'_m - x'|.
\eeqs
so that \eqref{eq:Step41} follows as long as 
\beqs
\delta \leq \min\left(\frac{1}{32  \big(1+C_0|p|\big)}, \frac{\epsilon_0}{  \big(1+2C_0|p|\big)}\right).
\eeqs
Since $\delta>0$ is independent of $R$, Point (i) of Condition \ref{cond:ray} holds with $\crayone= \widetilde{c}_{\rm ray,1}/R$, with 
$\widetilde{c}_{\rm ray,1}$ independent of $R$, and the proof is complete.
\end{proof}

\subsection{Bounding below the reflection coefficient \eqref{eq:reflectioncoefficient} for rays satisfying Condition \ref{cond:ray}}\label{sec:reflectionbound} 

In the follow result, we use the subscripts $D$ and $\tr$ on $\mathcal{H}$ to denote the hyperbolic set on $\Gamma_D$ and $\GammaIR$, respectively.

\ble[Lower bound on the reflection coefficient for general $\GammaIR$]\label{lem:reflectionbound1}
Let $\cV_\tr \subset \cH_{\tr}$. Given $(x',\xi') \in \mathcal{V}_{\tr}$, let 
\beq\label{eq:theta}
\theta(x',\xi'):= \sin^{-1} \big(|\xi'|_g\big)\in [0,\pi/2);
\eeq
observe that $\theta$ is well-defined since $r(x',\xi'):= 1- |\xi'|_g^2 >0$ on $\cH_\tr$.

Let $\{\psi_j\}_{j=1}^\mvanish$ be defined be \eqref{eq:psi_j}.
Suppose that 
\beq\label{eq:angleboundedbelow}
\theta\geq c_3 \quad\tand\quad 
\min_{j=1,\ldots,m}|\theta-\psi_j|\geq c_4, 
\eeq
and $\bcN$ and $\bcD$ satisfy Assumption \ref{ass:Pade} with \emph{either} $\MPade=\NPade$ \emph{or} $\MPade=\NPade+1$.
Then there exists $\Crefone= \Crefone(\MPade,\NPade)>0$ such that
\beq\label{eq:reflectionbound1}
\left|
\frac{
\sqrt{r} \sigma(\bcN) - \sigma(\bcD)
}{
\sqrt{r} \sigma(\bcN) + \sigma(\bcD)
}
\right|
\geq \Crefone 
\min\Big( |c_3|^{2\morderzero},\, |c_4|^{\mmult}\Big)
\quad\ton \cV_{\tr}.
\eeq
\ele

We make three remarks.
\bit
\item
The rationale behind the definition of $\theta$ \eqref{eq:theta} is that later we apply it to sets $\cV_\tr$ whose elements are of the form $\pi_{\GammaIR}(x,\xi)$ where $(x,\xi) \in S^*_{\overline{\domain}}\Rea^d$ (so that $|\xi|=1$). In this case, $\theta$ is the angle the vector $\xi$ makes with the normal to $\GammaIR$. 
\item We have denoted the constants in \eqref{eq:angleboundedbelow} by $c_3$ and $c_4$ since we later apply this lemma with $c_3 = \craythree$ and $c_4= \crayfour$.
\item We highlight that $\Crefone$ only depends on $\MPade$ and $\NPade$, and not on $\GammaIR$.
\eit

\bpf[Proof of Lemma \ref{lem:reflectionbound1}]
By Assumption \ref{ass:Pade},
\beq\label{eq:Pade1}
 \sigma(\bcN)(x',\xi')\sqrt{r(x',\xi')} - \sigma(\bcD)(x',\xi') = q\big(|\xi'|_g^2\big) \sqrt{1-|\xi'|^2_g}-p\big(|\xi'|_g^2\big).
\eeq
Since $\bcN$ and $\bcD$ satisfy Assumption \ref{ass:Pade} with \emph{either} $\MPade=\NPade$ \emph{or} $\MPade=\NPade+1$, Part (a) of Lemma \ref{lem:TH} implies that there exists $C_1=C_1(\MPade,\NPade)>0$ such that $|\sqrt{r} \sigma(\bcN) + \sigma(\bcD)|\geq C_1$ on $\mathcal{V}_\tr$.

By the definitions in \S\ref{sec:setup} of $p(t), q(t)$, $\morderzero$, $\{t_j\}_{j=1}^\mvanish$, and $\mmult$, there exists $C_2= C_2(\MPade,\NPade)>0$ such that 
\beq\label{eq:pqbounded}
\big| q(t) \sqrt{1-t} - p(t)\big| \geq C_1 \min \Big( |t|^{\morderzero}, 
\Big(\min_{j=1,\ldots,\mvanish} |t-t_j|\Big)^{\mmult}\Big) 
\quad\tfa t\in [0,1].
\eeq
Since $\sin x \geq 2x/\pi$ for $x\in[0,\pi/2]$ and there exists $C_2=C_2(\psi_{\min})>0$ such that $|\sin^2 \theta- \sin^2 \psi_j|\geq C_2 |\theta-\psi_j|$ for $j=1,\ldots,\mvanish$, the inequalities in \eqref{eq:angleboundedbelow} imply that 
\beqs
|\xi'|_g^2 \geq \big(c_3\big)^2\left(\frac{2}{\pi}\right)^2 \quad\tand \quad\min_{j=1,\ldots,\mvanish}\big| |\xi'|_g^2 - t_j\big| \geq C_2 c_4.
\eeqs
The bound \eqref{eq:reflectionbound1} then follows from combining these bounds with \eqref{eq:Pade1} and \eqref{eq:pqbounded}.
\epf

\subsection{Proof of Theorem \ref{th:lower} (the qualitative lower bound)}\label{sec:lowerboundproof1}

Similar to above, 
we use the subscripts $D$ and $\tr$ on $\mathcal{H}$ (and subsets of it) to denote the hyperbolic set on $\Gamma_D$ and $\GammaIR$, respectively; we use analogous notation for boundary measures.

\begin{proof}[Proof of Theorem \ref{th:lower}]
By Part (i) of Corollary \ref{cor:reduction}, we only need to show that $\mu(\mathcal I) > 0$. 
We now follow the steps outlined in \S\ref{sec:outlinerays};
seeking a contradiction, we assume that $\mu(\mathcal I)=0$. The inequality  (\ref{eq:key2b}) from Point (ii) of Corollary \ref{cor:explain} implies that $\muin_D=0$. Therefore \eqref{eq:muout}
implies that 
$$
\muout_{D}=2\sqrt{r(x',\xi')}\,\nu_{d,D} 
$$
and Lemma \ref{lem:nuD} therefore gives that
\beq \label{eq:muoutqual}
\muout_{D}=2\sqrt{r(x',\xi')} \ \dvol({x'}) \otimes \delta_{\xi' =(a_{T(x')})^{\flat}}.
\eeq
Given $M$ and $N$, let $\{\psi_j\}_{j=1}^\mvanish$ be defined by \eqref{eq:psi_j}; i.e., $\{\psi_j\}_{j=1}^\mvanish$ is the set of non-zero angles at which the reflection coefficient \eqref{eq:reflectioncoefficient} vanishes.
 Let the set $V_D \subset \Gamma_D$ be given by Lemma \ref{lem:ray1}; i.e., the rays emanating from $V_D$ are non-tangent to $\Gamma_D$ and hit $\GammaI$ directly and at angles bounded away from $\{0,\psi_1,\ldots,\psi_\mvanish\}$.
Let
$$
\mathcal V_D := \Big\{ \big( x', (a_{T(x')})^{\flat}\big), \; x' \in V_D \Big\} \subset \mathcal H_{D}.
$$
By (\ref{eq:muoutqual}), $\muout_D(\mathcal V_D) > 0$. Therefore, using the equality \eqref{eq:key1} from Point (i) of Corollary \ref{cor:explain} and the fact that $r>0$ on $\mathcal{H}$,
\beq\label{eq:29_1}
(2\sqrt{r}\muin_\tr)(\mathcal V_\tr) = (2\sqrt{r}\muout_D)(\mathcal V_D)>0,
\eeq
where 
\beqs
\mathcal{V}_\tr := 
 \bigcup_{q \in \mathcal{V}_D}\pi_{\GammaIR}\Big( \varphi_{\tout (q)}\big(\pout(q)\big)\Big)
 \subset \mathcal H_\tr,
\eeqs
where $\tout$ and $\pout$ are defined in \eqref{eq:tout} and \eqref{eq:pinout} respectively, and $\pi_{\GammaIR}$ equals $\pi_\boundary$ restricted to 
$T^*_{\GammaIR} \Rea^d$; observe that 
$\sup_{q \in \mathcal{V}_D} t^{\rm out}(q)<\infty$ since $\GammaI$ is (strictly) convex.

Corollary \ref{cor:reflection} then implies that
 $$ 
(2\sqrt{r}\muout)(\mathcal V_\tr) =
\left|
\frac{\sqrt{r}\sigma(\bcN) -\sigma(\bcD)}
{\sqrt{r}\sigma(\bcN) +\sigma(\bcD)}
\right|^{2}
(2\sqrt{r}\muin)(\mathcal V_\tr),
$$
where we have used the fact that $|\sigma(\bcN)|>0$ on $\mathcal V_\tr$ by Corollary \ref{cor:TH}. 

Since the rays emanating from $V_D$ hit $\GammaI$ directly and at angles bounded away from $\{0,\psi_1,\ldots,\psi_k\}$, Lemma \ref{lem:reflectionbound1} implies that $(2\sqrt{r}\muout)(\mathcal V_\tr) \geq C (2\sqrt{r} \muin)(\mathcal V_\tr)$ for $C>0$. 
Combining this inequality with \eqref{eq:29_1}, we have $(2\sqrt{r}\muout)(\mathcal V_\tr)>0$.
By the inequality  (\ref{eq:key2a}) in Point (ii) of Corollary \ref{cor:explain}, $\mu(\mathcal I)>0$, which is the desired contradiction.

Finally, the fact that $C$ in Theorem \ref{th:quant2} depends continuously on $\GammaIR$ follows from the fact that $c_{{\rm ray},j}, j=1,\ldots,4$, depend continuously on $\GammaIR$, and $\Crefone$
in Lemma \ref{lem:reflectionbound1} is independent of $\GammaIR$.
\end{proof}

\subsection{Proof of the lower bounds in Theorem \ref{th:quant}, Theorem \ref{th:quant2}, Theorem \ref{th:quant3}, Theorem \ref{th:local_ball}, and Theorem \ref{th:local_square}}\label{sec:lowerboundproof2}

Recall from Corollary \ref{cor:reduction} that to prove the lower bounds  in Theorems \ref{th:quant}, \ref{th:quant2}, \ref{th:local_ball}, and \ref{th:local_square}, we only need to bound $\mu(\mathcal I)$ and $\mu\big(\mathcal I \cap S^*_{B(0, 3/2)}\big)$ below; the following lemma provides the necessary lower bounds.

\begin{lem} \label{lem:quant_mes} 
Suppose Condition \ref{cond:ray} holds for $R \geq R_0$ with $\craytwo$ independent of $R$
and $\crayfive \geq \widetilde{c}_5 R$ with $\widetilde{c}_5>0$ independent of $R$.
 Then, 
 there exists $C>0$ such that, for all $R\geq R_0$,

(i)  
\beq\label{eq:lowerboundmugeneral}
\mu(\mathcal I) \geq C R \Big( \min \big( |\craythree|^{2\morderzero},\, |\crayfour|^{\mmult}\big)\Big)^2 \crayone.
\eeq
(ii) If, in addition, there exists $N_{\rm ref}\geq 1$ such that, for the interior billiard flow in $\domain$, these rays are reflected on $\GammaIR$ $N_{\rm ref}$ times, without being reflected on $\Gamma_D$ in between, and after their $N_{\rm ref}$th reflection all of these rays intersect $B(0, 3/2)  \setminus  B(0, 5/4)$ without being reflected before, then  
\beq\label{eq:lowerboundmulocal}
\mu\big(\mathcal I \cap S^*_{B(0, 3/2)} \mathbb R ^d\big) \geq C 
\Big( \min \big( |\craythree|^{2\morderzero},\, |\crayfour|^{\mmult}\big)\Big)^2
\crayone.
\eeq
\end{lem}

%

\noi\emph{Proof of (i).} 
As in the proof of Theorem \ref{th:lower}, we argue by contradiction and follow the steps in \S\ref{sec:outlinerays}.
Suppose that Condition \ref{cond:ray} holds for $R\geq R_0$, but,
for any $\epsilon>0$, there exists $R\geq R_0$ such that 
\beq \label{eq:quant_contr1}
\mu(\mathcal I)\leq \epsilon R
\Big( \min \big( |\craythree|^{2\morderzero},\, |\crayfour|^{\mmult}\big)\Big)^2
 \crayone.
\eeq
Let
\beq\label{eq:V_D}
\mathcal V_D := \Big\{ \big(x', (a_{T(x')})^\flat\big) \in T^{*}\Gamma_D, \; x' \in V_D\Big\} \subset \mathcal H_{D}.
\eeq
We now claim that
\beq\label{eq:c_delta}
\mu(\mathcal I) \geq  \left(\frac{\delta}{M}\right) R\, (2\sqrt{r}\muin)(\mathcal{V}_D)\quad\tfa R\geq 1. 
\eeq
Indeed, Part (ii) of Corollary \ref{cor:explain} implies that
\beqs
\mu(\mathcal I) \geq \dist(\GammaIR, \Omegaminus)  (2\sqrt{r}\muin)(\mathcal V_D),
\eeqs
and then to prove \eqref{eq:c_delta} we only need to show that
\beq\label{eq:dist}
\dist(\GammaIR, \Omegaminus)
\geq \left(\frac{\delta}{M}\right) R.
\eeq
Let $\delta = \dist(\Omegaminus, \partial B(0,1))$. Then, since $\tildedomain \supset B(0,M^{-1}R) \cup B(0,1)$ and $\Omegaminus \subset B(0,1)$,
if $R\geq M$,
\beqs
\dist(\GammaIR, \Omegaminus)
\geq 
 \big(M^{-1}R-1+\delta\big) 
= \left(M^{-1} - \frac{ (1-\delta)}{R}\right) R 
   \geq \left(\frac{\delta}{M}\right) R
 \eeqs
 and then \eqref{eq:dist} follows for $R\geq M$. On the other hand, if $R \leq M$, then 
 \beqs
\dist(\GammaIR, \Omegaminus)\geq \delta \geq  \left(\frac{\delta}{M}\right)  R,
\eeqs
and then \eqref{eq:dist} follows for $R\leq M$.

Combining \eqref{eq:quant_contr1} and \eqref{eq:c_delta}, we have
\begin{equation} \label{eq:low_muinV}
 (2\sqrt{r}\muin)(\mathcal{V}_D) \leq  \epsilon\frac{M}{\delta}
\Big( \min \big( |\craythree|^{2\morderzero},\, |\crayfour|^{\mmult}\big)\Big)^2
 \crayone.
\end{equation}

We now use Lemmas \ref{lem:key_Miller} and \ref{lem:nuD} to obtain a lower bound on $\muout(\mathcal{V}_D)$.
The two equations in \eqref{eq:key_Miller0} imply that
\begin{equation} \label{eq:muout_eq}
\muout = \sqrt{r} \nu_d + \frac{1}{\sqrt{r}}\nu_n-  \muin
\end{equation}
(see \eqref{eq:system2}). 
By Lemma~\ref{lem:nuD} and Part (i) of Condition \ref{cond:ray}, 
\begin{equation} \label{eq:nuD_vol}
\nu_d(\mathcal V_D) =  \vol (V_D)\geq \crayone.
\end{equation}
By the assumption that Condition \ref{cond:ray} holds (with $\craytwo$ independent of $R$), $|n(x') \cdot a|\geq \craytwo>0$ on $V_D$.
By the definitions of $V_D$ \eqref{eq:V_D} and $r(x',\xi')$ \eqref{eq:r}, $r(x',(a_{T(x')})^{\flat}=|n(x')\cdot a|$ for $x'\in V_D$ and thus $r \geq \craytwo>0$ on $\mathcal V_D$.
 Combining (\ref{eq:muout_eq}) with (\ref{eq:nuD_vol}) and (\ref{eq:low_muinV}), and using the facts that $\nu_n$ is nonnegative and 
 $\craythree, \crayfour \leq \pi/2$, we have 
\begin{align*}
(2\sqrt{r}\muout) (\mathcal V_D)&\geq 
2r \nu_d(\mathcal V_D) - (2\sqrt{r}\muin) (\mathcal V_D) \\
&\geq 2\sqrt{\craytwo}\left( \sqrt{\craytwo} -  \epsilon\frac{M}{\delta} 
\left(\frac{\pi}{2}\right)^{\max(2\morderzero,\mmult)}
\right) \crayone.
\end{align*}
If
\beq\label{eq:epsilon1}
\epsilon \leq \frac{\sqrt{\craytwo}}{2}
\frac{\delta}{M} 
\left(\frac{2}{\pi}\right)^{\max(2\morderzero,\mmult)}
\eeq
(observe that, since $\craytwo$ is assumed independent of $R$, this upper bound on $\epsilon$ is independent of $R$),
then
\beqs
(2\sqrt{r}\muout) (\mathcal V_D)\geq \craytwo\, \crayone.
\eeqs
We now use Corollary \ref{cor:explain} to propagate this lower bound on $\Gamma_D$ to a lower bound on $\GammaIR$. Indeed, Part (i) of Corollary \ref{cor:explain} then implies that
\beq \label{eq:muVI_uper}
(2\sqrt{r}\muin) (\mathcal V_{\tr}) =  (2\sqrt{r}\muout)(\mathcal V_D)  \geq  \craytwo\, \crayone.
\eeq
where 
\beqs
\mathcal{V}_\tr := 
 \bigcup_{q \in \mathcal{V}_D}\pi_{\GammaIR}\Big( \varphi_{\tout (q)}\big(\pout(q)\big)\Big)
 \subset \mathcal H_\tr,
\eeqs
where $\tout$ and $\pout$ are defined in \eqref{eq:tout} and \eqref{eq:pinout} respectively, and $\pi_{\GammaIR}$ equals $\pi_\boundary$ restricted to 
$T^*_{\GammaIR} \Rea^d$.

Combining Corollary \ref{cor:reflection}, Lemma \ref{lem:reflectionbound1}, and Point (iii) of Condition \ref{cond:ray}, we have
\beq\label{eq:xmas2}
\muout (\mathcal V_{\tr}) =\left| \frac{ \sqrt{r}\sigma(\bcN)- \sigma(\bcD)}{\sqrt{r}\sigma(\bcN)+ \sigma(\bcD)}\right|^2
\muin (\mathcal V_{\tr})
\geq \Big(\Crefone \min \big( |\craythree|^{2\morderzero},\, |\crayfour|^{\mmult}\big)\Big)^2\muin (\mathcal V_{\tr}).
\eeq

Finally, using Part (ii) of Corollary \ref{cor:explain} with Point (iv) of Condition \ref{cond:ray}, and then using
\eqref{eq:xmas2} and \eqref{eq:muVI_uper}, we have, 
\begin{align}\nonumber
\mu(\mathcal{I})\geq  
\widetilde{c}_5 \, R\,
(2\sqrt{r}\muout)(\mathcal V_{\tr}) 
&\geq 
\widetilde{c}_5 \, R\,
(2\sqrt{r}\muout)  (\mathcal V_{\tr})\\
&
\geq 
\widetilde{c}_5 \, R\,\Big(\Crefone \min \big( |\craythree|^{2\morderzero},\, |\crayfour|^{\mmult}\big)\Big)^2 \craytwo\, \crayone.
 \label{eq:contradict2}
\end{align}
We now restrict $\epsilon$ so that, in addition to satisfying \eqref{eq:epsilon1}, $\e$ satisfies
\beqs
\e < \widetilde{c}_5 \big(\Crefone\big)^2 \,\craytwo
\eeqs
(observe that, since $\widetilde{c}_5$ and $\craytwo$ are assumed independent of $R$, this upper bound is independent of $R$). Thus 
$\epsilon$ can be chosen sufficiently small (independent of $R$) such that \eqref{eq:contradict2} contradicts \eqref{eq:quant_contr1}, which is the desired contradiction.

\bpf[Proof of (ii)]
If the assumption of (ii) holds, then our contradiction argument also assumes that for all $\e>0$ there exists $R\geq R_0$ such that
\beq \label{eq:quant_contr2}
\mu(\mathcal I\cap S^*_{B(0, 3/2)} \mathbb R ^d)\leq 
\epsilon \Big( \min \big( |\craythree|^{2\morderzero},\, |\crayfour|^{\mmult}\big)\Big)^2\crayone.
\eeq
Applying Part (i) of Corollary \ref{cor:explain} $N_{\rm ref}-1$ more times and using \eqref{eq:xmas2}, we construct
$\mathcal{V}^1_{\tr}, \ldots, \mathcal{V}^{N_{\rm ref}}_{\tr} \subset T^{*}\GammaIR$, satisfying 
\begin{align}\nonumber
&\mathcal{V}^1_{\tr} := \mathcal{V}_{\tr}, \qquad (2\sqrt{r}\muin)({\mathcal{V}^{j+1}_{\tr}}) = (2\sqrt{r}\muout)({\mathcal{V}^{j}_{\tr}}), 
\\
&(2\sqrt{r}\muout)({\mathcal{V}^{j}_{\tr}}) \geq \Big(\Crefone \min \big( |\craythree|^{2\morderzero},\, |\crayfour|^{\mmult}\big)\Big)^2 (2\sqrt{r}\muin)(\mathcal{V}^{j}_\tr), \label{eq:square_reflsucc}
\end{align}
and so that for any $q \in \mathcal V_{\tr}^{N_{\rm ref}}$, $\big\{ \varphi^{\mathbb R^d}_t (\pout(q))\big\}_{t\geq 0}$ intersects $B(0, \frac 32)  \setminus  B(0, \frac 54)$ before hitting $\Gamma_D$ or $\GammaIR$.
Therefore, by (\ref{eq:square_reflsucc}) and (\ref{eq:muVI_uper})
\begin{equation} \label{eq:VNlower}
(2\sqrt{r}\muout)({\mathcal{V}^{N}_{\tr}}) \geq  
\Big(\Crefone \min \big( |\craythree|^{2\morderzero},\, |\crayfour|^{\mmult}\big)\Big)^{2 N_{\rm ref}}
\craytwo \, 
\crayone.
\end{equation}
Finally, since any ray entering $ B(0, \frac 32)  \setminus  B(0, \frac 54)$ spends a time at least $\frac 12 (\frac 32 - \frac 54) = \frac 18$ in this annulus, Part (ii) of Corollary \ref{cor:explain} implies that
\begin{align}\nonumber
\mu\big(\mathcal I \cap S^*_{B(0, 3/2)  \setminus  B(0, 5/4)} \mathbb R ^d\big) 
&\geq \frac 18 (2\sqrt{r}\muout)({\mathcal{V}^{N_{\rm ref}}_{\tr}}) \\
&\geq \frac18
\Big(\Crefone \min \big( |\craythree|^{2\morderzero},\, |\crayfour|^{\mmult}\big)\Big)^{2 N_{\rm ref}}
\craytwo \, \crayone,
\label{eq:xmas4}
\end{align}
where we have used (\ref{eq:VNlower}). Therefore, if 
\beqs
\epsilon < (\Crefone)^{2N_{\rm ref}} \craytwo ,
\eeqs
then \eqref{eq:xmas4} contradicts (\ref{eq:quant_contr2}), which is the desired contradiction.
(Observe that, similar to in Part (i), 
the upper bound on $\epsilon$ is independent of $R$ since $\craytwo$ and $\Crefone$ are independent of $R$.)
\end{proof}

\begin{proof}[Proof of the lower bounds in Theorems \ref{th:quant}, \ref{th:quant2}, \ref{th:local_ball}, and \ref{th:local_square}]
The lower bounds will follow from combining Corollary \ref{cor:reduction}, Lemma \ref{lem:quant_mes}, and the ray constructions in Lemmas \ref{lem:ray1}-\ref{lem:ray4}.

For Theorem \ref{th:quant2} (for generic $\GammaIR$), 
Lemma \ref{lem:ray3} implies that the assumptions of Part (i) of Lemma \ref{lem:quant_mes} are satisfied with $\crayone, \craythree, \crayfour$ independent of $R$, and $R_0$ sufficiently large; the required lower bound \eqref{eq:lowerboundmu2} on $\mu(\cI)$ then follows by inserting this (lack of) $R$-dependence into \eqref{eq:lowerboundmugeneral}. 

For the lower bound in Theorem \ref{th:quant} (for $\GammaIR =\partial B(0,R)$), 
Lemma \ref{lem:ray2} implies that the assumptions of Part (i) of Lemma \ref{lem:quant_mes} are satisfied with $\crayone, \crayfour$ independent of $R$, $\craythree = \widetilde{c}_3/R$ with $\widetilde{c}_3>0$ independent of $R$, and $R_0$ sufficiently large.
The required lower bound \eqref{eq:lowerboundmu} $\mu(\cI)$ then follows by inserting these $R$-dependences into \eqref{eq:lowerboundmugeneral}, and 
observing that, for $R$ sufficiently large,
\beq\label{eq:asymp}
\min \big( |\craythree|^{2\morderzero},\, |\crayfour|^{\mmult}\big) = \left|\frac{\widetilde{c}_3}{R}\right|^{2\morderzero}.
\eeq

For Theorem \ref{th:local_ball} (i.e.~the local error for $\GammaIR =\partial B(0,R)$), 
Point (iv)$^\prime$ in Lemma \ref{lem:ray2} implies that the assumptions of Part (ii) of Lemma \ref{lem:quant_mes} are satisfied $N_{\rm ref}=1$
and $R_0$ sufficiently large. The required lower bound on $\mu(\mathcal I \cap S^*_{B(0, 3/2)} \mathbb R ^d)$ \eqref{eq:lowerboundmu3} then follows from 
\eqref{eq:lowerboundmulocal} using \eqref{eq:asymp} and the fact that $\crayone$ is independent of $R$. The fact that the result holds with $R_0=2$ when 
$\MPade=\NPade=0$ follows from Remark \ref{rem:MN0}.

Finally, for Theorem \ref{th:local_square} (i.e.~the local error for the hypercube), 
Lemma \ref{lem:ray4} implies that the assumptions of Part (ii) of Lemma \ref{lem:quant_mes} are satisfied
with $\craythree, \crayfour$ independent of $R$, $\crayone =\widetilde{c}_1/R^{d-1}$ with $\widetilde{c}_1$ independent of $R$, 
and $R_0=4$.
The required lower bound on 
$\mu(\mathcal I \cap S^*_{B(0, 3/2)} \mathbb R ^d)$ \eqref{eq:lowerboundmu4} then follows from 
\eqref{eq:lowerboundmulocal} by inserting these $R$-dependences.
\end{proof}

\section{Proof of the trace bounds (Theorem \ref{th:higherOrderBound})}
\label{sec:traceHO}

\subsection{Strategy of the proof}
To illustrate some of the main ideas, consider the BVP \eqref{eq:genBC} with $\bcN=\bcD=I$, $\overline{M}$ compact, and the boundary condition imposed on the whole of $\boundary $, i.e.,
\begin{equation}
\label{eq:genBC2}
\begin{cases}
(-h^2\Delta-1)u=hf&\text{ in }M\\
hD_n u-u =g&\text{ on }\Gamma:=\boundary .
\end{cases}
\end{equation}
In the notation of Theorem \ref{th:higherOrderBound}, we have $m_{0,i}=m_{1,i}=0$, and the bounds \eqref{eq:trace1} and \eqref{eq:trace2} in the case $\ell_i=0$ are that
\beq\label{eq:example1}
\N{u}_{L^2(\Gamma)} + \N{h D_n u}_{L^2(\Gamma)} \leq C \Big( 
\N{u}_{L^2(M)} + \N{f}_{L^2(M)} + \N{g}_{L^2(\Gamma)}
\Big)
\eeq
and 
\beq\label{eq:example2}
\N{u}_{H^1_h(M)}\leq C \Big( 
\N{u}_{L^2(M)} + h\N{f}_{L^2(M)} + \N{g}_{L^2(\Gamma)}
\Big).
\eeq
We now show how to obtain these bounds;
pairing the PDE in \eqref{eq:genBC2} with $u$ and integrating by parts, we have
\beq\label{eq:example3}
h^2 \N{\nabla u}^2_{L^2(M)} - \N{u}^2_{L^2(M)} - h\big\langle f , u \big\rangle_{L^2(M)}= h i \N{u}^2_{L^2(\Gamma)}  + h \big\langle g, u \rangle_{L^2(\Gamma)}.
\eeq
Taking the imaginary part of \eqref{eq:example3}, we find that 
\beq\label{eq:example4}
 \N{u}^2_{L^2(\Gamma)} \leq\N{g}^2_{L^2(\Gamma)} + 
 \N{f}^2_{L^2(M)} + 
 \N{u}^2_{L^2(M)}.
\eeq
Taking the real part of \eqref{eq:example3} and adding $2\|u\|^2_{L^2(M)}$ to both sides of the resulting equation, we find that
\beq\label{eq:example5}
\N{u}^2_{H^1_h(M)} \leq \frac{5}{2} \N{u}^2_{L^2(M)} + \frac{h^2}{2} \N{f}^2_{L^2(M)} + \frac{h^2}{2} \N{u}^2_{L^2(\Gamma)} + \frac{1}{2} \N{g}^2_{L^2(\Gamma)}.
\eeq
Combining the inequality \eqref{eq:example4} with the boundary condition in \eqref{eq:genBC2}, we obtain the first result \eqref{eq:example1}. Then, using \eqref{eq:example4} in \eqref{eq:example5}, we obtain the second result \eqref{eq:example2}.

The proof of Theorem \ref{th:higherOrderBound} follows similar steps; indeed, the two main ingredients are (i) 
bounds on the traces in terms of the data and $H^{1}_h$ norms of $u$, and (ii) a bound the $H^{1}_h$ norm of $u$ in term of the traces and the data. The bound in (ii) is obtained by considering $\Re\langle ( -h^2 \Delta_g-1)u,u\rangle_{L^2(M)}$ and integrating by parts, similar to above, with the inequality \eqref{eq:analogy1} the generalisation of the inequality \eqref{eq:example5}.
The bounds in (i) are obtained by considering $\Im \langle(-h^2 \Delta_g -1)u,u\rangle_{L^2(M)}$, similar to above, but also 
$\Im \langle(-h^2 \Delta_g -1)u,hD_{\nu} u\rangle_{L^2(M)}$ (with Lemma \ref{lem:intByParts} above
considering a general commutator, and Lemma \ref{lem:basicbound} specialising to the case of a normal derivative). 

The additional complications for the bounds in (i) are because we need to consider the cases where $\bcD$ and $\bcN$ are both elliptic  (Lemma \ref{lem:ellipticEst}), where $\bcD$ is small and $\bcN$ elliptic (Lemma \ref{lem:elephant1}), and where $\bcD$ is elliptic and $\bcN$ small  (Lemma \ref{lem:elephant2}), 
These three cases are considered in \S\ref{sub:a_priori}, and then in \S\ref{sub:proofhOB} we show that, under the assumptions \eqref{eq:conditions2a}-\eqref{eq:conditions2d}, 
the bounds in these three cases cover all of $T^* \Gamma$.

\subsection{A priori estimates} \label{sub:a_priori}
We begin by proving some a-priori estimates for~\eqref{eq:genBC}. 
As usual, we work near $\Gamma$ where $M$ is locally given by $x_1>0$, as in \S\ref{subsec:geo}.
We repeatedly use the integration by parts result in Lemma \ref{lem:intByParts}.

\begin{lem}
\label{lem:basicbound}
If $u$ solves \eqref{eq:genBC}, then, for all $\e>0$ and for all $\ell$,
$$
\|hD_{x_1}u\|_{H_h^\ell(\Gamma_i)}\leq C\Big(\|u\|_{H_h^{\ell+1}(\Gamma_i)}+\|u\|_{H_h^{\ell+1}(M)}+\e^{-1}\|f\|_{H_h^{\ell}(M)}+\e\|u\|_{H_h^1(M)}  \Big).
$$
\end{lem}
\begin{proof}
Let $\chi\in C_c^\infty((-2\delta,2\delta);[0,1])$ with $\chi\equiv 1$ on $[-\delta,\delta]$. Let 
$$
B_1(x,hD_{x'}):=\chi(x_1)\langle hD_{x'}\rangle^{2\ell}\quad\tand\quad B_0(x,hD_{x'}):=\frac{1}{2}hD_{x_1}B_1 = \frac{h}{2i}\chi'(x_1) \langle hD_{x'}\rangle^{2\ell}.
$$
Then \eqref{eq:B0B1} holds, and $B$ satisfies the assumption of Lemma~\ref{lem:intByParts}; since $B_0|_{x_1=0}=0$, \eqref{eq:ibps} implies that 
\begin{align}\nonumber
\frac{i}{h}\langle [P,B]u,u\rangle_{L^2(M)}+\frac{2}{h}\Im \langle Pu,Bu\rangle_{L^2(M)}&=\langle h(B_1 a_1 -\overline{a_1}B_1)hD_{x_1}u,u\rangle_{L^2(\Gamma_i)}\\
&\quad
+\langle B_1(R-ha_0)u,u\rangle_{L^2(\Gamma_i)}+\langle B_1 hD_{x_1}u,hD_{x_1}u\rangle_{L^2(\Gamma_i)}.\label{eq:ibps_end1}
\end{align}
Now, observe that 
$$
[P,B]=h(\widetilde{B}_2 (hD_{x_1})^2+\widetilde{B}_1hD_{x_1}+\widetilde{B}_0)
$$
where 
$$
\widetilde{B}_2\in C_c^\infty((\delta,2\delta);\Psi^{2\ell}(\Gamma_i)),\quad \widetilde{B}_1\in C_c^\infty((\delta,2\delta);\Psi^{2\ell+1}(\Gamma_i)), \quad\widetilde{B}_0\in C_c^\infty((-2\delta,2\delta);\Psi^{2\ell+2}(\Gamma_i)).
$$
In particular, using the elliptic parametrix construction in the interior of $M$, we have
$$
\big\|[P,B]u\big\|_{H_h^s(M)}\leq C_sh\Big(\|Pu\|_{H_h^{s+2\ell}(M)}+\|u\|_{H_h^{s+2\ell}(M)}\Big).
$$
Therefore, 
\begin{equation*}
\begin{aligned}
&\Big|\big\langle B_1(1-R)u,u\big\rangle_{L^2(\Gamma_i)}+\big\langle B_1 hD_{x_1}u,hD_{x_1}u\big\rangle_{L^2(\Gamma_i)} \Big|\\
&\qquad\leq C h\|u\|_{H_h^\ell(\Gamma_i)}^2 +C\|u\|_{H_h^\ell(M)}^2+C\e^{-1}\|f\|_{H_h^\ell}^2+\e\|u\|^2_{H_h^{1}(M)}.\end{aligned}
\end{equation*}
and hence
\begin{equation*}
\begin{aligned}
\|hD_{x_1}u\|_{H_h^\ell}^2 
&\leq C \|u\|_{H_h^{\ell+1}(\Gamma_i)}^2 +C\|u\|_{H_h^\ell(M)}^2+C\e^{-1}\|f\|_{H_h^\ell}^2+\e\|u\|^2_{H_h^{1}(M)}.\end{aligned}
\end{equation*}
\end{proof}

\begin{rem*}
When $\ell=0$ the bound in Lemma \ref{lem:basicbound} is valid for Lipschitz domains and goes back to Ne\v{c}as; see \cite[\S5.1.2]{Ne:67}, \cite[Theorem 4.24 (i)]{Mc:00}.
\end{rem*}

We now show a bound where $\bcD$ and $\bcN$ are both elliptic. 
\begin{lem}
\label{lem:ellipticEst}
Suppose that $\WF(E)\subset \Ell(\bcD)\cap \Ell(\bcN)$. Then for any $B'\in \Psi^0$ with 
$$
\WF(E)\cap \WF(\Id-B')=\emptyset,\qquad \WF(B')\subset \Ell(\bcN)\cap \Ell(\bcD) 
$$
there exist $C>0$ and $h_0>0$ such that for any $\e>0$, $0<h<h_0$,
\begin{align*}
&\|Eu\|_{H_h^{\ell+m_0}(\Gamma_i)}+\|EhD_{x_1}u\|_{H_h^{\ell+m_1}(\Gamma_i)}\\
&\qquad\leq C\Big(\|u\|_{H_h^{\frac{2\ell+m_1+m_0+1}{2}}(M)}+\|u\|_{L^2(M)}+\|f\|_{H_h^{\frac{2\ell+m_1+m_0-1}{2}}(M)}+\|f\|_{L^2(M)}\Big) \\
&\qquad\qquad+\e \Big(\|B'u\|_{H_h^{\ell+m_0}(\Gamma_i)}+ \|B'hD_{x_1}u\|_{H_h^{\ell+m_1}(\Gamma_i)}\Big)+C\e^{-1}\|B'g_i\|_{H_h^{\ell}(\Gamma_i)}\\
&\qquad\qquad+O\Big(h^\infty\big(\|u\|_{H_h^{-N}(\Gamma_i)}+\|hD_{x_1}u\|_{H_h^{-N}(\Gamma_i)}+\|g\|_{H_h^{-N}(\Gamma_i)}\big)\Big).\end{align*}
\end{lem}
\begin{proof}
Let $B_0\in \Psi^{\ell_0}(\Gamma_i)$ self-adjoint with $\WF(b_0(x',hD_{x'}))\subset \WF(E)$. Let $B'\in \Psi^0(\Gamma_i)$ with 
$$ 
\WF(E)\subset \Ell(B')\subset \WF(B')\subset \Ell(\bcN)\cap \Ell(\bcD).
$$
We can assume without loss of generality that $B'$ is microlocally the identity in a neighbourhood of $\WF(E)$. 
Next, let $B_1=0$ and $\bcN^{-1}$ and $\bcD^{-1}$ denote microlocal inverses for $\bcN$ and $\bcD$ on $\WF(B')$. Then, by Lemma \ref{lem:intByParts}, 
\begin{align*}
\Big|\langle B_0hD_{x_1}u,u\rangle_{L^2(\Gamma_i)}+\langle h\overline{a_1}B_0u,u\rangle_{L^2(\Gamma_i)}&+\langle B_0u,hD_{x_1}u\rangle_{L^2(\Gamma_i)}\Big|\\
&\leq |2\langle f,Bu\rangle_{L^2(M)}|+h^{-1}|\langle [P,B]u,u\rangle_{L^2(M)}|
\end{align*}
First, note that 
$$
[P,B]= h(\widetilde{B}_1hD_{x_1}+\widetilde{B}_2)
$$
where 
$$
\widetilde{B}_1\in C_c^\infty((\delta,2\delta); \Psi^{\ell_0}(\Gamma_i)),\qquad \widetilde{B}_2\in C_c^\infty((-2\delta,2\delta); \Psi^{\ell_0+1}(\Gamma_i)),
$$
In particular, by the standard elliptic parametrix construction, for all $s\in \mathbb{R}$,
$$
\|[P,B]\|_{H_h^s(M)}\leq C_sh\Big(\|Pu\|_{H_h^{s+\ell_0-1}(M)}+\|u\|_{H_h^{s+\ell_0+1}(M)}+\|u\|_{L^2(M)}\Big).
$$
Therefore, 
\begin{align*}
&\Big|\langle B_0hD_{x_1}u,u\rangle_{L^2(\Gamma_i)}+\langle h\overline{a_1}B_0u,u\rangle_{L^2(\Gamma_i)}+\langle B_0u,hD_{x_1}u\rangle_{L^2(\Gamma_i)}\Big|\\
&\leq C\Big(\|f\|_{H_h^{\frac{\ell_0-1}{2}}(M)}+\|u\|_{H_h^{\frac{\ell_0+1}{2}}(M)}+\|u\|_{L^2(M)}+\|f\|_{L^2(M)}\Big)\Big(\|u\|_{H_h^{\frac{\ell_0+1}{2}}(M)}+\|u\|_{L^2(M)}\Big).
\end{align*}
Now, using~(\ref{eq:genBC}),
$$
\langle B_0hD_{x_1}u,u\rangle_{L^2(\Gamma_i)}=\langle B_0 \bcN^{-1}(\bcD u+g_i),u\rangle_{L^2(\Gamma_i)}+O\Big(h^\infty\big(\|u\|^2_{H_h^{-N}(\Gamma_i)}+\|hD_{x_1}u\|_{H_h^{-N}(\Gamma_i)}^2\big)\Big)
$$
and
$$
\langle B_0u,hD_{x_1}u\rangle_{L^2(\Gamma_i)}=\langle B_0u, \bcN^{-1}(\bcD u+g_i)\rangle_{L^2(\Gamma_i)}+O\Big(h^\infty\big(\|u\|^2_{H_h^{-N}(\Gamma_i)}+\|hD_{x_1}u\|_{H_h^{-N}(\Gamma_i)}^2\big)\Big).
$$
In particular, letting $B'\in\Psi^0$ with $\WF(B_0)\subset \Ell(B')$,
\begin{align*}
&\Big|\Big\langle[ (\bcN^{-1}\bcD )^*B_0+B_0(\bcN^{-1}\bcD )] u,  u\Big\rangle_{L^2(\Gamma_i)}\Big|\\
& \qquad\leq C\Big(\|f\|_{H_h^{\frac{\ell_0-1}{2}}(M)}+\|f\|_{L^2(M)}+\|u\|_{H_h^{\frac{\ell_0+1}{2}}(M)}+\|u\|_{L^2(M)}\Big)\Big(\|u\|_{H_h^{\frac{\ell_0+1}{2}}(M)}+\|u\|_{L^2(M)}\Big)\\
&\qquad\qquad +O(h)\|B'u\|_{H_h^{\frac{\ell_0}{2}}(\Gamma_i)}^2 +\e \|B'u\|_{H_h^{\frac{m_0-m_1+\ell_0}{2}}(\Gamma_i)}^2+C\e^{-1}\|B'g_i\|_{H_h^{\frac{\ell_0-m_1-m_0}{2}}(\Gamma_i)}^2\\
&\qquad\qquad+O\Big(h^\infty\big(\|u\|^2_{H_h^{-N}(\Gamma_i)}+\|hD_{x_1}u\|_{H_h^{-N}(\Gamma_i)}^2+\|g\|_{H_h^{-N}(\Gamma_i)}^2\big)\Big).
\end{align*}
Now, choose $b_0(x',hD_{x'})\in \Psi^{m_1-m_0+2\ell}$ self adjoint (i.e. $\ell_0=m_1-m_0+2\ell$) such that $B_0$ is elliptic on $\WF(E)$.
Then, since $\bcD $ and $\bcN$ have real-valued symbols and $-\bcN^{-1}\bcD $ is elliptic on $\WF(E)$, 
\begin{align*}
- \Re \langle B_0 \bcN^{-1}\bcD u,u\rangle \geq C\|E u\|^2_{H_h^\ell(\Gamma_i)}- Ch\|B'u\|^2_{H_h^{\frac{\ell-1}{2}}(\Gamma_i)}-O(h^\infty)\|u\|_{H_h^{-N}(M)}^2,
\end{align*}
and
\begin{equation}
\label{eq:squids}
\begin{aligned}
\|Eu\|_{H_h^{\ell}}^2& \leq C\Big(\|f\|_{H_h^{\frac{m_1-m_0+2\ell-1}{2}}(M)}+\|f\|_{L^2(M)}+\|u\|_{H_h^{\frac{m_1-m_0+2\ell+1}{2}}(M)}+\|u\|_{L^2(M)}\Big)\\
&\hspace{5cm}
\times\Big(\|u\|_{H_h^{\frac{m_1-m_0+2\ell+1}{2}}(M)}+\|u\|_{L^2(M)}\Big)\\ 
&\quad+O(h)\|B'u\|_{H_h^{\frac{m_1-m_0+2\ell}{2}}(\Gamma_i)}^2 +\e \|B'u\|_{H_h^{\ell}(\Gamma_i)}^2+C\e^{-1}\|B'g_i\|_{H_h^{k-m_0}(\Gamma_i)}^2\\
&\qquad\qquad+O\Big(h^\infty\big(\|u\|^2_{H_h^{-N}(\Gamma_i)}+\|hD_{x_1}u\|_{H_h^{-N}(\Gamma_i)}^2+\|g\|_{H_h^{-N}(\Gamma_i)}^2\big)\Big).
\end{aligned}
\end{equation}

Let $E'\in \Psi^0$ with 
$$
\WF(E)\cap \WF(\Id-E')=\emptyset,\qquad  \WF(E')\subset \Ell(B')\cap \Ell(\bcN)\cap \Ell(\bcD).
$$ 
By~(\ref{eq:squids}),
\begin{align*}
\|E'u\|_{H_h^{\ell}(\Gamma_i)}^2& \leq C\Big(\|u\|_{H_h^{\frac{m_1-m_0+2\ell+1}{2}}(M)}^2+\|u\|_{L^2(M)}^2+C(\|f\|_{H_h^{\frac{m_1-m_0+2\ell-1}{2}}(M)}^2 +\|f\|_{L^2(M)}^2\Big) \\
&\qquad\qquad+O(h)\|B'u\|_{H_h^{\frac{2\ell-m_0+m_1}{2}}(\Gamma_i)}^2+\e \|B'u\|_{H_h^{\ell}(\Gamma_i)}^2+C\e^{-1}\|B'g_i\|_{H_h^{k-m_0}(\Gamma_i)}^2 \\
&\qquad\qquad+O\Big(h^\infty\big(\|u\|^2_{H_h^{-N}(\Gamma_i)}+\|hD_{x_1}u\|_{H_h^{-N}(\Gamma_i)}^2+\|g\|_{H_h^{-N}(\Gamma_i)}^2\big)\Big).
\end{align*}
Let $\bcN^{-1}$ denote a microlocal inverse for $\bcN$ on $\WF(B')$. Then,  
$$
E hD_{x_1}u= E \big(\bcN^{-1}(-\bcD E' u+B'g_i)\big)+O\big(h^\infty\|u\|_{H_{h}^{-N}(\Gamma_i)}\big)_{H_h^N},
$$
so
$$
\|EhD_{x_1}u\|_{H_h^\ell(\Gamma_i)}\leq C \|E'u\|_{H_h^{\ell+m_0-m_1}}+\|B'g_i\|_{H_h^{k-m_1}}+O\big(h^\infty\|u\|_{H_{h}^{-N}(\Gamma_i)}\big).
$$
In particular, 
\begin{align*}
\|EhD_{x_1}u\|_{H_h^{\ell}}^2& \leq C\Big(\|u\|_{H_h^{\frac{m_0-m_1+2\ell+1}{2}}(M)}^2+\|u\|_{L^2(M)}^2+\|f\|_{H_h^{\frac{m_0-m_1+2\ell-1}{2}}(M)}^2 +\|f\|_{L^2(M)}^2\Big) \\
&\qquad\qquad+O(h)\|B'u\|_{H_h^{\frac{2\ell+m_0-m_1}{2}}(\Gamma_i)}^2+\e \|B'hD_{x_1}u\|_{H_h^{\ell}(\Gamma_i)}^2+C\e^{-1}\|B'g_i\|_{H_h^{k-m_1}(\Gamma_i)}^2 \\
&\qquad\qquad+O\Big(h^\infty\big(\|u\|^2_{H_h^{-N}(\Gamma_i)}+\|hD_{x_1}u\|_{H_h^{-N}(\Gamma_i)}^2+\|g\|_{H_h^{-N}(\Gamma_i)}^2\big)\Big).
\end{align*}
\end{proof}

We now consider $\bcD$ small and $\bcN$ elliptic:
\begin{lem}
\label{lem:elephant1}
Let $K\Subset T^{*}\Gamma_i$. Then for all $\eta>0$ there is $\delta_0>0$ and $C>0$ such that for all $0<\delta <\delta_0$, $E\in \Psi^{0}$ with
\begin{equation}
\label{eq:saphire}\WF(E)\subset K\cap \Ell(\bcN)\cap \{|\sigma (\bcD)|< \delta \langle \xi\rangle^{m_0}\}\cap\{ |R(x',\xi')|>\eta\},
\end{equation}
and $B'\in \Psi^{0}$ with 
$$
\WF(E)\cap \WF(\Id-B')=\emptyset,\qquad  \WF(B')\subset \Ell(\bcN)\cap \{|\sigma (\bcD)|< \delta \langle \xi\rangle^{m_0}\},
$$
there is $h_0>0$ small enough such that for all $0<h<h_0$ and $0<\e<1$
\begin{equation}
\label{e:gook}
\begin{aligned}
&\|Eu\|_{H_h^{\ell+m_0}(\Gamma_i)}+\|EhD_{x_1}u\|_{H_h^{\ell+m_1}(\Gamma_i)}\\  &\leq C (\e+h)\|B'u\|_{H_h^{\ell+m_0}(\Gamma_i)}+ C(\e^{-1}+1)\|B'g\|_{H_h^\ell(\Gamma_i)}  \\
&\qquad+C\|u\|_{H_h^{\ell+\frac{m_1+m_0+1}{2}}(M)}+C\e^{-1}\Big(\|f\|_{H_h^{\ell+\frac{m_1+m_0-1}{2}}(M)}+\|f\|_{L^2(M)}\Big)+C\e \|u\|_{H_h^1(M)}\\
&\qquad+O\Big(h^\infty\big(\|u\|_{H_h^{-N}(\Gamma_i)}+\|hD_{x_1}u\|_{H_h^{-N}(\Gamma_i)}+\|g\|_{H_h^{-N}(\Gamma_i)}\big)\Big).
\end{aligned}
\end{equation}
Moreover, if $ m_0\leq m_1+1$~\eqref{e:gook} holds with $K=T^{*}\Gamma_i$
\end{lem}
\begin{proof}
Throughout the proof, we take $b_1(x',hD_{x'})$ self-adjoint with $b_1\in \Psi^{2(k+m_0-1)}$  if $m_0\leq m_1+1$ and $b_1\in \Psi^{\comp}$ otherwise. We assume that
$$
\WF(E)\subset \Ell(b_1(x',hD_{x'}))\subset \WF(b_1(x',hD_{x'}))\subset \Ell(\bcN)\cap \{|\sigma (\bcD)|< \delta \langle \xi\rangle^{m_0}\},
$$
As in Lemma \ref{lem:basicbound}, let $\chi\in C_c^\infty((-2\delta,2\delta);[0,1])$ with $\chi\equiv 1$ on $[-\delta,\delta]$.
Let
\beq\label{eq:choiceofb}
B_1(x,hD_{x'}):=\chi(x_1)b_1(x',hD_{x'}) \quad\tand\quad B_0(x',hD_{x'}):=\frac{1}{2}hD_{x_1}B_1.
\eeq
Then \eqref{eq:B0B1} holds, and $B$ satisfies the assumption of Lemma~\ref{lem:intByParts}; since $B_0|_{x_1=0}=0$, \eqref{eq:ibps} implies that \eqref{eq:ibps_end1} holds.

Since $\bcN$ is elliptic on $\WF B'$, there  exists $\bcN^{-1}\in \Psi^{-m_1}$ a microlocal inverse for $\bcN$ on $\WF(B')$; that is, for any $\widetilde{B}$ with $\WF(\widetilde{B})\subset\{B'\equiv \Id\}$,
\begin{equation}
\label{eq:isolate}
\widetilde{B}hD_{x_1}u= \widetilde{B}\bcN^{-1}(\bcD B'u+B'g) +O(h^\infty)_{\Psi^{-\infty}}g+O(h^\infty)_{\Psi^{-\infty}}u+O(h^\infty)_{\Psi^{-\infty}}hD_{x_1}u
\end{equation}
and hence, using the fact that we are working with compactly microlocalized operators on $\Gamma$ to see that all $H_h^s(\Gamma_i)$ norms are equivalent up to $h^\infty$ remainders, we have
\begin{equation}
\label{eq:aardvark}
\begin{aligned}
&\Big|\langle B_1Ru,u\rangle_{L^2(\Gamma_i)}+\langle B_1 hD_{x_1}u,hD_{x_1}u\rangle_{L^2(\Gamma_i)} \Big|\\
&\qquad\leq C h\|B'u\|_{H_h^{\ell+m_0}(\Gamma_i)}^2+ Ch\|B'g\|_{H_h^{\ell}(\Gamma_i)}^2 + \Big|ih^{-1}\langle [P,B]u,u\rangle_{L^2(M)}+2\Im \langle f,Bu\rangle_{L^2(M)}\Big|\\
&\qquad+O\Big(h^\infty\big(\|u\|^2_{H_h^{-N}(\Gamma_i)}+\|hD_{x_1}u\|^2_{H_h^{-N}(\Gamma_i)}+\|g\|^2_{H_h^{-N}(\Gamma_i)}\big)\Big).\end{aligned}
\end{equation}
Now, observe that 
$$
[P,B]=h(\widetilde{B}_2 (hD_{x_1})^2+\widetilde{B}_1hD_{x_1}+\widetilde{B}_0)
$$
where 
\begin{gather*}
\widetilde{B}_2\in C_c^\infty((\delta,2\delta);\Psi^{2(k+m_0-1)}(\Gamma_i)),\quad \widetilde{B}_1\in C_c^\infty((\delta,2\delta);\Psi^{2(k+m_0)-1}(\Gamma_i)), \\
\widetilde{B}_0\in C_c^\infty((-2\delta,2\delta);\Psi^{2(k+m_0)}(\Gamma_i)).
\end{gather*}
In particular, using the elliptic parametrix construction as before, we have
$$
\|[P,B]u\|_{H_h^s(M)}\leq C_sh(\|Pu\|_{H_h^{s+2(k+m_0)}(M)}+\|u\|_{H_h^{s+2(k+m_0)}(M)}+\|u\|_{L^2(M)}),
$$
so by~\eqref{eq:isolate} and ~\eqref{eq:aardvark},
\begin{equation}
\label{eq:aardvark2}
\begin{aligned}
&\Big|\langle B_1Ru,u\rangle_{L^2(\Gamma_i)}+\langle B_1 \bcN^{-1}\bcD u,\bcN^{-1}\bcD u\rangle_{L^2(\Gamma_i)} \Big|\\
&\leq C (\e+h)\|B'u\|_{H_h^{\ell+m_0}(\Gamma_i)}^2+ C(\e^{-1}+1)\|B'g\|_{H_h^\ell(\Gamma_i)}^2  \\
&\qquad+C\|u\|_{H_h^{\ell+\frac{m_1+m_0+1}{2}}(M)}^2+C\e^{-1}(\|f\|_{H_h^{\ell+\frac{m_1+m_0-1}{2}}(M)}^2+\|f\|_{L^2(M)}^2)+C\e \|u\|^2_{H_h^1(M)}\\
&\qquad+O\Big(h^\infty\big(\|u\|^2_{H_h^{-N}(\Gamma_i)}+\|hD_{x_1}u\|^2_{H_h^{-N}(\Gamma_i)}+\|g\|^2_{H_h^{-N}(\Gamma_i)}\big)\Big).\end{aligned}
\end{equation}

If $m_0>m_1+1$, we assume that $b_1\in S^{\comp}$. Therefore, for all $(m_0,m_1)$
$$
B_1R+(\bcN^{-1}\bcD)^*B_1\bcN^{-1}\bcD\in \Psi^{2(\ell+m_0)}.
$$
for our choice of $B_1$. Next, since $\bcD$ is elliptic on $\{R=0\}$, for any $K\subset T^{*}\Gamma_i$ compact, there exists $\delta_0>0$ small enough such that 
$$
\inf\Bigg\{\langle \xi'\rangle^{-2}\Big||\sigma(\bcN^{-1}\bcD)(x',\xi')|^2+R(x',\xi')\Big|\, \text{ where } \, |\sigma(\bcD)(x',\xi')|\leq \delta_0\langle \xi'\rangle^{m_0},\, (x',\xi')\in K\Bigg\}\geq c_K>0
$$ 
Moreover, if $m_0\leq m_1+1$, then there is $\delta_0>0$ small enough such that
$$
\inf\Bigg\{\langle \xi'\rangle^{-2}\Big||\sigma(\bcN^{-1}\bcD)(x',\xi')|^2+R(x,\xi)\Big|\,\text{ where } \, |\sigma(\bcD)(x',\xi')|\leq \delta_0\langle \xi'\rangle^{m_0},\, (x',\xi')\in T^{*}\Gamma_i\Bigg\}\geq c>0.
$$ 
In particular, since $R$ is real-valued, there is $B_1\in \Psi^{2(k+m_0-1)}$ self adjoint, elliptic on $\WF(E)$, such that 
\begin{gather*}
\sigma(B_1R+(\bcN^{-1}\bcD)^*B_1\bcN^{-1}\bcD)(x',\xi')\geq c\langle \xi'\rangle^{2(k+m_0)},\qquad (x'
,\xi')\in \WF(E).
\end{gather*} 
 In particular, then the sharp G\aa rding inequality \cite[Theorem 9.11]{Zworski_semi} gives
$$
\|Eu\|_{H_h^{\ell+m_0}(\Gamma_i)}^2\leq C\big\langle\big( B_1R+(\bcN^{-1}\bcD\big)^*B_1\bcN^{-1}\bcD)u,u\big\rangle_{L^2(\Gamma_i)}+Ch\|B'u\|_{H_h^{\ell+m_0-\frac{1}{2}}}^2+O(h^\infty)\|u\|^2_{H_h^{-N}(\Gamma_i)},
$$
and we obtain from~\eqref{eq:aardvark2},
\begin{align*}
\|Eu\|_{H_h^{\ell+m_0}(\Gamma_i)}^2&\leq C (\e+h)\|B'u\|_{H_h^{\ell+m_0}(\Gamma_i)}^2+ C(\e^{-1}+1)\|B'g\|_{H_h^\ell(\Gamma_i)}^2  \\
&\quad+C\|u\|_{H_h^{\ell+\frac{m_1+m_0+1}{2}}(M)}^2+C\e^{-1}\Big(\|f\|_{H_h^{\ell+\frac{m_1+m_0-1}{2}}(M)}^2+\|f\|_{L^2(M)}^2\Big)+C\e \|u\|^2_{H_h^1(M)}\\
&\quad+O\Big(h^\infty\big(\|u\|^2_{H_h^{-N}(\Gamma_i)}+\|hD_{x_1}u\|^2_{H_h^{-N}(\Gamma_i)}+\|g\|^2_{H_h^{-N}(\Gamma_i)}\big)\Big).\end{align*}
Next, we write, as above,
$$
EhD_{x_1}u= E\bcN^{-1}(\bcD E'u+E'g)+O(h^\infty)_{\Psi^{-\infty}}g+O(h^\infty)_{\Psi^{-\infty}}u+O(h^\infty)_{\Psi^{-\infty}}hD_{x_1}u
$$
to obtain
\begin{align*}
\|EhD_{x_1}u\|_{H_h^{\ell+m_1}(\Gamma_i)}^2 &\leq  C\|E'u\|_{H_h^{\ell+m_0}(\Gamma_i)}+C\|E'g\|_{H_h^\ell(\Gamma_i)}\\
&\qquad+O(h^\infty(\|u\|_{H_h^{-N}(\Gamma_i)}+\|hD_{x_1}u\|_{H_h^{-N}(\Gamma_i)}+\|g\|_{H_h^{-N}(\Gamma_i)}),
\end{align*}
and this finishes the proof.
\end{proof}

Finally, we consider the case $\bcD$ elliptic and $\bcN$ small.
\begin{lem}
\label{lem:elephant2}
For all $K\Subset T^{*}\Gamma_i$, there is $\delta_0>0$ and $C>0$ such that for all $0<\delta <\delta_0$, $E\in \Psi^0$ with
\begin{equation}
\label{eq:saphire2}\WF(E)\subset  K\cap \Ell(\bcD)\cap \{|\sigma (\bcN)|< \delta \langle \xi\rangle^{m_1}\},
\end{equation}
and $B'\in \Psi^{0}$ with 
$$
\WF(E)\cap\WF(I-B')=\emptyset,\qquad \WF(B')\subset \Ell(\bcD)\cap \{|\sigma (\bcN)|< \delta \langle \xi\rangle^{m_1}\},
$$
there is $h_0>0$ small enough such that have for $0<h<h_0$ and $0<\e<1$,
\begin{equation}
\label{e:gobble}
\begin{aligned}
&\|E hD_{x_1}u\|_{H_h^{\ell+m_1}(\Gamma_i)}+\|Eu\|_{H_h^{\ell+m_0}(\Gamma_i)}\\
&\leq  C \e\|B'hD_{x_1}\|_{H_h^{\ell+m_1}(\Gamma_i)}+ C\e^{-1}\|B'g\|_{H_h^\ell(\Gamma_i)}  \\
&\qquad +C\|u\|_{H_h^{\ell+\frac{m_1+m_0+1}{2}}(M)}+C\e^{-1}\Big(\|f\|_{H_h^{\ell+\frac{m_1+m_0-1}{2}}(M)}+\|f\|_{L^2(M)}\Big)+C\e \|u\|_{H_h^1(M)}\\&\qquad+O\Big(h^\infty\big(\|u\|_{H_h^{-N}(\Gamma_i)}+\|hD_{x_1}u\|_{H_h^{-N}(\Gamma_i)}+\|g\|_{H_h^{-N}(\Gamma_i)}\big)\Big).
\end{aligned}
\end{equation}
Moreover, if $m_1+1\leq m_0$, then~\eqref{e:gobble} holds with $K=T^{*}\Gamma_i$.
\end{lem}
\begin{proof}
Throughout the proof, we take $b_1(x',hD_{x'})$ self-adjoint with $b_1\in \Psi^{2(k+m_0)}$  if $m_1+1\leq m_0$ and $b_1\in \Psi^{\comp}$ otherwise. We assume that
$$
\WF(E)\subset \Ell(b_1(x',hD_{x'}))\subset \WF(b_1(x',hD_{x'}))\subset  \Ell(\bcD)\cap \{|\sigma (\bcN)|< \delta \langle \xi\rangle^{m_1}\}.
$$
Let $B_1$ and $B_0$ be defined by \eqref{eq:choiceofb}.

Since $\bcD$ is elliptic on $\WF (B')$, there exists $\bcD^{-1}\in \Psi^{-m_0}$ a microlocal inverse for $\bcD$ on $\WF(B')$; that is, for any $B$ with $\WF(B)\subset \{B'\equiv \Id\}$,
\begin{equation}
\label{eq:rewrite}
Bu= -B\bcD^{-1}(\bcN hD_{x_1}B'u-B'g)+O(h^\infty)_{\Psi^{-\infty}}g+O(h^\infty)_{\Psi^{-\infty}}u+O(h^\infty)_{\Psi^{-\infty}}hD_{x_1}u
\end{equation}
Arguing as in the proof of Lemma \ref{lem:elephant1}, we obtain the analogue of (\ref{eq:aardvark2}) with $B_1\in \Psi^{2(k+m_1)}(\Gamma_i)$, namely
\begin{equation*}
\begin{aligned}
&\Big|\big\langle B_1R\bcD^{-1}\bcN hD_{x_1}u,\bcD^{-1}\bcN hD_{x_1}u\big\rangle_{L^2(\Gamma_i)}+\big\langle B_1 hD_{x_1}u,hD_{x_1}u\big\rangle_{L^2(\Gamma_i)} \Big|\\
&\leq C (\e+h)\|B'hD_{x_1}\|_{H_h^{\ell+m_1}(\Gamma_i)}^2+ C(\e^{-1}+1)\|B'g\|_{H_h^\ell(\Gamma_i)}^2  \\
&\qquad +C\|u\|_{H_h^{\ell+\frac{m_1+m_0+1}{2}}(M)}^2+C\e^{-1}\Big(\|f\|_{H_h^{\ell+\frac{m_1+m_0-1}{2}}(M)}^2+\|f\|_{L^2(M)}^2\Big)+C\e \|u\|^2_{H_h^1(M)}\\&\qquad+O\Big(h^\infty\big(\|u\|^2_{H_h^{-N}(\Gamma_i)}+\|hD_{x_1}u\|^2_{H_h^{-N}(\Gamma_i)}+\|g\|^2_{H_h^{-N}(\Gamma_i)}\big)\Big).\end{aligned}
\end{equation*}
If $m_0<m_1+1$, we assume that $b_1\in S^{\comp}$. Therefore, for all $(m_0,m_1)$
$$
\big(\bcD^{-1}\bcN\big)^*B_1(1-R)\bcD^{-1}\bcN+B_1\in \Psi^{2(\ell+m_1)}.
$$
for our choice of $B_1$. 
Now, any $K\subset T^{*}\Gamma_i$ compact, there is $\delta_0>0$ small enough such that 
$$
\inf\Bigg\{\Big|1+\big|\sigma(\bcD^{-1}\bcN)(x',\xi')\big|^2R(x',\xi')\Big|\,\text{ where } \, |\sigma(\bcN)(x',\xi')|\leq \delta_0\langle \xi'\rangle^{m_1},\, (x',\xi')\in K\Bigg\}\geq c_K>0
$$ 
Moreover, if $m_0\leq m_1+1$, then there is $\delta_0>0$ small enough such that
$$
\inf\Bigg\{\Big|1+\big|\sigma(\bcD^{-1}\bcN)(x',\xi')\big|^2R(x',\xi')\Big|\,\text{ where } \, |\sigma(\bcN)(x',\xi')|\leq \delta_0\langle \xi'\rangle^{m_1},\, (x',\xi')\in T^{*}\Gamma_i\Bigg\}\geq c>0.
$$ 

Therefore, choosing $B_1$ with non-negative symbol such that $B_1$ is elliptic on $\WF(E)$, we have
$$
\Re \sigma\Big(\big(\bcD^{-1}\bcN\big)^*B_1R\bcD^{-1}\bcN+B_1\Big)(x',\xi')\geq c,\qquad (x',\xi')\in \WF(E).
$$
In particular, 
\begin{align*}
&\|EhD_{x_1}u\|_{H_h^{\ell+m_1}(\Gamma_i)}^2\leq C\Big\langle \big((\bcD^{-1}\bcN)^*B_1R\bcD^{-1}\bcN+B_1\big)hD_{x_1}u,hD_{x_1}u\Big\rangle_{L^2(\Gamma_i)}\\
&\hspace{4cm}
+Ch\|B'u \|_{H_h^{\ell+m_1-\frac{1}{2}}}^2+O(h^\infty)\|hD_{x_1}u\|_{H_h^{-N}(\Gamma_i)}^2
\end{align*}
Therefore, 
\begin{equation*}
\begin{aligned}
\|E hD_{x_1}u\|^2_{H_h^{\ell+m_1}(\Gamma_i)}&\leq C (\e+h)\|B'hD_{x_1}\|_{H_h^{\ell+m_1}(\Gamma_i)}^2+ C(\e^{-1}+1)\|B'g\|_{H_h^\ell(\Gamma_i)}^2  \\
&\qquad +C\|u\|_{H_h^{\ell+\frac{m_1+m_0+1}{2}}(M)}^2+C\e^{-1}\Big(\|f\|_{H_h^{\ell+\frac{m_1+m_0-1}{2}}(M)}^2+\|f\|_{L^2(M)}^2\big)
\\&
\qquad+C\e \|u\|^2_{H_h^1(M)}+O\Big(h^\infty\big(\|u\|^2_{H_h^{-N}(\Gamma_i)}+\|hD_{x_1}u\|^2_{H_h^{-N}(\Gamma_i)}+\|g\|^2_{H_h^{-N}(\Gamma_i)}\big)\Big).\end{aligned}
\end{equation*}

Then, using~(\ref{eq:rewrite}) again the second claim follows.
\end{proof}

\subsection{Proof of Theorem \ref{th:higherOrderBound}} \label{sub:proofhOB}

Throughout this section we assume that~\eqref{eq:conditions} holds. In particular, the union of the elliptic sets for $A_{0,i}$ and $A_{1,i}$ covers $T^{*}\Gamma_i$ and $A_{0,i}$ is elliptic on $S^*\Gamma_i$. 

\begin{proof} 

We start by briefly considering the conditions~\eqref{eq:conditions2a}--~\eqref{eq:conditions2d} separately. Suppose first that~\eqref{eq:conditions2a} holds. Then, fixing $\delta_0>0$, such that~Lemmas \ref{lem:elephant1} and~\ref{lem:elephant2} with $K=T^{*}\Gamma_i$ hold, there exist  $E_0\in \Psi^{0}$ satisfying~\eqref{eq:saphire} and $E_1\in \Psi^{0}$ satisfying~\eqref{eq:saphire2} (both with $K=T^{*}\Gamma_i$) such that 
$$
\{\sigma(\bcD)=0\}\subset \Ell(E_0),\qquad \{\sigma(\bcN)=0\}\subset \Ell(E_1), \qquad T^{*}\Gamma_i\subset \Ell(E_0)\cup \Ell(E_1).
$$

Next, if~\eqref{eq:conditions2c} holds, there exists $K_2\Subset T^{*}\Gamma_i$ such that 
$$
K_2\cup \Ell(\bcN)\supset T^{*}\Gamma_i.
$$
Fixing $\delta_0>0$, such that~Lemma \ref{lem:elephant1} holds with $K=T^{*}\Gamma_i$ and~\ref{lem:elephant2} holds with $K=K_2$, there exist  $E_0\in \Psi^{0}$ satisfying~(\ref{eq:saphire}) with $K=T^{*}\Gamma_i$ and $E_1\in \Psi^{\comp}$ satisfying~(\ref{eq:saphire2}) with $K=K_2$ such that 
$$
\{\sigma(\bcD)=0\}\subset \Ell(E_0),\qquad \{\sigma(\bcN)=0\}\subset \Ell(E_1),\qquad T^{*}\Gamma_i\subset \Ell(E_0)\cup \Ell(E_1).
$$

Finally, if~\eqref{eq:conditions2d} holds, there exists $K_3\subset T^{*}\Gamma_i$ such that 
$$
K_3\cup \Ell(\bcD)\supset T^{*}\Gamma_i.
$$
Fixing $\delta_0>0$, such that~Lemma \ref{lem:elephant1} holds with $K=K_3$ and~\ref{lem:elephant2} holds with $K=T^{*}\Gamma_i$, there exist  $E_0\in \Psi^{\comp}$ satisfying~(\ref{eq:saphire}) with $K=K_3$ and $E_1\in \Psi^{0}$ satisfying~(\ref{eq:saphire2}) with $K=T^{*}\Gamma_i$ such that 
$$
\{\sigma(\bcD)=0\}\subset \Ell(E_0),\qquad \{\sigma(\bcN)=0\}\subset \Ell(E_1),\qquad T^{*}\Gamma_i\subset \Ell(E_0)\cup \Ell(E_1).
$$

In particular, in all cases, there exist $h_0>0$, $E_0,E_1,E_2\in \Psi^0$ such that for $0<h<h_0$, the estimates of Lemma~\ref{lem:elephant1} hold for $E_0^*E_0$, those for~\eqref{lem:elephant2} hold for $E_1^*E_1$, and those of Lemma~\ref{lem:ellipticEst} hold for $E_2^*E_2$ such that 
$$
T^{*}\Gamma_i\subset \Ell(E_0)\cup \Ell(E_1)\cup \Ell(E_2).
$$

Therefore, by Lemma~\ref{lem:elephant1}
\begin{align*}
\|E_0^*E_0u\|_{H_h^{\ell+m_0}}&+\|E_0^*E_0hD_{x_1}u\|_{H_h^{\ell+m_1}}\\
& \leq C (\e+h)\|u\|_{H_h^{\ell+m_0}(\Gamma_i)}+ C(\e^{-1}+1)\|g\|_{H_h^\ell(\Gamma_i)}  \\
&\qquad+C\|u\|_{H_h^{\ell+\frac{m_1+m_0+1}{2}}(M)}+C\e^{-1}\Big(\|f\|_{H_h^{\ell+\frac{m_1+m_0-1}{2}}(M)}+\|f\|_{L^2(M)}\Big)+C\e \|u\|_{H_h^1(M)}\\
&\qquad+O\Big(h^\infty\big(\|u\|_{H_h^{-N}(\Gamma_i)}+\|hD_{x_1}u\|_{H_h^{-N}(\Gamma_i)}+\|g\|_{H_h^{-N}(\Gamma_i)}\big)\Big).
\end{align*}
Similarly, by Lemma~\ref{lem:elephant2}
\begin{align*}
\|E_1^*E_1u\|_{H_h^{\ell+m_0}}&+\|E_1^*E_1hD_{x_1}u\|_{H_h^{\ell+m_1}}\\
&\leq   C \e\|hD_{x_1}\|_{H_h^{\ell+m_1}(\Gamma_i)}+ C\e^{-1}\|g\|_{H_h^\ell(\Gamma_i)}  \\
&\qquad +C\|u\|_{H_h^{\ell+\frac{m_1+m_0+1}{2}}(M)}+C\e^{-1}\Big(\|f\|_{H_h^{\ell+\frac{m_1+m_0-1}{2}}(M)}+\|f\|_{L^2(M)}\Big)+C\e \|u\|_{H_h^1(M)}\\&\qquad+O\Big(h^\infty\big(\|u\|_{H_h^{-N}(\Gamma_i)}+\|hD_{x_1}u\|_{H_h^{-N}(\Gamma_i)}+\|g\|_{H_h^{-N}(\Gamma_i)}\big)\Big).
\end{align*}

Finally, using Lemma~\ref{lem:ellipticEst},
\begin{align*}
&\|E_2^*E_2u\|_{H_h^{\ell+m_0}(\Gamma_i)}+\|E_2^*E_2hD_{x_1}u\|_{H_h^{\ell+m_1}(\Gamma_i)}\\
&\qquad\leq C\Big(\|u\|_{H_h^{\frac{2\ell+m_1+m_0+1}{2}}(M)}+\|u\|_{L^2(M)}+\|f\|_{H_h^{\frac{2\ell+m_1+m_0-1}{2}}(M)}+\|f\|_{L^2(M)}\Big) \\
&\qquad\qquad+\e \big(\|u\|_{H_h^{\ell+m_0}(\Gamma_i)}+ \|hD_{x_1}u\|_{H_h^{\ell+m_1}(\Gamma_i)}\big)+C\e^{-1}\|g_i\|_{H_h^{\ell}(\Gamma_i)}\\
&\qquad\qquad+O\Big(h^\infty\big(\|u\|_{H_h^{-N}(\Gamma_i)}+\|hD_{x_1}u\|_{H_h^{-N}(\Gamma_i)}+\|g\|_{H_h^{-N}(\Gamma_i)}\big)\Big).
\end{align*}

Since 
$$
T^{*}\Gamma_i\subset\Ell(E_0^*E_0+E_1^*E_1+E_2^*E_2),
$$
we have all together 
\begin{equation}
\label{eq:nearlyFinal}
\begin{aligned}
&\|u\|_{H_h^{\ell+m_0}(\Gamma_i)}+\|hD_{x_1}u\|_{H_h^{\ell+m_1}(\Gamma_i)}\\
&\qquad\leq C\Big(\|u\|_{H_h^{\ell+\frac{m_1+m_0+1}{2}}(M)}+\|u\|_{L^2(M)}+\e\|u\|_{H_h^1(M)}+ \|f\|_{H_h^{\frac{k+m_1+m_0-1}{2}}(M)}+\|f\|_{L^2(M)}\Big) \\
&\qquad\qquad+\e \Big(\|u\|_{H_h^{\ell+m_0}(\Gamma_i)}+ \|hD_{x_1}u\|_{H_h^{\ell+m_1}(\Gamma_i)}\Big)+C\e^{-1}\|g_i\|_{H_h^{\ell}(\Gamma_i)}.
\end{aligned}
\end{equation}

Finally, observe that 
\beqs
\Re \langle -h^2\Delta u,u\rangle_{L^2(M)}=\|h\nabla u\|_{L^2(M)}^2+h\sum_i\Re \langle h\partial_\nu u,u\rangle_{L^2(\Gamma_i)}.
\eeqs
Letting $\psi\in \Psi^{\comp}$ with $\bcD$ elliptic on $\WF(\psi)$ and $\bcN$ elliptic on $\supp \WF(\Id-\psi)$, we have 
\begin{align*}
|\Re \langle h\partial_\nu u,u\rangle_{L^2(\Gamma_i)}|&=\Big|\Re i\Big( \big\langle  hD_{\nu} u,\psi u\big\rangle_{L^2(\Gamma_i)}+ \big\langle (\Id- \psi) hD_{\nu} u,u\big\rangle_{L^2(\Gamma_i)}\Big)\Big|\\
&=\Big|\Re i\Big( \big\langle  hD_{\nu} u,-\psi \bcD^{-1}(\bcN hD_n u-g)\big\rangle_{L^2(\Gamma_i)}+ \big\langle (\Id- \psi) \bcN^{-1}(g+\bcD u),u\big\rangle_{L^2(\Gamma_i)}\Big)\Big|\\
&\qquad\qquad+O(h^{\infty})\Big(\|u\|_{H_h^{-N}(\Gamma_i)}^2+\|hD_{x_1}u\|_{H_h^{-N}(\Gamma_i)}^2\Big)\\
&\leq Ch\|hD_{\nu}u\|_{H_h^{-N}(M)}^2+h^{-1}\|g\|_{H_h^{-m_0-s}(\Gamma_i)}^2+h\|u\|_{H_h^{\frac{m_0-m_1-1}{2}}(\Gamma_i)}^2+h\|u\|_{H_h^s(\Gamma_i)}^2.
\end{align*}
Therefore, for any $s$,
\begin{align}\nonumber
\|u\|^2_{H_h^1(M)}\leq& \frac{1}{2}h^2\|f\|^2_{L^2(M)}+\frac{5}{2}\|u\|^2_{L^2(M)}
\\&+C\bigg(\sum_ih^2\|hD_n u\|^2_{H_h^{-N}(\Gamma_i)}+h^2\|u\|^2_{H_h^{\max\left(s,\frac{m_0-m_1-1}{2}\right)}(\Gamma_i)}+\|g_i\|_{H_h^{-m_{1,i}-s}(\Gamma_i)}^2\bigg).\label{eq:analogy1}
\end{align}
Using this in~\eqref{eq:nearlyFinal} and taking 
\beq\label{eq:elli}
-\frac{m_{0,i}+m_{1,i}}{2}\leq \ell_i\leq  \frac{1}{2}-\frac{m_{0,i}+m_{1,i}}{2}, \quad s_i=-\ell_{i}-m_{1,i},
\eeq
 we obtain
\begin{equation*}
\begin{aligned}
&\sum_i\|u\|_{H_h^{\ell_i+m_{0,i}}(\Gamma_i)}+\|hD_{x_1}u\|_{H_h^{\ell_i+m_{1,i}}(\Gamma_i)}\\
&\leq C\|u\|_{L^2(M)}+C(\e^{-1}+\e h)\|f\|_{L^2(M)} \\
&\quad+\sum_i \e^{-1}\|g\|_{H_h^{\ell_i}(\Gamma_i)}+C\e\Big(\|u\|_{H_h^{\ell+m_{0,i}}(\Gamma_i)}+\|hD_{x_1}u\|_{H_h^{\ell+m_{1,i}}(\Gamma_i)}\Big)\\
&\qquad +\sum_i\Big(h\|hD_n u\|_{H_h^{-N}(\Gamma_i)}+h\|u\|_{H_h^{\max\left(-m_{1,i}-\ell_i,\frac{m_{0,i}-m_{1,i}-1}{2}\right)}(\Gamma_i)}\Big).
\end{aligned}
\end{equation*}
Shrinking $\e$ such that $C\e<1/2$ and taking $h_0$ small enough such that $Ch_0\leq \frac{1}{2}$, 
the proof is complete since the inequality \eqref{eq:ellineq} (i.e., the first inequality in \eqref{eq:elli})
 implies that the terms on the right can be absorbed into the left.

The final inequality in Theorem~\ref{th:higherOrderBound} follows from combining the result of Lemma~\ref{lem:basicbound} (with $\ell=-s$) with \eqref{eq:trace2}.
\end{proof}


\appendix
\section{Semiclassical pseudodifferential operators and notation}\label{sec:appendix}

We review the notation and definitions for semiclassical pseudodifferential operators on $\mathbb{R}^d$ and refer the reader to~\cite[Appendix E]{DyZw:19},~\cite[Chapter 14]{Zworski_semi} for details of how to adapt these definitions to manifolds.

Before we introduce these objects, we recall the notion of \emph{semiclassical Sobolev spaces} $H_h^s$; these are the standard Sobolev spaces $H^s$ with a norm weighted with $h$. We say that $u\in H_h^s(\mathbb{R}^d)$ if 
$$
 \|\langle \xi\rangle ^s \mathcal{F}_h(u)(\xi)\|_{L^2}<\infty, \quad\text{ where } \quad\langle \xi\rangle:=(1+|\xi|^2)^{\frac{1}{2}}
\quad\tand\quad
\mathcal{F}_h(u)(\xi):=\int_{\mathbb{R}^d} e^{-\frac{i}{h}\langle y,\xi\rangle}u(y)\,dy
$$
is the \emph{semiclassical Fourier transform}.

We next introduce the notion of symbols.  We say that $a\in C^\infty(T^*\mathbb{R}^d)$ is a symbol of order $m$ if 
$$
|\partial_x^\alpha \partial_\xi^\beta a(x,\xi)|\leq C_{\alpha\beta}\langle \xi\rangle^m,
$$
and write $a\in S^m(T^*\mathbb{R}^d)$. 
Throughout this section we fix $\chi_0\in C_c^\infty(\mathbb{R}))$ to be identically 1 near 0.  We then say that an operator $A:C_c^\infty(\mathbb{R}^d)\to \mathcal{D}'(\mathbb{R}^d)$ is a semiclassical pseudodifferential operator of order $m$, and write $A\in \Psi^m(\mathbb{R}^d)$, if $A$ can be written as
\begin{equation}
\label{e:basicPseudo}
Au(x)=\frac{1}{(2\pi h)^d}\int_{\Rea^d}\int_{\Rea^d} e^{\frac{i}{h}\langle x-y,\xi\rangle}a(x,\xi)\chi_0(|x-y|)u(y)dyd\xi +E
\end{equation}
where $a\in S^m(T^*\mathbb{R}^d)$ and $E=O(h^\infty)_{\Psi^{-\infty}}$, where an operator $E=O(h^\infty)_{\Psi^{-\infty}}$ if for all $N>0$ there exists $C_N>0$ such that
$$
\|E\|_{H_h^{-N}(\mathbb{R}^d)\to H_h^N(\mathbb{R}^d)}\leq C_Nh^N. 
$$
We also define 
$$
\Psi^{-\infty}:=\bigcap_m \Psi^m,\qquad S^{-\infty}:=\bigcap_m S^m,\qquad  \Psi^\infty:=\bigcup_m \Psi^m, \qquad S^\infty:=\bigcup_m S^m.
$$
 We say that $a\in S^{\comp}$ if $a\in S^{-\infty}$ and $a$ is compactly supported, and we say that $A\in \Psi^{\comp}$ if $A\in \Psi^{-\infty}$ and can be written in the form~\eqref{e:basicPseudo} with $a\in S^{\comp}$. We use the notation $a(x,hD_x)$ for the operator $A$ in~\eqref{e:basicPseudo} with $E=0$. 

We recall that there exists a map 
$$
\sigma_m:\Psi^m \to S^m/hS^{m-1}
$$
called the \emph{principal symbol map} and such that the sequence 
$$
0\to hS^{m-1}\overset{\operatorname{Op}_h}{\rightarrow} \Psi^{m}\overset{\sigma}{\rightarrow} S^m/hS^{m-1}\to 0
$$
is exact where $\operatorname{Op}_h(a)=a(x,hD)$. Moreover, 
\beq\label{eq:symbol}
\sigma(AB)=\sigma(A)\sigma(B),\qquad \sigma(A^*)=\overline{\sigma}(A),\qquad \sigma(-ih^{-1}[A,B])=\{\sigma(A),\sigma(B)\}
\eeq
where $\{\cdot,\cdot\}$ denotes the Poisson bracket; see \cite[Proposition E.17]{DyZw:19}.

\subsection{Wavefront sets and elliptic sets}

To introduce a notion of wavefront set that respects both decay in $h$ as well as smoothing properties of pseudodifferential operators, we introduce the set 
$$
\overline{T^*\mathbb{R}^d}:=T^*\mathbb{R}^d\sqcup ( \mathbb{R}^d\times S^{d-1})
$$
where $\sqcup$ denotes disjoint union and we view $\mathbb{R}^d\times S^{d-1}$ as the `sphere at infinity' in each cotangent fiber (see also~\cite[\S E.1.3]{DyZw:19} for a more systematic approach where $\overline{T^*\mathbb{R}^d}$ is introduced as the fiber-radial compactification of $T^*\mathbb{R}^d$). We endow $\overline{T^*\mathbb{R}^d}$ with the usual topology near points $(x_0,\xi_0)\in T^*\mathbb{R}^d$ and define a system of neighbourhoods of a point $(x_0,\xi_0)\in \mathbb{R}^d\times S^{d-1}$ to be
\begin{align*}
U_\e:=&\Big\{ (x,\xi)\in T^*\mathbb{R}^d\,\big|\, |x-x_0|<\e, |\xi|>\e^{-1}, \big|\tfrac{\xi}{\langle \xi\rangle }-\xi_0\big|<\e\Big\}\\
&\qquad\quad \sqcup \big\{ (x,\xi)\in \mathbb{R}^d\times S^{d-1}\,:\, |x-x_0|<\e.,\, |\xi-\xi_0|<\e\big\}.
\end{align*}

We now say that a point $(x_0,\xi_0)\in \overline{T^*\mathbb{R}^d}$ is not in the wavefront set of an operator $A\in \Psi^m$, and write $(x_0,\xi_0)\notin \WF(A)$, if there exists a neighbourhood $U$ of $(x_0,\xi_0)$ such that $A$ can be written as in~\eqref{e:basicPseudo} with 
$$
\sup_{(x,\xi)\in U} | \partial^\alpha_x \partial_\xi^\beta a(x,\xi)\langle \xi\rangle^N|\leq C_{\alpha \beta N} h^N.
$$

We define the elliptic set of a pseudodifferential operator $A\in \Psi^m$ as follows. We say that $(x_0,\xi_0)\in \overline{T^*\mathbb{R}^d}$ is in the elliptic set of $A$, and write $(x_0,\xi_0)\in \Ell(A)$, if there exists a neighbourhood $U$ of $(x_0,\xi_0)$ such that $A$ can be written as in~\eqref{e:basicPseudo} with 
$$
\inf_{(x,\xi)\in U} |a(x,\xi)\langle \xi\rangle^{-m}|\geq c >0.
$$

Next, we define the wavefront of a family of distributions $u_h$ depending on $h$. We say that $u_h$ is \emph{tempered} if for all $\chi\in C_c^\infty(\mathbb{R}^d)$ there exists $N>0$ such that
$$
\|\chi u\|_{H_h^{-N}}<\infty. 
$$
For a tempered family of functions, $u_h$  we say that $(x_0,\xi_0)\in \overline{T^*\mathbb{R}^d}$ is not in the wavefront set of $u_h$ and write $(x_0,\xi_0)\notin \WF(u_h)$ if there exists $A\in \Psi^0$ with $(x_0,\xi_0)\in \Ell(A)$ such that for all $N$ there is $C_N>0$ such that 
$$
\|Au_h\|_{H_h^N}\leq C_Nh^N.
$$

\subsection{Bounds for pseudodifferential operators}
We next review some bounds for pseudodifferential operators acting on Sobolev spaces.
\begin{lem}(\cite[Propositions E.19 and E.24]{DyZw:19}~\cite[Theorem 8.10]{Zworski_semi})
Suppose that $A\in \Psi^m$. Then
$$
\|Au\|_{H_h^{s}}\leq C\|u\|_{H_h^{s+m}}.
$$
Moreover, if $A=a(x,hD)\in \Psi^0$, then there exists $C>0$ such that 
$$
\|A\|_{L^2\to L^2}\leq \sup |a|+ Ch^{\frac{1}{2}}.
$$
\end{lem}

Finally, we recall the elliptic parametrix construction (see e.g.~\cite[Proposition E.32]{DyZw:19}).
\begin{lem}
\label{l:ellip}
Suppose that $A\in \Psi^{m_1}$ and $B\in \Psi^{m_2}$ with $\WF(A)\subset \Ell(B)$. Then there exist $E_1,E_2\in \Psi^{m_1-m_2}$ such that 
$$
A=E_1B+O(h^\infty)_{\Psi^{-\infty}},\qquad A=BE_2+O(h^\infty)_{\Psi^{-\infty}}.
$$
\end{lem}

\subsection{Tangential pseudodifferential operators}\label{sec:tangential}

It will sometimes be convenient to have families of pseudodifferential operators depending on one of the position variables. In this case, as in  \S\ref{subsec:geo}, we write $x=(x_1,x')\in \mathbb{R}^d$ and $\xi=(\xi_1,\xi')$ for the corresponding dual variables.
We then consider families $A\in C_c^\infty( I_{x_1};\Psi^m(\mathbb{R}^{d-1}))$, that is, smooth functions in $x_1$ valued in pseudodifferential operators of order $m$ and write $A=a(x,hD_{x'})$ for some $a\in C_c^\infty(I_{x_1};S^m(\mathbb{R}^{d-1})).$

\bibliographystyle{amsalpha}
\bibliography{biblio_GLS}

\end{document}